\documentclass{amsart}

\usepackage{url}
\usepackage{amsmath}
\usepackage{amsfonts}
\usepackage{amssymb}
\usepackage{xcolor}
\usepackage{graphicx}
\usepackage{float}
\usepackage{enumerate}
\usepackage{multirow}
\newcommand{\Second}{\textup{I}\!\textup{I}}
\newcommand{\norm}[1]{\left\lVert#1\right\rVert}
\newcommand{\abs}[1]{\left |#1\right |}
\usepackage{mathtools}
\usepackage[export]{adjustbox}
\usepackage{tabulary}
\usepackage{booktabs}
\usepackage{algorithmic}

\newtheorem{theorem}{Theorem}
\newtheorem{corollary}{Corollary}
\newtheorem{assumption}{Assumption}
\newtheorem{definition}{Definition}
\newtheorem{lemma}{Lemma}
\newtheorem{proposition}{Proposition}
\newtheorem{remark}{Remark}

\begin{document}

\title[Roseland]{Scalability and robustness of spectral embedding:  landmark diffusion is all you need}

\author{Chao Shen}
\address{Department of Mathematics, Duke University, Durham, NC, USA}
\email{chao.shen@duke.edu}

\author{Hau-Tieng Wu}
\address{Department of Mathematics and Department of Statistical Science, Duke University, Durham, NC, USA; Mathematics Division, National Center for Theoretical Sciences, Taipei, Taiwan}
\email{hauwu@math.duke.edu}

\maketitle

\begin{abstract}
	{While spectral embedding is a widely applied dimension reduction technique in various fields, so far it is still challenging to make it {\em scalable} to handle ``big data''. On the other hand, the {\em robustness} property is less explored and there exists only limited theoretical results.
	Motivated by the need of handling such data, recently we proposed a novel spectral embedding algorithm, which we coined {\em Robust and Scalable Embedding via Landmark Diffusion} (ROSELAND). In short, we measure the affinity between two points via a set of landmarks, which is composed of a small number of points, and ``diffuse'' on the dataset via the landmark set to achieve a spectral embedding. 
	Roseland can be viewed as a generalization of the commonly applied spectral embedding algorithm, the {\em diffusion map} (DM), in the sense that it shares various properties of DM.
	In this paper, we show that Roseland is not only numerically scalable, but also preserves the geometric properties via its diffusion nature under the manifold setup; that is, we theoretically explore the asymptotic behavior of Roseland under the manifold setup, including handling the U-statistics-like quantities, and provide a $L^\infty$ spectral convergence with a rate. Moreover, we offer a high dimensional noise analysis and show that Roseland is robust to noise. 
	We also compare Roseland with other existing algorithms with numerical simulations.}
  Graph Laplacian, Diffusion Maps, Nystr\"om, Landmark, Scalability, Robustness, Roseland\\
  2000 Math Subject Classification: 53Z50, 65D18, 62H99, 68U01, 58Z05
\end{abstract}

\section{Introduction}\label{Sect:Introduction}

Unsupervised learning is arguably the holy grail in the field of artificial intelligence, and it is arguably that the more data we have, the better we can learn. So far there have been many unsupervised learning algorithms proposed and it is still an active studying field. 
In general, those algorithms share a common ground. The learner designer constructs an optimization framework that captures the intended properties of the learning process, and then designs an algorithm to solve the optimization problem.
Based on the nature of an algorithm, it can be roughly classified into two classes -- spectral or not. Spectral algorithms include ISOMAP \cite{tenenbaum2000global}, locally linear embedding (LLE) \cite{roweis2000nonlinear}, Hessian LLE \cite{Donoho_Grimes:2003}, eigenmap \cite{belkin2003laplacian}, diffusion map (DM) \cite{Coifman2006} and {its variations like with bi-stochastic kernel \cite{marshall2019manifold} and PHATE \cite{moon2019visualizing}}, vector diffusion map (VDM) \cite{singer2012vector}, to name but a few. 
Those algorithms have been widely applied to various scientific fields and various theoretical foundations have been established to support those algorithms in the past decades. Under the manifold setup, we have had a rich knowledge about the geometric and asymptotic behavior of some of those algorithms. For example, DM and VDM are both based on the diffusion process \cite{Coifman2006,singer2012vector}, and asymptotically they converge to the Laplace-Beltrami operator or connection Laplacian so that the spectral geometry theory can be applied; the LLE algorithm is not diffusion-based and the underlying kernel is asymmetric and depends on the geometry of the dataset \cite{Wu_Wu:2017}.
However, there are still various open problems remaining toward a better unsupervised learning framework. One critical challenge is how to make an algorithm scalable, which is always a numerical challenge and is critical in this ``big data era''. Another critical challenge is how to handle the inevitable noise, particularly when the noise is large and high dimensional. %How to handle these two challenges together is the focus of this paper.

Take the DM algorithm as an example. DM is based on the eigendecomposition of the graph Laplacian (GL) matrix, and the GL matrix is constructed by determining the {\em affinity} between each pair of points in the database. The algorithm has been shown to perform well when the database is ``tiny'', like in the order of $10^3\sim 10^4$. 
However, when the database gets larger, like in the order of $10^6$ or above, the algorithm needs a modification. Specifically, if the GL is dense, a full eigendecomposition is not feasible, and the k-nearest neighbor (kNN) scheme is an usually applied trick. However, kNN is in general not robust to noise. 
Specifically, when the dataset is noisy and the neighboring information is {\em not} provided, obtaining a reliable kNN information is challenged by the noise.
Another practical solution is subsampling the dataset, and then recovering the information of interest by the {\em Nystr\"om extension} \cite{czaja2017overview,belabbas2009landmark, williams2001using}. This approach is also called the {\em Nystr\"om low-rank approximation} \cite{chang2013asymptotic}, the {\em kernel extension method} \cite{fowlkes2004spectral}, or in general the {\em interpolative decomposition} \cite{martinsson2011randomized}. This approach has been widely applied, and it has various theoretical backups, for example  \cite{chang2013asymptotic}. While it works well for some missions, this approach is limited by the information loss during the subsampling process. To the best of our knowledge, how it performs when combined with the above-mentioned spectral based unsupervised learning algorithms is not yet well explored, not to mention how it impacts the algorithm under the manifold setup, or when the data is noisy and nonuniformly sampled. 
Yet another approach is directly speeding up the matrix decomposition by taking randomization into account \cite{halko2011finding}. For example, we can construct a thin matrix by taking a random subset of columns of the GL matrix and speed up the algorithm by taking the singular value decomposition (SVD) into account.  \cite{martinsson2011randomized} provides an efficient algorithm to approximate the SVD. It is also possible to randomly select a few points out of the $K$ nearest neighbors \cite{linderman2017randomized} to construct a sparse matrix. While this approach has been widely applied, to the best of our knowledge, we have limited knowledge about how it helps the spectral embedding algorithms, and how robust it is to the inevitable noise.  

In this paper, we consider a novel algorithm proposed in \cite{lin2020robust} that is both robust and scalable. 
The algorithm is intuitive and can be summarized in three steps. First, we find a ``small'' subset of points from the dataset, either randomly or by design, or collect a separate clean point cloud of small size. We call this set a {\em landmark set}. Second, we construct an affinity matrix recording the affinities between points in the whole dataset and the landmark set, and normalize it properly. This normalized affinity matrix is thin; that is, there are fewer columns than rows. Third, evaluate the singular vectors and singular values of the normalized affinity matrix, and embed the dataset with the singular vectors and singular values. 
Since this algorithm has an interpretable geometric meaning, and is directly related to the diffusion process, we coined the proposed algorithm the {\em RObust and Scalable Embedding via LANdmark Diffusion} (Roseland). An application of Roseland to analyze long-time physiological signal and liver transplant blood pressure analysis, which is also a motivation of designing this algorithm, can be found in the accompanying applied paper \cite{lin2020robust}.

\subsection{Related work -- scalability}

To better position our contribution, we summarize various related work in this section. 
The review paper by \cite{czaja2017overview} contains a comprehensive categorization of numerical acceleration techniques for nonlinear dimension reduction. 
The acceleration algorithms are roughly classified into three categories.
In the first category, the dataset  is compressed so that the relationship between pairs of points are well preserved. 
For example, the random projection can be applied, or   
the dataset is converted to a well-designed basis under the compressed sensing framework. This step can save us a little bit of time when computing pairwise distances used for neighborhood search. 
In the second category, we may try to accelerate the kNN search step. A brutal force method for computing the exact kNN graph requires $\Theta(n^{2})$. Many faster algorithms, deterministic or randomized, exact or approximate, have been developed in the past decades. 
In the third category, we may accelerate the eigen-decomposition step. 
For example, the kernel decomposition is approximated by classical iteration-based algorithms and the matrix decomposition can be evaluated by a randomized algorithm, where rigorous analyses have been developed  \cite{rokhlin2009randomized}. However, it is indicated in \cite{czaja2017overview} that the error bounds are usually pessimistic when compared with results of numerical experiments. A summary with citations of methods in each category can be found in \cite{czaja2017overview}.

We now review various algorithms that are directly related to our work.
The closest algorithm to Roseland is the one introduced  in \cite{haddad2014texture} to handle the texture separation problem. We call this algorithm {\em HKC}, which stands for initials of three authors in \cite{haddad2014texture}. The authors first convert an image into a collection of small patches, and choose a collection of specific patterns of interest as the {\em reference set}. Then one can construct an affinity matrix associated with the set of patches, where the affinity between patches is based on the landmark set. However, the normalization in HKC is different from Roseland, and this difference turns out to be significant. Moreover, it is not clear how HKC performs under the manifold setup. HKC can be classified as the third category.
Another directly related algorithm is the common Nystr\"{o}m extension \cite{belabbas2009landmark, fowlkes2004spectral, williams2001using}. We run eigen-decomposition on a small subset of the whole database, and then extend the eigenvectors to the whole dataset. This Nystr\"om extension approach can be classified as the third category. 
{There are various extensions or refinements of the Nystr\"om extension method as well as applications to spectral embeddings, for example, \cite{bermanis2013multiscale, de2004sparse, kushnir2012anisotropic, long2019landmark, platt2005fastmap}}. From the theoretical perspective, to the best of our knowledge,  \cite{chang2013asymptotic} is the only existing literature in this field. The authors analyzed the asymptotic spectral error bounds between the ground truth spectrum of the kernel function, full kernel matrix and the Nystr\"om low-rank approximation of the full kernel matrix.

Yet another and fundamental approach is speeding up the basic eigendecomposition or SVD themselves. But this direction is out of the scope of our work.
There are some closely related but different algorithms in the field, for example, CUR decomposition \cite{mahoney2009cur}, ``UBV'' decomposition \cite{cheng2005compression}, or some studies focusing solely on accelerating the spectral clustering. For the readers' convenience, we summarize them in Appendix \ref{Appendix:MoreRelatedWork}. 
To the best of our knowledge, none of the above-mentioned work, except \cite{chang2013asymptotic}, provides theoretical analysis to answer questions like what is the {\em asymptotic behavior} of the algorithm? However, even in \cite{chang2013asymptotic}, it is not clear how much geometric information is lost. In general, if we model the nonlinearity of the dataset by a manifold, we would like to know if we still have a convergence to the Laplace-Beltrami operator. Without this understanding, we cannot answer questions like how to choose, or even design, a landmark set so that the performance is guaranteed to some extent.

\subsection{Related work -- robustness}\label{section: related work: robustness}

Compared with the scalability issue, there exists less work focusing on the robustness issue. 
One intuitive idea is ``denoising'' the dataset before applying the spectral embedding algorithm. However, it is in general a challenging problem since we usually do not know the structure of the dataset, and extracting the structure of the dataset is the main target. Under the manifold setup, there have been several proposals to denoise the dataset {(or sometimes called the ``manifold denoising'' problem)}; for example, the ``reverse diffusion'' scheme \cite{NIPS2006_2997}, {fitting scheme \cite{fefferman2018fitting,Shira2020,Yariv2020}, or nonlocal median approach \cite{Alagapan_Shin_Frohlich_Wu:2018}.} While these approaches might work for practical problems, they might not be scalable.

A commonly applied acceleration tool is kNN (the second category of acceleration) \cite{czaja2017overview}. However, stability is an issue. Particularly, when the neighboring information is not provided, we apply the kNN to construct a sparse affinity graph on the dataset, which also help accelerate the algorithm. However, it is well known that finding neighbors via the kNN is noise-sensitive unless the pairwise distance is robust to noise. Usually, unless the data point has extra structure so that a robust metric can be applied, for example, in the image analysis \cite{cheng2009learning}, it is challenging to achieve a robust pairwise distance. Some authors propose to take the tangent plane structure to determine neighbors \cite{wang2005adaptive}; however, it is well known that in the high dimensional setup, the tangent space estimation via the principle component analysis is biased \cite{johnstone2007high} and the benefit might be limited in the practical setup. 
Another approach is taking the label into account to improve the stability of the kNN scheme \cite{ROHBAN20121363}, but this approach is out of the scope of this work.
When the edge information is given, in \cite{Steinerburger2016}, the author proposes to design a self-consistency Markov chain before the spectral embedding by modifying the non-lazy random walk via diffusion.

Besides the above efforts to alleviate the impact of noise, to the best of our knowledge, the robustness of the GL-based algorithms was first studied in \cite{ElKaroui:2010a} under the random matrix framework. The result was later extended to handle a large noise setup \cite{el2016graph}, where the authors suggest taking a complete non-lazy random walk to stabilize the spectral embedding methods. Obviously this approach is not scalable. In that paper, it is also shown that if we want to stabilize the kNN approach, a large number of nearest neighbors should be chosen.

\subsection{Our contribution}

We provide a series of theoretical support for Roseland. First, we provide the spectral convergence of eigenvalues/eigenvectors of Roseland to the eigenvalues/eigenfunctions of the Laplace-Beltrami operator in the $L^\infty$ sense, extending the argument shown in \cite{DunsonWuWu2019}. As a result, the geometry recovery is guaranteed, including the geodesic distance. A convergence rate is also provided. We argue that the convergence rate is controlled by the size of the landmark set.

We also provide a robustness theory describing why Roseland is robust to noise by extending the arguments used in \cite{ElKaroui:2010a,el2016graph}. We also show its robustness to high dimensional (possibly colored and different from sample to sample) noise, either Gaussian or non-Gaussian. 
We thus conclude that Roseland is useful when we want to recover the intrinsic Laplace-Beltrami operator of the manifold from a noisy point cloud.

On the way toward the analysis, we observe a peculiar kernel behavior of Roseland; specifically, the ``effective kernel'' associated with Roseland is {\em not} fixed but adaptive to the chosen landmark set. This is different from the ordinary kernel method, where the applied kernel is universal.
Based on the theoretical results, we show a series of numerical simulations.
We also propose a {\em design-based} landmark set sampling scheme to handle the inevitable non-uniform sampling in the real world data.

\subsection{Organization of the paper}
In Section \ref{algs} we summarize Roseland and DM. In Sections \ref{Section:PointConvergence} and \ref{Section:RobustAnalysis}, we state our main theoretical results. In Section \ref{numerical}, we provide numerical results and analysis. In Section \ref{discuss}, discussion and conclusion are provided.

\section{Summary of the Roseland algorithm}\label{algs}

In this section, we assume we have a data set or point cloud $\mathcal{X}=\{x_{i}\}_{i=1}^{n}\subseteq\mathbb{R}^{q}$.
Take a set $\mathcal{Y}=\{y_{k}\}_{k=1}^{m}$, which might or might not be a subset of $\mathcal{X}$. We call $\mathcal{Y}$ the {\em landmark set}.
Fix a non-negative kernel function $K:\mathbb{R}_{\geq 0}\to \mathbb{R}_+$ with proper decay and regularity; for example, a Gaussian function $K(t)=\frac{1}{\sqrt{2\pi}}e^{-t^2/2}$.

\subsection{The Roseland algorithm}\label{roseland alg}
We now detail the Roseland algorithm \cite{lin2020robust}. 
Construct a {\em landmark-set affinity matrix} $W^{(\textup{r})}\in\mathbb{R}^{n\times m}$, which is defined as
\begin{equation}\label{landmark affinity matrix in Roseland}
W^{(\textup{r})}_{ik} = K_{\epsilon}(x_{i},y_{k}) := K\left(\frac{\|x_{i}-y_{k}\|_{\mathbb{R}^{q}}}{\sqrt{\epsilon}}\right)\,.
\end{equation}
That is, the $(i,k)$-th entry of $W^{(\textup{r})}$ is the affinity between $x_i\in \mathcal{X}$ and $y_k\in \mathcal{Y}$. Construct a diagonal matrix $D^{(\textup{r})}$ as
\begin{equation}
D^{(\textup{R})}_{ii}:=e_i^\top W^{(\textup{r})}(W^{(\textup{r})})^{\top}\boldsymbol{1}\,,
\end{equation}
where $\boldsymbol{1}$ is a $n\times 1$ vector with all entries $1$, and $e_i$ is the unit vector with $1$ in the $i$-th entry. With $W^{(\textup{r})}$ and $D^{(\textup{R})}$, we evaluate the SVD of $(D^{(\textup{R})})^{-1/2}W^{(\textup{r})}$:
\begin{equation}\label{DR-1/2Wr SVD in Roseland}
(D^{(\textup{R})})^{-1/2}W^{(\textup{r})}=U\Lambda V^\top, 
\end{equation}
where the singular values are $\sigma_1\geq \sigma_2\geq\ldots\geq\sigma_m\geq 0$. Set $\bar{U}:=(D^{(\textup{R})})^{-1/2}U$, and set $\bar{U}_m\in \mathbb{R}^{n\times m}$ to be a matrix consisting of the second to the $(m+1)$-th left singular vectors. Also set $L_m:=\texttt{diag}(\sigma^2_2,\ldots,\sigma^2_{m+1})$. The Roseland embedding is defined by  
\begin{equation}\label{embedding in Roseland}
\Phi^{(\textup{R})}_t:x_{i}\mapsto e_i^\top \bar{U}_{m}L_{m}^t\,,
\end{equation}
where $t>0$ is the chosen diffusion time. With the Roseland embedding, we have the associated Roseland diffusion distance (RDD)
\begin{equation}\label{DD in Roseland}
D^{(\textup{R})}_t(x_i,x_j):=\|\Phi^{(\textup{R})}_t(x_{i})-\Phi^{(\textup{R})}_t(x_j)\|_{\mathbb{R}^{m}}\,.
\end{equation} 

\begin{figure}[htb!]\centering
	\includegraphics[trim=10 120 10 100, clip, width=0.75\textwidth]{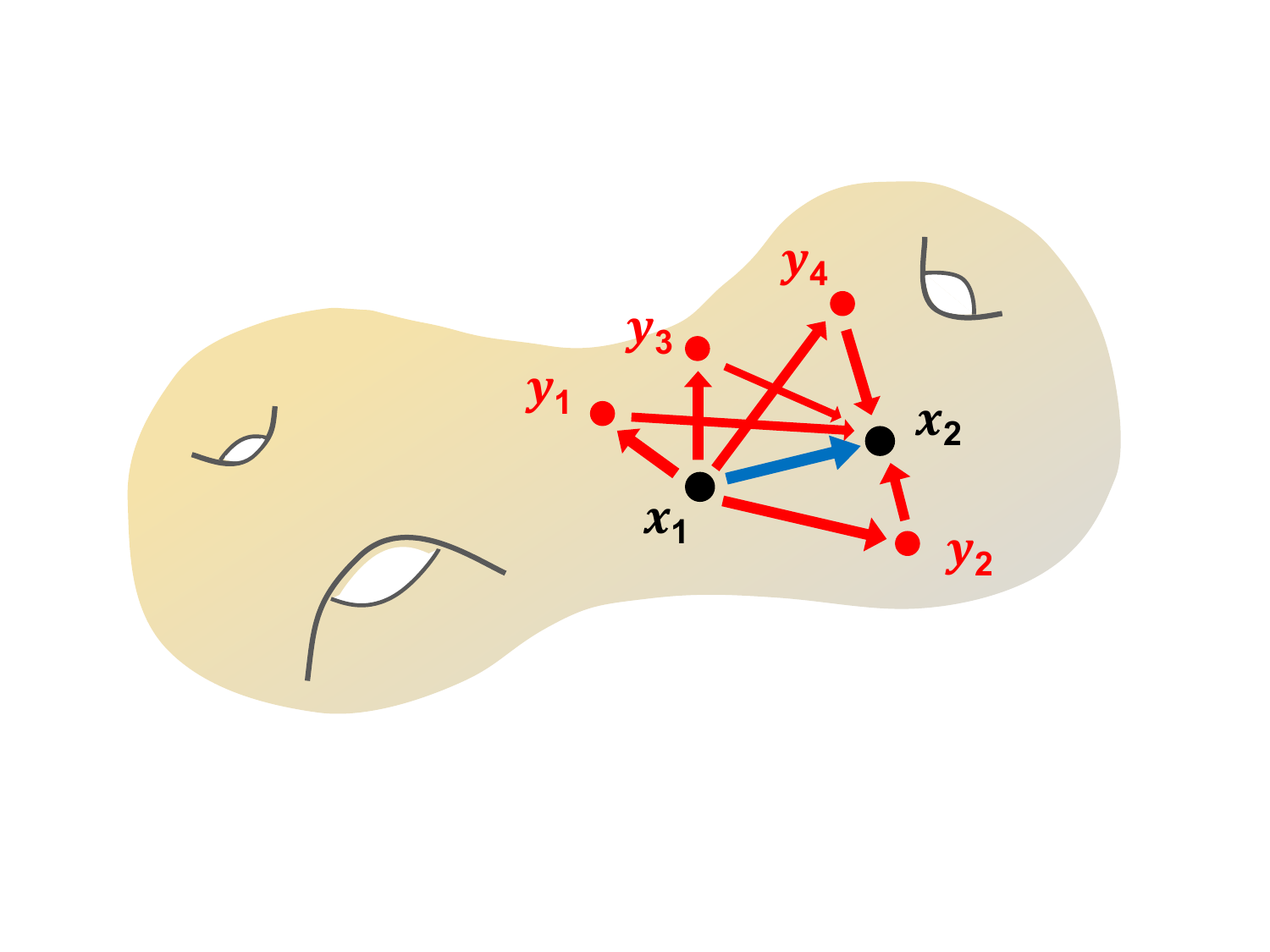}
	\caption{Main idea of Roseland: to measure the similarity between $x_{1}$ to $x_{2}$, instead of diffuse from $x_{1}$ to $x_{2}$ directly, we take a detour and first diffuse $x_{1}$ to the landmarks $y_{1},y_{2},y_{3}, y_4$, and then diffuse from the landmarks to $x_{2}$.\label{Figure:refDMillustration}}
\end{figure}

We now take a closer look at Roseland. Given $W^{(\textup{r})}\in \mathbb{R}^{n\times m}$, construct a new matrix     
\begin{equation}
W^{(\textup{R})} := W^{(\textup{r})}[W^{(\textup{r})}]^{\top}\in \mathbb{R}^{n\times n}\,,
\end{equation}
which can be viewed as a new affinity matrix. 
Indeed, since the kernel is chosen to be a non-negative function, $W^{(\textup{r})}$ is a matrix with non-negative entries, and so is $W^{(\textup{R})}$. Therefore, we can view $W^{(\textup{R})}$ as an affinity matrix defined on $\mathcal{X}$, where the affinity between $x_{i}$, $x_{j}$ via the landmark set $\mathcal{Y}$ is
\begin{equation}\label{definition WR in Roseland}
W^{(\textup{R})}_{ij}=\sum_{k=1}^{m}K_{\epsilon}(x_{i},y_{k})K_{\epsilon}(y_{k},x_{j}).
\end{equation}

We call it the {\em landmark-affinity matrix}. We mention that unlike the traditional affinity matrix, in general we cannot find a fixed kernel $\bar{K}$ and a bandwidth $\bar \epsilon$ so that $W^{(\textup{R})}_{ij}= \bar K_{\bar \epsilon}(x_i,x_j)$ for $i,j=1,\ldots,n$. Later, we will see how this ``new'' kernel function depends on $\mathcal{Y}$.
Thus, by construction, $D^{(\textup{R})}$ is nothing but the degree matrix associated with the landmark-affinity matrix $W^{(\textup{R})}$. Clearly, 
\begin{equation}\label{transition matrix definition of Roseland}
A^{(\textup{R})}:=(D^{(\textup{R})})^{-1}W^{(\textup{R})}
\end{equation} 
is a transition matrix on $\mathcal{X}$. In this sense, we have a Markov process, or diffusion, on $\mathcal{X}$, where if we want to diffuse from $x_i$ to $x_j$, we always go through $\mathcal{Y}$. Moreover, note that we have
\begin{align}
&(D^{(\textup{R})})^{-1/2}W^{(\textup{r})}[(D^{(\textup{R})})^{-1/2}W^{(\textup{r})}]^\top
=(D^{(\textup{R})})^{-1/2}W^{(\textup{R})}(D^{(\textup{R})})^{-1/2}\,,
\end{align}
where the right hand side is symmetric. If $(D^{(\textup{R})})^{-1/2}W^{(\textup{R})}(D^{(\textup{R})})^{-1/2}$ is non-negative definite, the SVD of $(D^{(\textup{R})})^{-1/2}W^{(\textup{r})}$ recovers the eigen-structure of the non-negative definite matrix $(D^{(\textup{R})})^{-1/2}W^{(\textup{R})}(D^{(\textup{R})})^{-1/2}$. 

\subsection{Relationship with the Diffusion Map}
Note that Roseland and a well-known algorithm, DM, are very close, except the diffusion via the landmark set. Specifically, in DM, the affinity matrix $W\in \mathbb{R}^{n\times n}$ is defined directly via
\begin{equation}\label{Definition affinity matrix W}
W_{ij} := 
K\left(\frac{\|x_{i}-x_{j}\|_{\mathbb{R}^{q}}}{\sqrt{\epsilon}}\right)\,
\end{equation}
and the corresponding {\em degree matrix} $D\in \mathbb{R}^{n\times n}$, is defined as 
\begin{equation}
D_{ii}:=\sum_{j=1}^nW_{ij}.
\end{equation}
The transition matrix associated with a Markov process on the dataset $\mathcal{X}$ is defined by
\begin{equation} \label{Eq A ordinary DM}
A:=D^{-1}W.
\end{equation}
Note that the landmark-affinity matrix $W^{(\textup{R})}$ defined in \eqref{definition WR in Roseland}, and the associated transition matrix $(D^{(\textup{R})})^{-1}W^{(\textup{R})}$ can be viewed as a different way of constructing a Markov process on the dataset.
Next, we run the eigen-decomposition of the transition matrix. Suppose the eigenvalues are ordered by $\lambda_1\geq\lambda_2\geq\ldots\geq\lambda_n$, and
the right eigenvectors of $A$ are denoted as $\phi_i$. Note that the decomposition in \eqref{DR-1/2Wr SVD in Roseland} is a parallel step of this eigendecomposition.
For a chosen {\em diffusion time $t$}, DM embeds $\mathcal{X}$ via the map
\begin{equation}
\Phi_t:x_{i}\mapsto (\lambda_2^t\phi_2(i),\ldots,\lambda_{q'+1}^t\phi_{q'+1}(i))\in \mathbb{R}^{q'},
\end{equation}
where $q'$ is the dimension chosen by the user. With DM, the diffusion distance (DD) with the diffusion time $t>0$ is defined as
\begin{equation}
D_t(x_i,x_j):=\|\Phi_t(x_{i})-\Phi_t(x_j)\|_{\mathbb{R}^{q'}}\,.
\end{equation}
Clearly, the Roseland embedding \eqref{embedding in Roseland} and RDD \eqref{DD in Roseland} are closely related to DM and DD.

\section{Asymptotic behavior of Roseland under the manifold setup}\label{Section:PointConvergence}

In this long section we show that the eigenvalues and eigenvectors of Roseland converge to the eigenvalues and eigenfunctions of the Laplace-Beltrami operator, and quantify the convergence rate, both {\em pointwisely} and {\em spectrally}. %It is safe for readers with interest in numerical results of Roseland to skip this section and jump directly to Section \ref{numerical}.
Before stating our main asymptotic results of Roseland, we briefly summarize existing literature about DM.

A celebrated spectral embedding  \cite{berard1994embedding} gives a solid foundation of various spectral based unsupervised learning algorithms, particular DM. It says that one can embed any given smooth closed $n$-dimensional Riemannian manifold by the eigenfunctions of its Laplace--Beltrami operator, and the embedding can be ``tuned'' to be as isometric as possible. However, this spectral embedding needs {\em all} eigenfunctions, which is not numerically affordable. 
To resolve this issue, it is proved in \cite{portegies2016embeddings} that for a given tolerable metric recovery error, we can achieve an almost isometric embedding with that tolerable error with a finite number of eigenfunctions of the Laplace--Beltrami operator, where the number only depends on the geometric bounds and the dimension. 
The above two results are on the continuous setup. To utilize these results, we need to link numerical finite sampling dataset to the continuous manifold setup. Specifically, if we are able to prove how the eigenvectors of the GL converges in the spectral sense to the eigenfunctions of the Laplace-Beltrami operator, we can apply the above-mentioned results in the continuous setup.
In \cite{singer2016spectral, trillos2018error,calder2019improved,DunsonWuWu2019}, the authors provide the spectral convergence of the GL constructed from random samples to the Laplace--Beltrami operator. A convergence rate of the eigenfunction in the $L^2$ sense is provided in \cite{trillos2018error,calder2019improved}, and a convergence rate of the eigenfunction in the $L^\infty$ sense is provided in \cite{DunsonWuWu2019}. 

At the first glance, it might be expected that the proofs are similar to those shown in \cite{singer2016spectral, trillos2018error,calder2019improved,DunsonWuWu2019}. However, as we will see below, we run into the dependence issue due to the landmark diffusion, so extra efforts and new technical tools are needed.

\subsection{Manifold model}\label{dm_setup}
Denote our observed data set the point cloud $\mathcal{X} = \{x_{i}\}_{i=1}^{n}\subseteq\mathbb{R}^{D}$, which are independent and identically distributed (i.i.d.) sampled from a random vector $X:(\Omega,\mathcal{F},\mathbb{P})\rightarrow\mathbb{R}^{D}$. We assume that the range $X$ is supported on a $d$-dimensional compact smooth 
Riemannian manifold $(M^d, g)$ without boundary that is isometrically embedded in $\mathbb{R}^D$ via $\iota:M^{d}\hookrightarrow \mathbb{R}^D$. Hence, $X$ induces a probability measure on $\iota(M^{d})$, denoted by $\widetilde{\mathbb{P}}_{X}$. Further assume $\widetilde{\mathbb{P}}_{X}$ is absolutely continuous with respect to the Riemannian measure on $\iota(M)$, denoted by $\iota_{\ast}dV(x)$. Then, by the Radon Nikodym theorem, we have $d\widetilde{\mathbb{P}}_{X}(x)=p_{X}(\iota^{-1}(x))\iota_{\ast}dV(x)$. Clearly, $p_X$ is a function defined on $M^{d}$. 

\begin{definition}
	We call $p_{X}$ defined above the probability density function (p.d.f.) associated with $X$. When $p_{X}$ is constant, $X$ is called uniform; otherwise non-uniform.
\end{definition}

We assume $p_{X}$ satisfies $p_{X} \in \mathcal{C}^4(M^{d})$ and $0 < \inf_{x\in M^{d}}p_{X}(x)\leq \sup_{x\in M^{d}}p_{X}(x)$.

\begin{definition}\label{Definition:kernel}
	A kernel function is any non-negative function $K:[0,\infty)\rightarrow\mathbb{R}^{+}$ so that it is $C^{3}$, $K(0)>0$ and decays exponentially fast. Denote $\mu_{r,l}^{(k)}\coloneqq \int_{\mathbb{R}^{d}}\|x\|^{l}\partial_{k}K^{r}(\|x\|)\,dx$, for $r,l,k=0,1,2,\ldots$, and assume $K$ is normalized so that $\mu_{1,0}^{0}=1.$
\end{definition}

%\subsection{Roseland in the manifold setup}
For the landmark set $\mathcal{Y} = \{y_{j}\}_{j=1}^{m}\subseteq\mathbb{R}^{D}$, we assume that $y_{j}$'s are i.i.d. samples from a random vector $Y:(\Omega,\mathcal{F},\mathbb{P})\rightarrow\mathbb{R}^{D}$, whose range is supported on the same manifold $M^{d}$, and has p.d.f. $p_{Y}$ on $M^{d}$. Moreover, we assume that $Y$ is independent of $X$. We also assume $p_{Y}$ satisfies $p_{Y} \in \mathcal{C}^4(M^{d})$. Clearly, we have $0 < \inf_{x\in M^{d}}p_{Y}(x)\leq \sup_{x\in M^{d}}p_{Y}(x)$.

In Roseland, the ``affinity'' between two data points $x_{i}$ and $x_{j}$ is measured via the landmark set $\mathcal{Y} = \{y_{j}\}_{j=1}^{m}$. Specifically, note that the affinity matrix $W^{(\textup{R})}=W^{(\textup{r})}[W^{(\textup{r})}]^{\top}\in\mathbb{R}^{n\times n}$, where $W^{(\textup{r})}\in\mathbb{R}^{n\times m}$ such that $W^{(\textup{r})}(i,k)= K_{\epsilon}(x_{i},y_{k})$ and $[W^{(\textup{r})}]^{\top}$ is the transpose of $W^{(\textup{r})}$. 

\begin{definition}\label{ref_affinity}
	Take the kernel function $K$. The affinity between any two points $x_{i}$ and $x_{j}$ via a landmark set $\mathcal{Y} = \{y_{j}\}_{j=1}^{m}$ is defined by
	\begin{equation}
	K_{\textup{ref},\epsilon,n}(x_{i},x_{j})\coloneqq \frac{1}{m}\sum_{k=1}^{m}K_{\epsilon}(x_{i},y_{k})K_{\epsilon}(y_{k},x_{j}).
	\end{equation}
\end{definition}
Note that compared with \eqref{definition WR in Roseland}, here is a $\frac{1}{m}$ factor in $K_{\textup{ref},\epsilon,n}(x_{i},x_{j})$. Due to the normalization, there is no difference if we put $\frac{1}{m}$ in front or not. 
See Figure \ref{Figure:refDMillustration} for an illustration of how this affinity is determined.
To study the asymptotic behavior of Roseland, we take the following expansion into account. 
For $f\in C(M)$, denote its discretization on $\{x_i\}_{i=1}^n$ as $\textbf{\textit{f}}\in\mathbb{R}^{n}$ such that $\textbf{\textit{f}}_{i}=f(x_{i})$. By a direct expansion, we have
\begin{eqnarray}\label{discrete_op2}
\big[(D^{(\textup{R})})^{-1}W^{(\textup{R})}\textbf{\textit{f}}\;\big](i)=\frac{\sum_{j=1}^{n}W^{(\textup{R})}_{ij}\textbf{\textit{f}}_{j}}{\sum_{j=1}^{n}W^{(\textup{R})}_{ij}}\,.%=\frac{\frac{1}{n}\sum_{j=1}^{n}W^{(\textup{R})}_{ij}\textbf{\textit{f}}_{j}}{\frac{1}{n}\sum_{j=1}^{n}(W^{(\textup{R})})_{ij}}
\end{eqnarray}
The denominator and numerator ring the bell of the law of large numbers. We thus expect that \eqref{discrete_op2} converges to the following integral operators when $n\to\infty$:

\begin{definition}\label{GL_cts}
	Take $f\in C(M), \epsilon>0$, define
	\begin{equation}
	T_{\textup{ref},\epsilon}f(x)\coloneqq\int_{M}\frac{K_{\textup{ref},\epsilon}(x,y)}{d_{\textup{ref},\epsilon}(x)}f(y)p_{X}(y)\,dV(y),
	\end{equation}
	where $K_{\textup{ref},\epsilon}:M\times M\to \mathbb{R}$ defined as
	\begin{align}
	K_{\textup{ref},\epsilon}(x,y)&\coloneqq\int_{M}K_{\epsilon}(x,z)K_{\epsilon}(z,y)p_{Y}(z)\,dV(z)
	\end{align}
	is called the {\em landmark-kernel} induced by the landmark set, and
	\begin{align}
	d_{\textup{ref},\epsilon}(x)&\coloneqq\int_{M}K_{\textup{ref},\epsilon}(x,y)p_{X}(y)\,dV(y)\,.
	\end{align}
\end{definition}

\subsection{Kernel behavior with the landmark set}

It is worth a bit more discussion of the landmark-kernel induced by the landmark set. Recall Definition \ref{ref_affinity}. The affinity between two points is now determined by passing through the landmark set. A direct consequence is that the kernel function may vary from point to point, depending on how two points are geologically related to the landmark set. The affinity of a point to itself might be smaller than the the affinity between a point and its close neighbor.

%To further illustrate this finding, take the Gaussian as the kernel function; that is, $K_{\epsilon}(x,y)=e^{-\|x-y\|_{\mathbb{R}^{D}}^{2}/\epsilon}$. 
%
To have a closer look at the landmark-kernel $K_{\textup{ref},\epsilon}$, we assume $K_{\epsilon}(y,x_{j})=\exp(-\|x_i-y\|_{\mathbb{R}^{D}}^2/\epsilon)$; that is, the kernel is Gaussian. {In this case, 
\[
K_{\epsilon}(x_{i},y)K_{\epsilon}(y,x_{j})=\exp\left(-\frac{2\|y-(x_i+x_j)/2\|_{\mathbb{R}^{D}}^2}{\epsilon}\right)\exp\left(-\frac{\|x_i-x_j\|_{\mathbb{R}^{D}}^2}{2\epsilon}\right)\,,
\]
and hence
\begin{align}
K_{\textup{ref},\epsilon}(x_i,x_j)&\,=\int_M  K_{\epsilon}(x_{i},y)K_{\epsilon}(y,x_{j})p_Y(y)dV(y)\nonumber\\
&\,=\exp\left(-\frac{\|x_i-x_j\|_{\mathbb{R}^{D}}^2}{2\epsilon}\right)\int_M \exp\left(-\frac{2\|y-(x_i+x_j)/2\|_{\mathbb{R}^{D}}^2}{\epsilon}\right)p_Y(y)dV(y)\,.\label{expansion landmark kernel}
\end{align}
Thus, we have 
\begin{equation}
\epsilon^{-d/2}K_{\textup{ref},\epsilon}(x_i,x_j)=c(x_i,x_j)\exp\left(-\frac{\|x_i-x_j\|_{\mathbb{R}^{D}}^2}{2\epsilon}\right)\,, \label{expansion landmark kernel2}
\end{equation}
where $c$ is smooth, depends on $x_i$ and $x_j$ and is of order $1$ since $M$ is compact.
This formula tells us that the landmark-kernel is not isotropic nor fixed for any pair of points.} Next, take a set of equally spaced samples from $S^{1}$, denoted as $\mathcal{X}$, and order them by their angles. Let the landmark set $\mathcal{Y}\subseteq\mathcal{X}$ contains $5\%$ equally spaced points of $\mathcal{X}$, also ordered by angle to its center. See Figure \ref{S1ker} $(a)$. By the symmetry of $\mathcal{X}$ and $\mathcal{Y}$, $|\mathcal{Y}|=5\%\times|\mathcal{X}|$ means we have $20$ difference kernel functions. Indeed, we have 20 data points between 2 consecutive landmarks, so there are totally 20 different geological relationships between the dataset and the landmark set.
See Figure \ref{S1ker} $(b)$ for plots of the $20$ landmark-kernels at the data points (in order) between two consecutive landmarks. Similarly, we plot the kernel functions when $|\mathcal{Y}|=a\times|\mathcal{X}|$, for $a=10\%,20\%$, see Figure \ref{S1ker} $(c)$ and $(d)$.
Note that when $N$ is fixed and $M$ increases, it is not surprising that the kernel looks more like a Gaussian. 

\begin{figure}[bht!]\centering
	\includegraphics[width=4.5cm]{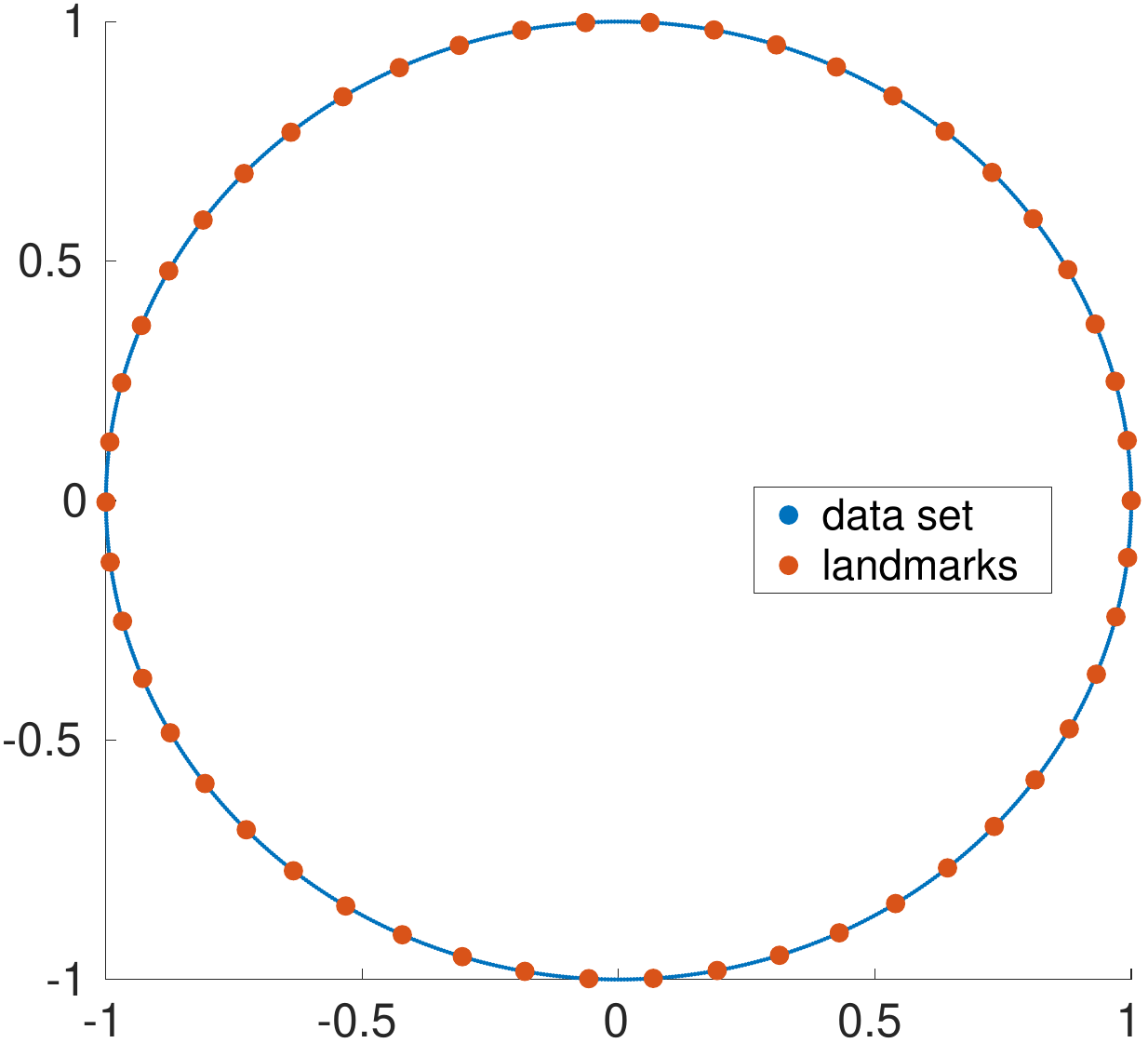}
	\includegraphics[width=5.5cm]{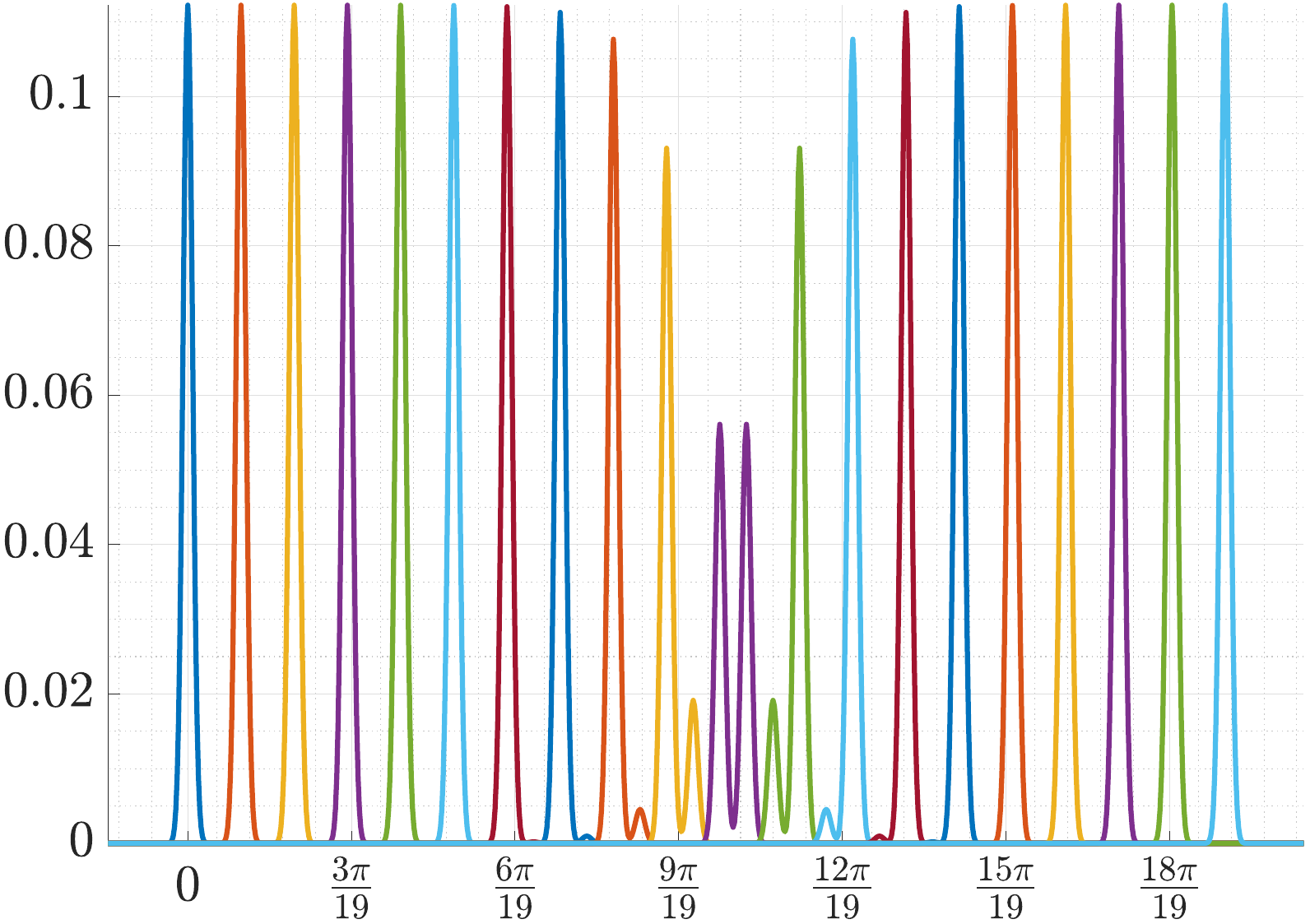}\\
	\includegraphics[width=5cm]{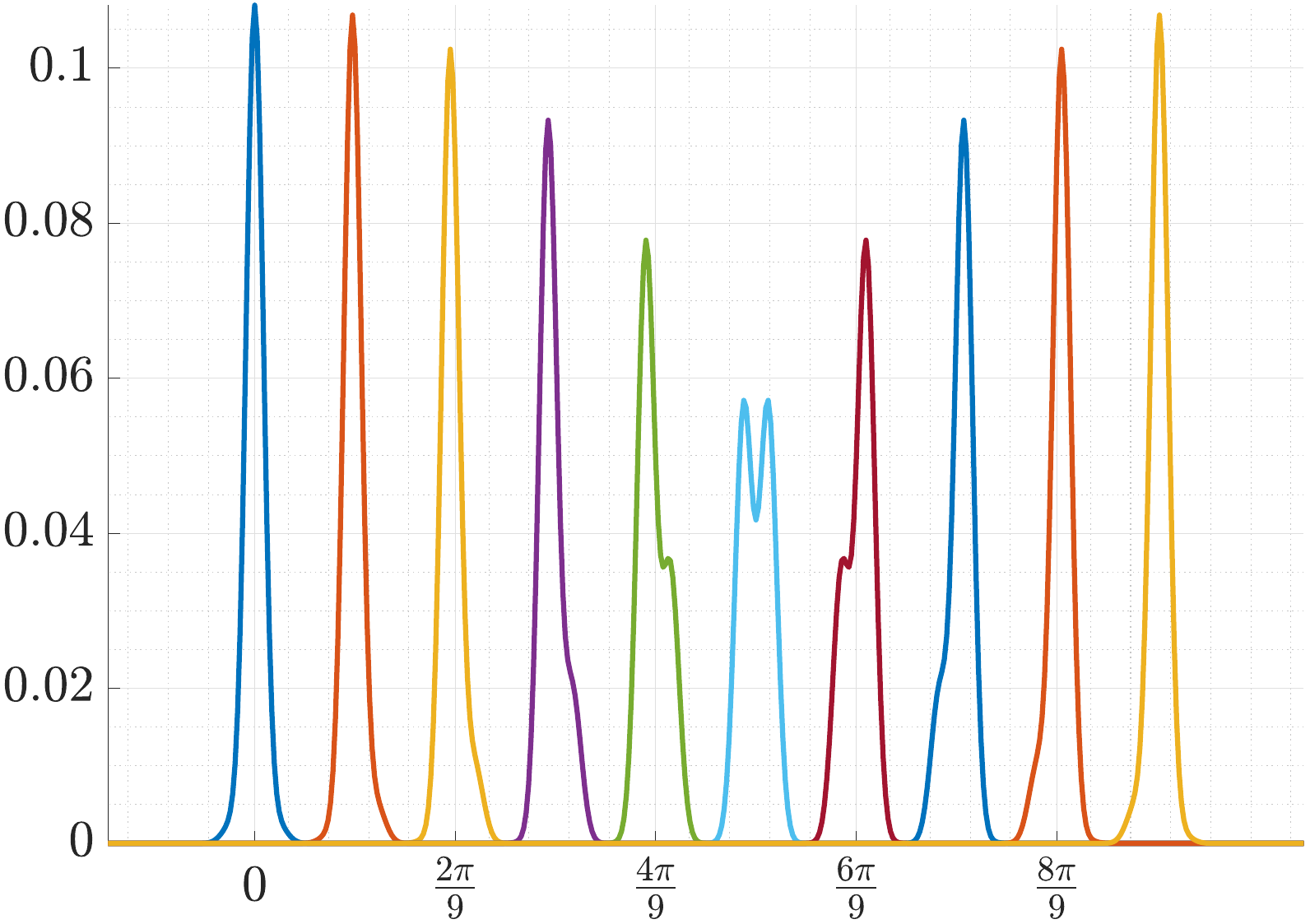}
	\includegraphics[width=5cm]{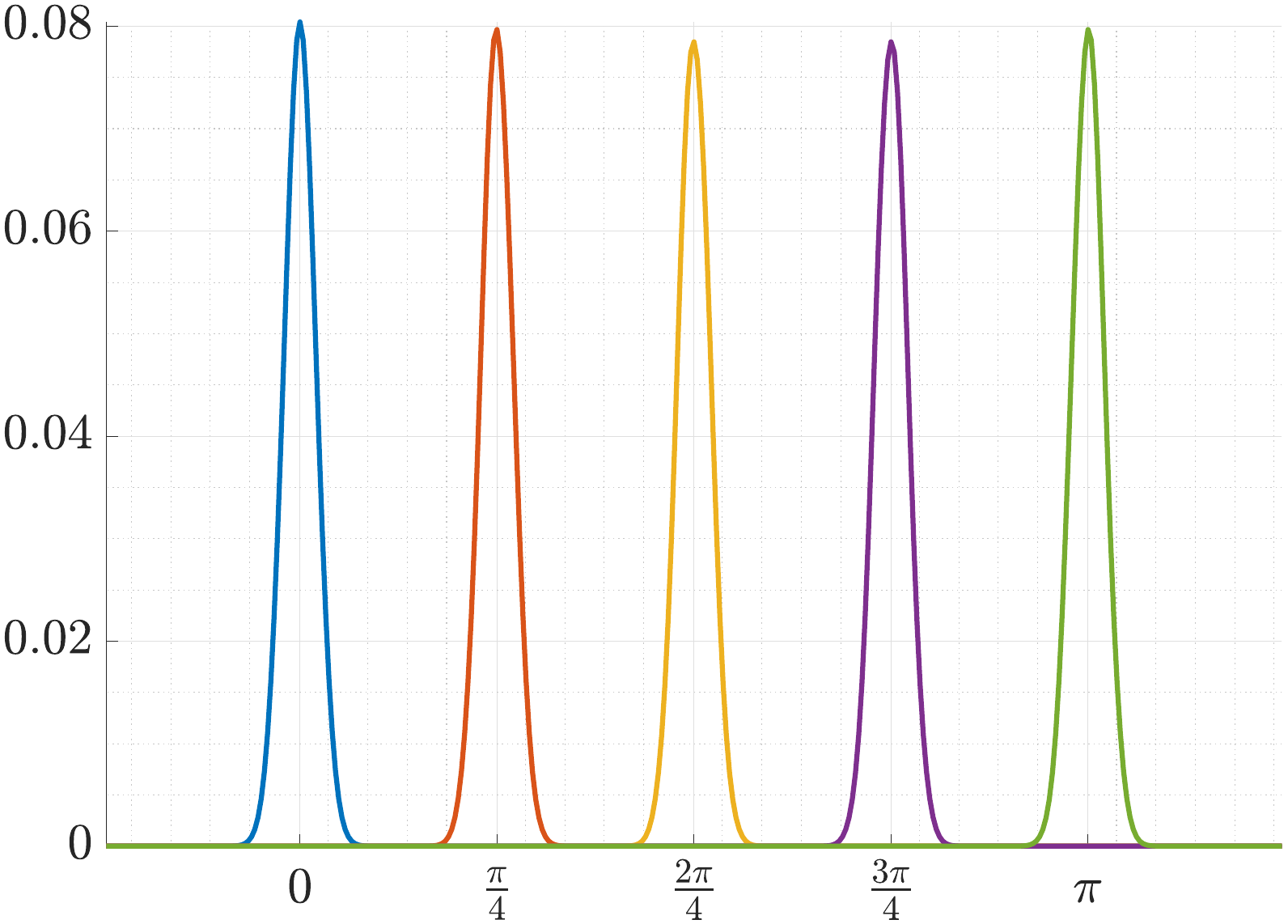}
	\caption{Top left: $S^{1}$ with 5\% equally spaced landmarks. Top right: $20$ different landmark-kernel functions from 5\% equal spaced landmarks as in top left. Bottom left: $10$ different landmark-kernels if choose 10\% equally spaced landmarks. Bottom right: $5$ different landmark-kernels if choose 20\% equal spaced landmarks.}\label{S1ker}
\end{figure}

\subsection{Pointwise convergence}

We first state the bias analysis of Roseland.
\begin{theorem}[Bias analysis]\label{biasthm}
	Take $f\in C^{4}(M^{d})$. Then, for all $x\in M^{d}$ we have
	\begin{align}
	T_{\textup{ref},\epsilon}f(x) - f(x) =&\, \frac{\epsilon\mu_{1,2}^{(0)}}{d}\left(\frac{2\nabla p_{X}(x)}{p_{X}(x)}+\frac{\nabla p_{Y}(x)}{p_{Y}(x)}\right)\cdot\nabla f(x) 
	+ \frac{\epsilon\mu_{1,2}^{(0)}}{d}\Delta f(x) + \mathcal{O}(\epsilon^{2})\,.\nonumber
	\end{align}
\end{theorem}

The proof is postponed to Section \ref{biasproof}.

\begin{remark} \label{remark: dm_thm}
	We compare the obtained result with the existing theorems for DM shown in \cite{Coifman2006}. Take $f\in C^{4}(M^{d})$. Recall the definition:
	$$
	T_{\epsilon,\alpha}f(x)\coloneqq\int_{M}\frac{K_{\epsilon,\alpha}(x,y)}{d_{\epsilon,\alpha}(x)}f(y)p_{X}(y)\,dV(y)\,,
	$$
	where 
	$K_{\epsilon,\alpha}(x,y)\coloneqq\frac{K_{\epsilon}(x,y)}{p_{X,\epsilon}^{\alpha}(x)p_{X,\epsilon}^{\alpha}(y)}$, $p_{X,\epsilon}(x)\coloneqq\int_{M}K_{\epsilon}(x,y)p_{X}(y)\,dV(y)$,
	and $d_{\epsilon,\alpha}(x)\coloneqq\int_{M}K_{\epsilon,\alpha}(x,y)p_{X}(y)\,dV(y)$.
	Then for all $x\in M^{d}$ we have:
	\begin{equation}
	T_{\epsilon,\alpha}f(x) - f(x) = \frac{\epsilon\mu_{1,2}^{(0)}}{2d}\left(\Delta f(x) + \frac{2\nabla f(x)\cdot\nabla p_{X}^{1-\alpha}(x)}{p_{X}^{1-\alpha}(x)}\right) + \mathcal{O}(\epsilon^{2})\,.
	\end{equation}
	Based on this result, we see that in DM, we can remove the impact of the non-uniformly sampling of the data set by letting $\alpha=1$ in the $\alpha$-normalization step. In Roseland, if $\frac{2\nabla p_{X}(x)}{p_{X}(x)}+\frac{\nabla p_{Y}(x)}{p_{Y}(x)}=0$, then we remove the impact of the non-uniformly sampling and recover the Laplace-Beltrami operator.
	Note that $\frac{2\nabla p_{X}(x)}{p_{X}(x)}+\frac{\nabla p_{Y}(x)}{p_{Y}(x)}=0$ suggests that we may want to have the landmark set following $p_{Y}(x)\propto\frac{1}{p^{2}_{X}(x)}$. This serves us as the guidance of how to design the landmark set.
	{While pursuing systematically how to handle the influence of the density function is not the focus of this paper, we mention that when we cannot design the landmark, we may consider to combine the $\alpha$ normalization in the DM framework and Roseland to remove the influence of the probability density of the data set and/or the landmark set (e.g. the two stage normalization similar to
equation (10) in \cite{haddad2014texture}). We will leave this possibility to our future work.}
\end{remark}

\begin{theorem}[Variance analysis]\label{variancethm}
	Take $\mathcal{X}=\{x_i\}_{i=1}^n$ and $\mathcal{Y}=\{y_j\}_{i=1}^m$, where $m=[n^{\beta}]$ for some $0<\beta\leq 1$ and $[x]$ is the nearest integer of $x\in\mathbb{R}$. 
	Take $f\in C(M^{d})$ and denote $\textbf{\textit{f}}\in\mathbb{R}^{n}$ such that $\textbf{\textit{f}}_{i}=f(x_{i})$. Let $\epsilon = \epsilon(n)$ so that $\frac{\sqrt{\log n}}{n^{\beta/2}\epsilon^{d/2+1/2}}\rightarrow 0$ and $\epsilon\rightarrow 0$ when $n\rightarrow\infty$. Then with probability higher than $1 - \mathcal{O}(1/n^{2})$, we have
	\begin{equation}
	\big[(I-(D^{(\textup{R})})^{-1}W^{(\textup{R})})\textbf{\textit{f}}\;\big](i) = f(x_{i}) - T_{\textup{ref},\epsilon}f(x_{i}) + \mathcal{O}\Big(\frac{\sqrt{\log n}}{n^{\beta/2}\epsilon^{d/2-1/2}}\Big)
	\end{equation}
	for all $i=1,2,\ldots,n$.
\end{theorem}
The proof is postponed to Section \ref{biasproof}.

\begin{remark}\label{Remark about the variance convergence rate}
	We compare the obtained result with the existing theorems shown in \cite{Coifman2006, singer2016spectral}. 
	For the variance analysis for DM, we have the following result when there are $n$ data points $\mathcal{X}=\{x_i\}_{i=1}^n$ from the manifold.
	Take $f\in C(M^{d})$.
	\begin{itemize}
		\item For $0<\alpha\leq 1$, let $\epsilon = \epsilon(n)$ so that $\frac{\sqrt{\log n}}{n^{1/2}\epsilon^{d/4+1/2}}\rightarrow 0$ and $\epsilon\rightarrow 0$ when $n\rightarrow\infty$. Then with probability higher than $1 - \mathcal{O}(1/n^{2})$, for all $i=1,2,\ldots,n$, we have
		\begin{equation}
		\big[(I-(D^{(\alpha)})^{-1}W^{(\alpha)})\textbf{\textit{f}}\;\big](i) =  f(x_{i})-T_{\epsilon,\alpha}f(x_{i}) + \mathcal{O}\Big(\frac{\sqrt{\log n}}{n^{1/2}\epsilon^{d/4}}\Big)\,.
		\end{equation}

		\item For $\alpha=0$, let $\epsilon = \epsilon(n)$ so that $\frac{\sqrt{\log n}}{n^{1/2}\epsilon^{d/4+1/2}}\rightarrow 0$ and $\epsilon\rightarrow 0$, when $n\rightarrow\infty$. Then with probability higher than $1 - \mathcal{O}(1/n^{2})$, for all $i=1,2,\ldots,n$, we have
		\begin{equation}\label{alpha=0 DM error bound}
		\big[(I-(D^{(0)})^{-1}W^{(0)})\textbf{\textit{f}}\;\big](i) =  f(x_{i})-T_{\epsilon,\alpha}f(x_{i}) + \mathcal{O}\Big(\frac{\sqrt{\log n}}{n^{1/2}\epsilon^{d/4-1/2}}\Big)\,.
		\end{equation}

	\end{itemize}
	Clearly, unlike DM, in Roseland, its convergence rate depends on $n^{\beta}$, which is the size of the landmark set. This pointwise convergence result tells us that the smaller the landmark set is, the faster the algorithm, but the slower the convergence rate to the Laplace-Beltrami operator. We should compare the rate of Roseland with the rate of DM when the alpha normalization is $0${; that is, $\alpha=0$}. The error term in DM is of order $ \mathcal{O}\Big(\frac{\sqrt{\log n}}{n^{1/2}\epsilon^{d/4-1/2}}\Big)$ while the error term in Roseland is of order $\mathcal{O}\Big(\frac{\sqrt{\log n}}{n^{\beta/2}\epsilon^{d/2-1/2}}\Big)$, where $n$ is the size of data, $n^{\beta}$ is the size of landmark set. 
	Note that even when we let $\beta=1$, the convergence rate of Roseland still does not recover the convergence rate of DM, where they differ by a factor of $\epsilon^{\frac{d}{4}}$. This is because Roseland introduces dependence among data points by diffusing through the landmark set, and this dependence relation results in a larger variance of the random variable to be analyzed. This fact gives a slower convergence rate when we apply the large deviation bound. See Section \ref{Discussion optimal variance convergence discrepancy} for more details.
\end{remark}

\subsection{Idea of analyzing the variance}\label{grid_samp_def}
Let $X$ and $Y$ be two independent random variables and $f:(X,Y)\rightarrow\mathbb{R}$. One way to compute $\mathbb{E}(f(X,Y))$ numerically is by i.i.d. sampling $n$ pairs of points $\{(x_{i},y_{i})\}_{i=1}^{n}$ from the joint distribution of $(X,Y)$. Then we have $\frac{1}{n}\sum_{i=1}^{n}f(x_{i},y_{i})\rightarrow\mathbb{E}(f(X,Y))$ almost surely by the law of large numbers. And there are standard techniques available to compute its convergence rate. Due to the nature of landmark set, this approach does not hold. Indeed, note that if we expand \eqref{discrete_op2}, we have
\begin{align}\label{Expansion:refDM in the summation formate}
\big[D^{(\textup{R})})^{-1}W^{(\textup{R})}\textbf{\textit{f}}\big](i)=\frac{\frac{1}{n}\sum_{j=1}^{n}\big[\sum_{k=1}^{m}K_{\epsilon}(x_{i},y_{k})K_{\epsilon}(y_{k},x_{j})\big]\,\,\textbf{\textit{f}}_{j}}{\frac{1}{n}\sum_{j=1}^{n}\big[\sum_{k=1}^{m}K_{\epsilon}(x_{i},y_{k})K_{\epsilon}(y_{k},x_{j})\big]}\,,
\end{align}
which generates dependence among the summands. 
%See Section \ref{varianceproof} for details. 
We have the following definition:
\begin{definition}\label{gridsamp}
	Let $X$ and $Y$ be two independent random variables. 
	%and $f:(X,Y)\rightarrow\mathbb{R}$. 
	We call $\{(x_i,y_j)|\;i=1,\ldots,n,\, j=1,\ldots,m\}$ a grid sampling if $\{x_{i}\}_{i=1}^{n}$ is i.i.d. sampled from $X$ and $\{y_{j}\}_{j=1}^{m}$ is i.i.d. sampled from $Y$. %to get $mn$ grid points $\{f(X_{i},Y_{j})\}_{i=1,j=1}^{n,m}$
\end{definition}
Clearly, the grid samples are not independent, %Thus standard large deviation theories don't apply.
%\begin{remark}\label{grid_indep_cond}
and we know that $(x_{i_{1}},y_{j_{1}})$ is independent of $(x_{i_{2}},y_{j_{2}})$ if and only if $i_{1}\neq i_{2}$ and $j_{1}\neq j_{2}$. 
%\end{remark}
In general, \eqref{Expansion:refDM in the summation formate} can be formulated in the following way. Given $f:(X,Y)\rightarrow\mathbb{R}$ and a grid sampling $\{(x_i,y_j)|\;i=1,\ldots,n,\, j=1,\ldots,m\}$, we ask how well we can approximate $\mathbb{E}(f(X,Y))$ from the sampling grid; that is, what is the convergence rate of 
\begin{eqnarray}\label{gridconvg}
\frac{1}{mn}\sum_{i=1}^{n}\sum_{j=1}^{m}f(x_{i},y_{j})\longrightarrow\mathbb{E}(f(X,Y))\,.
\end{eqnarray}
Clearly, we need to handle the dependence on the grid sampling.

To answer this question, we consider the work in \cite{janson2004large}, which provides a method of computing the convergence rate of this kind of sampling.
In general, we consider the random variable of the form
\begin{align}\label{dep_sum}
X=\sum_{\alpha\in\mathcal{A}}Y_{\alpha}
\end{align}
where $Y_{\alpha}$ are random variables, independent of not, and $\alpha$ ranging over some index set $\mathcal{A}$. We have the following definition.
\begin{definition}
	Given an index set $\mathcal{A}$ and $\{Y_{\alpha}\}_{\alpha\in\mathcal{A}}$, we make the following definitions.
	\begin{itemize}
		\item A subset $\mathcal{A}'$ of $\mathcal{A}$ is independent if the corresponding random variables $\{Y_{\alpha}\}_{\alpha\in\mathcal{A}'}$ are independent.
		\item A family $\{\mathcal{A}_{j}\}_{j}$ of subsets of $\mathcal{A}$ is a cover of $\mathcal{A}$ if $\bigcup_{j}\mathcal{A}_{j}=\mathcal{A}$.
		\item A cover is proper if each set $\mathcal{A}_{j}$ in it is independent.
		\item $\chi(\mathcal{A})$ is the size of the smallest proper cover of $\mathcal{A}$.
	\end{itemize}
\end{definition}
\noindent Then we have the first Hoeffding-like concentration inequality.
\begin{theorem}\label{hoeffding}
	Suppose X is defined in \eqref{dep_sum} with $a_{\alpha}\leq Y_{\alpha}\leq b_{\alpha}$ for every $\alpha\in\mathcal{A}$, where $a_{\alpha},b_{\alpha}\in \mathbb{R}$. Then for all  $t>0$,
	\begin{eqnarray}
	\nonumber\mathbb{P}(X-\mathbb{E}(X)\geq t)\leq\textup{exp}\left(\frac{-2t^{2}}{\chi(\mathcal{A})\sum_{\alpha\in\mathcal{A}}(b_{\alpha}-a_{\alpha})^{2}}\right).
	\end{eqnarray}
	The same estimate holds for $\mathbb{P}(X-\mathbb{E}(X)\leq -t)$.
\end{theorem}

%\begin{remark}\label{when_can_improve}
When $Y_{\alpha}$'s have variances that are substantially smaller than $(b_{\alpha}-a_{\alpha})^{2}/4$, we can improve theorem \ref{hoeffding} to the Bernstein-like concentration inequality.
%\end{remark}

\begin{theorem}\label{bernstein}
	Suppose X is defined in \eqref{dep_sum} with $Y_{\alpha}-\mathbb{E}(Y_{\alpha})\leq b$ for some $b>0$ for all $\alpha\in\mathcal{A}$. Suppose $S\coloneqq\sum_{\alpha\in\mathcal{A}}\textup{Var}\,Y_{\alpha}<\infty$. Then, for all $t>0$,
	\begin{eqnarray}
	\nonumber\mathbb{P}(X-\mathbb{E}(X)\geq t)\leq\textup{exp}\left(\frac{-8t^{2}}{25\chi(\mathcal{A})(S+bt/3)}\right).
	\end{eqnarray}
	The same estimates holds for $\mathbb{P}(X-\mathbb{E}(X)\leq -t)$ if also $Y_{\alpha}-\mathbb{E}(Y_{\alpha})\geq -b$ for $b>0$.
\end{theorem}

%In our case, Remark \ref{when_can_improve} can happen, so we can obtain a better bound by using theorem \ref{bernstein}.

With this general theory, we now come back to our setup. In our grid sampling scheme in Roseland, $\mathcal{A} = \{(j,k)\}_{j=1,k=1}^{n,m}$. We now claim that
\begin{equation}\label{Optimal bound of the independent set A}
\chi(\mathcal{A})\leq m+n-1=\mathcal{O}(\max(m,\,n))\,. 
\end{equation}
The easiest way of seeing it is by the following grid shown in Figure \ref{cover number plot for independent}, where the coordinate $(j_{a}, k_{b})$ corresponds to the random variable $(x_{j_{a}}, y_{k_{b}})$. Clearly, we know that $\{\mathcal{A}_{x}\}_{x=1}^{n+m-1}$ is a proper cover of $\mathcal{A}$. 
That means the convergence rate in (\ref{gridconvg}) should be the same as that of $\frac{1}{\min(n,m)}\sum_{i=1}^{\min(n,m)}f(X_{i},Y_{i})\rightarrow\mathbb{E}(f(X,Y))$, and hence the rate is dominated by $\min(m,n)$. 
%
%Note that the bound $m+n=\mathcal{O}(\max(m,\,n))$ is not optimal, particularly when $m$ is small, which can be seen by the special case where $m=1$.

\begin{figure}[bht!]
	\includegraphics[width=11cm]{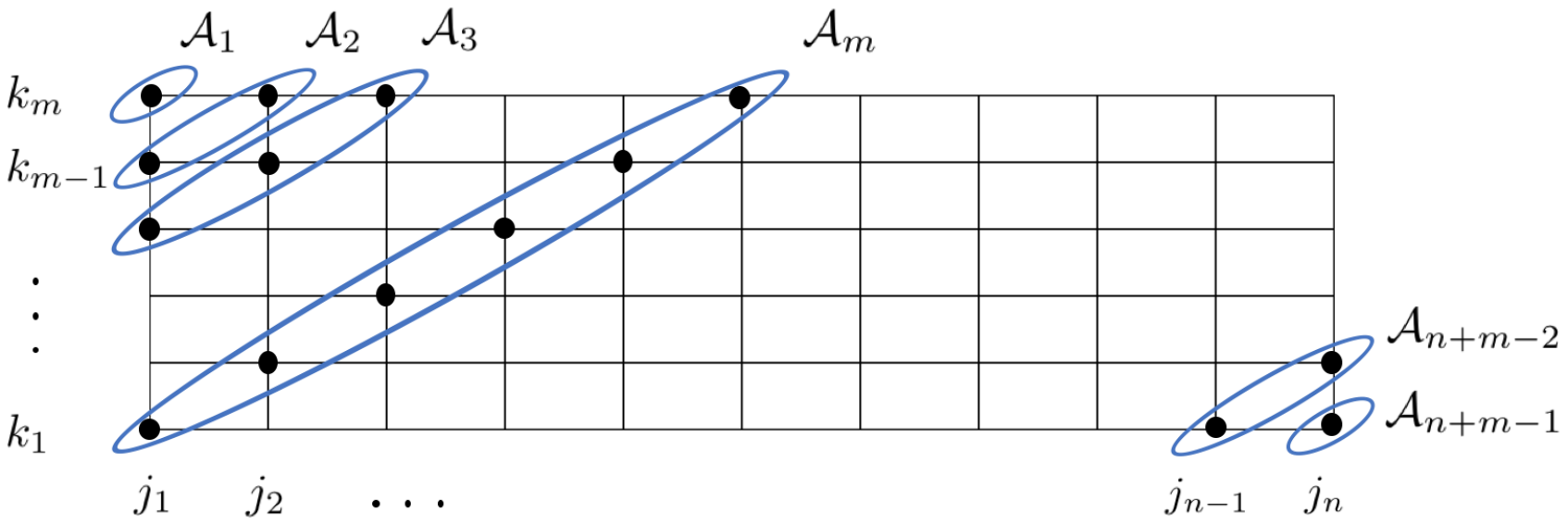}
	\centering
	\caption{Illustration of the grid argument, note the grid samples within each circle are independent. \label{cover number plot for independent}}
\end{figure}

The challenge we encounter with the grid sampling is directly related to the U-statistic or V-statistic. Take the ``kernel'' $h$ of $r$ variables, where $r\in \mathbb{N}$. For the dataset $x_{1},\ldots,x_{n}$, where $n\geq r$, an U-statistic of order $r\in \mathbb{N}$ is defined as
\begin{eqnarray}
U_{r}: = \frac{1}{{n\choose r}}\sum_{(i_1, i_2,\ldots, i_r)\in \langle n\rangle}h(x_{i_{1}},\ldots,x_{i_{r}})\,,\label{definition U statistics}
\end{eqnarray}
where $\langle n\rangle$ is the set of all permutations of $\{1,\ldots,n\}$ and $h$ is symmetric in its arguments. For example, when the kernel $h$ is of $1$ variable and $h(x)=x$, then the U-statistic $U_{1}(x)=(x_{1}+\ldots+x_{n})/n$ is the sample mean $\overline{x}$.
In practice, this statistic has a representation as the V-statistic:
\begin{eqnarray}
V_{m,n} := \frac{1}{n^m}\sum_{i_{1}=1}^{n}\cdot\cdot\cdot\sum_{i_{m}=1}^{n}h(x_{i_{1}},\ldots,x_{i_{m}})\,,
\end{eqnarray}
where $h$ is a symmetric kernel function. We call $V_{mn}$ a V-statistic of degree $m$. A typical example of a degree-2 V-statistic is the second central moment; that is, take $h(x,y):=(x-y)^{2}/2$, then
$V_{2,n} = \frac{1}{n^{2}}\sum_{i=1}^{n}\sum_{j=1}^{n}\frac{1}{2}(x_{i}-x_{j})^{2} = \frac{1}{n}\sum_{i=1}^{n}(x_{i}-\overline{x})^{2}$.
We should notice the difference between U-statistic, V-statistic and grid sampling. In the U-statistic, we need to take average over distinct ordered samples of size $r$ taken from $\{1,\ldots,n\}$; in the V-statistic, each argument of the function $h$ will run over all sample points. Clearly, both are not the case in the grid sampling we run into in Roseland. Moreover, in the grid sampling, $h$ need not to be symmetric. However, the U-statistic and the V-statistic are both special cases of the form in equation (\ref{dep_sum}).

\subsection{Spectral Convergence}\label{Section:spectral convergence}
The point-wise convergence of Roseland to the Laplace-Beltrami operator in Section \ref{Section:PointConvergence} does not guarantee the spectral convergence. To fully understand the spectral based methods, we need to establish the spectral convergence, and this is the focus of this section. 
Let $\{v_{n}\}_{n\in\mathbb{N}}$ be a set of eigenvectors of the transition matrix $(D^{(\textup{R})})^{-1}W^{(\textup{R})}$ associated with the point cloud $\mathcal{X}=\{x_{i}\}_{i=1}^{n}\subseteq\mathbb{R}^{D}$. We would like to study when $n\to \infty$, how will the eigenvectors $\{v_{n}\}_{n\in\mathbb{N}}$ converge to the eigenfunctions of the Laplace-Beltrami operator. 
Note that the vectors $v_{n}$ are in difference Euclidean spaces for different $n$, and the eigenfunctions of the Laplace-Beltrami operator are smooth functions on $M$. Clearly, they cannot be compared directly and we need to manipulate those quantities a bit so that we can compare them. In brief, we will find a sequence of functions $f_{n}\in C(M)$, such that the restriction of $f_{n}$ on the point cloud $\mathcal{X}$ equals to entries of $v_{n}$; that is, $f_{n}(x_{i}) = v_{n}(i)$, for $i=1,\ldots,n$. Then, we study the convergence of $\{f_{n}\}$ as $n\rightarrow\infty$. 
To state our spectral convergence theorem, we need the following definitions and results. %

\begin{definition}\label{Definition various operators discrete}
	Take $\mathcal{X}=\{x_{i}\}_{i=1}^{n}\subset M$. 
	Define the following functions
	\begin{align}
	\widehat{K}_{\textup{ref},\epsilon,n}(x,y)&\,:=\frac{1}{m}\sum_{j}^{m}K_{\epsilon}(x,z_{j})K_{\epsilon}(z_{j},y)\in C(M\times M),\nonumber\\
	\widehat{d}_{\textup{ref},\epsilon,n}(x)&\,:=\frac{1}{n}\sum_{i=1}^{n}\widehat{K}_{\textup{ref},\epsilon,n}(x,x_{i})\in C(M),\\
	\widehat{M}_{\textup{ref},\epsilon,n}(x,y)&\,:=\frac{\widehat{K}_{\textup{ref},\epsilon,n}(x,y)}{\widehat{d}_{\textup{ref},\epsilon,n}(x)}\in C(M\times M)\,.\nonumber
	\end{align}
	Also define the following operator:
	\begin{align}
	\widehat{T}_{\textup{ref},\epsilon,n}f(x)&\,:=\frac{1}{n}\sum_{i=1}^{n}\widehat{M}_{\textup{ref},\epsilon,n}(x,x_{i})f(x_{i})\,.
	\end{align}
	Moreover, define
	the restriction operator $\rho_{n}:C(M)\rightarrow\mathbb{R}^{n}$ 
	\begin{equation}
	\rho_{n}:f\mapsto [f(x_{1}), f(x_{2}),\ldots, f(x_{n})]^{\top}\,.
	\end{equation}
\end{definition}

The following Lemma describes the relationship between $(D^{(\textup{R})})^{-1}W^{(\textup{R})}$ and the integral operator $\widehat{T}_{\textup{ref},\epsilon,n}$.

\begin{proposition}\label{one to one corresp}
	Let $U_{n}:=(D^{(\textup{R})})^{-1}W^{(\textup{R})}$, then $U_{n}\circ\rho_{n}=\rho_{n}\circ\widehat{T}_{\textup{ref},\epsilon,n}$. 
	Moreover, we have the following one to one correspondence.
	\begin{enumerate}
		\item If $f\in C(M)$ is an eigenfunction of $\widehat{T}_{\textup{ref},\epsilon,n}$ with the eigenvalue $\lambda$, then the vector $v\coloneqq\rho_{n}f$ is an eigenvector of $U_{n}$ with the eigenvalue $\lambda$.
		Moreover, suppose $\lambda\neq 0$ is an eigenvalue of $\widehat{T}_{\textup{ref},\epsilon,n}$ with the eigenfunction $f$. If we let $v\coloneqq\rho_{n}f$, then $f$ satisfies
		\begin{equation}\label{extrapolation of eigenvector to the whole manifold}
		f(x)=\frac{\sum_{j=1}^{n}\widehat{K}_{\textup{ref},\epsilon,n}(x,x_{j})v_{j}}{\lambda\sum_{j=1}^{n}\widehat{K}_{\textup{ref},\epsilon,n}(x,x_{j})}\,.
		\end{equation}
		\item If $v$ is an eigenvector of $U_{n}$ with the eigenvalue $\lambda\neq 0$, then $f$ defined in \eqref{extrapolation of eigenvector to the whole manifold} is an eigenfunction of $\widehat{T}_{\textup{ref},\epsilon,n}$ with the eigenvalue $\lambda$.
	\end{enumerate}
\end{proposition}

The proof of this proposition is standard, and can be found in, for example \cite{von2008consistency,singer2016spectral}, so we omit it.
With this Lemma, we now can simply study how the eigenstructure of $T_{\textup{ref},\epsilon,n}$ converges to that of $-\Delta$.

Denote $(\lambda_{i}, u_{i})$ to be the $i$-th eigenpair of $-\Delta$, where $\lambda_i$ is the $i$-th smallest eigenvalue. Note that under our manifold setup, by the well known elliptic theory, the spectrum of $-\Delta$ is discrete with $\infty$ as the only accumulation point, and each eigenspace is of finite dimension. { Also,
	denote $(\lambda_{\epsilon,n,i}, u_{\epsilon, n,i})$ to be the $i$-th eigenpair of $\frac{I-\widehat{T}_{\textup{ref},\epsilon,n}}{\epsilon}$, where $\lambda_{\epsilon,n,i}$ is the $i$-th smallest eigenvalue.}
We assume that both $u_{i}$ and $u_{\epsilon, n,i}$ are normalized in the $L^2$ norm.
With the above preparation,
we are now ready to state the main theorem.

\begin{theorem}[Spectral convergence]\label{spec_cong}
	Fix $K\in \mathbb{N}$. Suppose the kernel is Gaussian; that is, $K_{\epsilon}(x,y)=e^{-\norm{x-y}^2/\epsilon}$.
	Suppose $\lambda_{i}$ is simple. Suppose $m=n^\beta$, where $\beta\in (0,1)$, $\epsilon=\epsilon(n)$ so that $\epsilon \to 0$ and $\frac{\sqrt{-\log\epsilon}+\sqrt{\log m}}{\sqrt{m}\epsilon^{d}} \rightarrow 0$, as $n \rightarrow \infty$, and $\sqrt{\epsilon} \leq \mathcal{K}_1 \min \Bigg(\bigg(\frac{\min(\Gamma_K,1)}{\mathcal{K}_2+\lambda_K^{d/2+5}}\bigg)^2, \frac{1}{(2+\lambda_K^{d+1})^2}\Bigg)$, where $\Gamma_K$, $\mathcal{K}_1$ and $\mathcal{K}_2>1$  are introduced in Proposition \ref{T epsilon and Delta}.
	Then, when $p_Y$ is properly chosen so that $\frac{2\nabla p_{X}(x)}{p_{X}(x)}+\frac{\nabla p_{Y}(x)}{p_{Y}(x)}=0$, there exists a sequence of signs $\{a_{n}\}$ such that with probability $1-\mathcal{O}(m^{-2})$, for all $i<K$, we have
	\begin{align}
	\norm{a_{n}u_{\epsilon,n,i}-u_{i}}_{L^\infty}&=\mathcal{O}(\epsilon^{1/2})+\mathcal{O}\left(\frac{\sqrt{-\log \epsilon}+\sqrt{\log m}}{\sqrt{m}\epsilon^{2d+3/2}}\right)\,,\\
	\abs{\lambda_{\epsilon,n,i}-\lambda_{i}}&=\mathcal{O}(\epsilon^{3/4})+\mathcal{O}\left(\frac{\sqrt{-\log \epsilon}+\sqrt{\log m}}{\sqrt{m}\epsilon^{2d+2}}\right)\nonumber\,,
	\end{align}
	where the implied constants depend on the kernel, the curvature of $M$, $p_X$ and $p_Y$
\end{theorem}

Based on this theorem, if $\epsilon\to 0$ and $\frac{\sqrt{-\log \epsilon}+\sqrt{\log m}}{\sqrt{m}\epsilon^{2d+2}}\to 0$ when $n\to \infty$, the error term converges to zero. Note that if we want $\frac{\sqrt{-\log \epsilon}+\sqrt{\log m}}{\sqrt{m}\epsilon^{2d+2}}$ to be of the same order of $\epsilon^{3/4}$, we can choose $\epsilon=\big(\frac{\log m}{m}\big)^{1/(4d+11/2)}$. Thus, $\norm{a_{n}u_{\epsilon,n,i}-u_{i}}_{L^\infty}=O\big(\big(\frac{\log m}{m}\big)^{3/(16d+22)}\big)$ and $\abs{\lambda_{\epsilon,n,i}-\lambda_{i}}=O\big(\big(\frac{\log m}{m}\big)^{1/(8d+11)}\big)$.
We mention that the obtained convergence rate should not be the optimal bound. In fact, it is much slower than what we observed numerically. Also, as we will show below, it seems that the eigenvector convergence should be faster than the eigenvalue convergence, but this is not reflected by the above spectral convergence rate. How to obtain the ``correct'' convergence rate is however out of the scope of this paper, and will be explored in our future work.

\section{Noise analysis of Roseland under the manifold setup}\label{Section:RobustAnalysis}

In this section, we show the robustness of Roseland. This section is a companion of Section \ref{Section:PointConvergence}. %, so it is safe for readers with interest in numerics to skip it and jump directly to Section \ref{numerical}.
Again, we briefly summarize existing literature about the robustness of GL and DM.
From the statistical perspective, it is interesting to study the spectral behavior of the GL under the null assumption; that is, the data is purely noise. There have been several work in this direction, like \cite{el2010spectrum,cheng2013spectrum,do2013spectrum} and several others. However, to our knowledge, there are limited work studying how the inevitable noise impacts the performance of DM, and more generally the GL, except \cite{ElKaroui:2010a,el2016graph}. 
The proof strategy for Roseland is an extension of \cite{ElKaroui:2010a,el2016graph}, while extra efforts and technical tools will be applied to better quantify the convergence behavior. %Specifically, we provide a convergence rate.

%\subsection{Noisy manifold model for Roseland}\label{noisy_setup} 
On top of the manifold model, we assume that the data set and the landmark set from the ambient space $\mathcal{X} = \{x_{i}\}_{i=1}^{n}$ are corrupted by additive ambient space noise. That is, we observe $\widetilde{\mathcal{X}} = \{\tilde{x_{i}}\}_{i=1}^{n}$ and $\widetilde{\mathcal{Y}} = \{\tilde{y_{j}}\}_{j=1}^{m}$ :
$$\tilde{x_{i}} = x_{i} + \xi_{i},\quad\tilde{y_{j}} = y_{j} + \eta_{j}\,,
$$
where $\xi_{i}$ and $\eta_j$ are noise. We assume that $x_i$, $y_j$, $\xi_i$ and $\eta_j$ are independent. 
For practical purpose, we consider the case when the ambient space dimension grows asymptotically as the dataset size $n$. To capture the manifold structure for the high dimensional dataset, we consider the following model.
\begin{assumption}[High dimensional model]\label{Assumption manifoldH}
	Fix a compact smooth $d$-dim Riemannian manifold $M$ isometrically embedded into $\mathbb{R}^D$ via $\iota$. Assume $K=\max_{x,y}\|\iota(x)-\iota(y)\|_{\mathbb{R}^D}>0$. Assume $q=q(n)\asymp n$ when $n\to \infty$. Fix an isometrical embedding $\bar{\iota}_q:\mathbb{R}^D\to \mathbb{R}^q$. Further, assume that $x_i$ and $y_j$ are i.i.d. sampled from $\iota_q(M)$, where $\iota_q:=\bar{\iota}_q\circ \iota$. 
\end{assumption}

There are several possibilities to model a high dimensional data. For example, we can embed $\mathbb{R}^D$ into the first $D$ axes of $\mathbb{R}^q$ via $\bar{\iota}_q$, and sample the data. But depending on the problem, we may need a different model. For example, if the point cloud represents images, and the ambient space dimension represents the image resolution, then obviously the manifold representing the clean image is not embedded in only the first few axes. 
The model in Assumption \ref{Assumption manifoldH} is suitable for that purpose. 
We now consider a noise model on top of Assumption \ref{Assumption manifoldH}.

\begin{assumption}[Noise model]\label{Assumption Noise}
	Suppose the noise contaminating the dataset is $\xi_{i}\in \mathbb{R}^q$ with mean $0$ and covariance $\Sigma$
	and the noise contaminating the landmark set is $\eta_{j}\in \mathbb{R}^q$ with mean $0$ and covariance $\overline{\Sigma}$, which is independent of $\xi_i$. 	Suppose $\norm{\Sigma}_{2}\leq\sigma_{q}^{2}$ and $\|\overline{\Sigma}\|_{2}\leq\overline{\sigma}_{q}^{2}$, where $\sigma_q\geq 0$ and $\overline{\sigma}_q\geq 0$.
	Assume for all convex $1$-Lipschitz function $f$, $\mathbb{P}(|f(\xi_{i}-\eta_j)-m_{f(\xi_{i}-\eta_j)}|>t)\leq 2\exp(-c_{ij}t^{2})$, where $m_{f(\xi_{i}-\eta_j)}$ is the median of $f(\xi_{i}-\eta_j)$ and $c_{ij}>0$.	
\end{assumption}

The commonly considered white Gaussian noise satisfies the previous assumptions with $c_{ij} = 1/\sqrt{2}$, independently of the dimension. 
Under the Gaussian noise model, the noisy data can be viewed as sampled from a Gaussian mixture model $\tilde{X}$, where the law $p_{\tilde{X}}$ can be written as:
\begin{align}
p_{\tilde{X}}(\tilde{x}) =&\, \int_{M}\mathcal{N}(\tilde{x}-\iota(z),\Sigma)p_{X}(z)\,dV(z) 
=  \frac{1}{(2\pi)^{-q/2}\sqrt{\text{det}(\Sigma)}}\int_{M}e^{-\frac{1}{2}(\tilde{x}-\iota(z))\Sigma^{-1}(\tilde{x}-\iota(z))}p_{X}(z)\,dV(z)\,.\nonumber 
\end{align}

Recall the Roseland algorithm in Section \ref{roseland alg}. Denote the Roseland embedding with diffusion time $t$ from clean data $\mathcal{X}$ and noisy data $\widetilde{\mathcal{X}}$ by $\Phi L^{t}\subset\mathbb{R}^{n\times q'}$ and $\widetilde{\Phi}\tilde{L}^{t}\subset\mathbb{R}^{n\times q'}$ respectively. We now investigate the discrepancy between the two embeddings. Since the embeddings are free up to rotations and reflections, we quantify $\|\Phi L^{t}O - \widetilde{\Phi}\widetilde{L}^{t}\|_{F},$ where $O\in\mathbb{R}^{q'\times q'}$ is some orthogonal matrix {that aligns the column subspaces spanned by $\Phi$ and $\widetilde{\Phi}$}, and $\norm{\cdot}_{F}$ is the Frobenius norm. 
We have the following main theorem.

\begin{theorem}[Robustness of Roseland]\label{robust thm roseland}
	Assume the point clouds $\{x_{i}\}_{i=1}^{n}$ and $\{y_{j}\}_{j=1}^{m}\subseteq\mathbb{R}^{q}$ are i.i.d. sampled from the high dimensional model satisfying Assumption \ref{Assumption manifoldH}. 
	Let $\tilde{x_{i}} = x_{i} + \xi_{i}$ and $\tilde{y}_{j} = y_{j} + \eta_{j},$ where the noises $\xi_i$ and $\eta_j$ satisfy Assumption \ref{Assumption Noise} and are independent of $x_i$ and $y_j$.
	Assume $\sup_{i,j}\sqrt{(\sigma^2_{q}+\overline{\sigma}^2_{q})/c_{ij}}\sqrt{\log nm}\to 0$.
	Set
	\begin{align}
	\delta_q:=\sqrt{\log nm}\sqrt{\sigma_{q}^{2}+\overline{\sigma}_{q}^{2}}\left[   \sup_{i,j}  \sqrt{c^{-1}_{ij}}\left(\sqrt{q(\sigma_{q}^{2}+\overline{\sigma}_{q}^{2})}\vee 1\right)+K\right]\,.\label{Assumption:CorollarySI34}
	\end{align}
	Fix $q'\in \mathbb{N}$. According to Theorem \ref{spec_cong}, pick a sufficiently small $\epsilon=\epsilon(q)>0$ so that the first $q'$ non-trivial singular values are sufficiently away from zero when $n$ is sufficiently large. 
	Denote $W^{(\textup{r})}$ and $\widetilde{W}^{(\textup{r})}$ to be the landmark-set affinity matrices from clean and noisy datasets respectively. Denote $\Phi L^{t}\in\mathbb{R}^{n\times q'}$ and $\widetilde{\Phi}\widetilde{L}^{t}\in\mathbb{R}^{n\times q'}$ to be Roseland embeddings from $W^{(\textup{r})}$ and $\widetilde{W}^{(\textup{r})}$ respectively. 
	Then, %there exists $f>0$ such that 
	for fixed $t>0$ and $q'\in \mathbb{N}$, we have
	\begin{align}
	\norm{\Phi L^{t}O - \widetilde{\Phi}\widetilde{L}^{t}}_{F}
	=\mathcal{O}_{P}\left(\frac{\delta_q}{\sqrt{m}} \frac{1+ts_2^{2t-2}+s_2^{2t}}{\epsilon^{2d+1} }
	\right)	\,,\nonumber
	\end{align}
	where {the implied constant depends on $q'$}, $O\in\mathbb{R}^{q'\times q'}$ is an orthogonal matrix {that aligns the column subspaces spanned by $\Phi$ and $\widetilde{\Phi}$}, and $s_{2}$ are the largest non-trivial singular value of Roseland from the clean data.
\end{theorem}

Note that the term $q(\sigma_q^2+\overline{\sigma}_q^2)$ in $\delta_q$ can be viewed as the total energy of noise. Thus, the result in the Theorem says that the bandwidth $\epsilon$ should chose ``large'' enough so that embedding is less impacted. This result can be intuitively understood as ``noise can be canceled when we have sufficient noise information''.

Next, suppose the noise levels of dataset and landmark set are the same and $\sup_{i,j}\sqrt{c^{-1}_{ij}}$ is bounded. In this case, if $\frac{\delta_q}{\sqrt{m}\epsilon^{2d+1}}\to 0$, we have $\norm{\Phi L^{t}O - \widetilde{\Phi}\widetilde{L}^{t}}_{F}\to 0$. 
This fact indicates that Roseland can tolerate big noise. To take a closer look at this claim, note that the choice of $\epsilon=\epsilon(q')$ in the theorem should satisfies $\epsilon\to 0$ and $\frac{\sqrt{-\log\epsilon}+\sqrt{\log m}}{\sqrt{m}\epsilon^{2d+2}} \rightarrow 0$ so that the spectral convergence in Theorem \ref{spec_cong} holds.
Suppose we choose $\epsilon=\big(\frac{\log m}{m}\big)^{2/(8d+11)}$. Therefore, if $\delta_q=O(1)$, we have $\frac{\delta_q}{\sqrt{m}\epsilon^{2d+1}}\to 0$.
Since $\sqrt{\log(nm)}\sqrt{q}\sigma_{q}^{2}\to 0$ holds when $\sigma_{q}=q^{-(1/4+a)}$ for any small constant $a>0$, and the total noise energy in this setup satisfies $\sigma_{q}^{2}q\to\infty$, we know that $\delta_q\to 0$. The relationships $\sigma_{q}=q^{-(1/4+a)}$ and $\sigma_{q}^{2}q\to\infty$ mean that while entrywisely the noise shrinks, the total noise energy blows up.

Moreover, when the landmark set is clean, for example, we can collect a clean landmark set by a high quality equipment, we could achieve a better convergence since $\overline{\sigma}_q=0$. See Section \ref{numerical} for numerical results.

When $t=0$, we recover the eigenmap algorithm \cite{belkin2003laplacian}. In this case, we know $ts_2^{2t-2}=0$ and $s_2^{2t}=1$. One the other hand, when $t\to \infty$, all non-trivial eigenvalues are gone, and we only receive the topological information of the manifold, which in our setup is the connectivity. Indeed, when the manifold is connected, since the trivial singular value and singular vector are not considered in the Roseland embedding, we get $\Phi L^tO\to 0$ and $\widetilde{\Phi}\widetilde{L}^t\to 0$, and $ts_2^{2t-2}\to 0$ and $s_2^{2t}\to 0$ when $t\to \infty$.

\section{Numerical Results}\label{numerical}

{
	
To illustrate how Roseland performs, in addition to showing the dimension reduction and geometric recovery results, we also compare the results with the Nystr\"om extension and HKC. 
For a fair comparison, in all the following simulations, the subset used in the Nystr\"om extension and the reference set used in HKC to embed the dataset are the same as the landmark set used in Roseland. As a result, the ranks of the matrices associated with Roseland, HKC and the Nystr\"om extension are the same. {All of the simulations were done on a Linux machine with 4-core 3.5Ghz i5 CPUs and 16GB memory.}
{The Matlab code used in this paper can be found in \url{https://github.com/shenchaojerry/Roseland_code} for the reproducibility issue.}

\subsection{Related algorithms and complexity analysis}
\subsubsection{The HKC algorithm}

HKC was proposed for the texture separation problem. The authors proposed to first divide an image into a collection of small patches, from which to choose a subset consists of specific patterns of interest as the {\em reference set}. 
Note that in \cite{haddad2014texture}, the reference set plays the same role as the landmark set in Roseland.
Then one can construct an affinity matrix associated of the patches based on the landmark set. HKC is the closest algorithm to Roseland among others. However, the normalization in HKC is different from Roseland, and this difference turns out to be significant. Moreover, it is not clear how does HKC performs under the manifold setup. 

We now summarize the HKC algorithm \cite{haddad2014texture}. Firstly, form the affinity matrix between the data set and the landmark set just like \eqref{landmark affinity matrix in Roseland} 
in Roseland; that is, set $W^{(\textup{HKC})}=W^{(\textup{r})}$. 
HKC then compute a $n\times n$ diagonal matrix by
$D^{(\textup{HKC})}_{ii}=\sum_{j=1}^m W^{(\textup{HKC})}_{i,j}$,
where $i=1,\ldots,n$. Then, convert $W^{(\textup{HKC})}$ to be row stochastic by:
\begin{align}\label{HKC affinity}
A^{(\textup{HKC})} = (D^{(\textup{HKC})})^{-1} W^{(\textup{HKC})}\in \mathbb{R}^{n\times m}\,.
\end{align}
One should notice the difference between Roseland and HKC when computing the degree matrix. In Roseland, the degree matrix is computed from the row sum of the matrix $W^{(\textup{HKC})}(W^{(\textup{HKC})})^{\top}$ instead of the row sum of $W^{(\textup{HKC})}$. Therefore, Roseland defines a Markov process on the data set, but HKC does not.
This normalization step plays a significant role.
Finally, HKC embeds the data via the eigenvectors $\psi_{j}$ of the matrix $\bar W^{(\textup{HKC})} = A^{(\textup{HKC})}(A^{(\textup{HKC})})^{\top}\in\mathbb{R}^{n\times n}$, which can be computed efficiently by
\begin{align}\label{HKC extend}
\psi^{(\textup{HKC})}_{j}=(\lambda^{(\textup{HKC})}_{j})^{-1/2}A^{(\textup{HKC})}\phi^{(\textup{HKC})}_{j}
\end{align}
where $\phi^{(\textup{HKC})}_{j}$ is the $j$-th eigenvectors of the matrix $\tilde W^{(\textup{HKC})}:=(A^{(\textup{HKC})})^{\top}A^{(\textup{HKC})}\in\mathbb{R}^{m\times m}$ associated with the eigenvalue $\lambda^{(\textup{HKC})}_{j}$. The denominator $(\lambda^{(\textup{HKC})}_{j})^{1/2}$ is to ensure that $\|\psi^{(\textup{HKC})}_{j}\|_{2}=\|\phi^{(\textup{HKC})}_{j}\|_{2}$.
In summary, we see that HKC is close to Roseland with a different normalization.

\subsubsection{Nystr\"{o}m Extension}

Another widely applied algorithm aiming to scale up spectral embedding is the Nystr\"{o}m extension \cite{belabbas2009landmark,fowlkes2004spectral,williams2001using,belabbas2009spectral, coifman2006geometric}. The idea is simple but effective. First, run the eigen-decomposition on a subset of the given dataset. Then, extend the eigenvectors to the whole dataset. There are some variants of the Nystr\"om extension method, for example, \cite{bermanis2013multiscale}.
A direct application is out of sample embedding.

We are interested in applying the Nystr\"om extension to the eigen-decomposition of the transition matrix defined in \eqref{Eq A ordinary DM} for the spectral embedding purpose. Similarly, we consider 
the symmetric kernel matrix $M=D^{-1/2}WD^{-1/2}$. Note that if we want to apply Nystr\"om extension directly on $M$, we would have to compute the affinity matrix $W$ and the degree matrix $D$, which is expensive and the kNN scheme is needed. 
We thus follow the existing literature \cite{lafon2006data,singer2012vector} and apply the modified Nystr\"om extension. 

Suppose we have $n$ data points. First, run DM on a chosen subset, also called the landmark set, which is of size $L=n^\beta$, where $\beta\in (0, 1)$. Denote the affinity matrix associated with this landmark set as $W_{L}$, and run the eigen-decomposition of the matrix $D^{-1/2}_{L}W_{L}D_L^{-1/2} = V_{L}\mathcal{L}_{L}V_{L}^{\top}$, where $D_{L}$ is the degree matrix associated with $W_{L}$ and $\mathcal{L}_L=\texttt{diag}\begin{bmatrix}\ell_1 &\ldots &\ell_L \end{bmatrix}\in \mathbb{R}^{L\times L}$. Let $\widetilde{U}_{L}=D_{L}^{-1/2}V_{L}$ to be the eigenvectors of $D^{-1}_{L}W_{L}$. We then extend it to the rest $n-L$ points by:
\begin{equation}
\check{U}_{\textup{ext}} = D^{-1}_{n-L}E\widetilde{U}_{L}\mathcal{L}_{L}^{-1}\in \mathbb{R}^{(n-L)\times L}\,,
\end{equation}
where $E\in \mathbb{R}^{(n-L)\times L}$ is the affinity matrix between the remaining $n-L$ data points and the landmark set. In other words, $E_{i,j}$ is the similarity between point $x_i$ in the remaining dataset and $x_j$ in the landmark set; $D_{n-L}$ is a $(n-L)\times (n-L)$ diagonal matrix such that $D_{n-L}(i,i) = \sum_{j=1}^{L}E(i,j)$. Hence the eigenvectors used to embed the whole dataset is:
\begin{align}
\check{U} &= 
\begin{bmatrix} 
\widetilde{U}_{L}\\
\check{U}_{\textup{ext}}
\end{bmatrix} 
=\begin{bmatrix} 
D^{-1}_{L} &  \\
& D^{-1}_{n-L}
\end{bmatrix}
\begin{bmatrix} 
W_{L}  \\
E
\end{bmatrix}
\widetilde{U}_{L}\mathcal{L}^{-1}_{L}\,.
\end{align}
Roughly speaking, the embedding coordinates of a data point $x$ outside the landmark set is simply the average of all of the landmarks' embeddings, weighted by the similarity between $x$ and all the landmarks.

While it is slightly different from the original Nystr\"om extension, we still call it the Nystr\"om extension. Note that in practice, we only need to calculate $W_{L}$ and $E$ instead of $W$ and $D$, which is more efficient in the sense of both time and spatial complexities.
With the estimated eigenvectors on the whole dataset, we can define the associated embedding and hence the distance just as in Roseland and DM. Specifically, suppose we have $\check{U}=\begin{bmatrix} \tilde{u}_1 & \ldots& \tilde{u}_L \end{bmatrix}\in \mathbb{R}^{N\times L}$ and $\check{\mathcal{L}}=\texttt{diag}\begin{bmatrix}\ell_1 &\ldots &\ell_L &0&\ldots &0\end{bmatrix}\in \mathbb{R}^{n\times n}$. Then we can define the associated embedding by %and DD by
\begin{equation}
\Phi^{(\textup{Nystr\"om})}_t:x_{i}\mapsto e_i^\top \check{U}_{q'}\check{\mathcal{L}}_{q'}^t\,,
\end{equation}
where $t>0$ is the chosen diffusion time, $\check{U}_{q'}\in \mathbb{R}^{n\times q'}$ to be a matrix consisting of $\tilde{u}_2,\ldots,\tilde{u}_{q'+1}$ and $\check{\mathcal{L}}_{q'}:=\texttt{diag}(\ell_2,\ldots,\ell_{q'+1})$.

\subsubsection{Complexity analysis}

Let $n$ be the size of the dataset and $n^\beta$ the size of the landmark set, where $\beta\leq 1$ throughout. We set $k=n^\beta$ whenever kNN is applied.
For a dense affinity matrix, the { space} complexity of DM is $O(n^2)$. If kNN is applied, the { space} complexity of DM is $O(n^{1+\beta})$. 
On the other hand, no matter what kernel is chosen, compactly supported or not, the { space} complexity of Roseland and the Nystr\"om extension is $O(n^{1+\beta})$.

For the computational complexity, it can be divided into two parts. The first part is forming the affinity matrix and the corresponding degree matrix; the second part is performing the eigen-decomposition or SVD. In the ordinary DM, the construction of the affinity matrix and the degree matrix is $O(n^2)$. If the kNN construction is considered and the k-d tree based algorithm is applied, the averaged time complexity of constructing the affinity matrix and the degree matrix is $O(n\log(n)+n^{1+\beta})=O(n^{1+\beta})$. In the Nystr\"om extension, the construction of the  $W_L$ and hence its degree matrix is $O(n^{2\beta})$ when $L=n^\beta$ for $\beta\leq 1$, while the construction of $E$ and $D_{n-L}$ is {  $O(n^{1+2\beta})$}. Thus, the first part complexity for the Nystr\"om extension is {  $O(n^{1+2\beta})$}.
In Roseland, the construction of the landmark-set affinity matrix and its associated degree matrix is $O(n^{1+\beta})$.
For the second part, it falls in the discussion of the complexity of the general eigen-decomposition and SVD. For a symmetric kernel matrix $M\in\mathbb{R}^{N\times N}$, the eigen-decomposition complexity is usually $O(N^3)$,\footnote{Theoretically, it can reach ${O}(N^{\omega+\eta})$, where the $N^{\omega}$ part comes from the algorithm of matrix multiplication, and an arbitrary $\eta>0$ \cite{Demmel2007}. Note that when $M$ is dense, $\omega=\omega_0\approx 2.376$ \cite{Coppersmith1990}. However, the implied constant in these asymptotic is too large and cannot be practical \cite{le2012faster}.} and when $M$ is $k$ sparse, where $k\leq n$, the complexity can be improved to $O(N^{2+\eta'})$ for an arbitrary $\eta'>0$ when $k\leq N^{0.14}$ \cite{Yuster2004}. In our application, even if we make $k\leq N^{0.14}$, the eigen-decomposition of the $M$ is roughly ${O}(N^{2+\eta'})$. On the other hand, for a matrix of size $N\times N'$, where $N\geq N'$, then the complexity of the SVD for  is ${O}(NN'^2)$. 
Hence, the overall computational complexity for the ordinary DM is $O(n^3)$ and is $O(n^\omega)$ for DM with the kNN scheme, where $\omega>2$ depends on the chosen $\beta$,  {  $O(n^{1+\beta}+n^{3\beta})$}  for the Nystr\"om extension, and $O(n^{1+2\beta})$ for Roseland. The complexity of HKC is the same as that of Roseland. 
To summarize, both the Nystr\"om extension and Roseland are more efficient than the ordinary DM with or without the kNN scheme. While Roseland is not faster than the traditional Nystr\"om extension approach, it is comparable, particularly for small $\beta$.

}

\subsection{Scalability of Roseland}

We take the dataset consisting of random projections of the two-dimensional Shepp-Logan phantom \cite{singer2013two}. 
A phantom is a $2$-dim image function $\psi$ compactly supported on $\mathbb{R}^{2}$ without any symmetry assumption. It is commonly applied in medical imaging society as a benchmark. 
Suppose we uniformly sample $n$ points from $S^{1}$, $\theta_1\ldots \theta_n\in S^1$, as the projection angles.
Then we generate a high dimensional data set by taking the Radon transform of $\psi$, denoted as $R_{\psi}:S^1\to L^2(\mathbb{R})$, followed by discretizing the projection image into $p\in \mathbb{N}$ points; that is, we have the dataset $\mathcal{X}:=\{D_pR_{\psi}(\theta_i)\}_{i=1}^n\subset \mathbb{R}^p$, where $D_p$ is the discretization operator. We refer readers with interest to \cite{singer2013two} for details. 
In this simulation, we fix the number of discretization points $p=128$ and let the number of projections $n$ vary. 
We run DM, Roseland, HKC and the Nystr\"om extension with $n=10,000$ and $m=n^\beta$, where $\beta=0.5$, and show the 3-dim embedding of $\mathcal{X}$ in Figure \ref{Figure:phantom}. 
Clearly, both DM and Roseland recover the $S^1$ structure, while Roseland is distorted, and HKC and the Nystr\"om extension are confused and lead to erroneous embeddings.
The computational times of different algorithms with $\beta=0.3$ are shown for a comparison. When $n=1.28\times 10^6$, Roseland finishes in about 2.5 minutes.

\begin{figure}[hbt!]\centering
	\includegraphics[width=2.95cm]{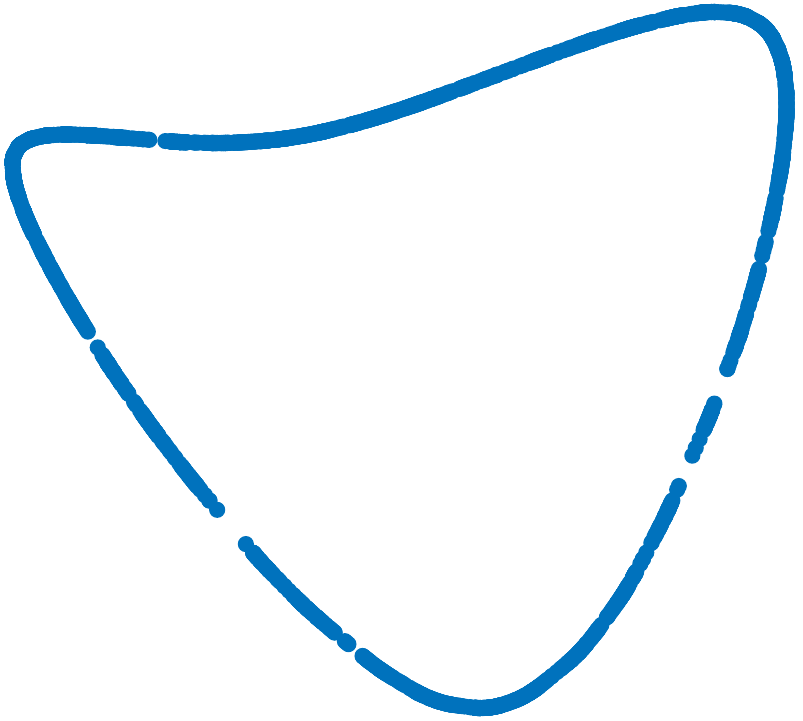}
	\includegraphics[width=2.95cm]{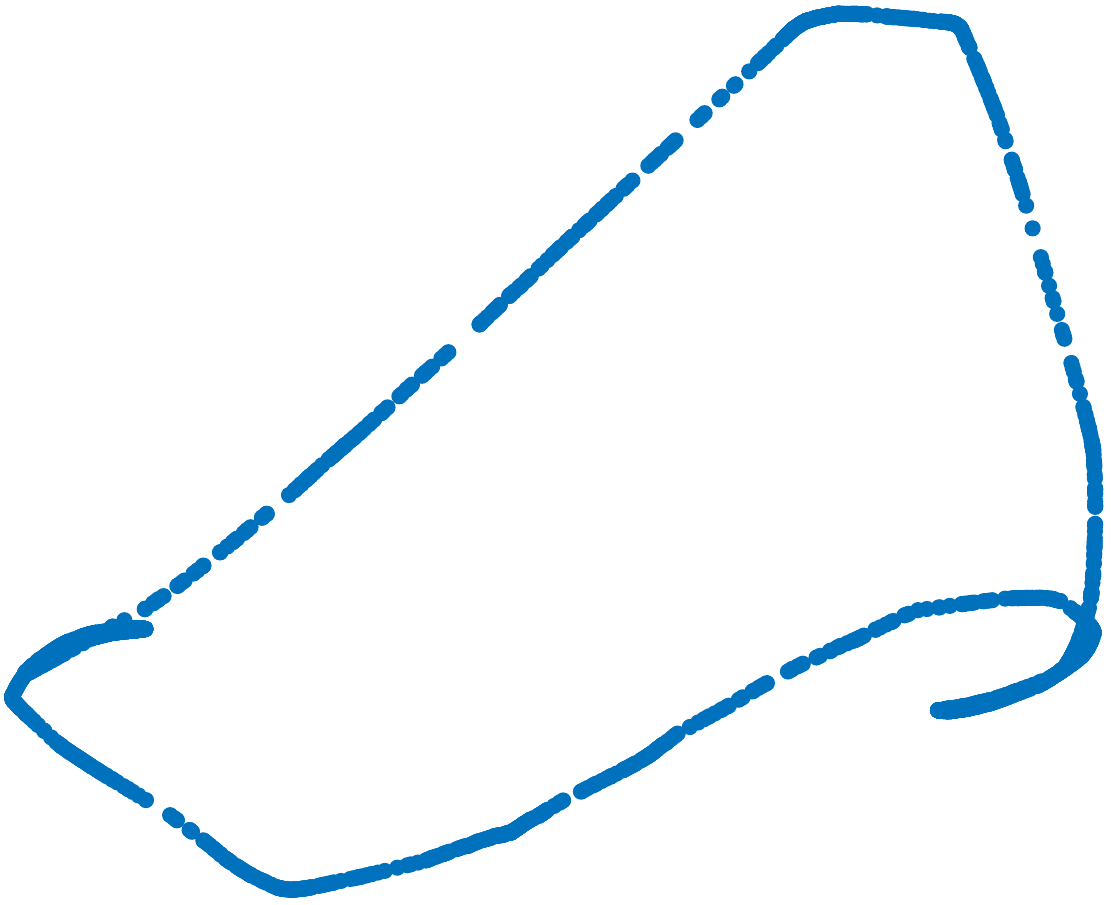}
	\includegraphics[width=2.95cm]{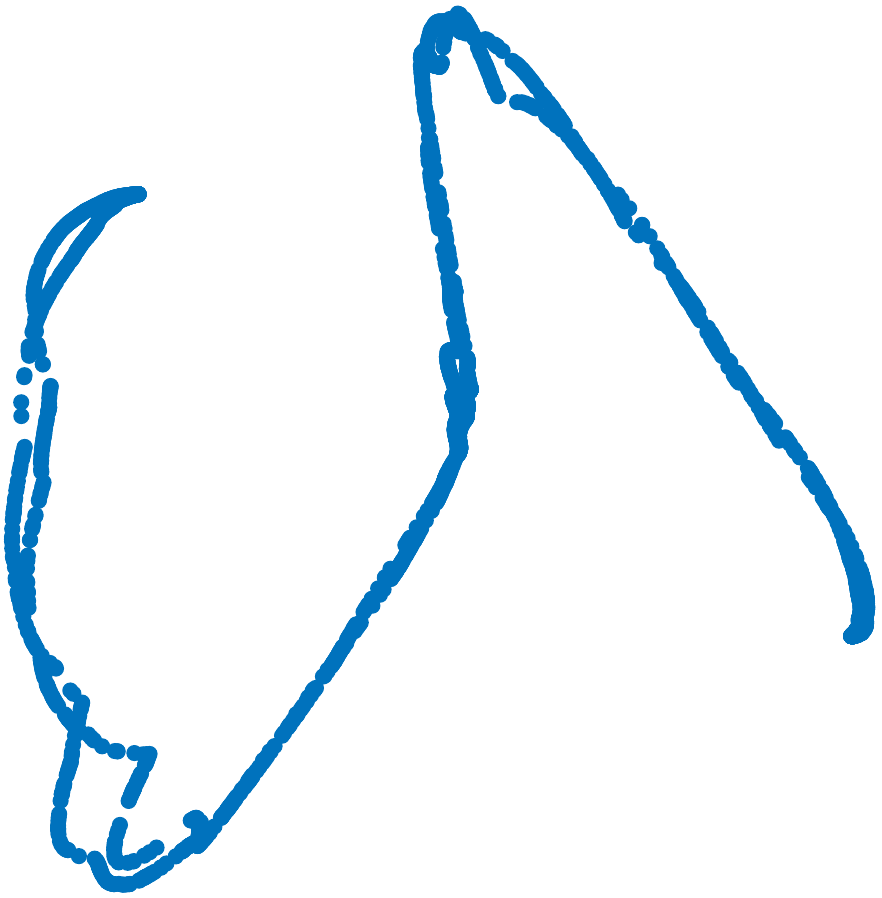}
	\includegraphics[width=2.95cm]{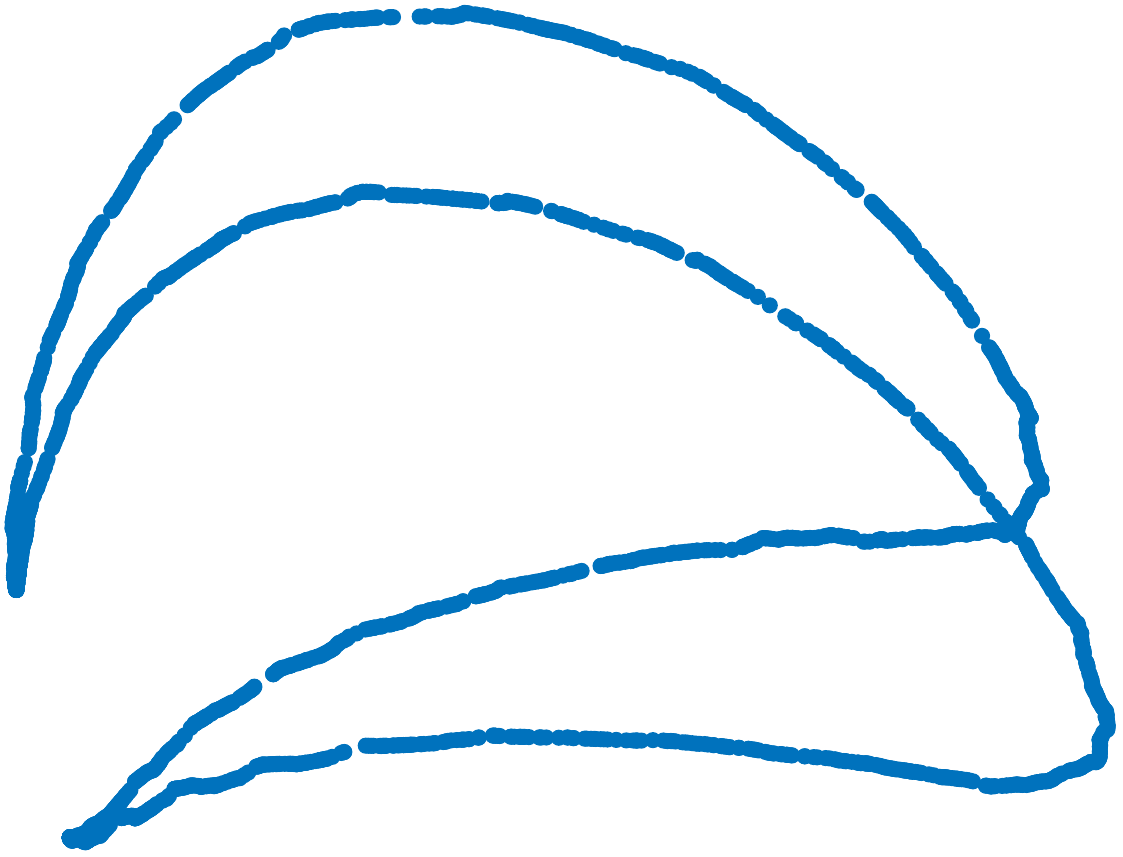}\\
	\includegraphics[width=5cm]{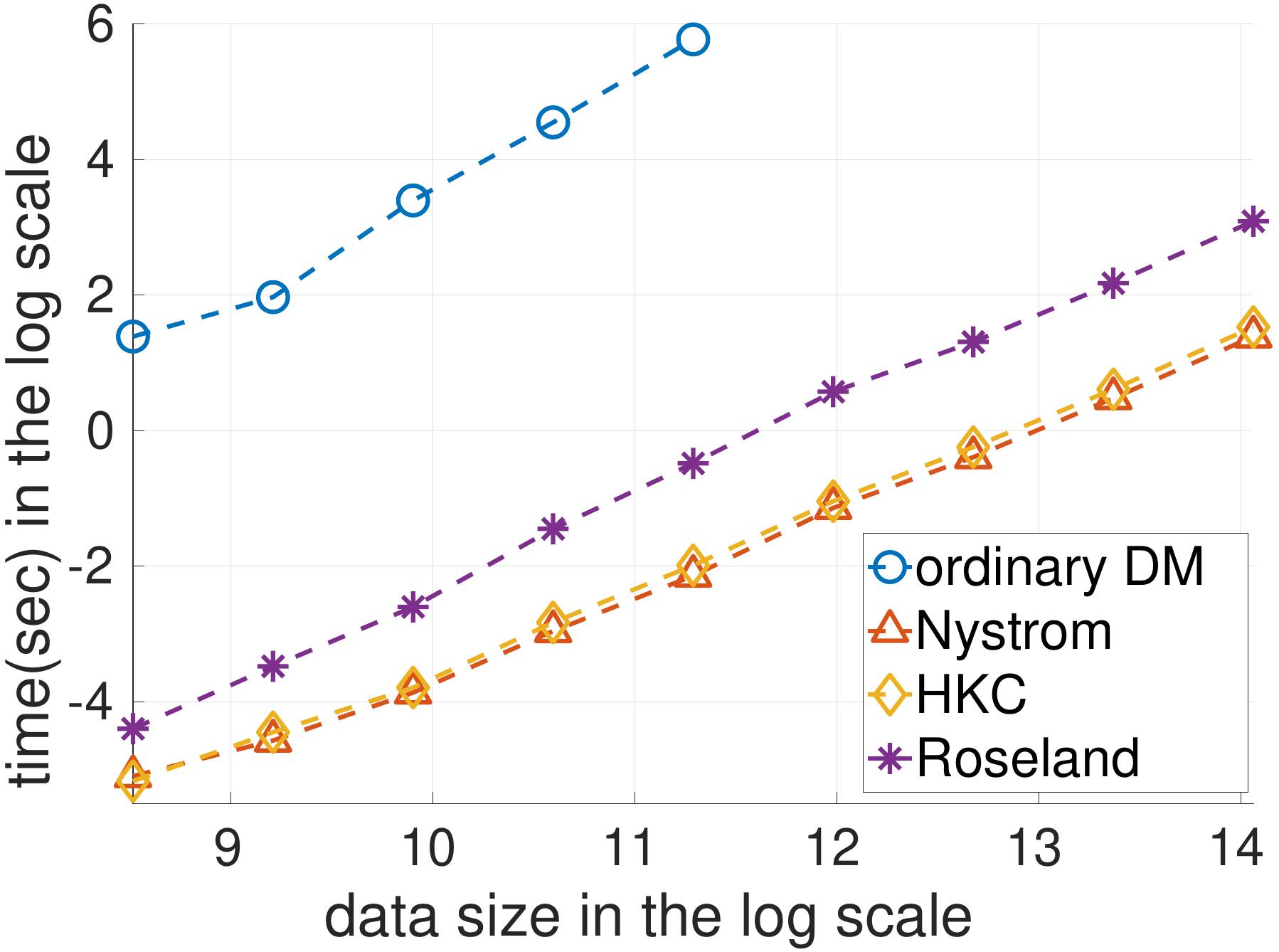}
	\caption{The random projection data from the phantom image, where the data size is $n=10,000$ and the dimension is $p=128$. Top row, from left to right: the DM embedding, Roseland, HKC, and the Nystr\"om extension, where we take $\beta =0.5$ for Roseland, the Nystr\"om, and HKC. All embeddings are 3-dim, and have been rotated to optimize the visualization. Bottom row: the runtime comparison of various algorithms when $\beta=0.3$. %on the left and $\beta=0.5$ on the right. 
		The x-axis is in the natural log unit, and the largest database size is $1,280,000$. %In the bottom row, we show the relationship between the computational time and the data size. 
		\label{Figure:phantom}}
\end{figure}

\subsection{Robustness of Roseland}\label{numerical section noisy data and subset}

We compare performance of the Nystr\"om extension, HKC, and Roseland from the aspect of spectral embedding when the data is noisy.
We consider the standard $S^1$ model, which is the one-dimensional canonical $S^1$ embedded in the first two coordinates of $\mathbb{R}^{100}$, since all ground truths can be analytically calculated. 
Specifically, we uniformly sample $n=90,000$ points from the $S^1$ to be the dataset and independently sample another $m=3 00$ points uniformly to be the landmark set; that is, $\beta=0.5$. Then, embed all points to $\mathbb{R}^{p}$, where $p=100$, and add independent Gaussian noise $\epsilon_{i}$ to both the dataset and the landmark set, where $\epsilon_{i}$ are i.i.d. sampled from $\mathcal{N}(0, \frac{1}{\sqrt{p}} I_{p\times p})$.
The visualization results are shown in Figures \ref{S1noisy-noisy_embedding}. %, . 
Clearly, while the Nystr\"om extension and HKC embed $S^1$ successfully, the embedding by Roseland is cleaner. 

\begin{figure}[bht!]\centering
	\includegraphics[width=2.95cm]{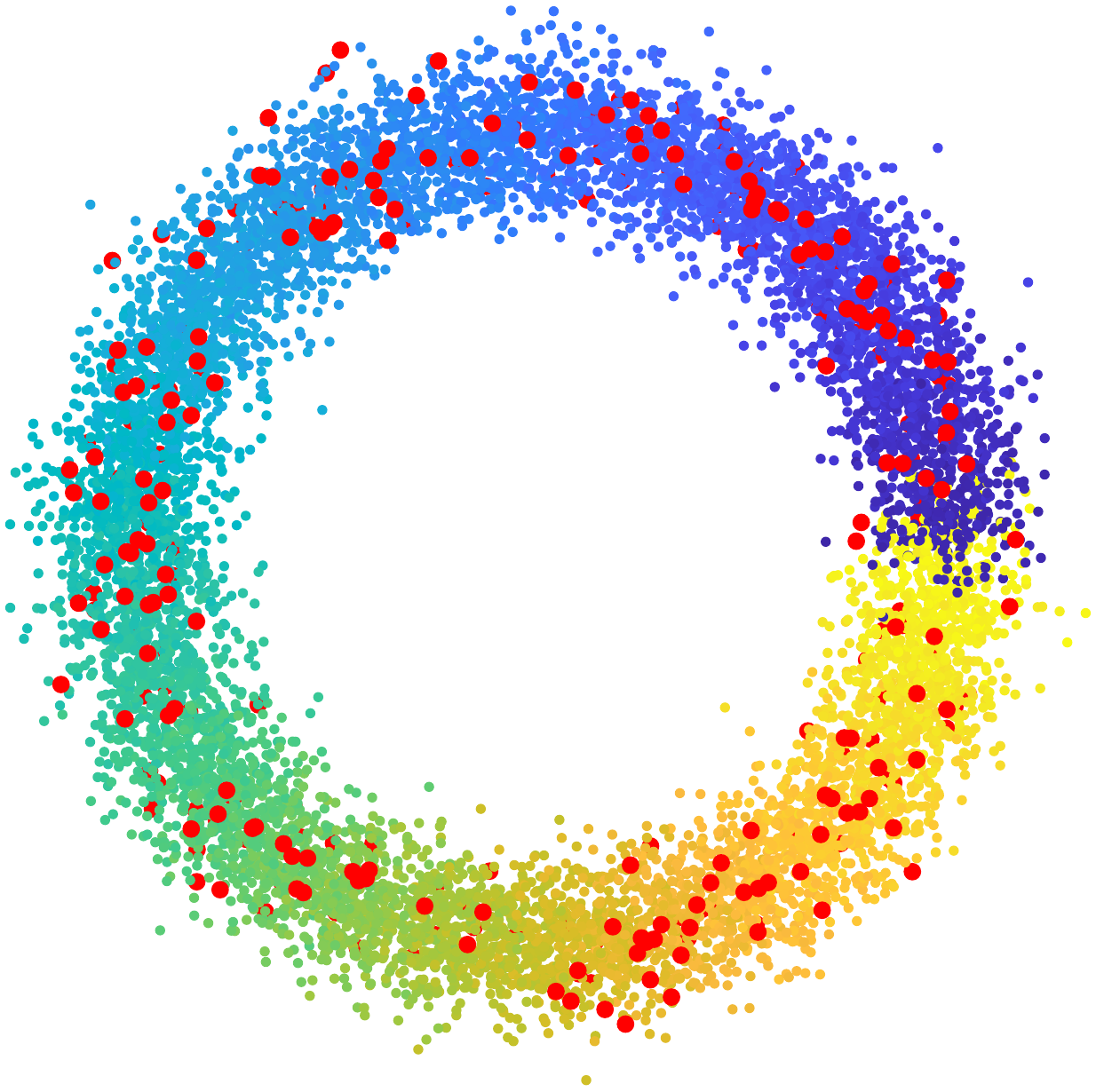}
	\includegraphics[width=2.95cm]{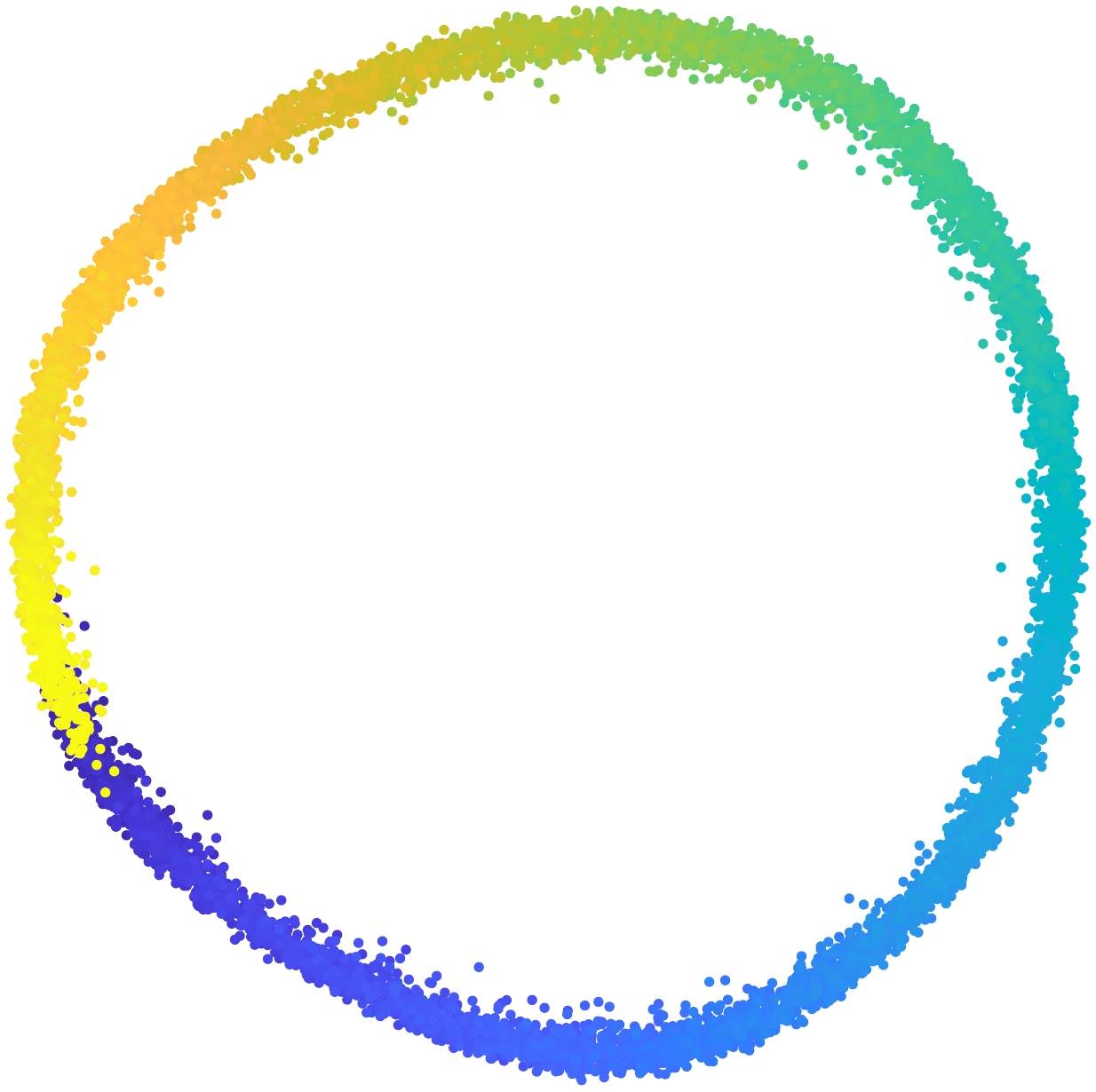}
	\includegraphics[width=2.95cm]{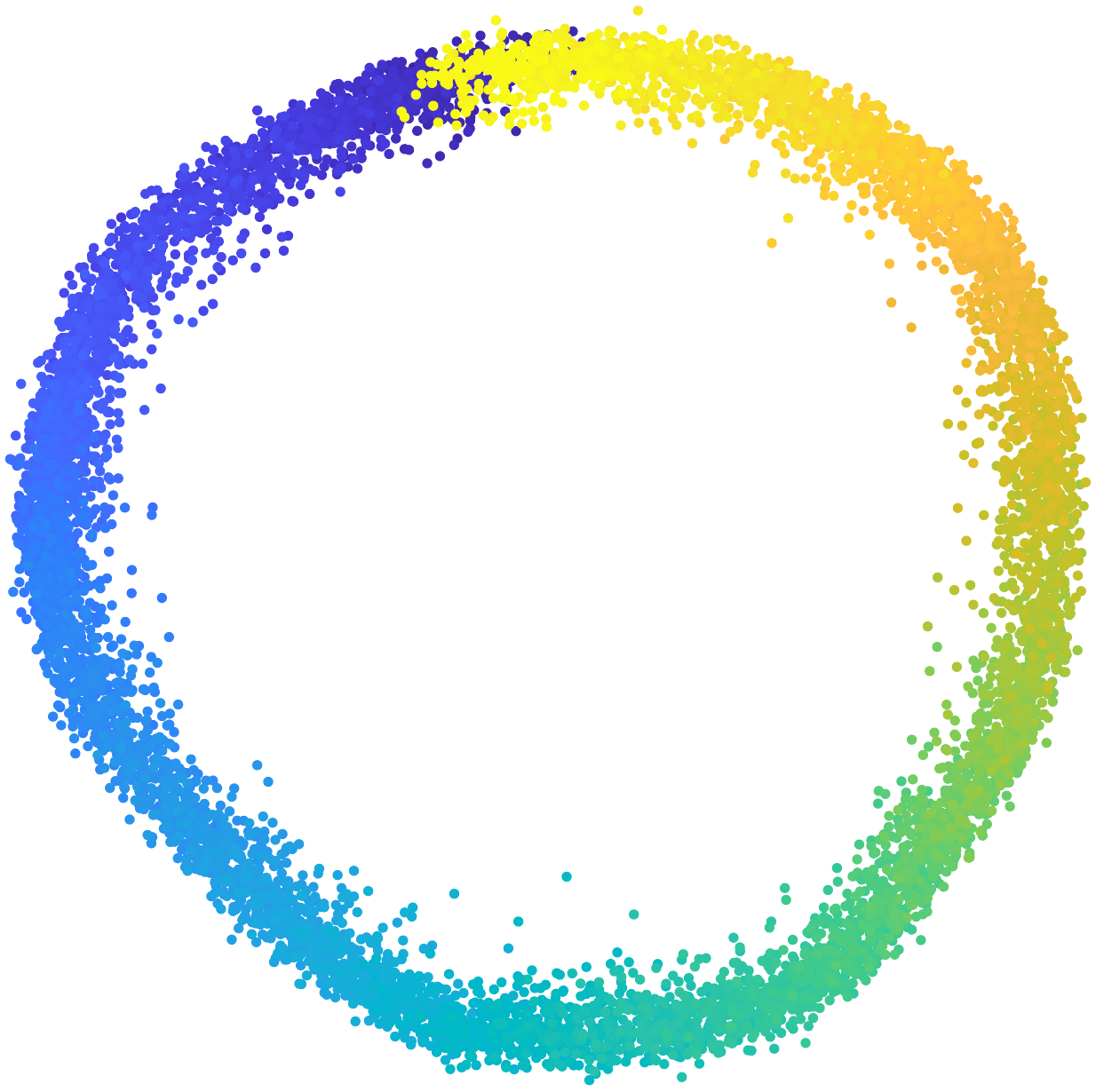}\includegraphics[width=2.95cm]{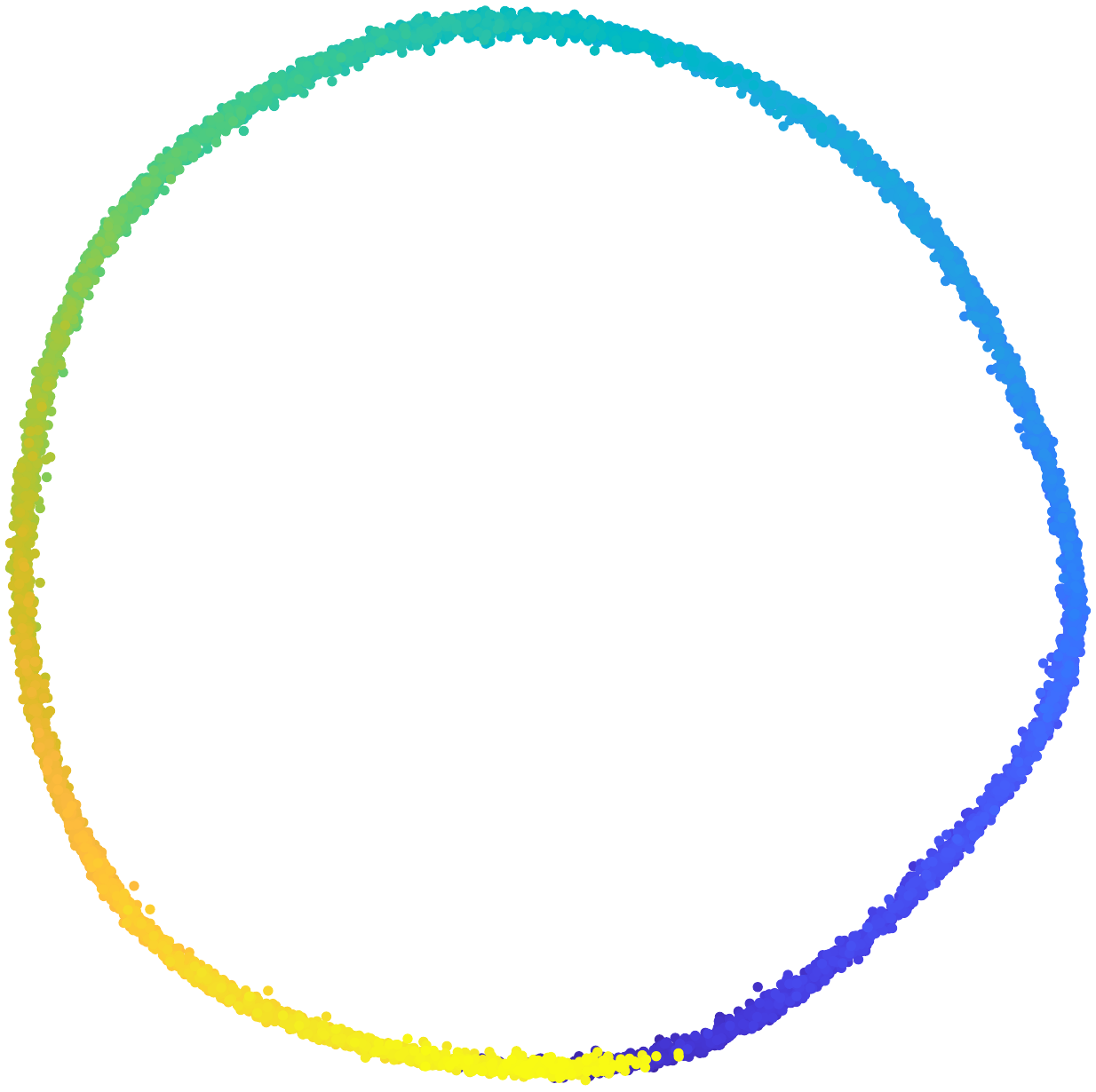}
	\vspace{-5pt}
	\caption{Left: noisy data and noisy subset (only the first two coordinately are shown). Middle left: the Nystr\"om embedding. Middle right: the HKC embedding. Right: the Roseland embedding. \label{S1noisy-noisy_embedding}}
\end{figure}

Next, the recovered eigenvectors are shown in Figures \ref{S1noisy-noisy_eigenvector} and \ref{S1noisy-noisy_eigenvalue}. Clearly, the Nystr\"om method can only successfully recover first few eigenfunctions of the Laplace-Beltrami operator (visually, only the first 8 look reasonably), while HKC and Roseland can recover more eigenfunctions (visually, the first 10 are reasonably well). Since HKC is not designed to recover the Laplace-Beltrami operator of $S^1$, we do not consider it in Figure \ref{S1noisy-noisy_eigenvalue}. Compared with eigenvectors, only the first 7 or 8 eigenvalues of Laplace-Beltrami operator can be well approximated in both Nystr\"om method and Roseland.

\begin{figure}[bht!]\centering
	\includegraphics[trim=4 0 0 0, clip, width=12cm]{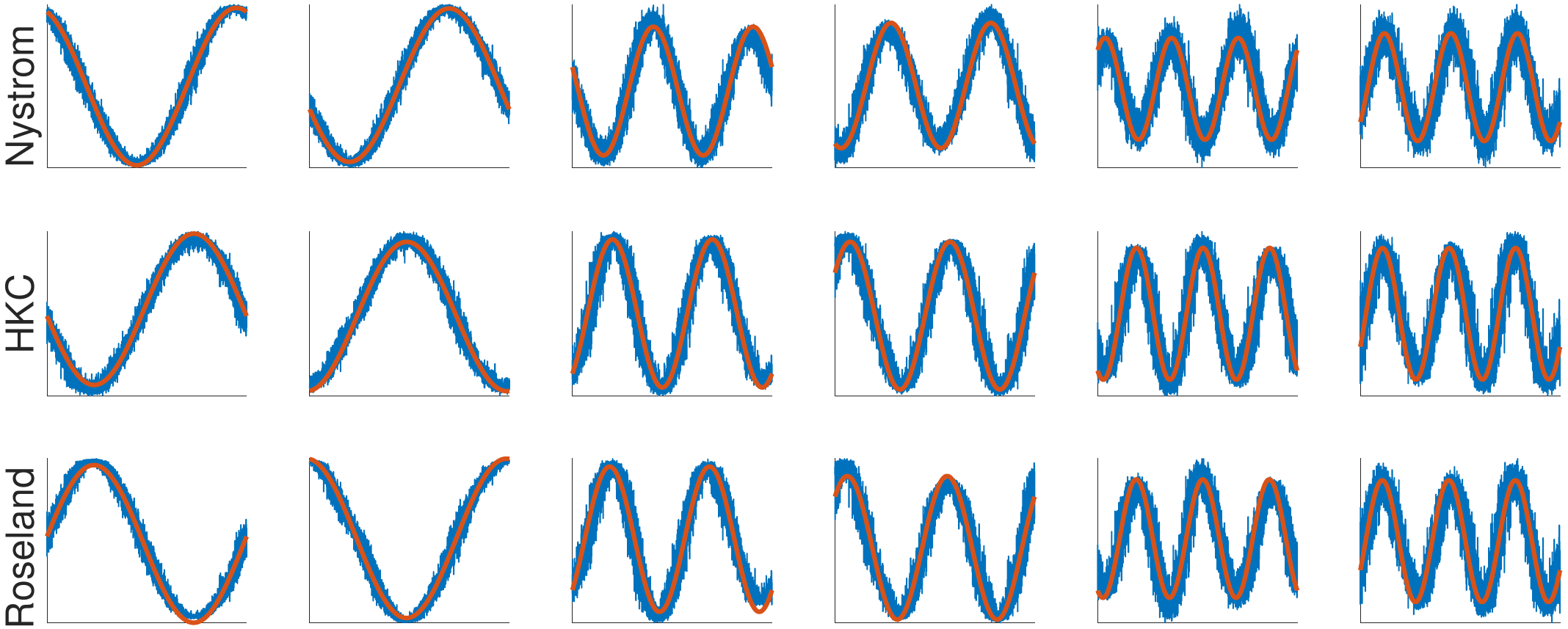}
	\includegraphics[trim=4 0 0 0, clip, width=12cm]{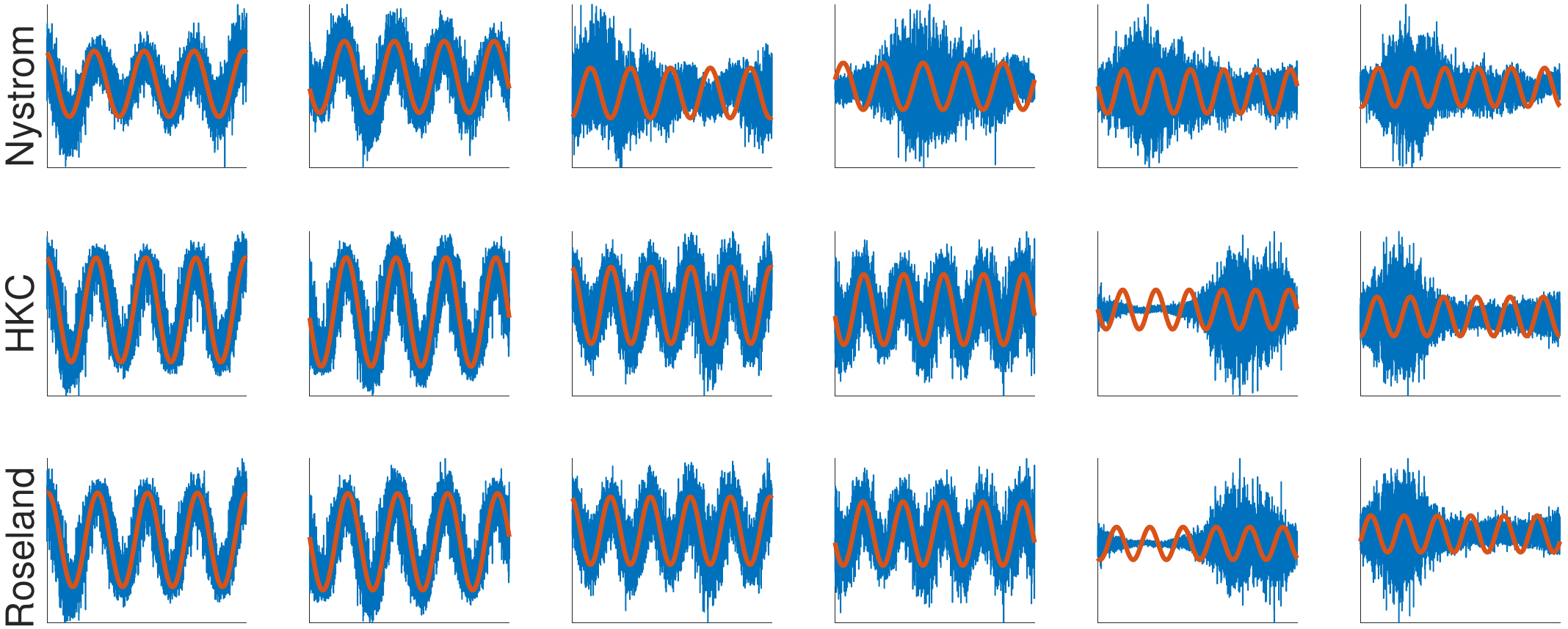}
	\caption{Noisy data set and subset. Superimpose the top 12 non-trivial eigenvectors by the Nystr\"om, HKC and Roseland with the ground truth (superimposed in red). Top three rows: the top 6 eigenvectors; bottom three rows: the $7^{\textup{th}}$ to the $12^{\textup{th}}$ eigenvectors. \label{S1noisy-noisy_eigenvector}}
\end{figure}

\begin{figure}[bht!]\centering
	\includegraphics[width=3.7cm]{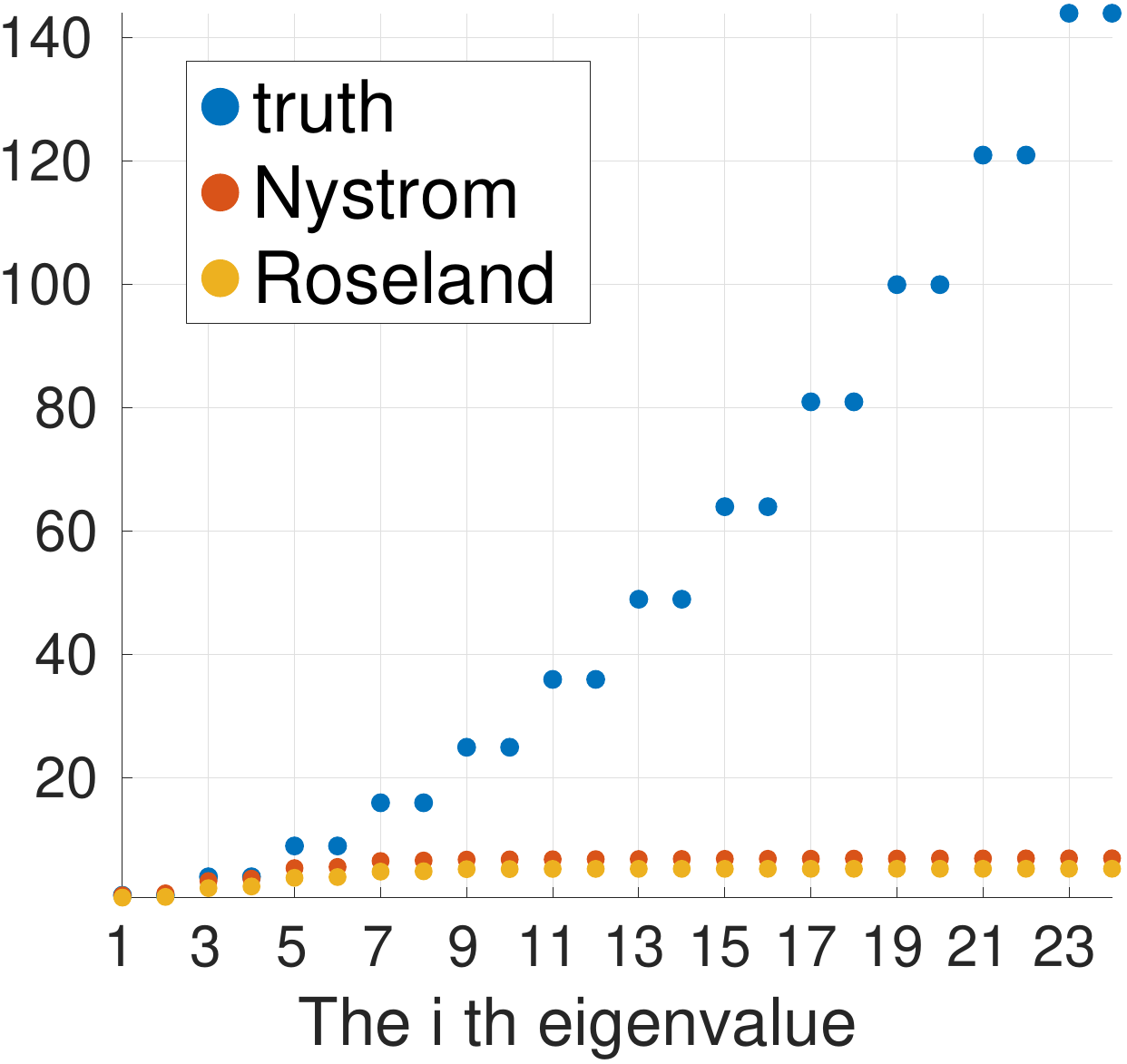}
	\includegraphics[width=3.7cm]{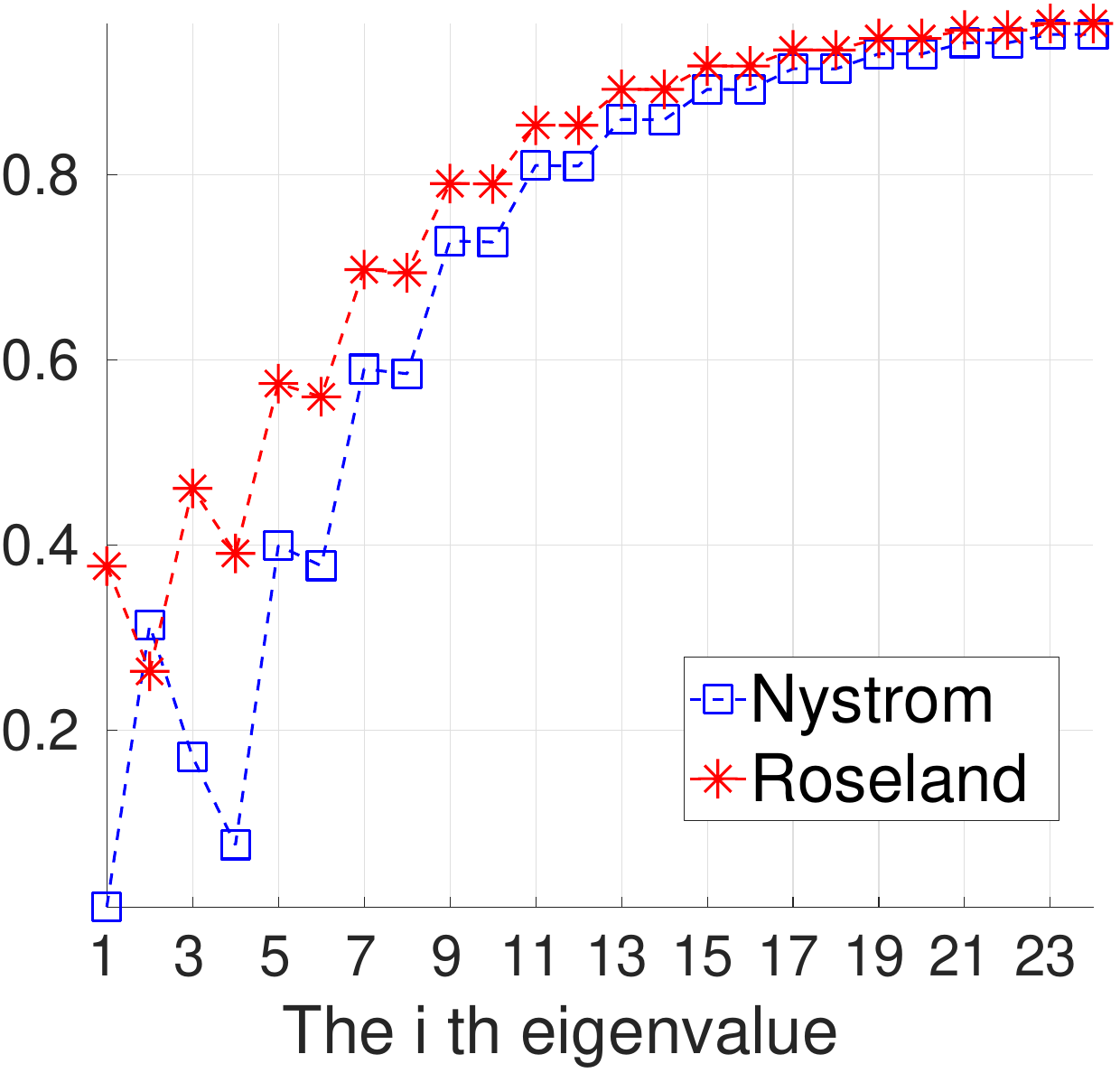}
	\includegraphics[width=4.1cm]{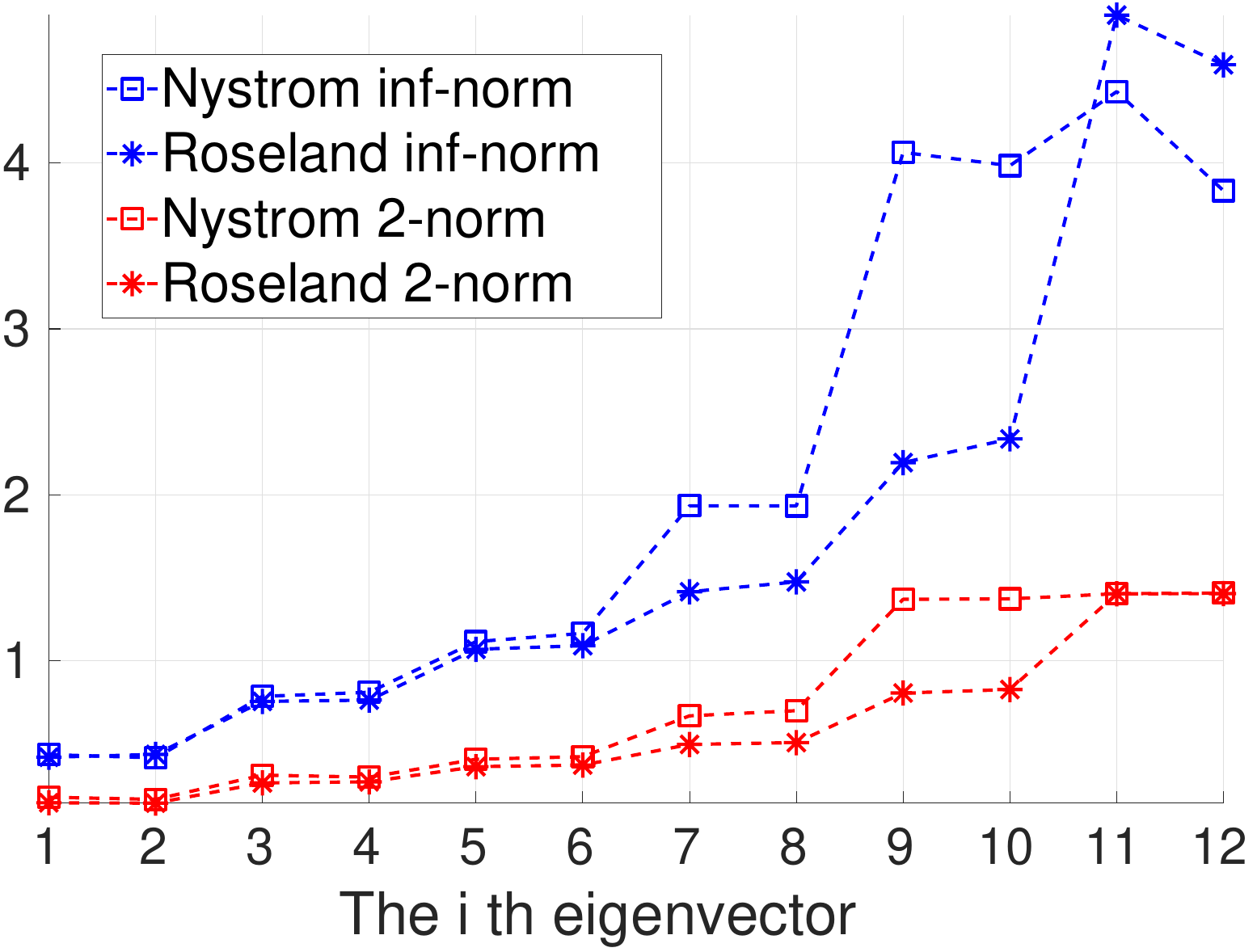}\caption{Illustration of the Nystr\"om method and Roseland on the noisy dataset and landmark set. Left: the top 18 non-trivial eigenvalues by the Nystr\"om and Roseland with the ground truth. Middle: relative error of eigenvalues. Right: relative $L^{\infty}$ and $L^{2}$ error of the top 12 non-trivial eigenvectors by the Nystr\"om and Roseland with the ground truth. \label{S1noisy-noisy_eigenvalue}}
\end{figure}

``Visually'' the qualities of the first two non-trivial eigenfunctions of the Nystr\"om extension, HKC and Roseland are similar,  
but the qualities of embeddings are different. To understand this discrepancy, we consider the following quantities. Note that the first two non-trivial eigenvectors, $v_1,v_2\in \mathbb{R}^n$ from either the Nystr\"om extension, HKC or Roseland, if successfully recovered the eigenfunctions of the Laplace-Beltrami operator, should be $\sin(\theta+\phi)$ and $\cos(\theta+\phi)$ for some $\phi\in(0,2\pi]$ respectively. Here, the phase $\phi$ comes from the uncertainty nature of the spectral embedding methods. We then plot $\arctan(v_1(i)/v_2(i))$ and $\sqrt{v_1(i)^2+v_2(i)^2}$ against $\theta_i$, where $\theta_i$ is the angle of the $i$-th sampled point. The results are shown in Figure \ref{S1noisy-noisy_eigenvector_phase}. Clearly, the amplitude eigenvectors of the Nystr\"om extension and HKC fluctuates more than those of Roseland, while the phase recovery qualities are similar. This difference comes from the different normalization steps of Roseland and HKC.

\begin{figure}[bht!]\centering
	\includegraphics[width=4cm]{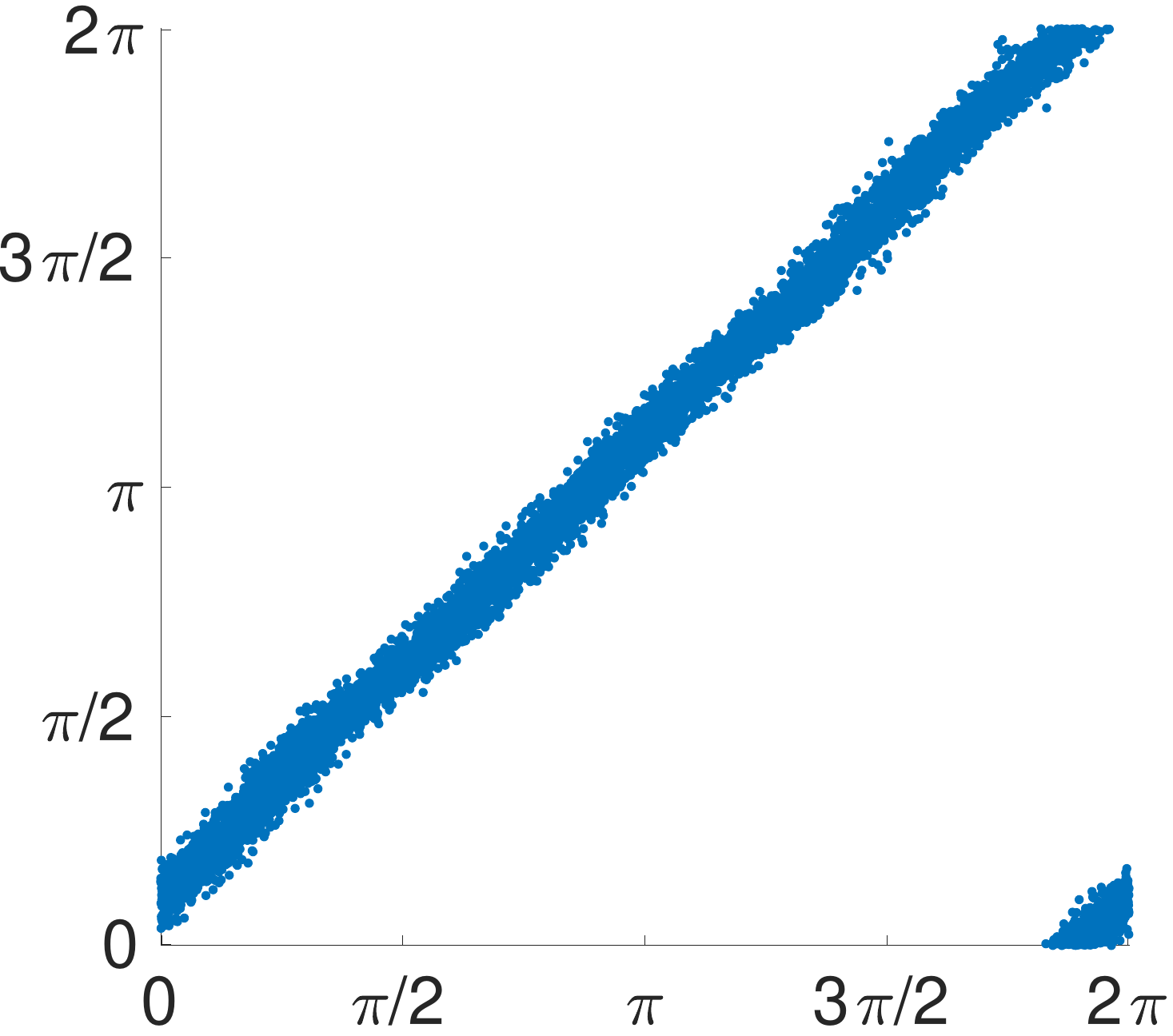}\includegraphics[width=4cm]{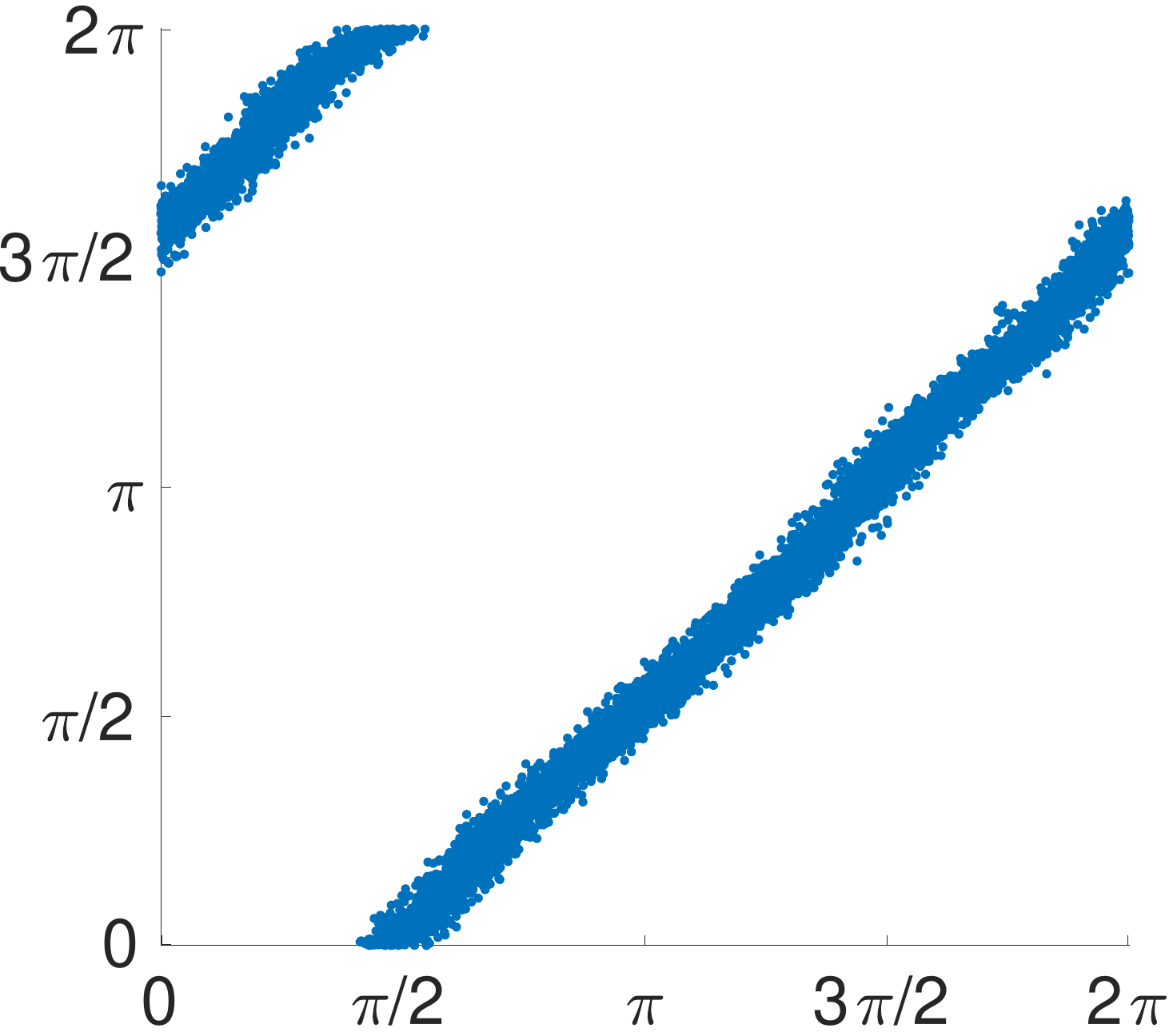}
	\includegraphics[width=4cm]{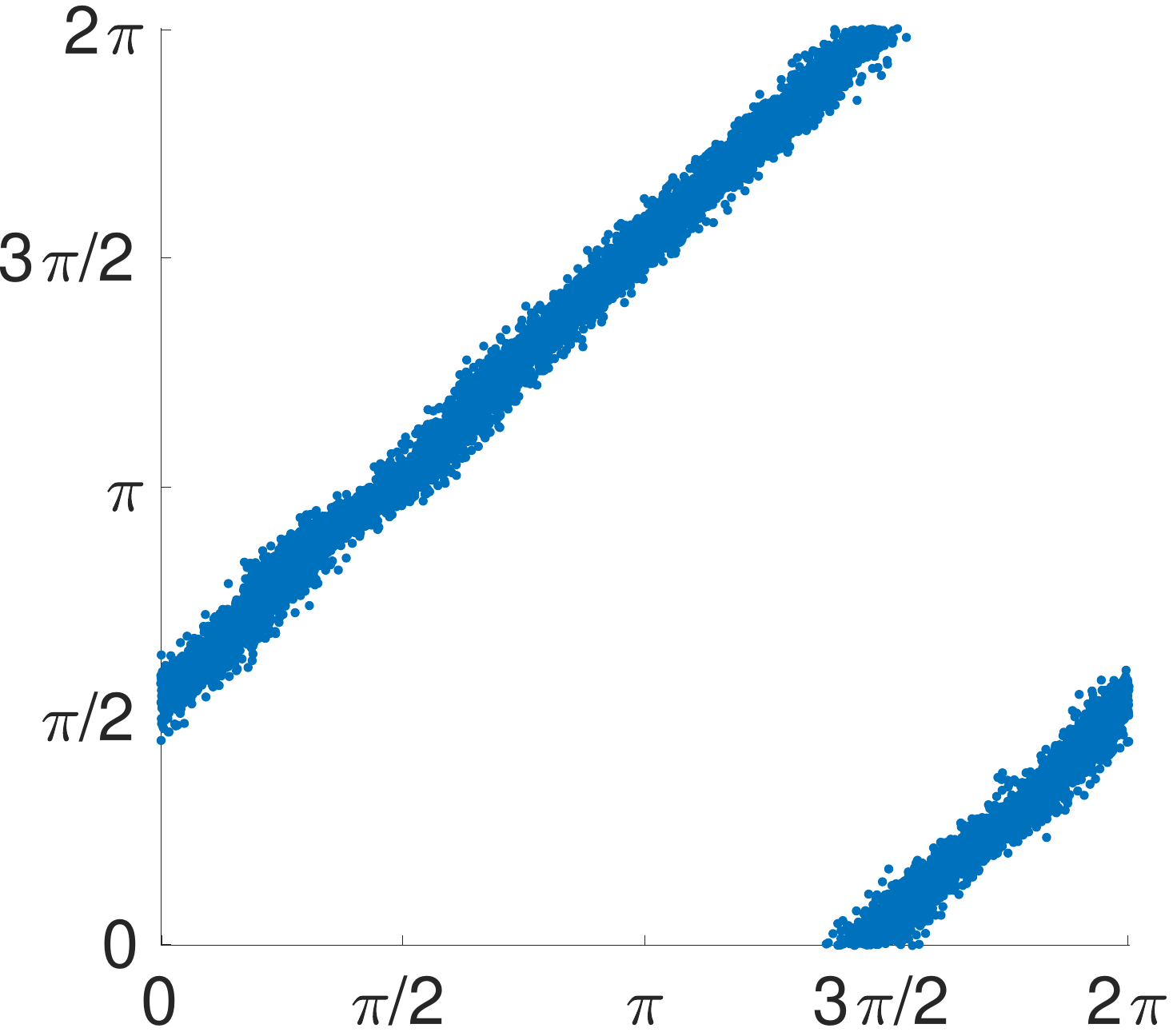}
	\\\vspace{10pt}
	\includegraphics[width=4cm]{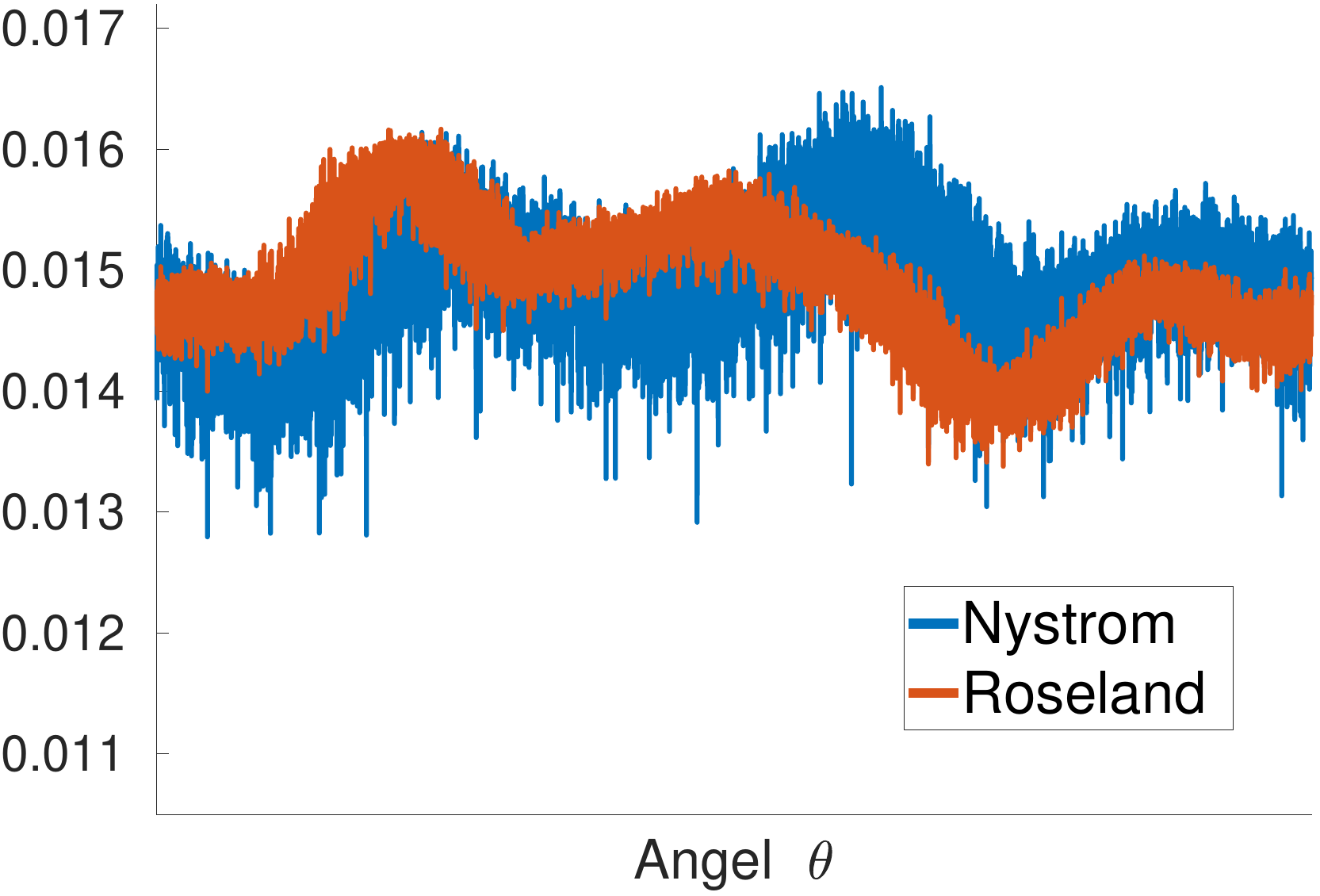}
	\includegraphics[width=4cm]{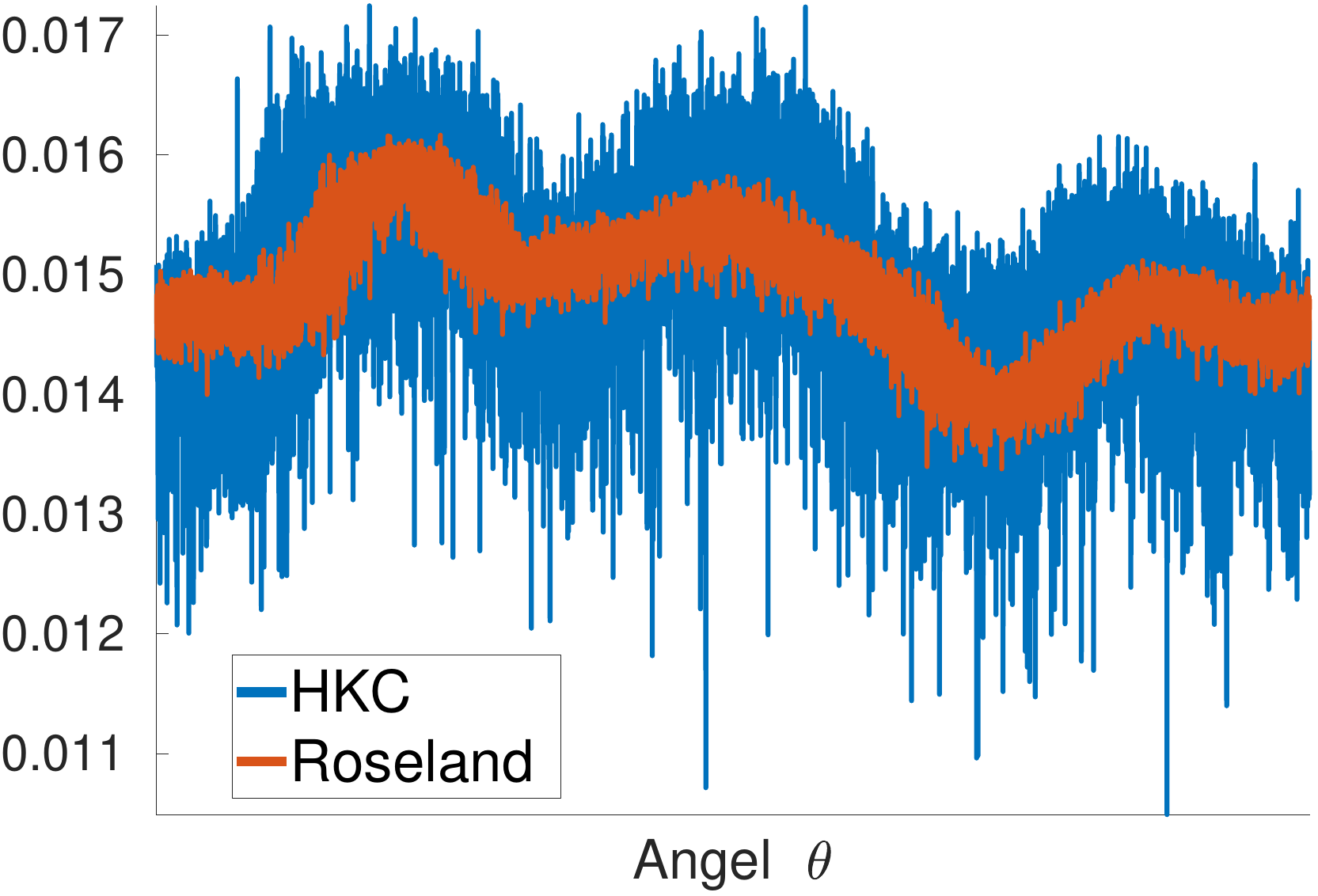}
	\vspace{-5pt}
	\caption{Noisy data set and subset. Top left: the phase of the embedding by Roseland. Top middle: the phase of the embedding by HKC. Top right: the phase of the embedding by the Nystr\"om extension.  
		Bottom left: the amplitude of the embedding by Roseland and the Nystr\"om. 
		Bottom right: the amplitude of the embedding by Roseland and HKC. 
		The phase of the embedding is determined by $\arctan(v_1(i)/v_2(i))$ and the amplitude is determined by $\sqrt{v_1(i)^2+v_2(i)^2}$, where $v_1$ and $v_2$ are the first non-trivial eigenvectors determined by the Nystr\"om extension or Roseland.
		\label{S1noisy-noisy_eigenvector_phase}}
\end{figure}

\subsection{Geometric structure recovery -- Geodesic distance estimation}

We now show that if we want to recover DM from the aspect of geodesic distance estimation, then Roseland outperforms the Nystr\"om method. Since HKC is not designed for this purpose, we do not compare it here. For a fair comparison, the subset used in the Nystr\"om extension is the same as the landmark set used in Roseland.
First, we describe the comparison methodology.
\begin{itemize}
	\item Uniformly and independently sample $2,500$ points from $S^{1}$ as the dataset. Uniformly and independently sample another $m=n^{\beta}$ points from $S^{1}$ as the landmark set for Roseland. In this experiment, we chose $m=50$, where $\beta=0.5$.
	\item Fix $K\in \mathbb{N}$. Denote $\{x_{i}\}_{i=1}^{2,500}$ to be the dataset, and denote $\{y_{i}\}_{i=1}^{2,500}\subset \mathbb{R}^l$ to be the embedded dataset, where $l\in \mathbb{N}$ is the dimension of Roseland. Denote $x_i^{(K)}$ (resp. $y_i^{(K)}$) to be the $K$-th nearest neighbor of $x_i$ (resp. $y_i$). The relative errors of the geodesic distance between $x_i$ and its $K$-th nearest neighbor is calculated by 
	\begin{equation}
	\frac{\big|D_t(x_{i},x_{i}^{(K)}) - d(x_{i},x_{i}^{(K)})\big|}{d(x_{i},x_{i}^{(K)})}\,, 
	\end{equation}
	where $d(x_{i},x_{i}^{(K)})$ is the ground truth geodesic distance between $x_{i}$ and $x_{i}^{(K)}$, and $D_t(x_{i},x_{i}^{(K)})$ can be DD determined by DM, Roseland, or the Nystr\"om extension.
	\item We compare the relative errors of the geodesic distance between $x_i$ and its $K$-th nearest neighbor by embeddings from DM, the Nystr\"om method and Roseland.
\end{itemize}

The results are shown in Figure \ref{geo_recov_50}. Clearly, the geodesic distance can be well recovered by Roseland, but the Nystr\"om method is limited. This result is more dramatic if we use the ground truth eigenvalues to estimate DD. This result indicates that it is critical to recover the eigenvalues.
The result is not surprising from the theoretical standpoint. Note that it is shown in \cite{portegies2016embeddings} that one can obtain an almost isometric embedding of the manifold as long as one has enough number of eigenvalues and eigenfunctions of the Laplace-Beltromi operator. Specifically, let $\varepsilon>0$ be any tolerable error given, then there exists some $t_{0}$, which depends on the manifold's intrinsic dimension, Ricci curvature, injectivity radius and $\varepsilon$ such that for all $0<t<t_{0}$, there is a $N_{E}$ depends on the manifold's intrinsic dimension, Ricci curvature, injectivity radius, volume, $\varepsilon$ and $t$ such that if $N>N_{E}$, the spectral embedding
\begin{align}
x\mapsto 2t^{(d+2)/4}\sqrt{2}(4\pi)^{d/4} \begin{bmatrix}e^{-\lambda_{1}t}\phi_{1}(x)&\ldots& e^{-\lambda_{N}t}\phi_{N}(x)\end{bmatrix}^\top
\end{align}
is almost isometric with the error controlled by $\varepsilon$, where $\lambda_{i}$ and $\phi_{i}$ is the $i^{\textup{th}}$ eigenpair of the Laplace Beltromi operator of the manifold. The above theorem essentially says that we need to recover enough eigenfunctions if we want to have an accurate geodesic distance estimate by DD. Since the diffusion structure is not taken into account in the Nystr\"om extension, it is limited in recovering higher order eigenfunctions. On the other hand, Roseland preserves the diffusion property, and hence the eigenfunctions of the Laplace-Beltrami operator.

\begin{figure}[bht!]\centering
	\includegraphics[width=6cm]{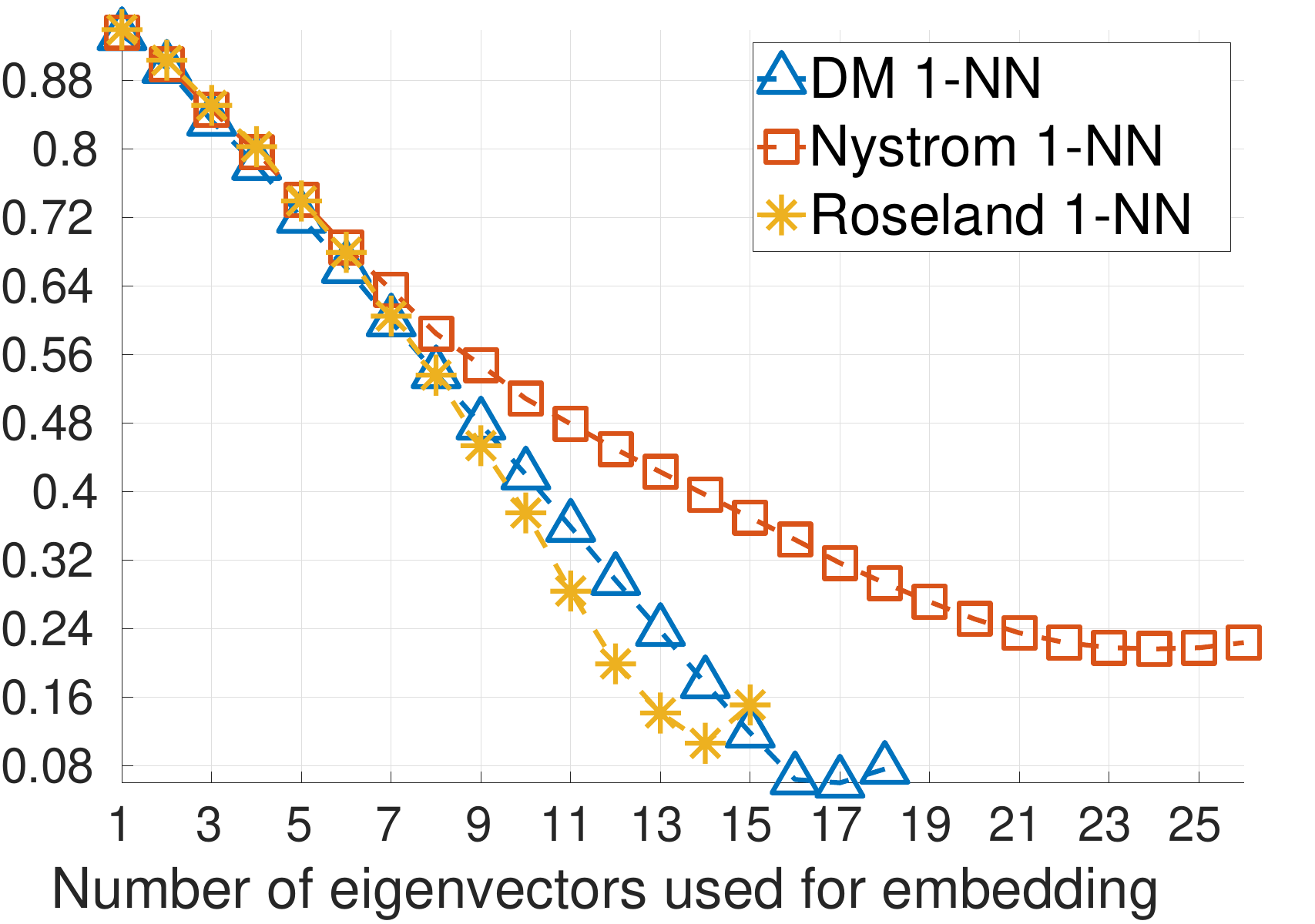}
	\includegraphics[width=6cm]{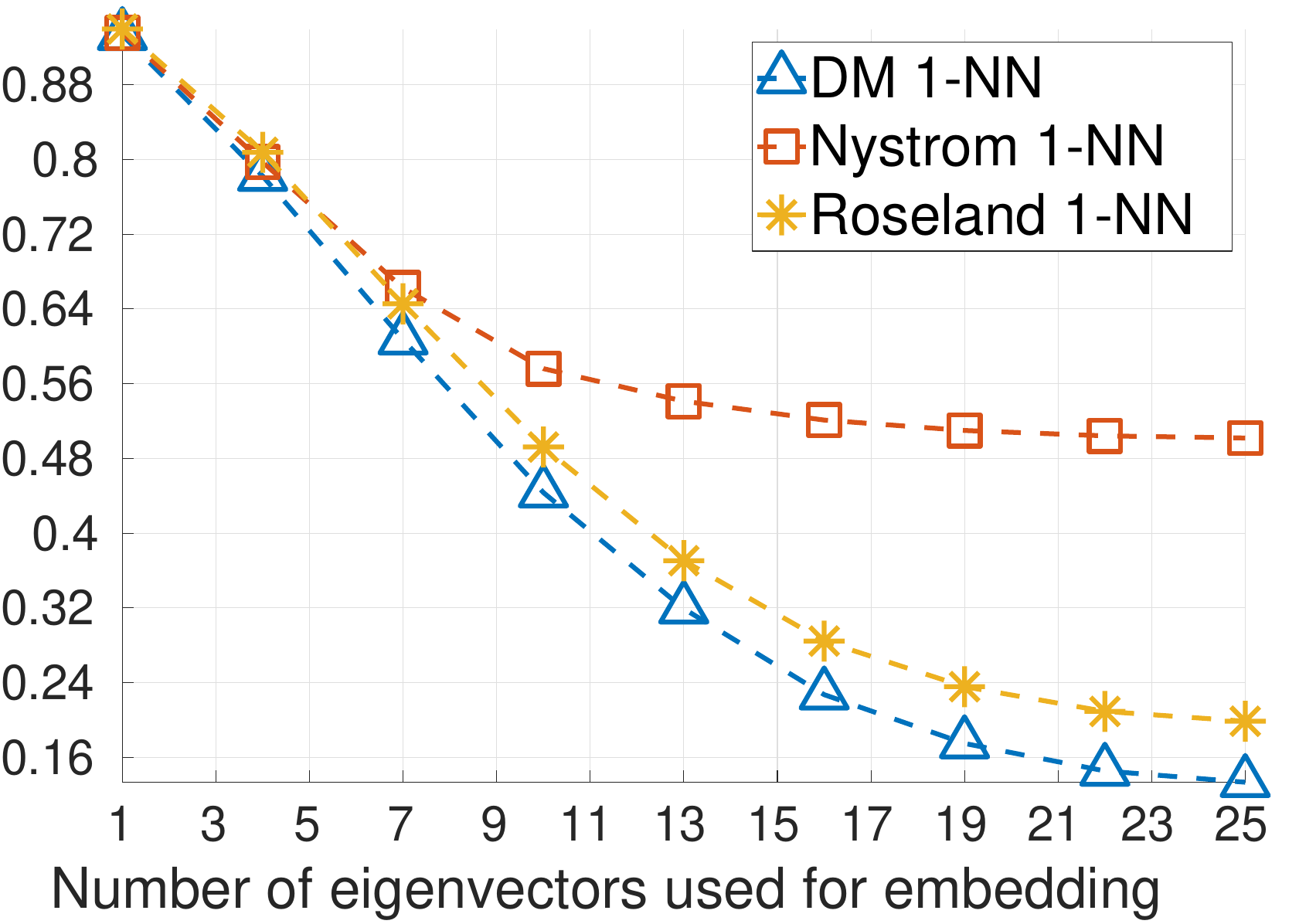}
	\caption{Subset size = 50, so $\beta=0.5$. 
		Left: relative errors of geodesic recovery by DM, the Nystr\"om extension and Roseland using their own eigenpairs. Right: relative errors of geodesic recovery by DM, the Nystr\"om extension and Roseland using their own eigenvectors and the ground truth eigenvalues. \label{geo_recov_50}}
\end{figure}

\subsection{Control non-uniform sampling by designing the landmark set}

Recall Remark \ref{remark: dm_thm} -- in Roseland, if we can design the sampling scheme for the landmark set so that $\frac{2\nabla p_{X}(x)}{p_{X}(x)}+\frac{\nabla p_{Y}(x)}{p_{Y}(x)}=0$, then we remove the impact of the non-uniformly sampling and recover the Laplace-Beltrami operator.  
The condition $\frac{2\nabla p_{X}(x)}{p_{X}(x)}+\frac{\nabla p_{Y}(x)}{p_{Y}(x)}=0$ suggests that we may want to sample the landmark set following the density function $p_{Y}(x)\propto\frac{1}{p^{2}_{X}(x)}$. 

To illustrate this fact, we use a dataset non-uniformly sampled from the canonical $S^{1}$ as an example. See Figure \ref{Figure:nonuniform-design} for the result. It is clear that if we do not design the landmark set, the first two nontrivial eigenfunctions estimated from Roseland are deviated from the ground truth. However, if the landmark set is well designed according to the developed theory, or could incorporate the background knowledge (like the HKC algorithm designed for the texture separation problem), we may better recover the desired ground truth. This preliminary result warrants a further consideration of this ``design problem'' in our future work.

\begin{figure}[bht!]\centering
	\includegraphics[width=2.9cm]{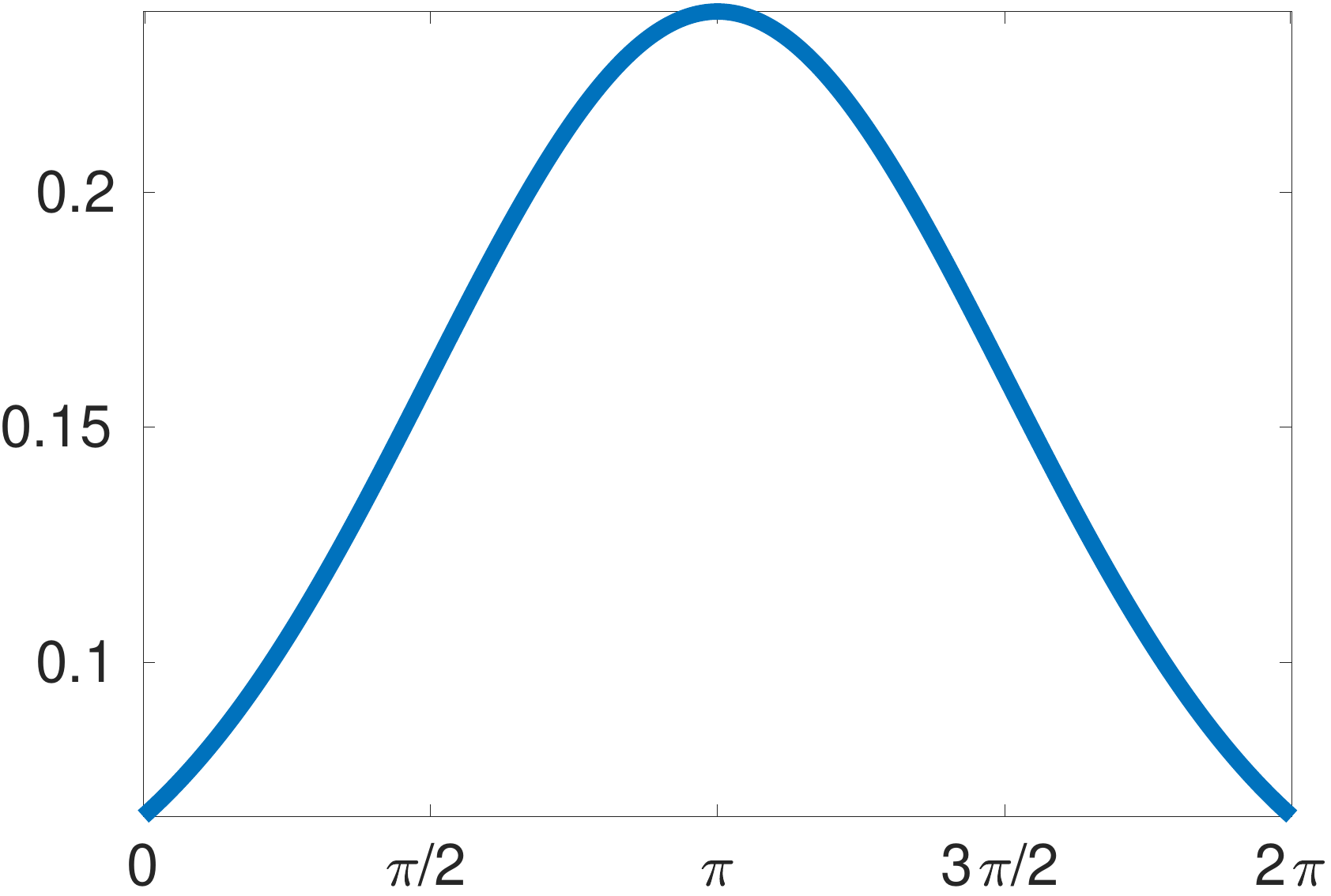}
	\includegraphics[width=2.9cm]{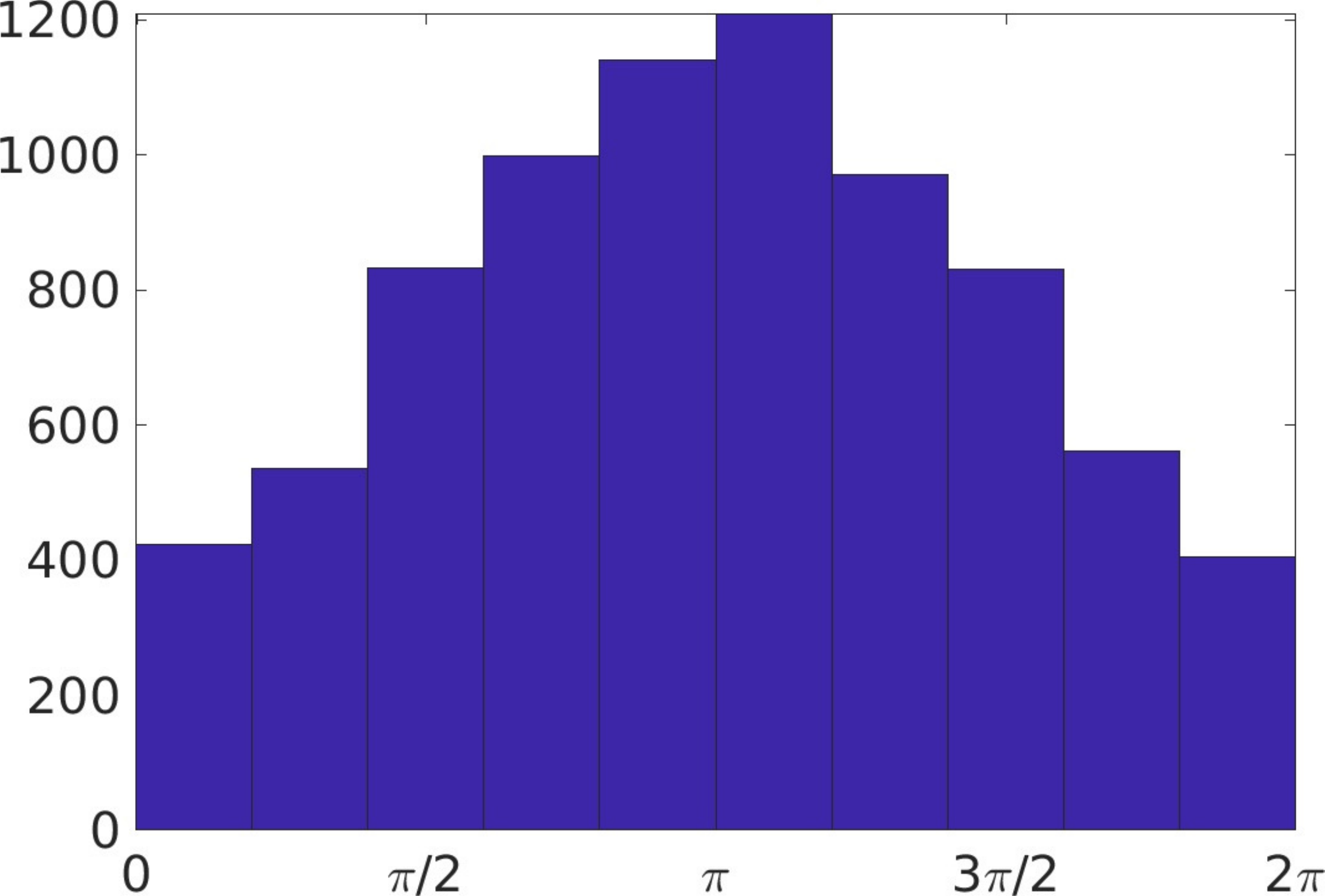}
	\includegraphics[width=2.9cm]{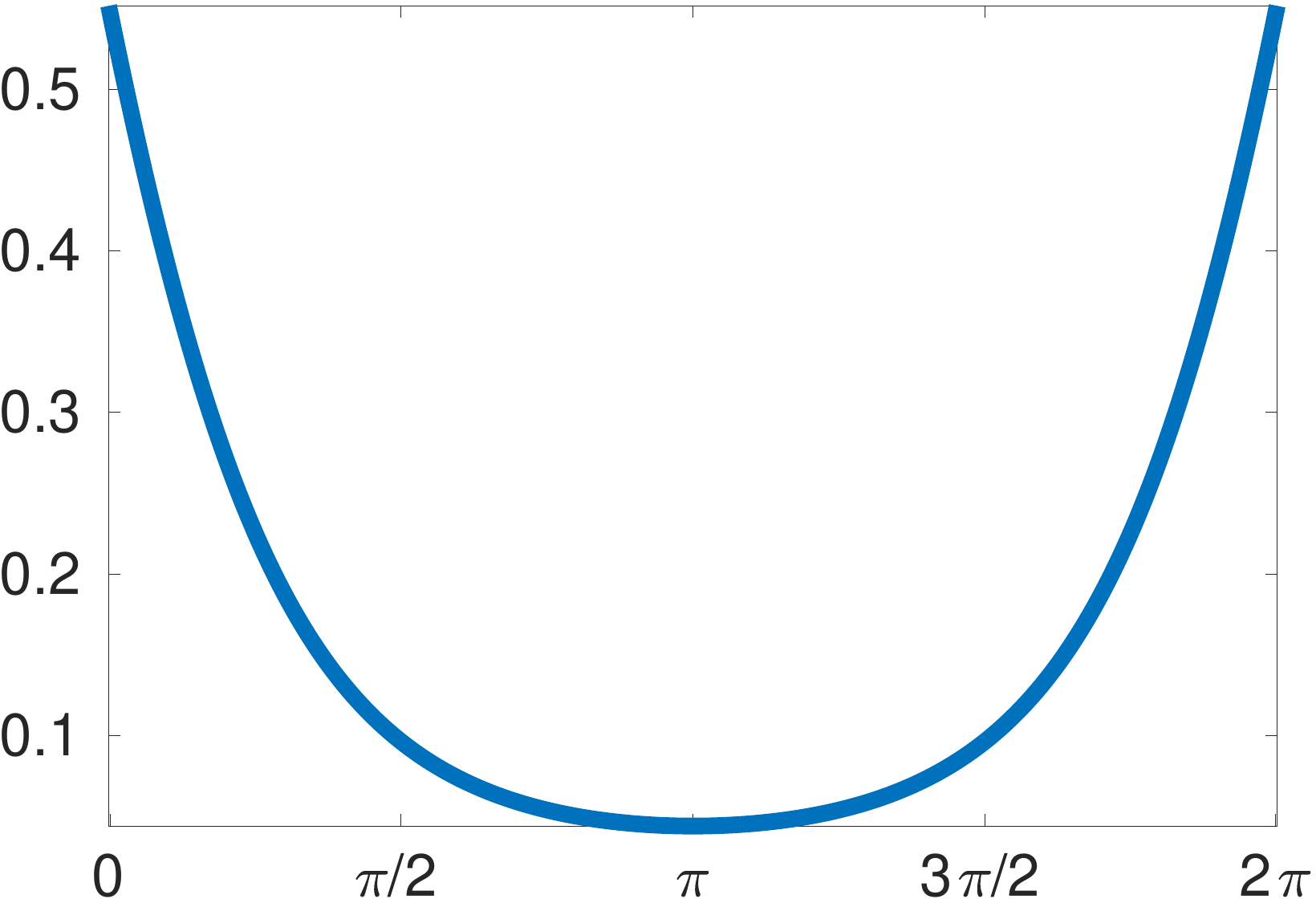}
	\includegraphics[width=2.9cm]{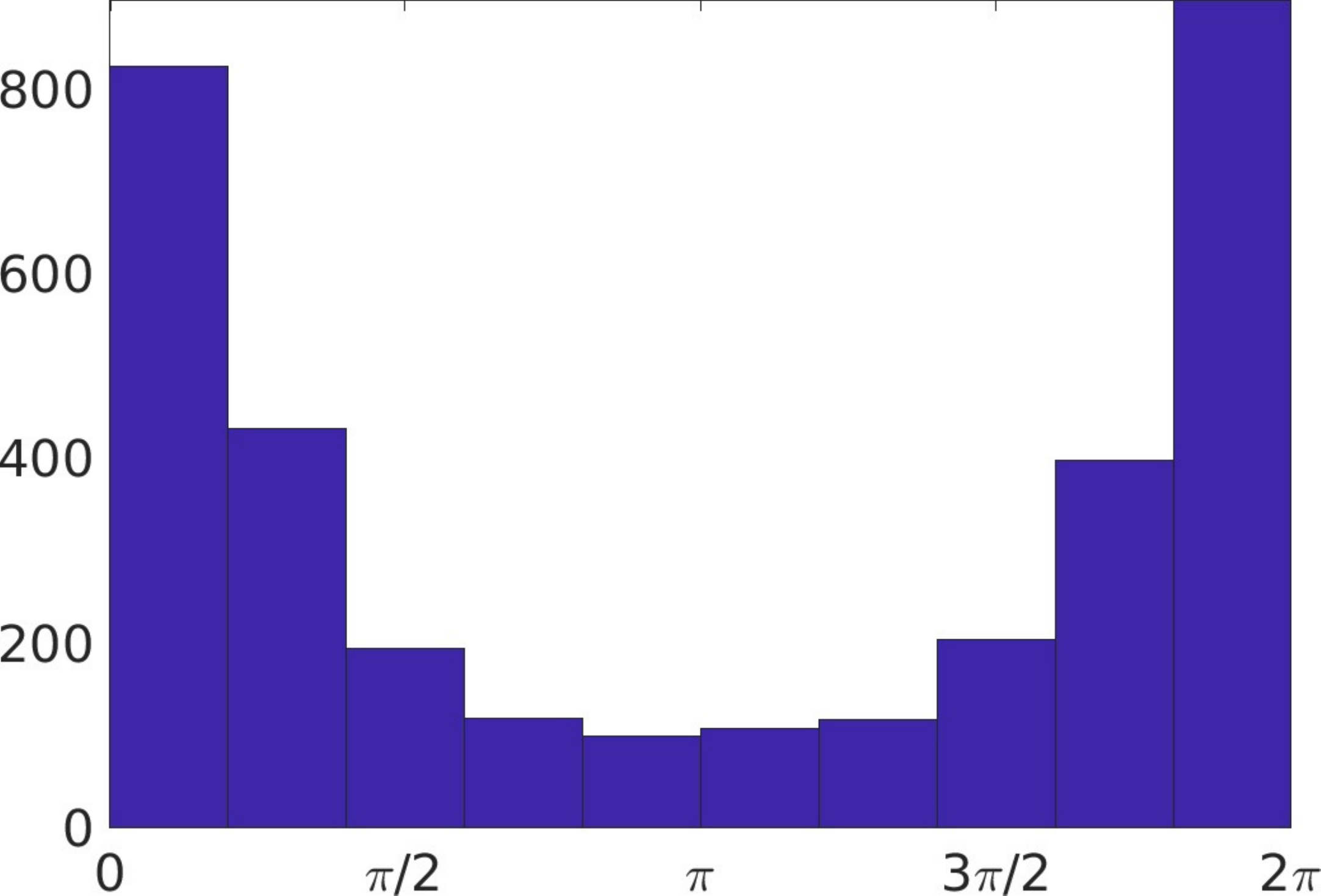}
	\includegraphics[width=2.9cm]{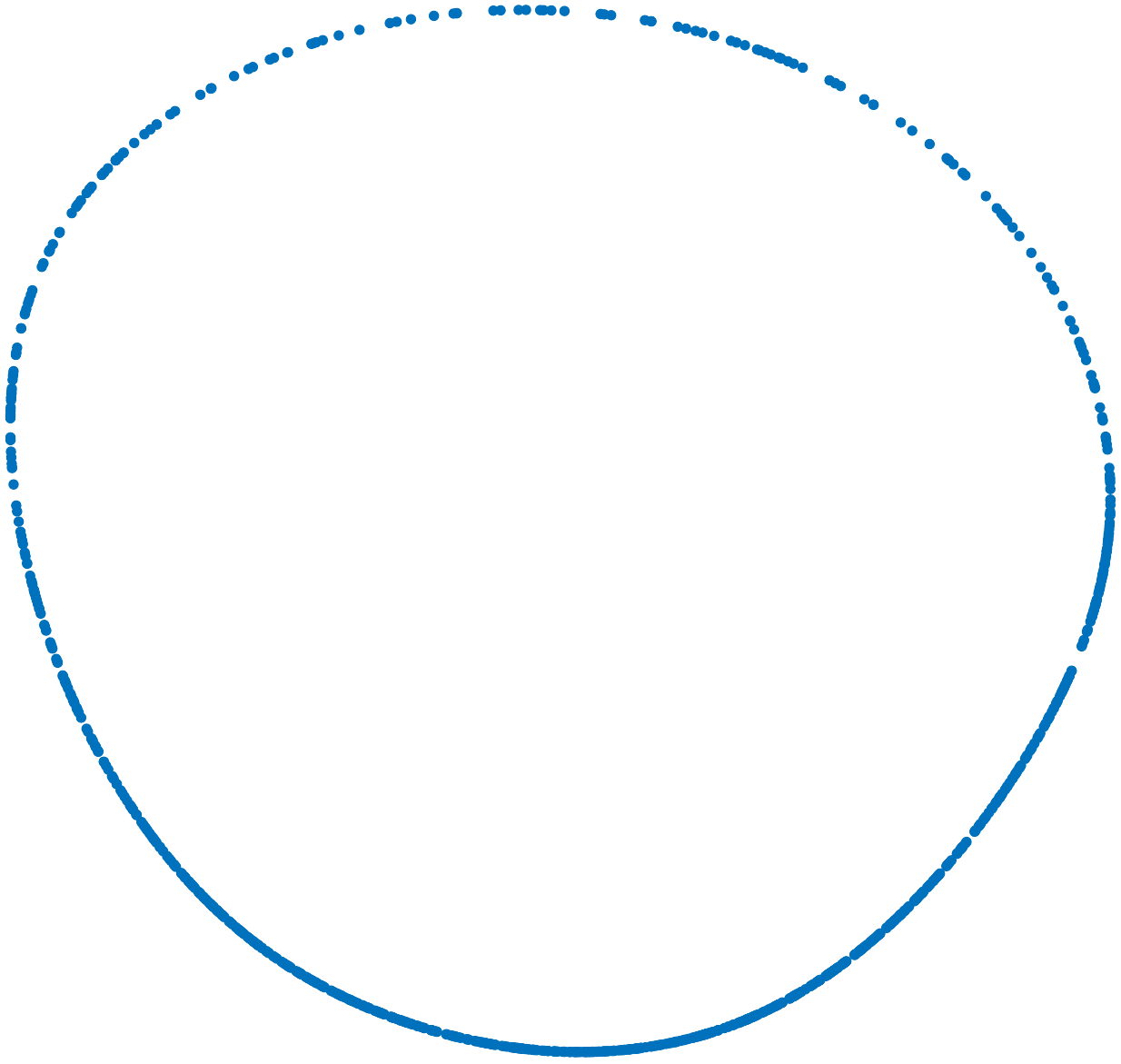}
	\includegraphics[width=2.9cm]{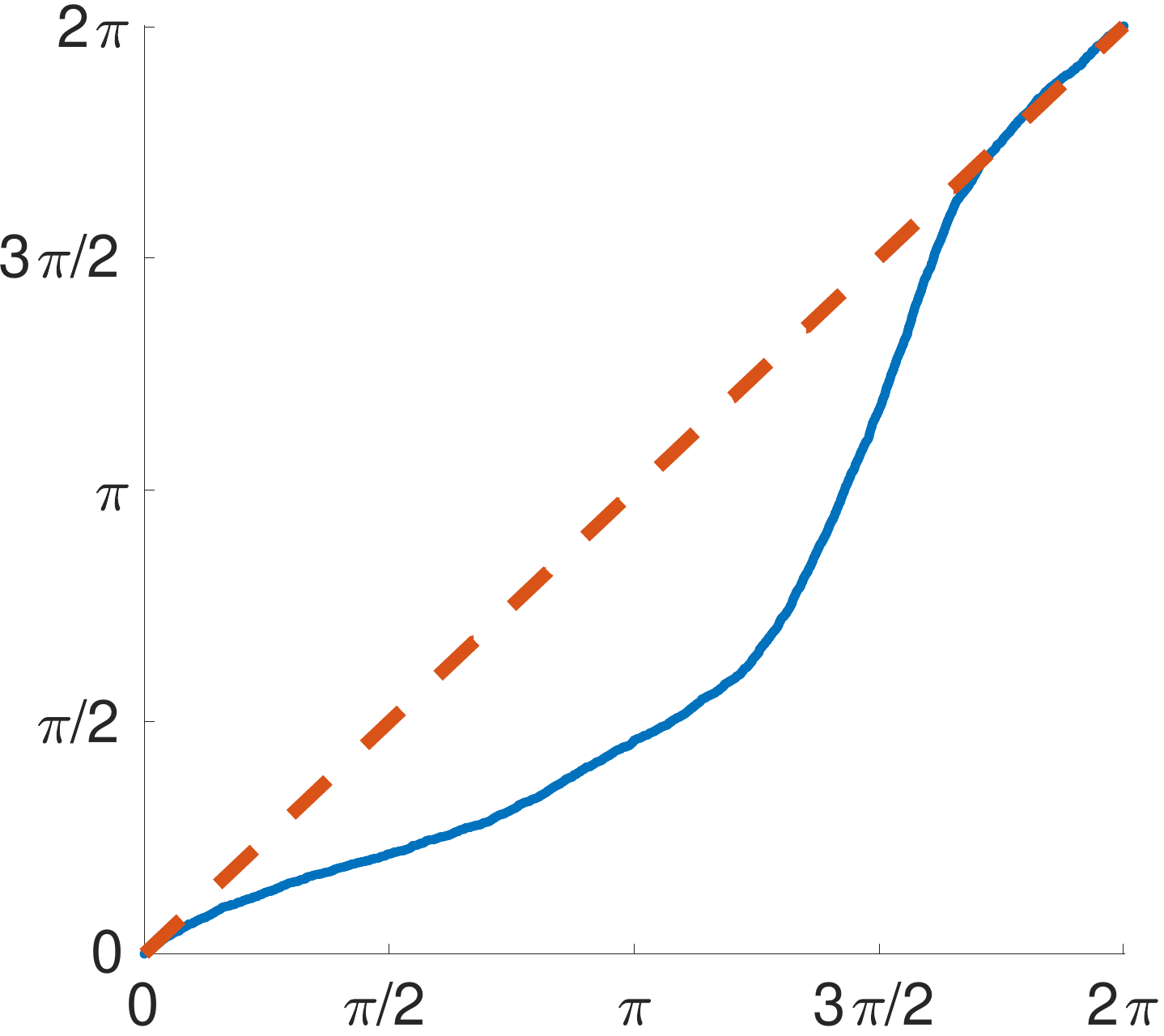}
	\includegraphics[width=2.9cm]{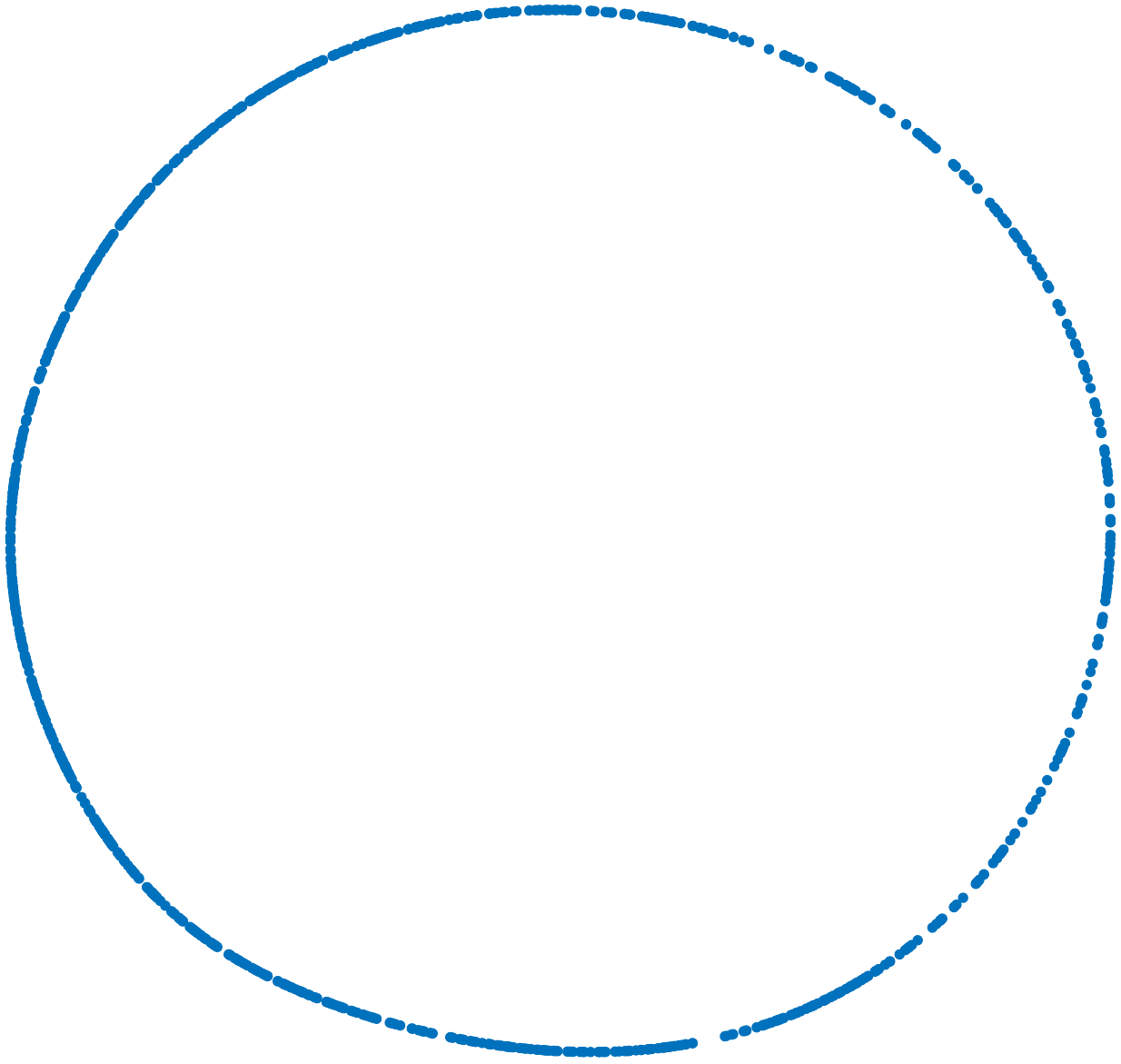}
	\includegraphics[width=2.9cm]{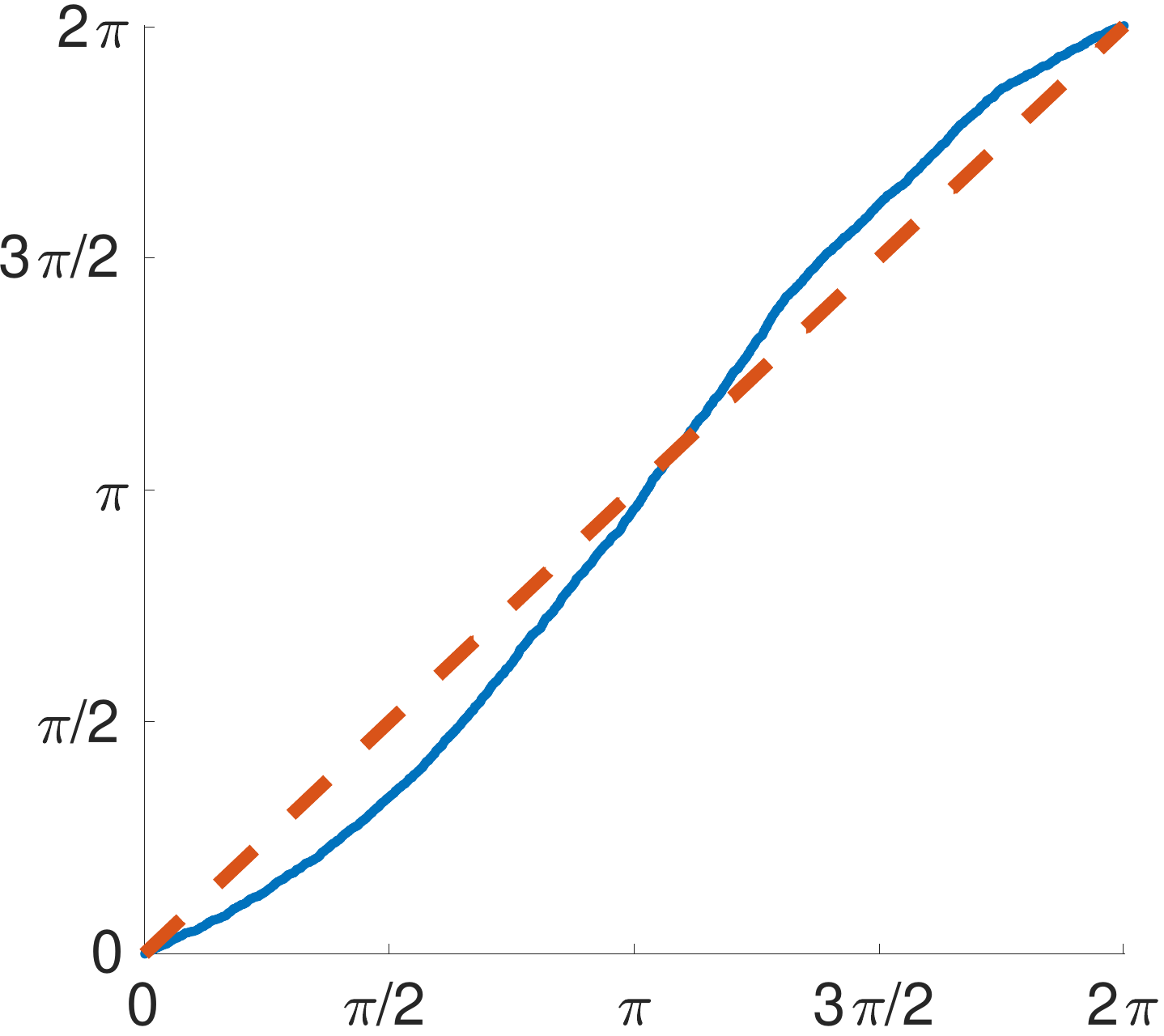}
	\caption{Top row: the left subplot shows the true p.d.f. $p(\theta)$ we use to sample points from $S^{1}$, which is parametrized by $\theta\in [0,2\pi)$, the left middle subplot shows the histogram of the sampled data, the right middle subplot shows the designed p.d.f. $q(\theta)$ that is proportional to $1/p^2(\theta)$, and the right subplot shows the histogram of the landmark set sampled $\theta$ according to $q(\theta)$. Bottom row: the left subplot is the embedding determined by Roseland, where the landmark set is chosen randomly, the left middle subplot shows the  scattering plot of the recovered angles from Roseland with the randomly chosen landmarks, where the x-axis is the estimated angles of $S^1$, and the y-axis is the associated true angles, the right middle subplot is the embedding determined by Roseland with the designed landmark set following the law of $q(\theta)$, and the right subplot is the scattering plot of the recovered angles from Roseland with the designed landmark set, where the x-axis is the estimated angles of $S^1$, and the y-axis is the associated true angles. \label{Figure:nonuniform-design}}
\end{figure}

\section{Discussion and conclusion}\label{discuss}

In this paper, we introduce a new algorithm based on the landmark set to accelerate the DM algorithm. In addition to providing a series of theoretical justification, we also provide a series of numerical examples to support the potential of the algorithm.

\subsection{Optimal variance control}\label{Discussion optimal variance convergence discrepancy}

Note that in the pointwise convergence, the variance is bounded by the large deviation, while the bound might not be the optimal one. We discuss this problem from a theoretical aspect and a numerical aspect. 
Recall the definition of the U-statistics in \eqref{definition U statistics}, which is a special case of (\ref{dep_sum}). \cite{hoeffding1994probability} proved that for all $t>0$, we have the following bound for the $r$-degree U-statistics:
\begin{equation}
\mathbb{P}(|U_{r}-\mathbb{E}(U_{r})|\geq t)\leq 2\textup{exp}\left(\frac{-nt^{2}}{r}\right).
\end{equation} 
Later, when $\sigma=\textup{Var}(h(X_{1},\ldots,X_{r}))$ is finite, \cite{arcones1995bernstein} refined it to a Bernstein-like inequalities; that is, for all $t>0$,
$$
\mathbb{P}(|U_{r}-\mathbb{E}(U_{r})|\geq t)\leq a\textup{exp}\left(\frac{-(n/r)t^{2}}{2\sigma^{2}+bt}\right),
$$ 
where $a,b$ are some constants. Note that this bound is not better than the one shown in Theorem \ref{bernstein} \cite{janson2004large}. We comment that since the U-statistic is a special form of (\ref{dep_sum}), we would expect that the convergence rate provided in Theorems \ref{hoeffding} and \ref{bernstein} is good enough since they are of the same order as those for the U-statistics. 
While showing the optimal bound, either order or constant, for the U-statistics or the more general form like (\ref{dep_sum}) is not the focus of this paper, we could numerically evaluate the quality of the bound provided in Theorems \ref{bernstein} \cite{janson2004large}. See Figure \ref{largedev} for a simulation, where we observe that the empirical convergence rate of $M\times N$ grid samples is ``faster'' than $M$ i.i.d. samples. This at least suggests that we may have a better constant in front of the convergence order term. We will explore this large deviation rate in our future work.

\begin{figure}[bht!]\centering
	\includegraphics[width=4.5cm]{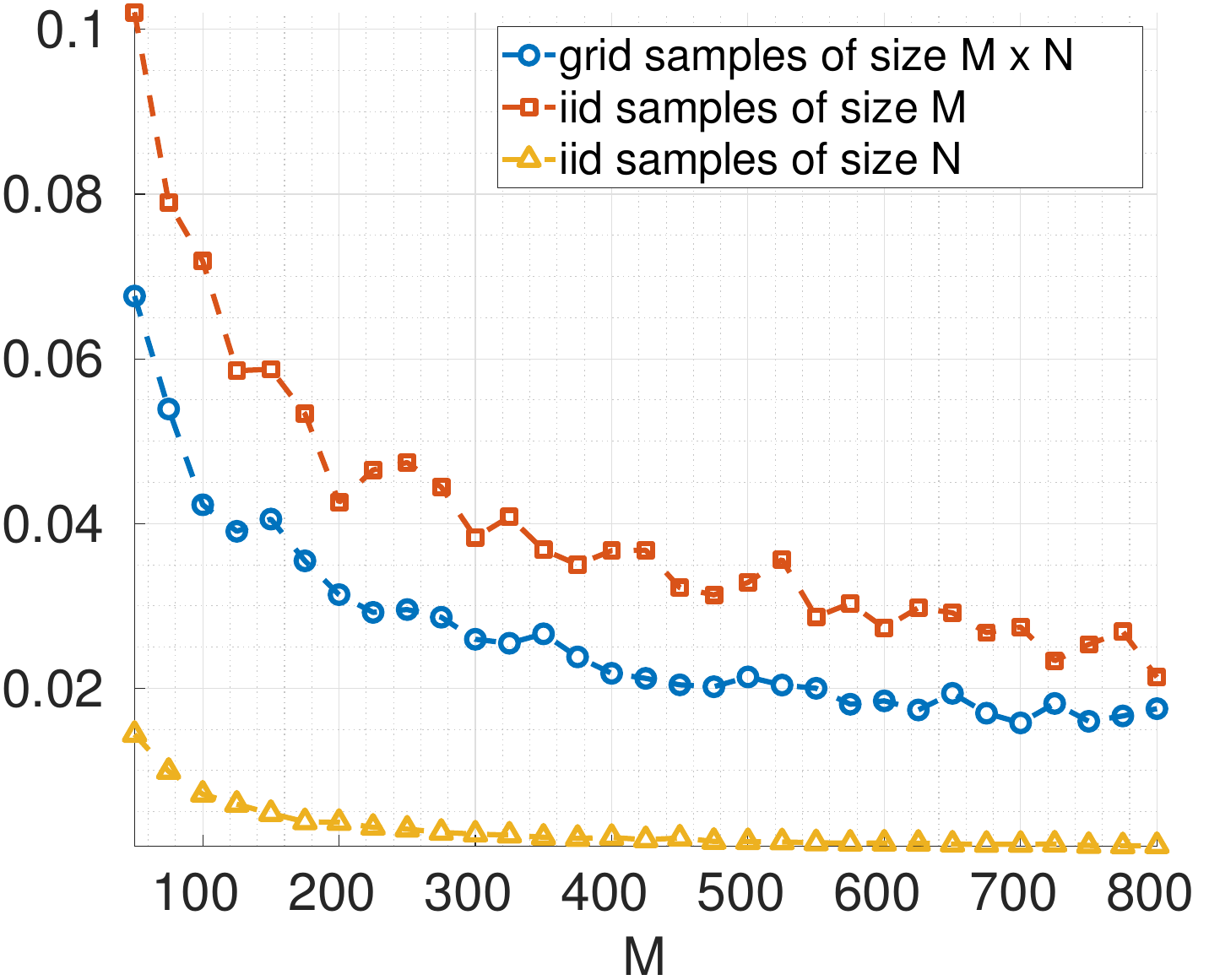}
	\includegraphics[width=4.5cm]{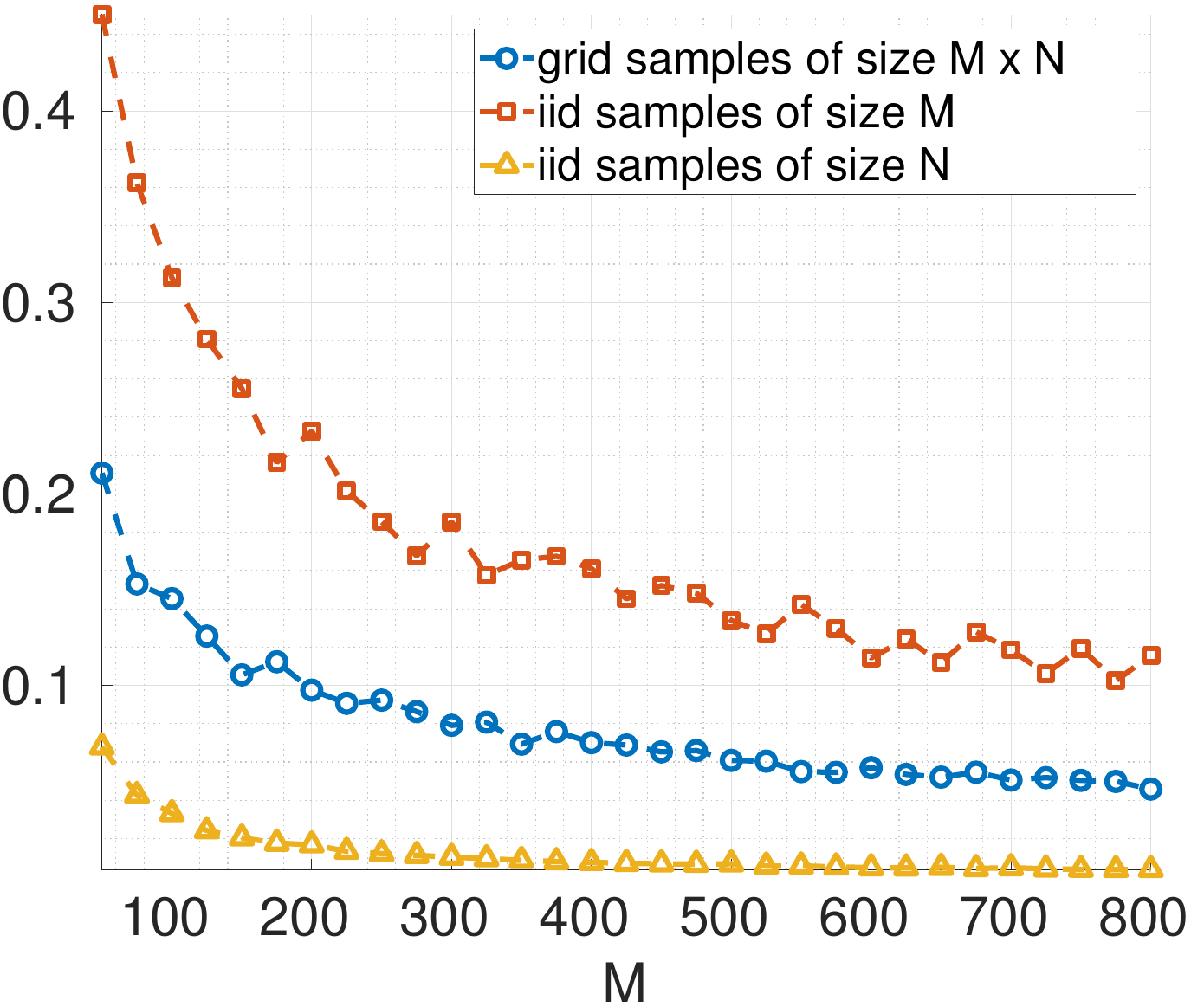}
	\caption{We run two simulations to evaluate the optimality of the convergence rate and plot the relative errors. In both cases, we let $M=N^{1/2}$ and compare the convergence of the empirical mean using i.i.d. $M$ samples, i.i.d. $N$ samples, and $N\times M$ grid samples. Left: $X,Y\sim U[0,6]$ and $f(X,Y)=XY$. Right: $X,Y\sim U[0,1]$ and $f(X,Y)=100e^{-100[(Y-0.5)^{2}+(X-Y)^{2}]}$.}\label{largedev}
\end{figure}

Another natural issue we shall further elaborate is the discrepancy raised in Remark \ref{Remark about the variance convergence rate}.  
{Ideally, when $m$ is big, like $\beta=1$, we would expect that 
\[
{K}_{\textup{ref},\epsilon,n}(x_i,x_j)=\frac{1}{m}\sum_{k=1}^{m}\epsilon^{-d/2} K_{\epsilon}(x_{i},y_{k})K_{\epsilon}(y_{k},x_{j})\approx  \int_M \epsilon^{-d/2} K_{\epsilon}(x_{i},y)K_{\epsilon}(y,x_{j})p_Y(y)dV(y)={K}_{\text{ref},\epsilon}(x_i,x_j)
\]
where ${K}_{\text{ref},\epsilon}:M\times M\to \mathbb{R}$ is a kernel with the bandwidth of order $\sqrt{\epsilon}$. See \eqref{expansion landmark kernel} and \eqref{expansion landmark kernel2} for the behavior of $K_{\text{ref},\epsilon}$. If this approximation is accurate when $m$ is large, then Roseland can be understood and analyzed as the vanilla DM with $\alpha=0$ and the landmark-kernel ${K}_{\text{ref},\epsilon}$. If this holds, we would expect the variance of Roseland to be the same as that of DM with $\alpha=0$ shown in \eqref{alpha=0 DM error bound}. However, it is not the case according to our theorem. This counterintuitive finding thus needs some discussion. First, the finite approximation of $K_{\text{ref},\epsilon}$ is not accurate with a non-negligible variance even when $m=n$. Indeed, with probability greater than $1-O(m^{-2})$, for all $x_i$ and $x_j$ the finite approximation satisfies
\begin{equation}
\frac{1}{m}\sum_{k=1}^{m}\epsilon^{-d/2} K_{\epsilon}(x_{i},y_{k})K_{\epsilon}(y_{k},x_{j})={K}_{\text{ref},\epsilon}(x_i,x_j)+O\left(\frac{\sqrt{\log(n)}}{n^{\beta/2}\epsilon^{d/4}}\right)\,.\label{kernel ref convergence finite variance}
\end{equation}
See Section \ref{Roseland kernel discrepancy proof} for details.
While it is intuitive that this non-negligible variance might be canceled when we continue to analyze Roseland with this approximation,
however, this cancelation is not as efficient as expected due to the dependence. This dependence has been detailed in the discussion of the variance analysis of Roseland  shown in Section \ref{grid_samp_def}. 
Indeed, even if $\beta=1$, if we traced the proof of the variance analysis stated in Theorem \ref{variancethm}, we know that the ``slower'' convergence rate incurred by the worse variance $\mathcal{O}\big(\frac{\sqrt{\log(n)}}{n^{1/2}\epsilon^{d/2-1/2}}\big)$ comes from Step 2 in the proof, where we have a double integral but not a single integral like that in the DM analysis. 
As a result, the variance shown in \eqref{Proof Variance Theorem key variance quantity}, which is slower than that of DM with $\alpha=0$ by $\epsilon^{-d/4}$ even if $\beta=1$. We shall also mention that if we are able to take $m$ to be much larger than $n$, (e.g. $\beta>1$), then this variance is better controlled, and ``ideally'' the variance of Roseland is the same as that of DM with $\alpha=0$. But this is not the region we are concerned with in practice, so we do not pursue the related convergence result in this paper. 
To provide more insights,}
in Figure \ref{m=n figures}, we show a numerical result for an illustration. {The experiment is set as follows. We first i.i.d. uniformly sample $n$ points from $S^1$ as the dataset, and  i.i.d. uniformly sample $m=n$ data points from $S^1$ as the landmark set, which is independent of the dataset. Then we run DM on the dataset ($n$ points) only, and run Roseland on the dataset and landmark set ($2n$ points in total). Since we know the true eigenfunctions and eigenvalues of $S^1$, we are able to compute the $L^2$ and $L^\infty$ error of eigenvectors and absolute error of eigenvalues. We see that even when $m=n=2,500$, the performance of Roseland is ``slightly'' worse than that of DM. This indicates that even if we sacrifice the computational efficiency by setting $m$ to be as large as $m=n$, we do not gain benefit from Roseland compared with the vanilla DM. Note that in practice, if we have that many points, we would prefer to shift some points in the landmark set to the dataset to improve both the computational efficiency and accuracy.} 

%Another interesting point related to the finding in Figure \ref{m=n figures} is the following question -- intuitively, if choosing the landmark set independent from the data set is the reason that increases the variance, is it advantageous to choose the landmark set as a subset of of the data set?

\begin{figure}[bht!]\centering
	\includegraphics[width=7cm]{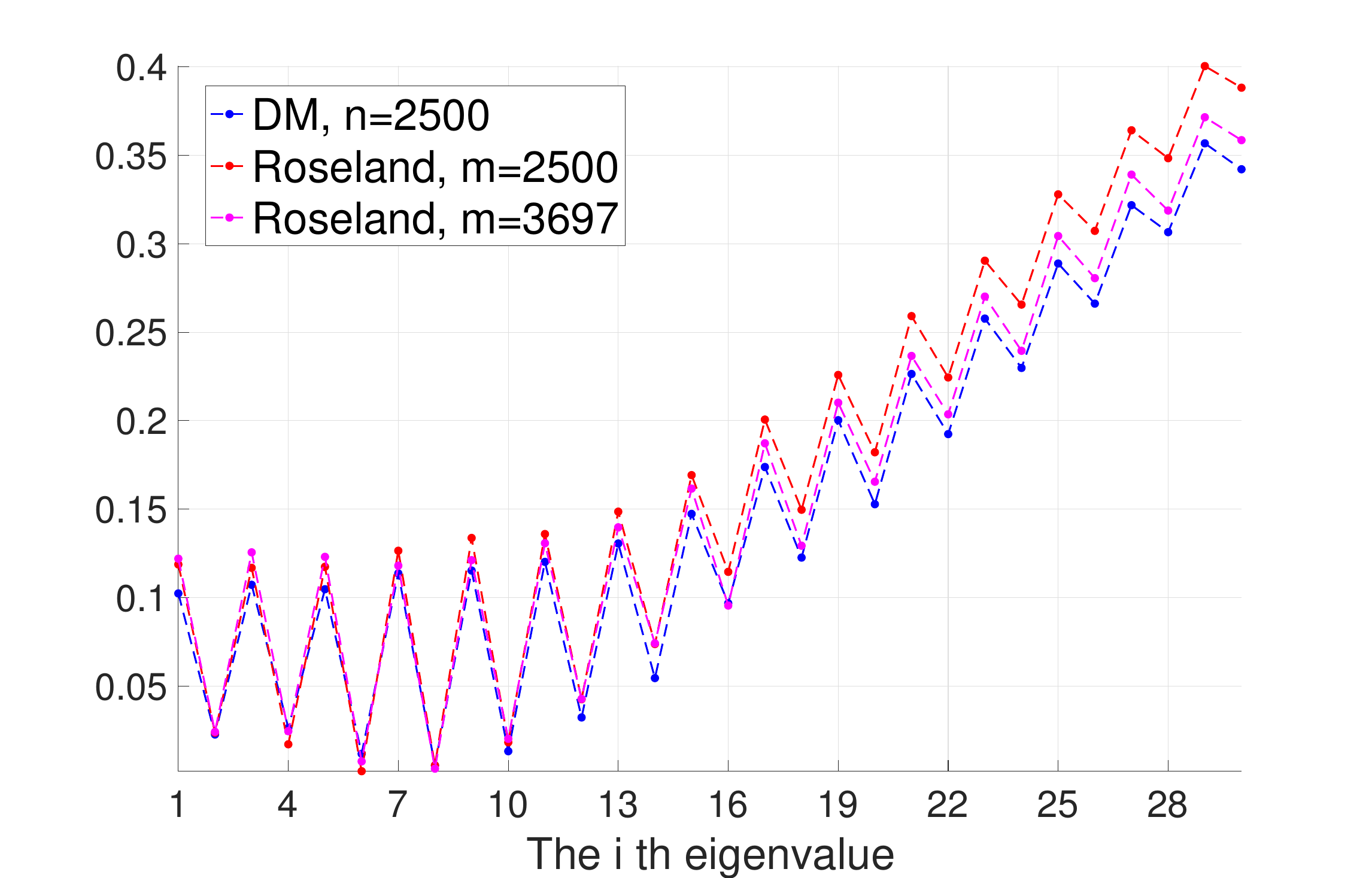}
	\includegraphics[width=7cm]{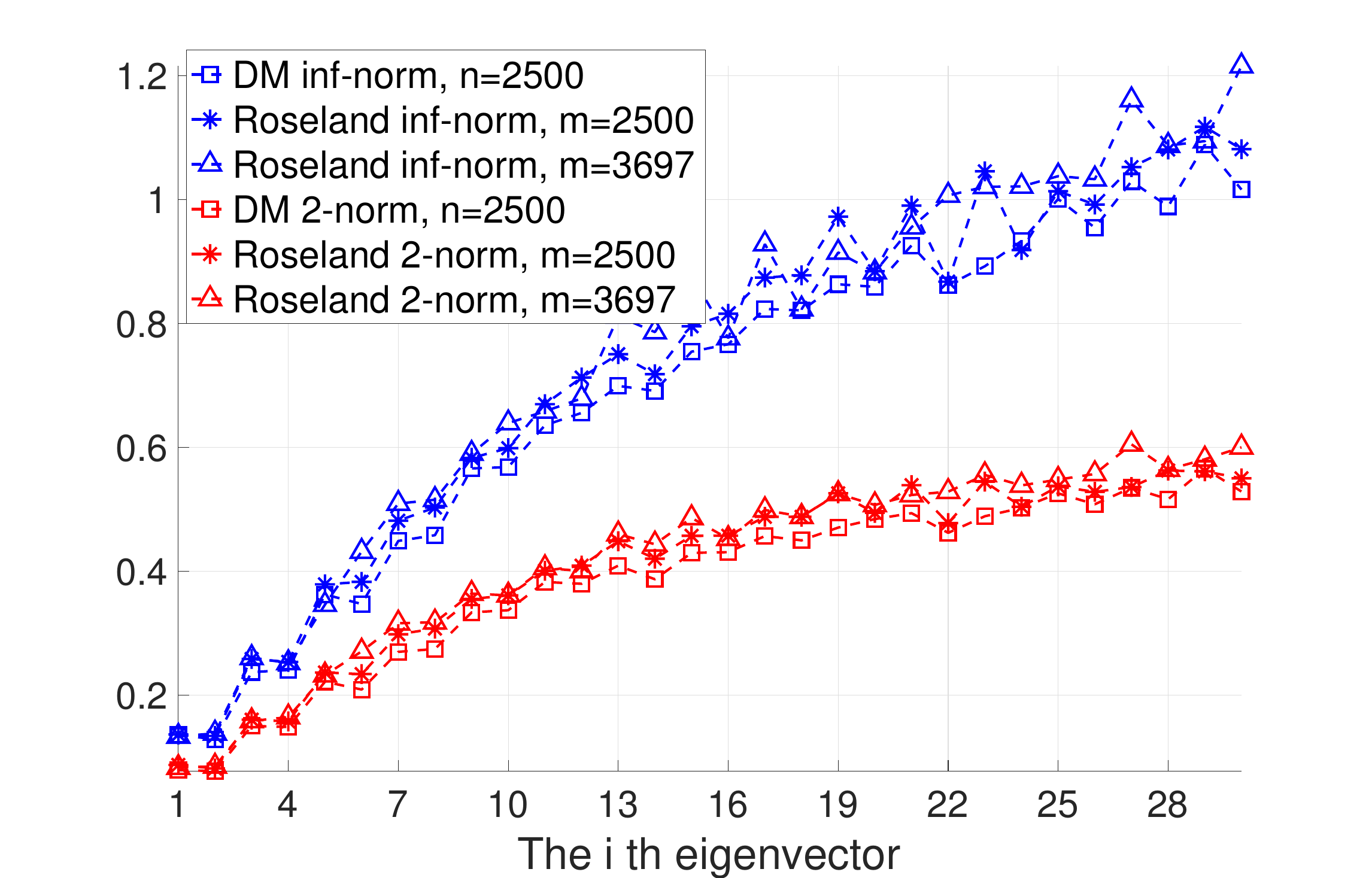}
	\caption{{ The dataset is uniformly sampled from $S^1$. We compare DM with $n=2,500$, Roseland with $n=2500$ and $m=n=2500$ landmarks, also Roseland with $n=2500$ and $m=3,697\approx n^{1.05}$ landmarks. Left column: the relative errors of eigenvalues; right column: the relative errors of eigenfunctions. Note that even when $m=n$, the performance of Roseland is worse than that of DM. The performance increases when we increase the landmarks, but it is still not as good as DM.}}\label{m=n figures}
\end{figure}

\subsection{Optimal spectral convergence rate}

Another relevant topic we need to discuss is spectral convergence. To the best of our knowledge, there have been several papers studying the spectral convergence rate, for example \cite{von2008consistency,wang2015spectral,shi2015convergence,trillos2018error,calder2019improved}, but the ``optimal'' spectral convergence rate is not yet known. Among these papers, our approach and setup are closest to \cite{von2008consistency,wang2015spectral}, where we studied the connection between the graph laplacian and the heat kernel, and utilized the covering number of the kernel function. In \cite{shi2015convergence}, the authors consider a different setup. In \cite{trillos2018error,calder2019improved}, the setup is similar, but the approach is different. The optimal transport approach considered in \cite{trillos2018error} and an improvement in \cite{calder2019improved} both have a potential to be combined with the analysis strategy considered in this paper. See \cite{DunsonWuWu2019} for more discussion. It is interesting to ask if we are able to determine the ``optimal'' spectral convergence rate under the manifold setup, and design an algorithm to achieve it. 
Last but not the least, we shall mention that without the knowledge of ``optimal spectral convergence rate'', we are not able to select bandwidth to recover the spectral structure. In our numerical simulation, we provide a fair comparison based on a naive bandwidth setup. However, we consistently find that if we select a smaller bandwidth for Roseland, then the eigenvalue reconstruction is better, while the Nyst\"om extension fails. 
To sum up, this topic is critical to further understand the algorithm and for the purpose of statistical inference, and we will report our exploration in our future work.

\section*{Acknowledgments}
Chao Shen thanks Xiucai Ding and Nan Wu for various discussion of the topic.
{The authors would like to thank the anonymous reviewers their constructive and helpful comments that improve the overall quality of this paper.}

\section*{Data Availability statement}

No new data were generated or analyzed in support of this review

\bibliographystyle{plain}
\bibliography{mybibfile}

%\end{document}

%\setcounter{equation}{0}
%\setcounter{figure}{0}
%\setcounter{theorem}{0}
%\setcounter{lemma}{0}
%\setcounter{definition}{0}
%\setcounter{remark}{0}
%\setcounter{proposition}{0}
%\renewcommand{\theequation}{SI.\arabic{equation}}
%\renewcommand{\thelemma}{SI.\arabic{lemma}}
%\renewcommand{\theproposition}{SI.\arabic{proposition}}
%\renewcommand{\thedefinition}{SI.\arabic{definition}}
%\renewcommand{\thetable}{SI.\arabic{table}}
%\renewcommand{\thetheorem}{SI.\arabic{theorem}}
%\renewcommand{\thefigure}{SI.\arabic{figure}}
%\renewcommand{\theremark}{SI.\arabic{remark}}

\section{More relevant work}\label{Appendix:MoreRelatedWork}

There are some works focusing on accelerating the spectral clustering. Since they are related to our focus, we summarize them here. 
In \cite{pham2018large}, the authors propose to use a smaller subset as landmarks, and take the co-clustering idea into account for the spectral clustering purpose. The landmarks are chosen by uniform sampling or by the k-means clustering. Although this landmark idea is the same to our proposed algorithm at the first glance, they are different since the affinity matrix between data and landmarks are normalized differently. Specifically, the authors normalize the affinity matrix by both row and column sums. This approach can thus be classified as the third category. 
In \cite{yan2009fast}, the authors apply the k-means or random projection tree first to pre-group the data sets and obtain $m$ centroids and hence $m$ groups. Then, apply the spectral clustering on these $m$ centroids to obtain $k$ new centroids. Finally, merge the original $m$ groups according to these $k$ centroids. This is thus also classified as the third category.
Although these approaches do not focus on recovering the intrinsic geometry, like the geodesic distance or almost isometric embedding, we still consider them as relative studies since the spectral clustering is directly related to the GL. Indeed, under the manifold setup, the spectral clustering can be understood as finding the connected components, which is related to recovering the null space of the Laplace-Beltrami operator. Moreover, the behavior of those algorithms on the manifold setup is not clear.

\cite{salhov2015approximately} propose a dictionary-based method, where they incrementally construct an approximative map by using a single scan of the data. This algorithm is greedy, inefficient and sensitive to the scan order, as is commented in their followup paper \cite{bermanis2016incomplete}. \cite{bermanis2016incomplete} provide an incomplete pivoted QR-based deterministic method for dimensionality reduction {\em after} running DM.  
In general, given a generic matrix $A$, the authors select a subset of columns of $A$ called pivots, and then use them to perform an incomplete QR factorization to approximate $A$. This method is applied to reduce the dimension of DM, and is claimed to preserve the data geometry up to a user-specified distortion rate.  
Based on the nature of this algorithm, while it looks relevant at the first glance, this approach is not directly related to our acceleration mission.

Another useful and closely related algorithms are the subspace sampling \cite{drineas2006subspace} and the CUR decomposition \cite{mahoney2009cur,wang2013improving}. Those algorithms aim to find interpretable $k$ rank approximation of the original data matrix with respect to the matrix Frobenius norm. For example, the SVD would give us \cite{mahoney2009cur}:  $(1/2)\textup{age}- (1/\sqrt{2})\textup{height}+(1/2)\textup{income}$. This kind of linear combination of uncorrelated features are difficult to interpret in some situations. The CUR on the other hand, is interpretable as it decomposes the original data matrix into a small number of actual columns and rows.  In general, the variable selection algorithms, like LASSO, can also be considered as dimension reduction algorithms, although the purpose of variable selection is totally different and the computational complexity might not be suitable for our purpose. However, from the aspect of reducing the dimension of the dataset, in the broad sense they are closely related to algorithms in the first category.

In \cite{cheng2005compression}, the authors give a new decomposition form to compress a rank-deficient matrices when the SVD cannot be used efficiently. Specifically, a matrix $A$ of rank $k$ is represented as $A=UBV$, where $B$ is a $k\times k$ submatrix of $A$, and $U,V$ are well-conditioned matrices that each contain a $k \times k$ identity submatrix. The geometric interpretation of this decomposition is: columns (rows) of A are expressed as linear combinations of $k$ selected columns (rows) of A, the selection induces the matrix $B$, such that in the new coordinate system, the action of $A$ is represented by the action of $B$.

\section{Technical background for the proof}\label{lemma}

In this section we prepare some known technical lemmas and necessary results for our proof. To be succinct, detailed proofs will be skipped and we refer readers to the relative papers.

\subsection{Some basic differential geometry facts}

The following Lemma is critical for us to control the size of the eigenfunction.
\begin{lemma}[Homander's inequality \cite{hormander1968spectral}]\label{lemma hormander}
	Fix a compact Riemannian manifold $(M,g)$. For the $l$-th pair of eigenvalue $\lambda_l$ and eigenfunction $\phi_l$ of the Laplace-Beltrami operator, where $\|\phi_l\|_2=1$, we have  
	\begin{align}
	\|\phi_l\|_\infty \leq  C_1 \lambda_l^{\frac{d-1}{4}}\,,   \nonumber 
	\end{align}
	where $C_1$ is a constant depending on the injectivity radius and sectional curvature of the manifold $M$.
\end{lemma}

The following lemma is the well-known Weyl's law, which controls the eigenvalue growth.
\begin{lemma}[Weyl's law \cite{courant1920eigenwerte}]\label{lemma weyls law}
	Fix a compact and connected Riemannian manifold $(M,g)$. The eigenvalues of the Laplace-Beltrami operator, denoted as $0=\lambda_1< \lambda_2\leq\ldots$, satisfy
	\begin{align}
	l=C_2 \lambda^{d/2}_l+O\big(\lambda^{\frac{d-1}{2}}_l \log \lambda_l\big)\,,  \nonumber 
	\end{align}
	where $C_2$ is a constant depending on the volume of the manifold. 
\end{lemma}

The proof of the following truncation lemma can be found in, for example, \cite{Coifman2006}.

\begin{lemma}\label{trunc}
	Suppose $f\in L^{\infty}(M^{d})$ and $0<\gamma<1/2$. Then for any $x\in M^{d}$, when $\epsilon$ is sufficiently small, we have
	$$
	\left|\int_{M^{d}\backslash\widetilde{\mathcal{B}}_{\epsilon^{\gamma}}(x)}\epsilon^{-d/2}K_{\epsilon}(x,y)f(y)\,dV(y)\right|=\mathcal{O}(\epsilon^{2})\,,
	$$
	where $\mathcal{O}(\epsilon^{2})$ depends on $\|f\|_{\infty}$, and $\widetilde{\mathcal{B}}_{\epsilon^{\gamma}}(x)\coloneqq\iota^{-1}(\mathcal{B}_{\epsilon^{\gamma}}(x)\cap\iota(M^{d}))\subseteq M^{d}$, where $\mathcal{B}_{\epsilon^{\gamma}}(x)$ is the Euclidean ball with radius $\epsilon^{\gamma}$ centered at $x$.  
\end{lemma}

The proof of the following approximation of identity type lemma follows the standard argument, and can be found in \cite[Lemma B.3]{singer2016spectral}.

\begin{lemma}\label{for_numerator_est}
	If $f\in C^{4}(M^{d})$, then for all $x\in M^{d}$, we have
	$$
	\int_{M^{d}}\epsilon^{-d/2}K_{\epsilon}(x,y)f(y)\,dV(y)=f(x)+\frac{\epsilon\mu_{1,2}^{(0)}}{2d}(\Delta f(x)-w(x)f(x))+\mathcal{O}(\epsilon^{2})\,,
	$$
	where $w(x)=\frac{1}{3}s(x)-\frac{d}{12|S^{d-1}|}\frac{\mu_{1,3}^{(1)}}{\mu_{1,2}^{(0)}}\int_{S^{d-1}}\Second_{x}^{2}(\theta,\theta)\,d\theta$, $s(x)$ is scalar curvature at $x$, $\Second_{x}$ is the second fundamental form of the embedding at $x$, and $|S^{d-1}|$ is the volume of the canonical $(d-1)$-sphere.
\end{lemma}

\begin{lemma}\label{numerator_est}
	If $f\in C^{4}(M^{d})$, then for all $x\in M^{d}$, we have
	\begin{align}
	\nonumber &\int_{M^{d}}\epsilon^{-d/2}K_{\epsilon}(x,y)f(y)p_{X}(y)\,dV(y) 
	\\\nonumber =&\, p_{X}(x)\left[1+\frac{\epsilon\mu_{1,2}^{(0)}}{d}\left(\frac{\Delta p_{X}(x)}{2p_{X}(x)} - \frac{1}{2}w(x)\right)f(x)\right]
	+ \epsilon\frac{\mu_{1,2}^{(0)}p_{X}(x)}{d}\left(\frac{\Delta f(x)}{2}+\frac{\nabla f(x)\cdot\nabla p_{X}(x)}{p_{X}(x)}\right) + \mathcal{O}(\epsilon^{2})
	\end{align}
	where $w(x)$ as in lemma \ref{for_numerator_est}.
\end{lemma}
\begin{proof}
	It is an immediate consequence of lemma \ref{for_numerator_est} by replacing $f(y)$ by $f(y)p_{X}(y)$.
\end{proof}

\subsection{Glivenko-Cantelli class and entropy bound}
\begin{definition}
	A set $\mathfrak{F}\subseteq C(M)$ is called a Glivenko-Cantelli class if 
	$$\textup{sup}_{f\in\mathfrak{F}}|\mathbb{P}_{n}f - \mathbb{P}f|\longrightarrow 0\quad\textup{a.s.}$$
\end{definition}

To handle the randomness, we need to control the complexity of ``all possible'' functions that are related to the random samples. The following definition is the quantity we need.

\begin{definition}
	Let $(\mathcal{F},\norm{\cdot})$ be a subset of normed space of real functions $f:M\rightarrow\mathbb{R}$. Given two functions $l$ and $u$, the bracket $[l,u]$ is the set of all functions $f$ such that $l(x)\leq f(x)\leq u(x)$ for all $x\in M$. An $\epsilon$-bracket is a bracket such that $\norm{u-l}<\epsilon$. The bracketing number $N_{[]}(\epsilon,\mathcal{F},\norm{\cdot})$ is the minimum number of $\epsilon$-bracket needed to cover $\mathcal{F}$. The upper and lower bounds $u$ and $l$ of the brackets need not belong to $\mathcal{F}$ but are assumed to have finite norms.
\end{definition}

To proceed to the spectral convergence rate, we need the following results that control the rate of convergence from finite sample points. 
To this end, we need the following entropy bound \cite[Theorem 19]{von2008consistency}. Or see \cite{mendelson2003few} for a more systematic review of the topic.

\begin{theorem}[Entropy bound \cite{mendelson2003few, von2008consistency}]\label{entropy_bd}
	Let $(\mathcal{X},\mathcal{B},\mathbb{P})$ be an arbitrary probability space, $\mathcal{F}$ a class of real-valued functions on $\mathcal{X}$ with $\norm{f}_{\infty}\leq 1$. Let $(X_{n})$ be a sequence of i.i.d. random variables drawn from $\mathbb{P}$, and $(\mathbb{P}_{n})$ the corresponding empirical distributions. For $\delta>0$, there exists a constant $C_E>0$ such that, for all $ n\in\mathbb{N}$, with probability higher than $1-\delta$:
	$$\sup_{F\in\mathcal{F}}\abs{\mathbb{P}_{n}F-\mathbb{P}F}\leq\frac{C_E}{\sqrt{n}}\int_{0}^{\infty}\sqrt{\log N(\mathcal{F},r,L_{2}(\mathbb{P}_{n}))}\,dr + \sqrt{\frac{1}{n}\log\frac{2}{\delta}}.$$
\end{theorem}

The following lemma is the standard statement about the covering number when the kernel is Gaussian. Since the proof can be found in \cite[Lemma 4.1]{wang2015spectral}, we omit it here.

\begin{lemma}[Covering number of Gaussian]\cite[Lemma 4.1]{wang2015spectral}\label{cover_num_gau}
	Take $\epsilon>0$ and $k(x,y):=K_\epsilon(x,y)=e^{-\norm{x-y}^2/\epsilon}$ in Definition \ref{Definition different classes}. For $r>0$, we have the following bound:
	$$
	N(\mathcal{K},r,\norm{\cdot}_{\infty})\leq\Big(\frac{24\sqrt{2d}D_{M}}{r\epsilon}\Big)^{2d},
	$$
	where $D_{M}$ is the diameter of $M^{d}$.
\end{lemma}

\subsection{Facts we need for the spectral convergence}

Next we need the notion of collectively compact convergence. Recall that $C(M)$ is a Banach space with the $\norm{\cdot}_{\infty}$ norm.
\begin{definition}(\cite[p.122]{chatelin2011spectral})
	Let $(E, \|\cdot\|_{E})$ be an arbitrary Banach space.
	A sequence of operators $T_{n}:E\rightarrow E$ converges to $T:E\rightarrow E$ collectively compactly if and only if the following conditions are satisfied:
	\begin{enumerate}
		\item $T_{n}\rightarrow T$ pointwisely,
		\item the set $\cup_{n}(T_{n}-T)B$ is relatively compact in $E$, where $B$ is the closed unit ball centered at $0$ in $E$.
	\end{enumerate}
\end{definition}

We state the following theorem for the spectral convergence.

\begin{theorem}(\cite{von2008consistency,chatelin2011spectral})
	Let $(E,\norm{\cdot}_{E})$ be any Banach space, $\{T_{n}\}_{n}$ and $T$ be bounded linear operators on $E$ such that $T_{n}\rightarrow T$ compactly. Let $\lambda\in\sigma(T)$ be an isolated eigenvalue with finite multiplicity $m$, and $\mathcal N_\lambda\subset \mathbb{C}$ be an open neighborhood of $\lambda$ such that $\sigma(T)\cap \mathcal N_\lambda=\lambda$. Then:
	\begin{enumerate}
		\item (Convergence of eigenvalues) There exists an $N\in\mathbb{N}$ such that for all $n>N$, the set $\sigma(T_{n})\cap \mathcal N_\lambda$ is an isolated part of $\sigma(T_{n})$ consisting of at most $m$ different eigenvalues, and their multiplicities sum to $m$. Moreover, the sequence of sets $\{\sigma(T_{n})\cap \mathcal N_\lambda\}\rightarrow\{\lambda\}$ in the sense that every sequence $\{\lambda_{n}\}$ with $\lambda_{n}\in\sigma(T_{n})\cap \mathcal N_\lambda$ satisfies $\lim\lambda_{n}\rightarrow\lambda$.
		\item (Convergence of spectral projections) Let $P_\lambda$ be the spectral projection of $T$ corresponding to $\lambda$. Let $P_{\lambda,n}$ be the spectral projection of $T_{n}$ corresponding to $\sigma(T_{n})\cap \mathcal N_\lambda$, then $P_{\lambda,n}\rightarrow P_\lambda$ pointwisely.
	\end{enumerate}
\end{theorem}

Note that when $\lambda$ is simple with the eigenfunction $f$, this theorem can be simplified. Indeed, there exists $N\in\mathbb{N}$ such that for all $n>N$, the sets $\sigma(T_{n})\cap \mathcal N_\lambda$ consists of a simple eigenvalue $\lambda_{n}$, and $\lim\lambda_{n}\rightarrow\lambda$. Moreover, for the corresponding eigenfunctions $f_{n}$, there exists a sequence $a_{n}\in \{1,\,-1\}$ so that $\norm{a_{n}f_{n}-f}_{E}\rightarrow 0$.

Recall that the $L^2(M)$ space is a separable Hilbert space. The following lemma is the key toward the spectral convergence rate \cite{DunsonWuWu2019}. The proof can be found in \cite{DunsonWuWu2019}, so we omit it here. We mention that last part of this lemma is also considered in \cite{calder2019improved} to improve the $L^2$ convergence rate.

\begin{lemma}\label{ratio of eigenvalues}
	Let $A$ and $B$ be two compact self-adjoint operators from the separable Hilbert space $H$ to $H$. Let $(\cdot, \cdot)$ be the inner product of $H$. 
	Suppose the eigenvalues of $A$, denoted as $\lambda_l(A)$, $l=1,\ldots$, are simple and positive, and the eigenvalues of $B$, denoted as $\lambda_l(B)$, $l=1,\ldots$, are simple and bounded from below so that $1=\lambda_1(A) > \lambda_2(A) > \cdots\geq 0$ and $\lambda_1(B) >  \lambda_2(B) > \cdots $. Denote $\{u_i\}$ to be the orthonormal eigenfunctions of $A$ and $\{w_i\}$ to be orthonormal eigenfunctions of $B$. Furthermore, denote 
	\begin{equation}
	\gamma_i(B):=\min(\lambda_i(B)-\lambda_{i-1}(B), \,\lambda_{i+1}(B)-\lambda_i(B))\,.\label{Proof proposition 1 Definition gamma i} 
	\end{equation}
	Let $E: = A-B$. Then, for $\epsilon>0$ we have the following statements:
	\begin{enumerate}
		\item
		If $\big|\frac{(E f,f)}{(Af,f)}\big| \leq \epsilon$ for all $f \in L^2$, then for all $i$, we have 
		\begin{equation}
		\left|\frac{1-\lambda_i(B)}{1-\lambda_i(A)}-1\right|\leq \epsilon\,.\nonumber
		\end{equation}
		\item
		If $\|Bu_i-\lambda_i(B)u_i\|_2 \leq \epsilon$, then for $a=1$ or $-1$, we have 
		\begin{equation}
		\|aw_i- u_i\|_2 \leq \frac{2\epsilon}{\gamma_i(B)}. \nonumber
		\end{equation}
		Moreover, 
		\begin{equation}
		|(u_i,w_i)| \geq 1-\frac{\epsilon}{\gamma_i(B)}\,.\nonumber
		\end{equation}
		\item The eigenvalues satisfy
		\begin{equation}
		|\lambda_i(A)-\lambda_j(B)|\leq \frac{\|Ew_j\|_2}{|(u_i,w_j)|}\,.\nonumber
		\end{equation}
	\end{enumerate}
\end{lemma}

The following result describes how the spectral convergence, or more precisely, the eigenfunction convergence, happens when a sequence of operators converges. It is a restatement of  \cite[Equation (5) in,][Theorem 3]{atkinson1967numerical}. We refer the readers with interest in its derivation to \cite[Theorem SI.1]{DunsonWuWu2019}.

\begin{theorem}\label{atkinson1967bound}\cite{atkinson1967numerical}, or \cite[Theorem SI.1]{DunsonWuWu2019}
	Let $(E, \|\cdot\|_{E})$ be an arbitrary Banach space. Let $\{T_n\}_{n=1}^\infty$ and $T$ be compact linear operators on $E$ such that $\{T_n\}_{n=1}^\infty$ converges to $T$ collectively compactly. 
	For a nonzero eigenvalue $\lambda \in \sigma(T)$, denote the corresponding spectral projection by $\textup{Pr}_{\lambda}$. Let $D\subset \mathbb{C}$ be an open neighborhood of $\lambda$ such that $\sigma(T)\cap D$ =$\{ \lambda\}$. There exists some $N\in \mathbb{N}$ such that for all $n > N$, $\sigma(T_n)\cap D=\{\lambda_n\}$. Let $\textup{Pr}_{\lambda_n}$ be the corresponding spectral projections of $T_n$ for $\lambda_n$. Let $r <|\lambda|$ and $r<\textup{dist}(\{\lambda\}, \sigma(K)\setminus \{\lambda\})$. Then, for every $x \in \textup{Pr}_\lambda(E)$, we have
	\begin{equation}
	\|x-\textup{Pr}_{\lambda_n} x\|_{E} \leq \max_{z\in \Gamma_r(\lambda)} \frac{2r\|R_{z}(T)\|}{\min_{z\in \Gamma_r(\lambda)}|z|}\left(\|(T_n-T)x\|_{E} +\|R_{z}(T)x\|_{E}\|(T-T_n)T_n\|\right)\,,
	\nonumber 
	\end{equation}
	where $\Gamma_r(\lambda):=\{z\in \mathbb{C}|\,|z-\lambda|=r\}$.
\end{theorem}

We need the following lemma to connect the ultimate eigenfunction convergence to Theorem \ref{atkinson1967bound}.

\begin{lemma}[convergence of one-dim projections]\label{cong_1_dim_proj} \cite[Proposition 18]{von2008consistency}
	Let $\{v_{n}\}_{n=1}^\infty$ be a sequence of vectors in a Banach space $(E,\norm{\cdot})$ with $\norm{v_{n}}=1$. Denote $\textup{Pr}_{v_n}$ be the projection onto the one-dimensional subspaces spanned by $v_{n}$. Take $v\in E$ with $\norm{v}=1$. Then there exists a sequence of signs $(a_{n})$ such that
	$$
	\norm{a_{n}v_{n}-v}\leq2\norm{v-\textup{Pr}_{v_n}(v)}.
	$$
	In particular, if $\norm{v-\textup{Pr}_{v_n}(v)}\rightarrow 0$, then $v_{n}\rightarrow v$ in $(E,\norm{\cdot})$ up to a change of sign.
\end{lemma}

\begin{table}[htb!]\caption{Table of notation throughout the proof. Assume $f,g\in C(M)$.}
	\begin{center}
		\begin{tabular}{r c l}
			\toprule
			$K_{\textup{ref},\epsilon}(x,y)$ & &$\displaystyle\int_{M}K_{\epsilon}(x,z)K_{\epsilon}(z,y)p_{Y}(z)\,dV(z)$ \\
			$d_{\textup{ref},\epsilon}(x)$ & & $\displaystyle\int_{M}K_{\textup{ref},\epsilon}(x,y)p_{X}(y)\,dV(y)$\\
			$M_{\textup{ref},\epsilon}(x,y)$ & & $\displaystyle \frac{K_{\textup{ref},\epsilon}(x,y)}{d_{\textup{ref},\epsilon}(x)}$\\
			$T_{\textup{ref},\epsilon}f(x)$ & & $\displaystyle\int_{M}M_{\textup{ref},\epsilon}(x,y)f(y)p_{X}(y)\,dV(y)$\\
			$\widehat{K}_{\textup{ref},\epsilon,n}(x,y)$ & & $\displaystyle\frac{1}{m}\sum_{j}^{m}K_{\epsilon}(x,z_{j})K_{\epsilon}(z_{j},y)$\\
			$\widehat{d}_{\textup{ref},\epsilon,n}(x)$ & & $\displaystyle\frac{1}{n}\sum_{i=1}^{n}\widehat{K}_{\textup{ref},\epsilon,n}(x,x_{i})$\\
			$\widehat{M}_{\textup{ref},\epsilon,n}(x,y)$ &  & $\displaystyle\frac{\widehat{K}_{\textup{ref},\epsilon,n}(x,y)}{\widehat{d}_{\textup{ref},\epsilon,n}(x)}$\\
			$\widehat{T}_{\textup{ref},\epsilon,n}f(x)$ & & $\displaystyle\frac{1}{n}\sum_{i=1}^{n}\widehat{M}_{\textup{ref},\epsilon,n}(x,x_{i})f(x_{i})$\\
			$d_{\textup{ref},\epsilon,n}(x)$ & & $\displaystyle\frac{1}{n}\sum_{i=1}^n K_{\textup{ref},\epsilon}(x,x_i)$\\
			$T_{\textup{ref},\epsilon,n}f(x)$ & & $\displaystyle\frac{1}{n}\sum_{i=1}^{n}M_{\textup{ref},\epsilon}(x,x_{i})f(x_{i})$\\
			$\widehat{M}^{(d)}_{\textup{ref},\epsilon,n}(x,y)$ &  & $\displaystyle\frac{K_{\textup{ref},\epsilon}(x,y)}{\widehat{d}_{\textup{ref},\epsilon,n}(x)}$\\
			$\mathbb{P}f$ & & $\displaystyle\int f(x)p_X(x)\,dV(x)$\\
			$\mathbb{P}_{n}f$ & & $\displaystyle\frac{1}{n}\sum_{i=1}^{n}f(x_{i})$\\
			$\widetilde{\mathbb{P}}g$ & & $\displaystyle\int g(y)p_{Y}(y)\,dV(y)$\\
			$\widetilde{\mathbb{P}}_mg$ & & $\displaystyle\frac{1}{m}\sum_{l=1}^m g(y_l)$\\
			\bottomrule
		\end{tabular}
	\end{center}
	\label{tab:TableOfNotationForMyResearch}
\end{table}

{\allowdisplaybreaks

\section{Proof of Theorems \ref{biasthm} and \ref{variancethm} -- pointwise convergence} \label{biasproof}
The main theoretical contribution of this paper is handling how the overall diffusion behaves when the diffusion must goes through the landmark set. The result will form the base of the spectral convergence proof.

\begin{theorem}[Bias analysis]\label{Proof: Theorem1: bias analysis}
	Take $f\in C^{3}(M^{d})$. Then, for all $x\in M^{d}$ we have
	\begin{align}
	T_{\textup{ref},\epsilon}f(x) - f(x) =&\, \frac{\epsilon\mu_{1,2}^{(0)}}{d}\left(\frac{2\nabla p_{X}(x)}{p_{X}(x)}+\frac{\nabla p_{Y}(x)}{p_{Y}(x)}\right)\cdot\nabla f(x) + \frac{\epsilon\mu_{1,2}^{(0)}}{d}\Delta f(x) + \mathcal{O}(\epsilon^{3/2})\,,\nonumber
	\end{align}
	where the implied constant in $\mathcal{O}(\epsilon^{3/2})$ depends on the $C^3$ norm of $f$, the $C^2$ norms of $p_X$ and $p_Y$, and the Ricci curvature of the manifold.
\end{theorem}

\begin{proof}
	The bias analysis is almost the same as those shown in \cite{Coifman2006,singer2016spectral}, except the extra step handling the landmark set.
	By Definition \ref{GL_cts}, we have
	\begin{align}
	\nonumber T_{\textup{ref},\epsilon}f(x) &= \frac{\int_{M}K_{\textup{ref},\epsilon}(x,y)f(y)p_{X}(y)\,dV(y)}{\int_{M}K_{\textup{ref},\epsilon}(x,y)p_{X}(y)\,dV(y)}\,.
	\end{align}
	We first compute the numerator, which satisfies
	\begin{align}
	\nonumber\int_{M}K_{\textup{ref},\epsilon}(x,y)f(y)p_{X}(y)\,dV(y)
	=& \int_{M}\left(\int_{M}K_{\epsilon}(x,z)K_{\epsilon}(z,y)p_{Y}(z)\,dV(z)\right)f(y)p_{X}(y)\,dV(y)
	\\\nonumber =& \int_{M}\left(\int_{M}K_{\epsilon}(z,y)f(y)p_{X}(y)\,dV(y)\right)K_{\epsilon}(x,z)p_{Y}(z)\,dV(z)\,.
	\end{align}
	By Lemmas \ref{trunc} and \ref{numerator_est}, the right hand side can be expanded and organized as
	\begin{align}
	\nonumber &\, \epsilon^{d}\left(p_{X}(x)p_{Y}(x)+\frac{\epsilon\mu_{1,2}^{(0)}}{d}p_{Y}(x)\Delta p_{X}(x) + \frac{\epsilon\mu_{1,2}^{(0)}}{2d}p_{X}(x)\Delta p_{Y}(x)\right. \\
	\nonumber&\quad\left.-\frac{\epsilon\mu_{1,2}^{(0)}}{d}w(x)p_{X}(x)p_{Y}(x)+\frac{\epsilon\mu_{1,2}^{(0)}}{d}\nabla p_{X}(x)\cdot\nabla p_{Y}(x)\right)f(x)
	\\\nonumber &\quad+ \epsilon^{d}\left(\frac{2\epsilon\mu_{1,2}^{(0)}}{d}p_{Y}(x)\nabla p_{X}(x)+\frac{\epsilon\mu_{1,2}^{(0)}}{d}p_{X}(x)\nabla p_{Y}(x)\right)\cdot\nabla f(x) \nonumber\\
	&\quad+ \frac{\epsilon^{d+1}\mu_{1,2}^{(0)}}{d}p_{X}(x)p_{Y}(x)\Delta f(x)+\mathcal{O}(\epsilon^{d+2})\,.
	\end{align}
	Next, note that the denominator is just the numerator with $f(x)$ replaced by the constant function $1$. Hence, we have:
	\begin{align}
	\nonumber \int_{M}K_{\textup{ref},\epsilon}(x,y)p_{X}(y)\,dV(y)
	=&\, \epsilon^{d}\left(p_{X}(x)p_{Y}(x)+\frac{\epsilon\mu_{1,2}^{(0)}}{d}p_{Y}(x)\Delta p_{X}(x) + \frac{\epsilon\mu_{1,2}^{(0)}}{2d}p_{X}(x)\Delta p_{Y}(x)\right. \\
	\nonumber&\left.\qquad-\frac{\epsilon\mu_{1,2}^{(0)}}{d}w(x)p_{X}(x)p_{Y}(x)+\frac{\epsilon\mu_{1,2}^{(0)}}{d}\nabla p_{X}(x)\cdot\nabla p_{Y}(x)\right)+\mathcal{O}(\epsilon^{d+2})\,.
	\end{align}
	By putting them together, we have
	\begin{align}
	\nonumber T_{\textup{ref},\epsilon}f(x) &= f(x) + \frac{\epsilon\mu_{1,2}^{(0)}}{d}\left(\frac{2\nabla p_{X}(x)}{p_{X}(x)}+\frac{\nabla p_{Y}(x)}{p_{Y}(x)}\right)\cdot\nabla f(x) + \frac{\epsilon\mu_{1,2}^{(0)}}{d}\Delta f(x) + \mathcal{O}(\epsilon^{2})\,.
	\end{align}
\end{proof}

\begin{remark}
	Note the constants in front of $\frac{\nabla p_{X}(x)}{p_{X}(x)}$ and $\frac{\nabla p_{Y}(x)}{p_{Y}(x)}$ respectively. Intuitively, the $2$ in front of $\frac{\nabla p_{X}(x)}{p_{X}(x)}$ comes from the ``2'' steps diffusion. 
\end{remark}

The variance analysis is less trivial, and we need to apply the large deviation theorem when dependence exists. 

\begin{theorem}[Variance analysis]
	Take $\mathcal{X}=\{x_i\}_{i=1}^n$ and $\mathcal{Y}=\{y_j\}_{i=1}^m$, where $m=[n^{\beta}]$ for some $0<\beta\leq 1$ and $[x]$ is the nearest integer of $x\in\mathbb{R}$. 
	Take $f\in C(M^{d})$ and denote $\textbf{\textit{f}}\in\mathbb{R}^{n}$ such that $\textbf{\textit{f}}_{i}=f(x_{i})$. Let $\epsilon = \epsilon(n)$ so that $\frac{\sqrt{\log n}}{n^{\beta/2}\epsilon^{d/2+1/2}}\rightarrow 0$ and $\epsilon\rightarrow 0$ when $n\rightarrow\infty$. Then with probability higher than $1 - \mathcal{O}(1/n^{2})$, we have
	\begin{equation}
	\big[(I-(D^{(\textup{R})})^{-1}W^{(\textup{R})})\textbf{\textit{f}}\;\big](i) = f(x_{i}) - T_{\textup{ref},\epsilon}f(x_{i}) + \mathcal{O}\Big(\frac{\sqrt{\log n}}{n^{\beta/2}\epsilon^{d/2-1/2}}\Big)
	\end{equation}
	for all $i=1,2,\ldots,n$.
\end{theorem}

\begin{proof}
	Define $L:=I-(D^{(\textup{R})})^{-1}W^{(\textup{R})}$. Fix some $x_{i}\in M^{d}$, by Definition \ref{ref_affinity},
	\begin{align}
	\nonumber(L\textbf{\textit{f}})_{i} \,&=\frac{\sum_{j=1}^{n}W_{ij}(f(x_{i})-f(x_{j}))}{\sum_{j=1}^{n}W_{ij}}=\frac{\sum_{j=1}^{n}K_{\textup{ref},\epsilon}(x_{i},x_{j})(f(x_{i})-f(x_{j}))}{\sum_{j=1}^{n}K_{\textup{ref},\epsilon}(x_{i},x_{j})}
	\\ &= \label{grid_samp}
	\frac{\frac{1}{nm}\sum_{j=1,k=1}^{n,m}\epsilon^{-d} K_{\epsilon}(x_{i},y_{k})K_{\epsilon}(y_{k},x_{j})(f(x_{i})-f(x_{j}))}{\frac{1}{nm}\sum_{j=1,k=1}^{n,m}\epsilon^{-d}K_{\epsilon}(x_{i},y_{k})K_{\epsilon}(y_{k},x_{j})}\,.
	\end{align}
	Define two random variables
	\begin{eqnarray}\label{def_F}
	F\coloneqq\epsilon^{-d}K_{\epsilon}(x_{i},Y)K_{\epsilon}(Y,X)(f(x_{i})-f(X))
	\end{eqnarray}
	\begin{eqnarray}\label{def_G}
	G\coloneqq\epsilon^{-d}K_{\epsilon}(x_{i},Y)K_{\epsilon}(Y,X)
	\end{eqnarray}
	Recall that the landmark set $\mathcal{Y}=\{y_{k}\}_{k=1}^{m}$ are i.i.d. samples from the random vector $Y$, which has the p.d.f. $p_{Y}$. Also, the data set $\mathcal{X}=\{x_{j}\}_{j=1}^{n}$ are i.i.d. sampled from the random vector $X$, which has the p.d.f. $p_{X}$. Moreover, $Y$ is independent of $X$. Denote by $F_{k,j}$ one realization of $F$ when the realization of the random vector $(X,Y)$ is   $(x_j,\,y_{k})$; in other words, $F_{k,j}=\epsilon^{-d}K_{\epsilon}(x_{i},y_{k})K_{\epsilon}(y_{k},x_{j})(f(x_{i})-f(x_{j}))$. Similarly for $G_{k,j}$.
	Hence the numerator in \eqref{grid_samp} can be written as a random variable
	\begin{eqnarray}\label{F}
	\frac{1}{mn}\textbf{F}\coloneqq\frac{1}{mn}\sum_{j=1,k=1}^{n,m}F_{k,j}
	\end{eqnarray}
	and the denominator can be written as a random variable
	\begin{eqnarray}\label{G}
	\frac{1}{mn}\textbf{G}\coloneqq\frac{1}{mn}\sum_{j=1,k=1}^{n,m}G_{k,j}
	\end{eqnarray}
	By the law of large number, we would expect that 
	$(L\textbf{\textit{f}})_{i}\approx\frac{\mathbb{E}(F)}{\mathbb{E}(G)}$.
	We now justify this intuition, and get the convergence rates of
	\begin{eqnarray}\label{top_rate}
	\frac{1}{mn}\sum_{j=1,k=1}^{n,m}F_{k,j}\longrightarrow\mathbb{E}(F)
	%\end{eqnarray}
	\quad\mbox{and}\quad
	%\begin{eqnarray}
	\frac{1}{mn}\sum_{j=1,k=1}^{n,m}G_{k,j}\longrightarrow\mathbb{E}(G)\,,
	\end{eqnarray}
	and hence the convergence rate of
	%\begin{equation}
	$(L\textbf{\textit{f}})_{i}\longrightarrow\frac{\mathbb{E}(F)}{\mathbb{E}(G)}$.
	%\end{equation}
	%
	We start by solving \ref{top_rate} and there are three steps.\newline\newline
	\textbf{[Step1]}. We know $\chi(\mathcal{A})=\mathcal{O}(\max(m,n))$ by \eqref{Optimal bound of the independent set A}. \newline
	
	\noindent\textbf{[Step2]}. We want to compute $\textup{Var}(F)=\mathbb{E}(F^{2})-(\mathbb{E}(F))^{2}$. From Lemma \ref{numerator_est}, we have
	\begin{align}
	\nonumber\mathbb{E}(F) =&\, \epsilon^{-d}\int_{M}\left(\int_{M}K_{\epsilon}(x_{i},y)K_{\epsilon}(y,x)p_{Y}(y)\,dV(y)\right)(f(x_{i})-f(x))p_{X}(x)\,dV(x)
	\\\nonumber =&\, \epsilon^{-d}\left(\frac{\epsilon^{d+1}\mu_{1,2}^{(0)}}{d}(2p_{Y}\nabla p_{X}+p_{X}\nabla p_{Y})\cdot\nabla ((f(x_{i})-f(x))p_{X}(x))|_{x=x_{i}}\right. \\
	\nonumber&\left. +\frac{\epsilon^{d}\mu_{1,2}^{(0)}}{d}p_{X}p_{Y}\Delta ((f(x_{i})-f(x))p_{X}(x))|_{x=x_{i}}+\mathcal{O}(\epsilon^{d+2})\right)
	\\\nonumber =&\, \frac{\epsilon\mu_{1,2}^{(0)}}{d}(2p_{Y}\nabla p_{X}+p_{X}\nabla p_{Y})\cdot\nabla ((f(x_{i})-f(x))p_{X}(x))|_{x=x_{i}} \\
	\nonumber& +\frac{\epsilon\mu_{1,2}^{(0)}}{d}p_{X}p_{Y}\Delta ((f(x_{i})-f(x))p_{X}(x))|_{x=x_{i}}+\mathcal{O}(\epsilon^{2});
	\end{align}
	Similarly, by applying Lemma \ref{numerator_est} twice, we get
	\begin{align}
	\label{Proof Variance Theorem key variance quantity}\mathbb{E}(F^{2}) =&\, \epsilon^{-2d}\int_{M}\left(\int_{M}K_{\epsilon}^{2}(x_{i},y)K_{\epsilon}^{2}(y,x)p_{Y}(y)\,dV(y)\right) (f(x_{i})-f(x))^{2}p_{X}(x)\,dV(x)
	\\\nonumber =&\, \int_{M}\epsilon^{-d}\left(\int_{M}\epsilon^{-d}K_{\epsilon}^{2}(y,x)(f(x_{i})-f(x))^{2}p_{X}(x)\,dV(x)\right) K_{\epsilon}^{2}(x_{i},y)p_{Y}(y)\,dV(y)\nonumber
	\\\nonumber =&\, \frac{\epsilon^{1-d}\mu_{2,0}^{(0)}\mu_{2,2}^{(0)}}{d}\Delta((f(y)-f(x_{i}))^{2}p_{X}(y)p_{Y}(y))|_{y=x_{i}} + \mathcal{O}(\epsilon^{2-d})\,.
	\end{align}
	Without loss of generality, we assume from now that $\Delta((f(y)-f(x_{i}))^{2}p_{X}(y)p_{Y}(y))|_{y=x_{i}}$ is positive. Therefore, when $\epsilon>0$ is sufficiently small, $\mathbb{E}(F^{2})\asymp \epsilon^{1-d}$.  
	Since $[\mathbb{E}(F)]^2 =\mathcal{O}(\epsilon^{2})$, we have $\mathbb{E}(F^{2}) \gg  (\mathbb{E}(F))^{2}$ and hence $\textup{Var}(F)\asymp \epsilon^{1-d}$.\newline\newline
	\noindent\textbf{[Step3]}. We apply Theorem \ref{bernstein} to establish the large deviation bound. 
	From \textbf{[Step2]} we have $\textup{Var}(F_{k,j})\asymp \epsilon^{1-d}$, which is controlled by $|F_{k,j}|\asymp \epsilon^{-d}$ by \eqref{def_F}. 
	Moreover, $\mathbb{E}(F_{k,j})=\mathcal{O}(\epsilon)$, so $|F_{k,j}-\mathbb{E}(F_{k,j})|\leq b$, for some $b>0$ satisfying $b\asymp \epsilon^{-d}$ when $\epsilon$ is sufficiently small.
	Hence, by Theorem \ref{bernstein}, let $m=[n^{\beta}]$ for any $0<\beta\leq 1$, we have for all $ t>0$:
	\begin{align}
	\nonumber &\mathbb{P}\Big(\frac{1}{mn}\textbf{F}-\mathbb{E}(F)\geq t\Big)\leq\textup{exp}\left(\frac{-8(mnt)^{2}}{25\chi(\mathcal{A})(\sum_{k,j}\textup{Var}(F_{k,j})+bmnt/3)}\right)
	\\\nonumber\asymp &\,\textup{exp}\left(\frac{-8(mnt)^{2}}{25(m+n)(c_1mn\epsilon^{1-d}+c_2mnt\epsilon^{-d}/3)}\right)
	=\textup{exp}\left(\frac{-8mnt^{2}}{25(m+n)(c_1\epsilon^{1-d}+c_2t\epsilon^{-d}/3)}\right)\,,
	\end{align}
	where $c_1,c_2>0$ are implied constants in $\textup{Var}(F_{k,j})\asymp \epsilon^{1-d}$ and $b\asymp \epsilon^{-d}$ respectively. Since our goal is to estimate the Laplace-Beltrami term, which has the prefactor of order $\epsilon$, we ask $\frac{t}{\epsilon}\rightarrow 0$. As a result, the exponent becomes
	\begin{align}
	\label{aa}\frac{8mnt^{2}}{25(m+n)(c_1\epsilon^{1-d}+c_2t\epsilon^{-d}/3)} \geq \frac{c_3n^{1+\beta}t^{2}}{(n^{\beta}+n)\epsilon^{1-d}} \geq\frac{c_3n^{1+\beta}t^{2}}{n\epsilon^{1-d}}=\frac{c_3n^{\beta}t^{2}}{\epsilon^{1-d}}\,,
	\end{align}
	for some constant $c_3>0$. 
	Then, if we choose $n$ such that $\frac{c_3n^{\beta}t^{2}}{\epsilon^{1-d}}=3\log n$, we have
	\begin{eqnarray}
	\nonumber t\asymp \frac{\sqrt{\log n}}{n^{\beta/2}\epsilon^{d/2-1/2}}\,,
	\end{eqnarray}
	which satisfies the request that $\frac{t}{\epsilon} \asymp \frac{\sqrt{\log n}}{n^{\beta/2}\epsilon^{d/2+1/2}}\rightarrow 0$ as $n\rightarrow\infty$ by assumption.
	As a result, by the chosen $n$, we have
	\begin{eqnarray}
	\nonumber\mathbb{P}\Big(\frac{1}{mn}\textbf{F}-\mathbb{E}(F)\geq t\Big)\leq\textup{exp}(-3\log n) = \frac{1}{n^{3}}\,.
	\end{eqnarray}
	Recall that we have fixed $x_{i}$ for some $i$. In order for $i=1,2,\ldots,n$, we simply use the union bound to get:
	\begin{eqnarray}
	\nonumber\mathbb{P}\Big(\frac{1}{mn}\textbf{F}-\mathbb{E}(F)\geq t; \textup{for}\;i=1,2,\ldots,n\Big)\leq n\times\frac{1}{n^{3}}=\frac{1}{n^{2}}\,.
	\end{eqnarray}
	
	In general, when $\Delta((f(y)-f(x_{i}))^{2}p_{X}(y)p_{Y}(y))|_{y=x_{i}}=0$, the variance is smaller, and $t$ is smaller. 
	Hence, we conclude that with probability $1-\mathcal{O}(n^{-2})$, for all $x_i$, the numerator of $(L\textbf{\textit{f}})_{i}$ equals
	\begin{eqnarray}\label{num_LF}
	\mathbb{E}(F)+\mathcal{O}\left(\frac{\sqrt{\log n}}{n^{\beta/2}\epsilon^{d/2-1/2}}\right)\,.
	\end{eqnarray}
	The the denominator follows the same line and we list the computation here for the convenience of the reader:
	\begin{align}
	\nonumber\mathbb{E}(G)\, &= \epsilon^{-d}\int_{M}\left(\int_{M}K_{\epsilon}(x_{i},y)K_{\epsilon}(y,x)p_{Y}(y)\,dV(y)\right)p_{X}(x)\,dV(x)= p_{X}p_{Y}(x_{i}) + \mathcal{O}(\epsilon)
	\end{align}
	and 
	\begin{align}
	\nonumber\mathbb{E}(G^{2}) \,&= \epsilon^{-2d}\int_{M}\left(\int_{M}K_{\epsilon}^{2}(x_{i},y)K_{\epsilon}^{2}(y,x)p_{Y}(y)\,dV(y)\right)p_{X}(x)\,dV(x)= \epsilon^{-d}\mu_{2,0}^{(0)}\mu_{2,0}^{(0)}p_{X}p_{Y}(x_{i}) + \mathcal{O}(\epsilon^{1-d})\,.
	\end{align}
	By the same argument, with probability $1-\mathcal{O}(1/n^{2})$, we have for all $x_i$ that the denominator of $(L\textbf{\textit{f}})_{i}$ satisfies
	\begin{eqnarray}\label{denom_LF}
	\mathbb{E}(G)+\mathcal{O}\left(\frac{\sqrt{\log n}}{n^{\beta/2}\epsilon^{d/2}}\right)\,.
	\end{eqnarray}
	By combining \eqref{num_LF} and \eqref{denom_LF} and the binomial expansion, we conclude that for all $i=1,2,\ldots,n$, with probability $1-\mathcal{O}(1/n^{2})$, we have
	\begin{align}
	\nonumber(L\textbf{\textit{f}})_{i} \,=& \frac{\mathbb{E}(F)+\mathcal{O}\left(\frac{\sqrt{\log n}}{n^{\beta/2}\epsilon^{d/2-1/2}}\right)}{\mathbb{E}(G)+\mathcal{O}\left(\frac{\sqrt{\log n}}{n^{\beta/2}\epsilon^{d/2}}\right)}
	= \frac{\mathbb{E}(F)}{\mathbb{E}(G)} + \mathcal{O}\left(\frac{\sqrt{\log n}}{n^{\beta/2}\epsilon^{d/2-1/2}}\right)
	=
	f(x_{i}) - T_{\textup{ref},\epsilon}f(x_{i}) + \mathcal{O}\Big(\frac{\sqrt{\log n}}{n^{\beta/2}\epsilon^{d/2-1/2}}\Big)\,,
	\end{align}
	and hence the proof.
\end{proof}

\section{Proof of Theorems \ref{spec_cong} -- spectral convergence}

We extend the argument provided in \cite{belkin2007convergence,von2008consistency}, and apply tools from \cite{DunsonWuWu2019} to prove the spectral convergence in the $L^\infty$ sense and its corresponding rate. 
We first define some notations.
Recall the definition of $K_{\textup{ref},\epsilon}$ and $d_{\textup{ref},\epsilon}$ in Definition \ref{GL_cts}. 
Define the normalized landmark-kernel as
\begin{equation}
M_{\textup{ref},\epsilon}(x,y)=\frac{K_{\textup{ref},\epsilon}(x,y)}{d_{\textup{ref},\epsilon}(x)}\in C(M\times M)\,.
\end{equation}
Note that by this definition we have 
\begin{equation}
T_{\textup{ref},\epsilon}f(x)=\int_{M}M_{\textup{ref},\epsilon}(x,y)f(y)p_{X}(y)\,dV(y)\,.
\end{equation}
When we only have finite sample points $\{x_{i}\}_{i=1}^{n}$ and landmark set $\{z_{i}\}_{i=1}^{m}$, we need to handle various terms in Definition \ref{Definition various operators discrete}, and the following ``intermittent'' term:
\begin{align}
d_{\textup{ref},\epsilon,n}(x):=\frac{1}{n}\sum_{i=1}^n K_{\textup{ref},\epsilon}(x,x_i)\,,\quad
T_{\textup{ref},\epsilon,n}f(x):=\frac{1}{n}\sum_{i=1}^{n}M_{\textup{ref},\epsilon}(x,x_{i})f(x_{i})\,.\nonumber
\end{align}

The following lemma says that $T_{\textup{ref},\epsilon}$, $T_{\textup{ref},\epsilon,n}$, and $\widehat{T}_{\textup{ref},\epsilon,n}$ are all ``nice'' integral operators. The proof is similar to that shown in \cite{belkin2007convergence,von2008consistency}, so we omit it.

\begin{lemma}
	The integral operators $T_{\textup{ref},\epsilon}$, $T_{\textup{ref},\epsilon,n}$, and $\widehat{T}_{\textup{ref},\epsilon,n}$ are all compact.
\end{lemma}

The proof of Theorems \ref{spec_cong} is composed of two major parts.
\begin{itemize}
	\item Part 1. When $p_{Y}$ is well chosen, show that $\frac{T_{\textup{ref},\epsilon}-1}{\epsilon}\rightarrow \frac{\mu_{1,2}^{(0)}}{d}\Delta$ ``spectrally'' as $\epsilon\rightarrow 0$, and evaluate the rate that depends on $\epsilon$. 
	
	\item Part 2. For a fixed $\epsilon>0$, show that $\widehat{T}_{\textup{ref},\epsilon,n}\rightarrow T_{\textup{ref},\epsilon}$  compactly $\textup{a.s.}$ as $n\rightarrow\infty$, and evaluate the rate that depends on $\epsilon$ and $n$.   
\end{itemize}

Below, we prepare needed facts for these two major parts. With these facts, we put them together to finish the proof in the end.

\subsection{Facts for Part 1}

We need the following Proposition. This proposition is the key step toward the spectral convergence. Its proof is long and delicate, and can be found in \cite{DunsonWuWu2019}, we only provide key steps and refer readers with interest to \cite[Proposition 1]{DunsonWuWu2019} for details.

\begin{proposition}\label{T epsilon and Delta}
	Assume that all eigenvalues of $\Delta$ are simple. Denote $(\lambda_{i,\epsilon}, \phi_{i,\epsilon})$ to be the $i$-th eigenpair of $\frac{I-T_{\textup{ref},\epsilon}}{\epsilon}$ and $(\lambda_{i}, \phi_{i})$ to be the $i$-th eigenpair of $-\Delta$. 
	Assume both $\phi_{i,\epsilon}$ and $\phi_{i}$ are normalized in the $L^2$ norm. 
	Fix $K\in \mathbb{N}$. Denote 
	\begin{equation}
	\Gamma_K:=\min_{1 \leq i \leq K}\textup{dist}(\lambda_i, \sigma(-\Delta)\setminus \{\lambda_i\})\,.
	\end{equation}
	Suppose $\sqrt{\epsilon} \leq \mathcal{K}_1 \min \Bigg(\bigg(\frac{\min(\Gamma_K,1)}{\mathcal{K}_2+\lambda_K^{d/2+5}}\bigg)^2, \frac{1}{(2+\lambda_K^{d+1})^2}\Bigg)$, where $\mathcal{K}_1$ and $\mathcal{K}_2>1$  are the constants depending on $p_X$, $p_Y$, and the volume, the injectivity radius and the sectional curvature of the manifold. 
	Furthermore, assume $p_Y$ is properly chosen so that $\frac{2\nabla p_{X}(x)}{p_{X}(x)}+\frac{\nabla p_{Y}(x)}{p_{Y}(x)}=0$.
	Then, there are $a_i \in \{-1, 1\}$ such that for all $i < K$, 
	\begin{align}
	|\lambda_{i,\epsilon}-\lambda_{i}| &\, \leq \epsilon^{3/4} \,,  \label{Proposition main result statement} \\
	\|a_i\phi_{i,\epsilon}-\phi_{i}\|_{\infty} &\, \leq \epsilon^{1/2} \,. \nonumber 
	\end{align}
\end{proposition}

\begin{proof}
	Note that the kernel associated with $T_{\textup{ref},\epsilon}$ is
	\begin{align}
	K_{\textup{ref},\epsilon}(x,y)&\coloneqq\int_{M}K_{\epsilon}(x,z)K_{\epsilon}(z,y)p_{Y}(z)\,dV(z)\,.
	\end{align}
	While in general $K_{\textup{ref},\epsilon}$ is not Gaussian, it is smooth and decays exponentially fast. To proceed, note that by Theorem \ref{Proof: Theorem1: bias analysis}, when $p_Y$ is properly chosen so that $\frac{2\nabla p_{X}(x)}{p_{X}(x)}+\frac{\nabla p_{Y}(x)}{p_{Y}(x)}=0$, we have the pointwise convergence of the eigenvalue/eigenfunction of $\frac{1-T_{\textup{ref},\epsilon}}{\epsilon}$ to those of $-\Delta$. On the other hand, if we plug the eigenfunction of $-\Delta$ into Theorem \ref{Proof: Theorem1: bias analysis}, the error in the pointwise convergence depends on the $C^4$ norm of the eigenfunction. Therefore, by the standard Sobolev embedding (see \cite[Theorem 9.2]{palais1968foundations} or \cite[Lemma SI.8]{DunsonWuWu2019}), the error of the pointwise convergence is controlled in the uniform way.
	Therefore, by plugging the kernel $K_{\textup{ref},\epsilon}$ into the proof of \cite[Proposition 1]{DunsonWuWu2019}, where Lemma \ref{ratio of eigenvalues} is applied to control the deviations of eigenvalues and eigenfunctions, we obtain the result. Note that the implied constants associated with error bounds are different from those shown in \cite[Proposition 1]{DunsonWuWu2019} due to the different kernels we choose here. Also, note that the bandwidth used in \cite{DunsonWuWu2019} is $\epsilon$, while it is $\epsilon^{1/2}$ in this work.
\end{proof}

\subsection{Facts for Part 2}

This subsection is long and includes several details we need to discuss. 
Overall, to link the random finite samples to the continuous and deterministic setup; that is, link $\widehat{T}_{\textup{ref},\epsilon,n}$ to $T_{\textup{ref},\epsilon}$, we consider the Glivenko-Cantelli class commonly used in the empirical processes analysis \cite{wellner2013weak}. We need some more definitions.
For the probability measure $d\mathbb{P}_X=p_XdV$ associated with the dataset, and a function $f\in C(M)$, introduce the abbreviation 
\begin{equation}
\mathbb{P}f\coloneqq\int f(x)\,d\mathbb{P}_X(x)\,. 
\end{equation}
Let $x_{1},\ldots,x_{n}$ be i.i.d. sampled from $\mathbb{P}$, and denote by 
\begin{equation}
\mathbb{P}_{n}\coloneqq \frac{1}{n}\sum_{i=1}^{n}\delta_{x_{i}}
\end{equation}
the corresponding empirical distribution, where $\delta_{x_{i}}$ is the Dirac delta measure supported at $x_i$.
Note that we have $\mathbb{P}_{n}f = \frac{1}{n}\sum_{i=1}^{n}f(x_{i})$. Also, denote 
\begin{equation}
\widetilde{\mathbb{P}}f:=\int f(y)p_{Y}(y)\,dV(y)\,\,\,\mbox{ and }\,\,\,\widetilde{\mathbb{P}}_mf:=\frac{1}{m}\sum_{l=1}^m f(y_l)\,.
\end{equation}

We now prepare some bounds for later proof.
\begin{lemma}\label{unif-bd}
	Fix $\epsilon>0$. Set $\delta=\min K_{\epsilon}$. The following bounds hold for all $x\in M$: 
	\begin{align}
	\delta^2\leq K_{\textup{ref},\epsilon}(x,y)\leq\norm{K}^{2}_{\infty},&\;\; \delta^2\leq\widehat{K}_{\textup{ref},\epsilon,n}(x,y)\leq\norm{K}^{2}_{\infty}\,,\nonumber\\
	C_1\epsilon^d\leq d_{\textup{ref},\epsilon}(x)\leq C_2\epsilon^d, &\;\; \delta^2\leq\widehat{d}_{\textup{ref},\epsilon,n}(x)\leq\norm{K}^{2}_{\infty}\,,\label{Bound of d and dhat}\\
	\frac{C_2\epsilon^d}{\norm{K}_{\infty}}\leq M_{\textup{ref},\epsilon}(x,y)\leq \frac{\norm{K}^{2}_{\infty}}{C_1\epsilon^{d}},&\;\;\frac{\delta^2}{\norm{K}_{\infty}}\leq \widehat{M}_{\textup{ref},\epsilon,n}(x,y)\leq \frac{\norm{K}^{2}_{\infty}}{\delta^2}\,,\nonumber
	\end{align}
	where $C_1$ is a constant depending on the kernel, the curvature of the manifold and the minima of $p_X$ and $p_Y$. Similarly, $C_2$ is a constant depending on the kernel, the curvature of the manifold and the maxima of $p_X$ and $p_Y$. 
\end{lemma}

\begin{proof}
	The bounds for $K_{\textup{ref},\epsilon}(x,y)$ and $\hat K_{\textup{ref},\epsilon,n}(x,y)$ come from a trivial bound. For $d_{\textup{ref},\epsilon}(x)$, we have
	\begin{align}
	\nonumber&\quad\int_{M}K_{\textup{ref},\epsilon}(x,y)P_{X}(y)\,dV(y)\,\geq \inf_{x'\in M}P_{X}(x')\int_{M}K_{\textup{ref},\epsilon}(x,y)\,dV(y)\\
	\nonumber
	&\geq \inf_{x',y'\in M}P_X(x')P_Y(y')\int_M \int_{M}K_{\epsilon}(x,z)K_{\epsilon}(z,y)\,dV(z)dV(y) = C_1\epsilon^{d}\,,
	\end{align}
	Note that due to the randomness, $\widehat{d}_{\textup{ref},\epsilon,n}(x)$ can only be trivially bounded. $M_{\textup{ref},\epsilon}(x,y)$ and $\hat M_{\textup{ref},\epsilon,n}(x,y)$ are bounded by combining the above bounds.
\end{proof}

Below, we list some functional spaces we need for the analysis, and show that they are Glivenko-Cantelli classes.

\begin{definition}\label{Definition different classes}
	Let $u\in C(M)$ and $k$ be the chosen Gaussian kernel stated in Theorem \ref{spec_cong}. Define 
	\begin{align}
	\mathcal{K}&\coloneqq\{k(x,\cdot);x\in M\}, \nonumber\\
	\mathcal{K}\cdot\mathcal{K}&\coloneqq\{k(x,\cdot)k(\cdot,y);x,y\in M\}\,\nonumber\\
	u\cdot\mathcal{M}&\coloneqq\{u(\cdot)M_{\textup{ref},\epsilon}(x,\cdot);x\in M\},\\
	\mathcal{M}\cdot\mathcal{M}&\coloneqq\{M_{\textup{ref},\epsilon}(x,\cdot)M_{\textup{ref},\epsilon}(\cdot,y);x,y\in M\}\nonumber\\
	\int\mathcal{K}\cdot\mathcal{K}&\coloneqq\Big\{\int k(x,z)k(z,\cdot)p_{Y}(z)\,dV(z);x\in M\Big\}\,.\nonumber
	\end{align}
\end{definition}

\begin{lemma}\label{G-C class}
	The classes $\mathcal{K}$, $\mathcal{K}\cdot\mathcal{K}, f\cdot\mathcal{M}$, $\mathcal{M}\cdot\mathcal{M}$ and $\int\mathcal{K}\cdot\mathcal{K}$ are Glivenko-Cantelli classes.
\end{lemma}
The proof of Lemma \ref{G-C class} is standard, and can be found in, for example \cite[Proposition 11]{von2008consistency}, so we omit the details.

\begin{lemma}\label{Proposition: Tepsn conv to Teps compactly}
	For a fixed $\epsilon>0$, $\widehat{T}_{\textup{ref},\epsilon,n}$ converges to $T_{\textup{ref},\epsilon}$ collectively compactly $\textup{a.s.}$ as $n\rightarrow\infty$.
\end{lemma}

\begin{proof}
	We verify the collectively compact convergence. For (I), let $\widehat{M}_{\textup{ref},\epsilon,n}^{(d)}(x,y)\coloneqq\frac{K_{\textup{ref},\epsilon}(x,y)}{\widehat{d}_{\textup{ref},\epsilon,n}^{(d)}(x)}\in C(M\times M)$. Pick any $f\in C(M)$. By the triangle inequality, we have
	\begin{align}
	\nonumber \norm{\widehat{T}_{\textup{ref},\epsilon,n}f - T_{\textup{ref},\epsilon}f}_{\infty} =&\, \norm{\frac{1}{n}\sum_{i=1}^{n}\widehat{M}_{\textup{ref},\epsilon,n}(x,x_{i})f(x_{i}) - \int_{M}M_{\textup{ref},\epsilon}(x,y)f(y)p_{X}(y)\,dV(y)}_{\infty}
	\\\nonumber \leq&\, \sup_{x}\abs{\mathbb{P}_{n}\widehat{M}_{\textup{ref},\epsilon,n}(x,\cdot)f(\cdot) - \mathbb{P}M_{\textup{ref},\epsilon}(x,\cdot)f(\cdot)}
	\\ \label{term1} \leq&\, \sup_{x}\abs{\mathbb{P}_{n}\widehat{M}_{\textup{ref},\epsilon,n}(x,\cdot)f(\cdot) - \mathbb{P}_{n}\widehat{M}^{(d)}_{\textup{ref},\epsilon,n}(x,\cdot)f(\cdot)}
	\\ \label{term2} &\quad+\, \sup_{x}\abs{\mathbb{P}_{n}\widehat{M}^{(d)}_{\textup{ref},\epsilon,n}(x,\cdot)f(\cdot) - \mathbb{P}_{n}M_{\textup{ref},\epsilon}(x,\cdot)f(\cdot)}
	\\ \label{term3} &\quad+\, \sup_{x}\abs{\mathbb{P}_{n}M_{\textup{ref},\epsilon}(x,\cdot)f(\cdot) - \mathbb{P}M_{\textup{ref},\epsilon}(x,\cdot)f(\cdot)}\,.
	\end{align}
	We bound the three terms (\ref{term1}) (\ref{term2}) and (\ref{term3}) respectively by Lemma \ref{unif-bd}. By a direct expansion, 
	\begin{align}
	\nonumber (\ref{term1}) \,&= \sup_{x}\abs{\frac{1}{n}\sum_{i=1}^{n}\widehat{M}_{\textup{ref},\epsilon,n}(x,x_{i})f(x_{i}) - \frac{1}{n}\sum_{i=1}^{n}\widehat{M}^{(d)}_{\textup{ref},\epsilon,n}(x,x_{i})f(x_{i})}
	\\\nonumber &=
	\sup_{x}\abs{\frac{1}{n}\sum_{i=1}^{n}\left[\widehat{M}_{\textup{ref},\epsilon,n}(x,x_{i})-\widehat{M}^{(d)}_{\textup{ref},\epsilon,n}(x,x_{i})\right]f(x_{i})}\,,
	\end{align}
	which is bounded by
	\begin{align}
	\nonumber &\quad
	\sup_{x}\frac{1}{n}\sum_{i=1}^{n}\abs{\widehat{M}_{\textup{ref},\epsilon,n}(x,x_{i})-\widehat{M}^{(d)}_{\textup{ref},\epsilon,n}(x,x_{i})}\abs{f(x_{i})}
	\\\nonumber &\leq
	\norm{f}_{\infty}\sup_{x,y}\abs{\frac{\widehat{K}_{\textup{ref},\epsilon,n}(x,y)-K_{\textup{ref},\epsilon}(x,y)}{\widehat{d}_{\textup{ref},\epsilon,n}(x)}}
	\leq \frac{\norm{f}_{\infty}}{\delta^{2}}\sup_{x,y}\abs{\widehat{K}_{\textup{ref},\epsilon,n}(x,y)-K_{\textup{ref},\epsilon}(x,y)}\,,
	\end{align}
	where the last bound comes from \eqref{Bound of d and dhat}. Then, we spell out $K_{\textup{ref},\epsilon,n}(x,y)$ and $K_{\textup{ref},\epsilon}(x,y)$ so that (\ref{term1}) is further bounded by
	\begin{align}
	&\nonumber
	\frac{\norm{f}_{\infty}}{\delta^{2}}\sup_{x,y}\abs{\frac{1}{m}\sum_{j}^{m}K_{\epsilon}(x,z_{j})K_{\epsilon}(z_{j},y)-\int_{M}K_{\epsilon}(x,z)K_{\epsilon}(z,y)p_{Y}(z)\,dV(z)}
	\\\nonumber =&\,
	\frac{\norm{f}_{\infty}}{\delta^{2}}\sup_{x,y}\abs{\widetilde{\mathbb{P}}_{m}K_{\epsilon}(x,\cdot)K_{\epsilon}(\cdot,y) - \widetilde{\mathbb{P}}K_{\epsilon}(x,\cdot)K_{\epsilon}(\cdot,y)}
	\leq 
	\frac{\norm{f}_{\infty}}{\delta^{2}}\sup_{F\in\mathcal{K}\cdot\mathcal{K}}\abs{\widetilde{\mathbb{P}}_{m}F - \widetilde{\mathbb{P}}F} \,,
	\end{align}
	which goes to $0$ a.s. as $n\rightarrow\infty$
	since $\mathcal{K}\cdot\mathcal{K}$ is Glivenko-Cantelli class by Lemma \ref{G-C class} and $m=n^{\beta}$ for some $\beta\in(0,1)$.
	For term (\ref{term2}), we again expand it directly: 
	\begin{align}
	\nonumber (\ref{term2}) =& \sup_{x}\abs{\frac{1}{n}\sum_{i=1}^{n}\widehat{M}^{(d)}_{\textup{ref},\epsilon,n}(x,x_{i})f(x_{i}) - \frac{1}{n}\sum_{i=1}^{n}M_{\textup{ref},\epsilon}(x,x_{i})f(x_{i})}
	\\\nonumber \leq&\,
	\sup_{x}\frac{1}{n}\sum_{i=1}^{n}\abs{\widehat{M}^{(d)}_{\textup{ref},\epsilon,n}(x,x_{i})-M_{\textup{ref},\epsilon}(x,x_{i})}\abs{f(x_{i})}
	\\\nonumber \leq&\,
	\norm{f}_{\infty}\sup_{x,y}\abs{\frac{K_{\textup{ref},\epsilon}(x,y)}{\widehat{d}_{\textup{ref},\epsilon,n}(x)}-\frac{K_{\textup{ref},\epsilon}(x,y)}{d_{\textup{ref},\epsilon}(x)}} 
	\leq
	\frac{\norm{f}_{\infty}\norm{K}_{\infty}^{2}}{\delta^{4}}\sup_{x}\abs{\widehat{d}_{\textup{ref},\epsilon,n}(x) - d_{\textup{ref},\epsilon}(x)}\,,
	\end{align}
	where in the last bound we apply the fact that if $A,B\geq C>0$, then $\abs{A^{\beta}-B^{\beta}}\leq\frac{1}{C^{1-\beta}}\abs{A-B}$. Then, by spelling out $\widehat{d}_{\textup{ref},\epsilon,n}(x)$ and $d_{\textup{ref},\epsilon}(x)$, (\ref{term2}) is further bounded by
	\begin{align}
	\nonumber &\,
	\frac{\norm{f}_{\infty}\norm{K}_{\infty}^{2}}{\delta^{4}}\sup_{x}\abs{\mathbb{P}_{n}\widetilde{\mathbb{P}}_{m}K_{\epsilon}(x,\cdot)K_{\epsilon}(\cdot,\star)-\mathbb{P}\widetilde{\mathbb{P}}K_{\epsilon}(x,\cdot)K_{\epsilon}(\cdot,\star)}
	\\ \label{term4}\leq&\,
	\frac{\norm{f}_{\infty}\norm{K}_{\infty}^{2}}{\delta^{4}}\sup_{x}\abs{\mathbb{P}_{n}\widetilde{\mathbb{P}}_{m}K_{\epsilon}(x,\cdot)K_{\epsilon}(\cdot,\star)-\mathbb{P}_{n}\widetilde{\mathbb{P}}K_{\epsilon}(x,\cdot)K_{\epsilon}(\cdot,\star)}
	\\ \label{term5}&\quad+\,
	\frac{\norm{f}_{\infty}\norm{K}_{\infty}^{2}}{\delta^{4}}\sup_{x}\abs{\mathbb{P}_{n}\widetilde{\mathbb{P}}K_{\epsilon}(x,\cdot)K_{\epsilon}(\cdot,\star)-\mathbb{P}\widetilde{\mathbb{P}}K_{\epsilon}(x,\cdot)K_{\epsilon}(\cdot,\star)}
	\end{align}
	Clearly, the term $(\ref{term5})\rightarrow 0\quad\textup{a.s.}$ by lemma \ref{G-C class}. For term (\ref{term4}):
	\begin{align}
	\nonumber (\ref{term4}) =&\, \sup_{x}\abs{\mathbb{P}_{n}\widetilde{\mathbb{P}}_{m}K_{\epsilon}(x,\cdot)K_{\epsilon}(\cdot,\star)-\mathbb{P}_{n}\widetilde{\mathbb{P}}K_{\epsilon}(x,\cdot)K_{\epsilon}(\cdot,\star)}
	\\\nonumber =&\,
	\sup_{x}\abs{\frac{1}{n}\sum_{i=1}^{n}\frac{1}{m}\sum_{j=1}^{m}K_{\epsilon}(x,z_{j})K_{\epsilon}(z_{j},x_{i})-\frac{1}{n}\sum_{i=1}^{n}\int K_{\epsilon}(x,z)K_{\epsilon}(z,x_{i})p_{Y}(z)\,dV(z)}
	\\\nonumber \leq&\,
	\sup_{x}\frac{1}{n}\sum_{i=1}^{n}\abs{\frac{1}{m}\sum_{j=1}^{m}K_{\epsilon}(x,z_{j})K_{\epsilon}(z_{j},x_{i})-\int K_{\epsilon}(x,z)K_{\epsilon}(z,x_{i})p_{Y}(z)\,dV(z)}
	\\\nonumber \leq&
	\sup_{x,y}\abs{\frac{1}{m}\sum_{j=1}^{m}K_{\epsilon}(x,z_{j})K_{\epsilon}(z_{j},y)-\int K_{\epsilon}(x,z)K_{\epsilon}(z,y)p_{Y}(z)\,dV(z)}
	\\\nonumber =&\,
	\sup_{x,y}\abs{\widetilde{\mathbb{P}}_{m}K_{\epsilon}(x,\cdot)K_{\epsilon}(\cdot,y) - \widetilde{\mathbb{P}}K_{\epsilon}(x,\cdot)K_{\epsilon}(\cdot,y)}\,,
	\end{align}
	which tends to $0$ a.s. as $n\rightarrow\infty$.
	Hence, the term $(\ref{term2})\rightarrow 0\;\; \textup{a.s. as}\;n\rightarrow\infty$.
	Finally, the term $(\ref{term3})\rightarrow 0\;\; \textup{a.s. as}\;n\rightarrow\infty$ by Lemma \ref{G-C class}. So the condition (I) is verified.
	
	Next, we verify (II). Since $T_{\textup{ref},\epsilon}$ is compact, it is enough to show the set $\cup_{n}\widehat{T}_{\textup{ref},\epsilon,n}(B)$ is relatively compact, where $B\subset (C(M),\,\|\cdot\|_\infty)$ is the unit ball centered at $0$. By the Arzela-Ascoli theorem, we need to show that the set $\cup_{n}\widehat{T}_{\textup{ref},\epsilon,n}(B)$ is:
	\begin{enumerate}
		\item pointwisely bounded, and
		\item equicontinuous.
	\end{enumerate}
	For 1, pick any $x\in M$, by Lemma \ref{unif-bd}:
	\begin{align}
	\nonumber&\sup_{f\in B,\, n\in \mathbb{N}}\norm{\widehat{T}_{\textup{ref},\epsilon,n}f(x)}_{\infty}
	=
	\sup_{f\in B,\, n\in \mathbb{N}}\abs{\frac{1}{n}\sum_{i=1}^{n}\widehat{M}_{\textup{ref},\epsilon,n}(x,x_{i})f(x_{i})}
	\\\nonumber\leq&\,
	\sup_{f\in B,\, n\in \mathbb{N}}\norm{f}_{\infty}\frac{1}{n}\sum_{i=1}^{n}\abs{\widehat{M}_{\textup{ref},\epsilon,n}(x,x_{i})}
	\leq \frac{\norm{K}_{\infty}^{2}}{\delta^{2}} < \infty.
	\end{align}
	For 2, pick any $x,y\in M$ that are close, a direct expansion leads to
	\begin{align}
	\nonumber &\sup_{f\in B,\, n\in \mathbb{N}}\abs{\widehat{T}_{\textup{ref},\epsilon,n}f(y) - \widehat T_{\textup{ref},\epsilon,n}f(x)} 
	= \sup_{f\in B,\, n\in \mathbb{N}}\abs{\frac{1}{n}\sum_{i=1}^{n}\widehat{M}_{\textup{ref},\epsilon,n}(y,x_{i})f(x_{i}) - \frac{1}{n}\sum_{i=1}^{n}\widehat{M}_{\textup{ref},\epsilon,n}(x,x_{i})f(x_{i})}
	\\\nonumber \leq&\,
	\sup_{f\in B,\, n\in \mathbb{N}}\frac{1}{n}\sum_{i=1}^{n}\abs{\widehat{M}_{\textup{ref},\epsilon,n}(y,x_{i})-\widehat{M}_{\textup{ref},\epsilon,n}(x,x_{i})}\abs{f(x_{i})}
	\leq
	\sup_{f\in B}\norm{f}_{\infty}\sup_{z}\abs{\widehat{M}_{\textup{ref},\epsilon,n}(y,z)-\widehat{M}_{\textup{ref},\epsilon,n}(x,z)}\,.
	\end{align}
	Clearly, since $f\in B$, $\norm{f}_{\infty}\leq 1$. So $\sup_{f\in B,\, n\in \mathbb{N}}\abs{\widehat{T}_{\textup{ref},\epsilon,n}f(y) - \widehat T_{\textup{ref},\epsilon,n}f(x)}$ is further bounded by 
	\begin{align}
	\nonumber &\,
	\sup_{z}\abs{\frac{\widehat{K}_{\textup{ref},\epsilon,n}(y,z)}{\widehat{d}_{\textup{ref},\epsilon,n}(y)}-\frac{\widehat{K}_{\textup{ref},\epsilon,n}(x,z)}{\widehat{d}_{\textup{ref},\epsilon,n}(x)}}
	\leq 
	\frac{1}{\delta^{4}}\sup_{z}\abs{\widehat{d}_{\textup{ref},\epsilon,n}(y)\widehat{K}_{\textup{ref},\epsilon,n}(x,z)-\widehat{d}_{\textup{ref},\epsilon,n}(x)\widehat{K}_{\textup{ref},\epsilon,n}(y,z)}\,,
	\end{align}
	where the last inequality comes again from Lemma \ref{unif-bd}. The right hand side can further be bounded by
	\begin{align}
	\nonumber &\,
	\frac{1}{\delta^{4}}\sup_{z}\big|\widehat{d}_{\textup{ref},\epsilon,n}(y)[\widehat{K}_{\textup{ref},\epsilon,n}(x,z)-\widehat{K}_{\textup{ref},\epsilon,n}(y,z)]+[\widehat{d}_{\textup{ref},\epsilon,n}(y)-\widehat{d}_{\textup{ref},\epsilon,n}(x)]\widehat{K}_{\textup{ref},\epsilon,n}(y,z)\big|
	\\\label{Bound for equicontinuity part 1} \leq&\,
	\frac{\norm{K}_{\infty}^{2}}{\delta^{4}}\sup_{z}\abs{\widehat{K}_{\textup{ref},\epsilon,n}(x,z)-\widehat{K}_{\textup{ref},\epsilon,n}(y,z)}
	\\\label{Bound for equicontinuity part 2} &\qquad+
	\frac{\norm{K}_{\infty}^{2}}{\delta^{4}}\abs{\widehat{d}_{\textup{ref},\epsilon,n}(y)-\widehat{d}_{\textup{ref},\epsilon,n}(x)}\,.
	\end{align}
	To finish the equicontinuity argument, we bound \eqref{Bound for equicontinuity part 1} and \eqref{Bound for equicontinuity part 2}. By spelling out $\widehat{K}_{\textup{ref},\epsilon,n}(x,z)$ and $\widehat{K}_{\textup{ref},\epsilon,n}(y,z)$; that is,
	\begin{align}
	\nonumber &\sup_{z}\abs{\widehat{K}_{\textup{ref},\epsilon,n}(x,z)-\widehat{K}_{\textup{ref},\epsilon,n}(y,z)}
	=
	\sup_{z}\abs{\frac{1}{m}\sum_{j}^{m}K_{\epsilon}(x,z_{j})K_{\epsilon}(z_{j},z)-\frac{1}{m}\sum_{j}^{m}K_{\epsilon}(y,z_{j})K_{\epsilon}(z_{j},z)}\,,\nonumber
	\end{align}    
	\eqref{Bound for equicontinuity part 1} is bounded by:
	\begin{align}
	\nonumber &
	\sup_{z}\frac{1}{m}\sum_{j}^{m}\abs{K_{\epsilon}(x,z_{j})-K_{\epsilon}(y,z_{j})}\abs{K_{\epsilon}(z_{j},z)}
	\\\nonumber \leq&\,
	\norm{K_{\epsilon}}_{\infty}\frac{1}{m}\sum_{j}^{m}\abs{K_{\epsilon}(x,z_{j})-K_{\epsilon}(y,z_{j})}
	\leq
	\norm{K_{\epsilon}}_{\infty}\sup_{z}\abs{K_{\epsilon}(x,z)-K_{\epsilon}(y,z)}\,,
	\end{align}
	which is controlled by $Cd(x,y)$ for some constant $C>0$ due to the continuity of $K_{\epsilon}$ and the compactness of $M$.
	Similarly, by spelling out $\widehat{d}_{\textup{ref},\epsilon,n}(y)-\widehat{d}_{\textup{ref},\epsilon,n}(x)$; that is,
	\begin{equation}
	{\widehat{d}_{\textup{ref},\epsilon,n}(y)-\widehat{d}_{\textup{ref},\epsilon,n}(x)}
	= {\frac{1}{n}\sum_{i=1}^{n}\widehat{K}_{\textup{ref},\epsilon,n}(y,x_{i})-\frac{1}{n}\sum_{i=1}^{n}\widehat{K}_{\textup{ref},\epsilon,n}(x,x_{i})}\,,
	\end{equation}
	\eqref{Bound for equicontinuity part 2} is bounded by:
	\begin{align}
	\nonumber &
	\abs{\frac{1}{n}\sum_{i=1}^{n}\frac{1}{m}\sum_{j}^{m}K_{\epsilon}(y,z_{j})K_{\epsilon}(z_{j},x_{i})-\frac{1}{n}\sum_{i=1}^{n}\frac{1}{m}\sum_{j}^{m}K_{\epsilon}(x,z_{j})K_{\epsilon}(z_{j},x_{i})}
	\\\nonumber \leq&
	\frac{1}{nm}\sum_{i,j}\abs{K_{\epsilon}(y,z_{j})-K_{\epsilon}(x,z_{j})}\abs{K_{\epsilon}(z_{j},x_{i})}
	\leq
	\norm{K_{\epsilon}}_{\infty}\sup_{z}\abs{K_{\epsilon}(x,z)-K_{\epsilon}(y,z)}\,,
	\end{align}
	which is again controlled by $Cd(x,y)$. So equicontinuity is verified. Then condition (II) is verified by Arzela-Ascoli theorem. We thus finish the proof of collectively compact convergence.
\end{proof}

Next, we control various functional classes that concern us with the bound shown in Lemma \ref{cover_num_gau}. These controls are needed when we derive the convergence rate. While the proof is standard, we provide details of how the landmark set plays a role in the bound.

\begin{lemma}\label{cover_num_bd}
	Take $u\in C(M)$ and $\epsilon>0$. Take $k(x,y):=K_\epsilon(x,y)$ in Definition \ref{Definition different classes}. For $r>0$, we have the following bounds:
	$$N(\mathcal{K}\cdot\mathcal{K},r,\norm{\cdot}_{\infty})\leq N\Big(\mathcal{K},\frac{r}{2\norm{K_{\epsilon}}},\norm{\cdot}_{\infty}\Big)$$
	$$N(\int\mathcal{K}\cdot\mathcal{K},r,\norm{\cdot}_{\infty})\leq N\Big(\mathcal{K},\frac{r}{\norm{K_{\epsilon}}},\norm{\cdot}_{\infty}\Big)$$
	$$N(u\cdot\mathcal{M},r,\norm{\cdot}_{\infty})\leq N\Big(\mathcal{K},\frac{rC^{2}\epsilon^{2d}}{2\norm{u}\norm{K_{\epsilon}}^{3}},\norm{\cdot}_{\infty}\Big)$$
	$$N(\mathcal{M}\cdot\mathcal{M},r,\norm{\cdot}_{\infty})\leq N\Big(\mathcal{K},\frac{rC^{3}\epsilon^{3d}}{4\norm{K_{\epsilon}}^{5}},\norm{\cdot}_{\infty}\Big)\,.$$
\end{lemma}

\begin{proof}
	For the class $\mathcal{K}\cdot\mathcal{K}$, pick any $x_{1},x_{2},y_{1},y_{2}\in M$:
	\begin{align}
	\nonumber&\abs{K_{\epsilon}(x_{1},z)K_{\epsilon}(y_{1},z)-K_{\epsilon}(x_{2},z)K_{\epsilon}(y_{2},z)}
	\\\nonumber=&\,
	\big|K_{\epsilon}(x_{1},z)K_{\epsilon}(y_{1},z)-K_{\epsilon}(x_{1},z)K_{\epsilon}(y_{2},z)+K_{\epsilon}(x_{1},z)K_{\epsilon}(y_{2},z)-K_{\epsilon}(x_{2},z)K_{\epsilon}(y_{2},z)\big|
	\\\nonumber\leq&\,
	\norm{K_{\epsilon}}_\infty\big(\abs{K_{\epsilon}(y_{1},z)-K_{\epsilon}(y_{2},z)} + \abs{K_{\epsilon}(x_{1},z)-K_{\epsilon}(x_{2},z)}\big)\,.
	\end{align}
	This implies a $\frac{r}{2\norm{K_{\epsilon}}}$-cover of $\mathcal{K}$ induces a $r$-cover of $\mathcal{K}\cdot\mathcal{K}$. Hence, $N(\mathcal{K}\cdot\mathcal{K},r,\norm{\cdot}_{\infty})\leq N(\mathcal{K},\frac{r}{2\norm{K_{\epsilon}}},\norm{\cdot}_{\infty})$.
	
	For the class $\int\mathcal{K}\cdot\mathcal{K}$, pick any $x_{1},x_{2}\in M$:
	\begin{align}
	\nonumber&\abs{\int K_{\epsilon}(x_{1},z)K_{\epsilon}(z,y)\,d\widetilde{\mathbb{P}}(z)-\int K_{\epsilon}(x_{2},z)K_{\epsilon}(z,y)\,d\widetilde{\mathbb{P}}(z)}
	\leq
	\norm{K_{\epsilon}}\abs{K_{\epsilon}(x_{1},z)-K_{\epsilon}(x_{2},z)}\,,
	\end{align}
	which implies that a $\frac{r}{\norm{K_{\epsilon}}}$-cover of $\mathcal{K}$ induces a $r$-cover of $\int\mathcal{K}\cdot\mathcal{K}$. Hence, $N(\int\mathcal{K}\cdot\mathcal{K},r,\norm{\cdot}_{\infty})\leq N(\mathcal{K},\frac{r}{\norm{K_{\epsilon}}},\norm{\cdot}_{\infty})$.
	
	For the class $u\cdot\mathcal{M}$, pick any $x_{1},x_{2}\in M$:
	\begin{align}
	\nonumber&\quad\abs{u(y)M_{\textup{ref},\epsilon}(x_{1},y)-u(y)M_{\textup{ref},\epsilon}(x_{2},y)}
	\leq
	\norm{u}_\infty \abs{\frac{K_{\textup{ref},\epsilon}(x_{1},y)}{d_{\textup{ref},\epsilon}(x_{1})}-\frac{K_{\textup{ref},\epsilon}(x_{2},y)}{d_{\textup{ref},\epsilon}(x_{2})}}
	%    \\\nonumber&\leq
	%    \norm{u}\norm{\frac{d_{\textup{ref},\epsilon}(x_{2})K_{\textup{ref},\epsilon}(x_{1},y)-d_{\textup{ref},\epsilon}(x_{1})K_{\textup{ref},\epsilon}(x_{2},y)}{d_{\textup{ref},\epsilon}(x_{1})d_{\textup{ref},\epsilon}(x_{2})}}
	\\\nonumber&\leq
	\frac{\norm{u}_\infty}{C^{2}\epsilon^{2d}}\left(\norm{d_{\textup{ref},\epsilon}}\abs{K_{\textup{ref},\epsilon}(x_{1},y)-K_{\textup{ref},\epsilon}(x_{2},y)}+\norm{K_{\textup{ref},\epsilon}}\norm{d_{\textup{ref},\epsilon}(x_{1})-d_{\textup{ref},\epsilon}(x_{2})}\right)
	\\\nonumber&\leq
	\frac{2\norm{u}_\infty\norm{K_{\epsilon}}_\infty^{2}}{C^{2}\epsilon^{2d}}\abs{K_{\textup{ref},\epsilon}(x_{1},y)-K_{\textup{ref},\epsilon}(x_{2},y)}
	\\\nonumber&=
	\frac{2\norm{u}_\infty\norm{K_{\epsilon}}_\infty^{2}}{C^{2}\epsilon^{2d}}\abs{\int K_{\epsilon}(x_{1},z)K_{\epsilon}(z,y)\,d\widetilde{\mathbb{P}}(z)-\int K_{\epsilon}(x_{2},z)K_{\epsilon}(z,y)\,d\widetilde{\mathbb{P}}(z)}
	\\\nonumber&\leq
	\frac{2\norm{u}_\infty\norm{K_{\epsilon}}_\infty^{3}}{C^{2}\epsilon^{2d}}\abs{K_{\epsilon}(x_{1},z)-K_{\epsilon}(x_{2},z)}\,,
	\end{align}
	which implies that a $\frac{rC^{2}\epsilon^{2d}}{2\norm{u}_\infty\norm{K_{\epsilon}}_\infty^{3}}$-cover of $\mathcal{K}$ induces a $r$-cover of $u\cdot\mathcal{M}$, hence $N(u\cdot\mathcal{M},r,\norm{\cdot}_{\infty})\leq N\big(\mathcal{K},\frac{rC^{2}\epsilon^{2d}}{2\norm{u}_\infty\norm{K_{\epsilon}}_\infty^{3}},\norm{\cdot}_{\infty}\big)$.
	
	For the class $\mathcal{M}\cdot\mathcal{M}$, pick any $x_{1},x_{2},y_{1},y_{2}\in M$ and carry out a similar bound:
	\begin{align}
	\nonumber&\quad\abs{M_{\epsilon}(x_{1},z)M_{\epsilon}(y_{1},z)-M_{\epsilon}(x_{2},z)M_{\epsilon}(y_{2},z)}
	\\\nonumber&\leq
	\norm{M_{\epsilon}}_\infty\big(\abs{M_{\epsilon}(y_{1},z)-M_{\epsilon}(y_{2},z)} + \abs{M_{\epsilon}(x_{1},z)-M_{\epsilon}(x_{2},z)}\big)
	\leq
	\frac{4\norm{K_{\epsilon}}^{5}_\infty}{C^{3}\epsilon^{3d}}\abs{K_{\epsilon}(x_{1},z)-K_{\epsilon}(x_{2},z)}\,,
	\end{align}
	which implies that a $\frac{rC^{3}\epsilon^{3d}}{4\norm{K_{\epsilon}}^{5}}$-cover of $\mathcal{K}$ induces a $r$-cover of $\mathcal{M}\cdot\mathcal{M}$. Hence, $N(\mathcal{M}\cdot\mathcal{M},r,\norm{\cdot}_{\infty})\leq N\big(\mathcal{K},\frac{rC^{3}\epsilon^{3d}}{4\norm{K_{\epsilon}}^{5}},\norm{\cdot}_{\infty}\big)$. The constant C above is the C in Lemma \ref{unif-bd}, depends on the kernel $K$, the curvature of the manifold and the density functions $p_{X}, p_{Y}$.
\end{proof}

With Lemma \ref{cover_num_bd}, we can derive the following technical lemma. We provide a detailed proof to show how the landmark set impacts the bound.

\begin{lemma}\label{cant_class_rate}
	Take $\epsilon>0$, $u\in C(M)$, and $k(x,y):=K_\epsilon(x,y)=e^{-\norm{x-y}^2/\epsilon}$ in Definition \ref{Definition different classes}.
	Let $\mathcal{F}_{\epsilon}\coloneqq(\mathcal{K}\cdot\mathcal{K})\cup(\int\mathcal{K}\cdot\mathcal{K})\cup(u\cdot\mathcal{M})\cup(\mathcal{M}\cdot\mathcal{M})$.
	With probability $1-\mathcal O(n^{-2})$, we have:
	\begin{equation}\label{GC class bound for all classes}
	\sup_{f\in\mathcal{F}_{\epsilon}}\abs{\mathbb{P}_{n}f-\mathbb{P}f}=\mathcal{O}\left(\frac{\sqrt{-\log \epsilon}+\sqrt{\log n}}{\sqrt{n}}\right)\,,
	\end{equation}
	where the implied constant depends on $d$, $\|u\|_\infty$ and the constants shown in the entropy bound shown in Theorem \ref{entropy_bd}.
\end{lemma}

Note that the probability event space that \eqref{GC class bound for all classes} holds depends on the chosen $u$, and the implied constant depends on $\|u\|_\infty$. This is critical when we carry out the final spectral convergence proof.

\begin{proof}
	By plugging $\delta=1/n^{2}$ into Theorem \ref{entropy_bd}, we have
	\begin{align}
	\nonumber\sup_{F\in\mathcal{K}\cdot\mathcal{K}}\abs{\mathbb{P}_{n}F-\mathbb{P}F}\,&\leq\frac{C_{E}}{\sqrt{n}}\int_{0}^{\infty}\sqrt{\log N(\mathcal{K}\cdot\mathcal{K},r,L_{2}(\mathbb{P}_{n}))}\,dr + \sqrt{\frac{1}{n}\log(2n^{2})}
	\\\nonumber&=
	\frac{C_{E}}{\sqrt{n}}\int_{0}^{1}\sqrt{\log N(\mathcal{K}\cdot\mathcal{K},r,L_{2}(\mathbb{P}_{n}))}\,dr + c_{1}\sqrt{\frac{\log n}{n}}\,,
	\end{align}
	where we use the fact that $N(\mathcal{K}\cdot\mathcal{K},r,L_{2}(\mathbb{P}_{n}))=1$ when $r>1$. Indeed, for any $x,y\in M$ and $n\in\mathbb{N}$, we have
	\begin{align}
	\|K_\epsilon(x,\cdot)K_\epsilon(\cdot,y)\|^2_{L_{2}(\mathbb{P}_{n})}
	\leq \frac{1}{n}\sum_{l=1}^n|K_\epsilon(x,z_l)K_\epsilon(z_l,y)|^2\leq \|K_\epsilon\|_{\infty}^2=1\,.\nonumber
	\end{align}
	Similarly, note that for any $x,x',y,y'\in M$ and $n\in \mathbb{N}$, we have
	\begin{equation}
	\|K_\epsilon(x,\cdot)K_\epsilon(\cdot,y)-K_\epsilon(x',\cdot)K_\epsilon(\cdot,y')\|^2_{L_{2}(\mathbb{P}_{n})}\leq \|K_\epsilon(x,\cdot)K_\epsilon(\cdot,y)-K_\epsilon(x',\cdot)K_\epsilon(\cdot,y')\|^2_{\infty}\nonumber\,,
	\end{equation}
	so we immediately have
	\begin{equation}
	N(\mathcal{K}\cdot\mathcal{K},r,L_{2}(\mathbb{P}_{n}))\leq N(\mathcal{K}\cdot\mathcal{K},r,\norm{\cdot}_{\infty})\nonumber\,
	\end{equation}
	and hence $\int_{0}^{1}\sqrt{\log N(\mathcal{K}\cdot\mathcal{K},r,L_{2}(\mathbb{P}_{n}))}\,dr\leq
	\int_{0}^{1}\sqrt{\log N(\mathcal{K}\cdot\mathcal{K},r,\norm{\cdot}_{\infty})}\,dr$.
	By Proposition \ref{cover_num_bd}, we have the bound that
	\begin{align}
	\nonumber &\quad   \int_{0}^{1}\sqrt{\log N(\mathcal{K}\cdot\mathcal{K},r,\norm{\cdot}_{\infty})}\,dr 
	\leq
	\int_{0}^{1}\sqrt{\log N\Big(\mathcal{K},\frac{r}{2\norm{K_{\epsilon}}_\infty},\norm{\cdot}_{\infty}\Big)}\,dr 
	\\\nonumber&\leq \sqrt{2d}\int_0^1\sqrt{[48\sqrt{2d} D_M-\log \epsilon]-\log r}dr
	=\sqrt{2d}\Big(\frac{\sqrt{\pi}}{2}e^{c_\epsilon}\textup{erfc}(c_\epsilon)+\sqrt{c_\epsilon}\Big)\,,
	\end{align}
	where $\textup{erfc}$ is the  complementary error function, $c_\epsilon:= 48\sqrt{2d} D_M-\log \epsilon$, the second inequality comes from Theorem \ref{cover_num_gau} and the fact that $\|K_\epsilon\|_\infty=1$. Note that since $c_\epsilon\approx -\log \epsilon$ when $\epsilon$ is small and  $\textup{erfc}(x)\approx \frac{e^{-x^2}}{\sqrt{\pi}x}$, we know that $\frac{\sqrt{\pi}}{2}e^{c_\epsilon}\textup{erfc}(c_\epsilon)\to 0$ when $\epsilon$ tends to $0$. 
	As a result, 
	\begin{align}
	\nonumber \int_{0}^{1}\sqrt{\log N(\mathcal{K}\cdot\mathcal{K},r,\norm{\cdot}_{\infty})}\,dr&\leq
	2\sqrt{2d}\sqrt{-\log \epsilon}
	\end{align}
	when $\epsilon$ is sufficiently small.
	By combining the above bounds, we have
	\begin{align}
	\nonumber \sup_{F\in\mathcal{K}\cdot\mathcal{K}}\abs{\mathbb{P}_{n}F-\mathbb{P}F}\,&\leq  \frac{2\sqrt{2d}c\sqrt{-\log \epsilon}}{\sqrt{n}} + \frac{c_1\sqrt{\log n}}{\sqrt{n}}=\mathcal{O}\Big(\frac{\sqrt{-\log \epsilon}+\sqrt{\log n}}{\sqrt{n}}\Big)\,.
	\end{align}
	By a similar argument we have the bound for $\sup_{F\in\int\mathcal{K}\cdot\mathcal{K}}\abs{\mathbb{P}_{n}F-\mathbb{P}F}$ , $\sup_{F\in u\cdot\mathcal{M}}\abs{\mathbb{P}_{n}F-\mathbb{P}F}$ and $\sup_{F\in \mathcal{M}\cdot\mathcal{M}}\abs{\mathbb{P}_{n}F-\mathbb{P}F}$. Note that the implied constant of the bound for $\sup_{F\in u\cdot\mathcal{M}}\abs{\mathbb{P}_{n}F-\mathbb{P}F}$ depends on $\|u\|_\infty$. The result follows.
	
\end{proof}

Below, we prepare several technical lemmas to control the spectral convergence rate.

\begin{lemma}\label{d_hat_bd}
	Suppose $m=n^\beta$ for $\beta\in (0,1)$ and take $\epsilon=\epsilon(n)$ so that $\epsilon\to 0$ when $n\to \infty$. If we further assume that $\frac{\sqrt{-\log\epsilon}+\sqrt{\log m}}{\sqrt{m}\epsilon^{d}}\rightarrow 0$, we have with probability $1-\mathcal{O}(m^{-2})$:
	\begin{align}
	&\;\; C_1 \epsilon^{d}/2\leq\widehat{d}_{\textup{ref},\epsilon,n}(x)\leq 2C_2 \epsilon^{d} \label{Bound of d and dhat prob}\,,
	\end{align}
	and with probability $1-\mathcal{O}(n^{-2})$:
	\begin{align}
	&\;\; C_1 \epsilon^{d}/2\leq d_{\textup{ref},\epsilon,n}(x)\leq 2C_2 \epsilon^{d} \nonumber\,,
	\end{align}
	where $C_1,C_2>0$ are constants defined in Lemma \ref{unif-bd}. 
\end{lemma}

Note the difference between this lemma and Lemma \ref{unif-bd}. This lemma says that while usually the kernels involved in the analysis have a wide range, with high probability, the range is well controlled.

\begin{proof}
	By the same calculation (e.g. Step 2 in the proof of Theorem \ref{spec_cong}) we have 
	\begin{align}
	\nonumber&
	\sup_{x}\abs{\widehat{d}_{\textup{ref},\epsilon,n}(x) - d_{\textup{ref},\epsilon}(x)}
	=
	\sup_{x}\abs{\mathbb{P}_{n}\widetilde{\mathbb{P}}_{m}K_{\epsilon}(x,\cdot)K_{\epsilon}(\cdot,\star)-\mathbb{P}\widetilde{\mathbb{P}}K_{\epsilon}(x,\cdot)K_{\epsilon}(\cdot,\star)}
	\\\nonumber \leq&\,
	\sup_{x}\abs{\mathbb{P}_{n}\widetilde{\mathbb{P}}_{m}K_{\epsilon}(x,\cdot)K_{\epsilon}(\cdot,\star)-\mathbb{P}_{n}\widetilde{\mathbb{P}}K_{\epsilon}(x,\cdot)K_{\epsilon}(\cdot,\star)}
	+
	\sup_{x}\abs{\mathbb{P}_{n}\widetilde{\mathbb{P}}K_{\epsilon}(x,\cdot)K_{\epsilon}(\cdot,\star)-\mathbb{P}\widetilde{\mathbb{P}}K_{\epsilon}(x,\cdot)K_{\epsilon}(\cdot,\star)}\,,
	\end{align}
	which is further bounded by taking Lemma \ref{cant_class_rate} into account:
	\begin{align}
	\nonumber &\,
	\sup_{f\in\mathcal{K}\cdot\mathcal{K}}\abs{\widetilde{\mathbb{P}}_{m}f - \widetilde{\mathbb{P}}f} + 
	\sup_{f\in\int\mathcal{K}\cdot\mathcal{K}}\abs{\mathbb{P}_{n}f - {\mathbb{P}}f}
	\\\nonumber =&\,
	\mathcal{O}\left(\frac{\sqrt{-\log\epsilon}+\sqrt{\log m}}{\sqrt{m}}\right)+ \mathcal{O}\left(\frac{\sqrt{-\log\epsilon}+\sqrt{\log n}}{\sqrt{n}}\right) =
	\mathcal{O}\left(\frac{\sqrt{-\log\epsilon}+\sqrt{\log m}}{\sqrt{m}}\right)
	\end{align}
	with probability $1-\mathcal{O}(m^{-2})-\mathcal{O}(n^{-2})=1-\mathcal{O}(m^{-2})$, where we use the fact that $m<n$.
	Then, since $d_{\textup{ref},\epsilon}(x)\geq C_1\epsilon^{d}$ by Lemma \ref{unif-bd}, 
	by the assumption $\frac{\sqrt{-\log\epsilon}+\sqrt{\log m}}{\sqrt{m}\epsilon^{d}}\rightarrow 0$, we have $\widehat{d}_{\textup{ref},\epsilon,n}(x)\geq C_1\epsilon^{d}/2$ with probability $1-\mathcal{O}(m^{-2})$ when $m$ is sufficiently large.
	
	Likewise, we have
	\begin{align}
	\nonumber&
	\sup_{x}\abs{d_{\textup{ref},\epsilon,n}(x) - d_{\textup{ref},\epsilon}(x)}
	=
	\sup_{x}\abs{\mathbb{P}_{n}\widetilde{\mathbb{P}}K_{\epsilon}(x,\cdot)K_{\epsilon}(\cdot,\star)-\mathbb{P}\widetilde{\mathbb{P}}K_{\epsilon}(x,\cdot)K_{\epsilon}(\cdot,\star)}
	\\\nonumber \leq&\,
	\sup_{f\in\int\mathcal{K}\cdot\mathcal{K}}\abs{\mathbb{P}_{n}f - \widetilde{\mathbb{P}}f}=\mathcal{O}\left(\frac{\sqrt{-\log\epsilon}+\sqrt{\log n}}{\sqrt{n}}\right)\,.
	\end{align}
	Again, by the assumption $\frac{\sqrt{-\log\epsilon}+\sqrt{\log m}}{\sqrt{m}\epsilon^{d}}\rightarrow 0$, we have $d_{\textup{ref},\epsilon,n}(x)\geq C_1\epsilon^{d}/2$ with probability $1-\mathcal{O}(n^{-2})$ when $n$ is sufficiently large.
\end{proof}

Next, we control the other terms we need for the spectral convergence rate.

\begin{lemma}\label{Proposition: bound of several operators}
	Suppose $m=n^\beta$ for $\beta\in (0,1]$ and take $\epsilon=\epsilon(n)$ so that $\epsilon\to 0$ when $n\to \infty$. If we further assume that $\frac{\sqrt{-\log\epsilon}+\sqrt{\log n†}}{\sqrt{n}\epsilon^{d}}\rightarrow 0$, we have with probability $1-\mathcal{O}(n^{-2})$ the following bound: $$
	\norm{T_{\textup{ref},\epsilon,n}}\leq \frac{2C_2}{C_1}\,,
	$$
	where $C_1$ and $C_2$ are constants defined in Lemma \ref{unif-bd}.
	Moreover, the following two bounds always hold:
	$$
	\norm{T_{\textup{ref},\epsilon}}\leq 1\,,\,\,\norm{\widehat{T}_{\textup{ref},\epsilon,n}}\leq 1\,.
	$$
\end{lemma}
\begin{proof}
	Take $f\in C(M)$ so that $\|f\|_\infty\leq 1$. Since the kernel is positive, by definition,
	\begin{align}
	|\widehat{T}_{\textup{ref},\epsilon,n}f(x)|\leq \frac{\frac{1}{n}\sum_{i=1}^{n}\widehat{K}_{\textup{ref},\epsilon,n}(x,x_i)|f(x_{i})|}{\widehat{d}_{\textup{ref},\epsilon,n}(x)} \leq \|f\|_\infty\,,\nonumber
	\end{align}
	\begin{equation}
	|T_{\textup{ref},\epsilon}f(x)|\leq \frac{\int_{M}K_{\textup{ref},\epsilon}(x,y)|f(y)|p_{X}(y)\,dV(y)}{|d_{\textup{ref},\epsilon}(x)|}\leq \|f\|_\infty\,.\nonumber
	\end{equation}
	Similarly, by Lemma \ref{d_hat_bd}, with probability $1-\mathcal{O}(n^{-2})$:
	\begin{align}
	|T_{\textup{ref},\epsilon,n}f(x)|\leq \frac{\frac{1}{n}\sum_{i=1}^{n}K_{\textup{ref},\epsilon}(x,x_i)}{d_{\textup{ref},\epsilon}(x)}\|f\|_\infty= \frac{d_{\textup{ref},\epsilon,n}(x)}{d_{\textup{ref},\epsilon}(x)}\leq \frac{2C_2}{C_1} \,.\nonumber
	\end{align}
	We hence finish the proof.
\end{proof}

\begin{lemma}\label{tech_bd}
	Take $g\in C(M)$. 
	Suppose $m=n^\beta$ for $\beta\in (0,1]$ and take $\epsilon=\epsilon(n)$ so that $\epsilon\to 0$ when $n\to \infty$. If we further assume that $\frac{\sqrt{-\log\epsilon}+\sqrt{\log m}}{\sqrt{m}\epsilon^{d}}\rightarrow 0$, when $n$ is sufficiently large, we have with probability $1-\mathcal{O}(m^{-2})$:
	\begin{align}
	\norm{\widehat{T}_{\textup{ref},\epsilon,n}-T_{\textup{ref},\epsilon,n}}&\,\leq \frac{2(\sqrt{-\log \epsilon}+\sqrt{\log m})}{C^2_{1}\sqrt{m}\epsilon^{2d}}\,,\label{Proposition SI3 bound Tnhat and Tn}
	\end{align}
	and with probability $1-\mathcal{O}(n^{-2})$: 
	\begin{align}
	\norm{(T_{\textup{ref},\epsilon}-T_{\textup{ref},\epsilon,n})T_{\textup{ref},\epsilon,n}}&\,\leq C_3\frac{\sqrt{-\log \epsilon}+\sqrt{\log n}}{\sqrt{n}}\,,\\
	\norm{(T_{\textup{ref},\epsilon,n}-T_{\textup{ref},\epsilon})g}_{\infty}&\,\leq C_4\frac{\sqrt{-\log \epsilon}+\sqrt{\log n}}{\sqrt{n}}\,,\nonumber
	\end{align}
	where $C_3>0$ is a constant depending on the kernel, the curvature of the manifold and the minima of $p_X$ and $p_Y$, and the constants shown in the entropy bound shown in Theorem \ref{entropy_bd}, and $C_4>0$ is a constant depending on the kernel, the curvature of the manifold and the minima of $p_X$ and $p_Y$, $\|g\|_\infty$ and the constants shown in the entropy bound shown in Theorem \ref{entropy_bd}.
\end{lemma}

Note that the bound for $\norm{\widehat{T}_{\textup{ref},\epsilon,n}-T_{\textup{ref},\epsilon,n}}$ is dominated by $\epsilon^{-2d}$ since we need to control the term $\widehat{d}_{\textup{ref},\epsilon,n}(x)$. Also, the condition $\frac{\sqrt{-\log\epsilon}+\sqrt{\log m}}{\sqrt{m}\epsilon^{d}}\to 0$ does not imply that $\norm{\widehat{T}_{\textup{ref},\epsilon,n}-T_{\textup{ref},\epsilon,n}}\to 0$.

\begin{proof}
	Take $f\in C(M)$ so that $\|f\|_\infty\leq 1$. 
	By the same calculation (e.g. Step 2 in the proof of Theorem \ref{spec_cong}) and Lemma \ref{d_hat_bd}, we have
	\begin{align}
	\nonumber&\, \norm{(\widehat{T}_{\textup{ref},\epsilon,n}-T_{\textup{ref},\epsilon,n})f}_{\infty}
	=\max_{x\in M}\abs{\mathbb{P}_{n}\widehat{M}_{\textup{ref},\epsilon,n}(x,\cdot)f(\cdot) - \mathbb{P}_{n}M_{\textup{ref},\epsilon}(x,\cdot)f(\cdot)}
	\\\nonumber\leq&\,\max_{x\in M}\abs{\mathbb{P}_{n}\widehat{M}_{\textup{ref},\epsilon,n}(x,\cdot)f(\cdot) -
		\mathbb{P}_{n}\widehat{M}^{(d)}_{\textup{ref},\epsilon,n}(x,\cdot)f(\cdot)}
	+
	\max_{x\in M}\abs{\mathbb{P}_{n}\widehat{M}^{(d)}_{\textup{ref},\epsilon,n}(x,\cdot)f(\cdot)-\mathbb{P}_{n}M_{\textup{ref},\epsilon}(x,\cdot)f(\cdot)}\,,
	\end{align}
	where $\widehat{M}^{(d)}_{\textup{ref},\epsilon,n}(x,y):=\frac{K_{\textup{ref},\epsilon}(x,y)}{\widehat{d}_{\textup{ref},\epsilon,n}(x)}$.
	By \eqref{Bound of d and dhat prob}, with probability $1-\mathcal{O}(m^{-2})$, we have
	\begin{align}
	\nonumber&\max_{x\in M}\abs{\mathbb{P}_{n}\widehat{M}_{\textup{ref},\epsilon,n}(x,\cdot)f(\cdot) -
		\mathbb{P}_{n}\widehat{M}^{(d)}_{\textup{ref},\epsilon,n}(x,\cdot)f(\cdot)}\leq \frac{2\|f\|_\infty}{C_{1}\epsilon^{d}}\max_{x\in M}\frac{1}{n}\sum_{l=1}^n|K_{\textup{ref},\epsilon}(x,x_l)-\widehat K_{\textup{ref},\epsilon}(x,x_l)|\\
	\nonumber= &\, \frac{2\|f\|_\infty}{C_{1}\epsilon^{d}}\frac{1}{n}\sum_{l=1}^n\max_{x\in M}|\mathbb{P}_{m}K_{\epsilon}(x,\cdot)K_{\epsilon}(\cdot,x_l)-\widehat{\mathbb{P}}_{m}K_{\epsilon}(x,\cdot)K_{\epsilon}(\cdot,x_l) |\,.
	\end{align}
	We bound the right hand side by Lemma \ref{cant_class_rate}; that is, with probability $1-\mathcal{O}(m^{-2})$, we have
	\begin{align}
	\nonumber&\, \frac{2\|f\|_\infty}{C_{1}\epsilon^{d}}\sup_{F\in\mathcal{K}\cdot\mathcal{K}}\abs{\widetilde{\mathbb{P}}_{m}F - \widetilde{\mathbb{P}}F}\leq \frac{2C_3\|f\|_\infty }{C_{1}\epsilon^{d}}\frac{\sqrt{-\log \epsilon}+\sqrt{\log m}}{\sqrt{m}}\,,
	\end{align}
	where $C_3$ is the implied constant in Lemma \ref{cant_class_rate}.
	For the other term, with probability $1-\mathcal{O}(m^{-2})$, we have
	\begin{align}
	\nonumber&\sup_{x}\abs{\mathbb{P}_{n}\widehat{M}^{(d)}_{\textup{ref},\epsilon,n}(x,\cdot)f(\cdot)-\mathbb{P}_{n}M_{\textup{ref},\epsilon}(x,\cdot)f(\cdot)}\\
	\nonumber\leq\,& \|f\|_\infty \max_{x\in M}\frac{1}{n}\sum_{l=1}^n\frac{K_{\textup{ref},\epsilon}(x,x_l)}{d_{\textup{ref},\epsilon}(x)\hat d_{\textup{ref},\epsilon,n}(x)}\abs{\hat d_{\textup{ref},\epsilon,n}(x)-d_{\textup{ref},\epsilon}(x)}\leq   \frac{2\|f\|_\infty(\sqrt{-\log \epsilon}+\sqrt{\log m})}{C^2_{1}\sqrt{m}\epsilon^{2d}}\,,
	\end{align}
	where we use the fact the $K_\epsilon$ is positive, $\|K_{\textup{ref},\epsilon}\|_\infty=1$, and Lemma \ref{d_hat_bd}. 
	As a result, by combining the above two bounds with a union probability bound, we conclude that when $m$ is sufficiently large, with probability $1-\mathcal{O}(m^{-2})$,
	\begin{align}
	\norm{\widehat{T}_{\textup{ref},\epsilon,n}-T_{\textup{ref},\epsilon,n}} \nonumber\leq
	\frac{2(\sqrt{-\log \epsilon}+\sqrt{\log m})}{C^2_{1}\sqrt{m}\epsilon^{2d}}\,.
	\end{align}

	The second statement follows the same argument, but with more terms to control:
	\begin{align}
	\nonumber&\quad
	\norm{(T_{\textup{ref},\epsilon}-T_{\textup{ref},\epsilon,n})T_{\textup{ref},\epsilon,n}f}_{\infty}
	\\\nonumber&=
	\norm{T_{\textup{ref},\epsilon}\left(\frac{1}{n}\sum_{i=1}^{n}M_{\textup{ref},\epsilon}(x,x_{i})f(x_{i})\right)-T_{\textup{ref},\epsilon,n}\left(\frac{1}{n}\sum_{i=1}^{n}M_{\textup{ref},\epsilon}(x,x_{i})f(x_{i})\right)}_{\infty}
	\\\nonumber&=
	\norm{\int M_{\textup{ref},\epsilon}(y,z)\left(\frac{1}{n}\sum_{i=1}^{n}M_{\textup{ref},\epsilon}(z,x_{i})f(x_{i})\right)d\mathbb{P}(z)\right.\\
		&\nonumber\left.\qquad\qquad\qquad-\frac{1}{n}\sum_{j=1}^{n}M_{\textup{ref},\epsilon}(y,z_{j})\left(\frac{1}{n}\sum_{i=1}^{n}M_{\textup{ref},\epsilon}(z_{j},x_{i})f(x_{i})\right)}_{\infty}\,,
	\end{align}
	which, by noting that the term $\frac{1}{n}\sum_{i=1}^{n}f(x_{i})$ can be isolated, can be bounded by
	\begin{align}
	\nonumber&\quad
	\norm{f}_\infty\sup_{y}\abs{\int M_{\textup{ref},\epsilon}(y,z)M_{\textup{ref},\epsilon}(z,x_{i})d\mathbb{P}(z)-\frac{1}{n}\sum_{j=1}^{n}M_{\textup{ref},\epsilon}(y,z_{j})M_{\textup{ref},\epsilon}(z_{j},x_{i})}
	\\\nonumber&\leq
	\|f\|_\infty \sup_{F\in\mathcal{M}\cdot\mathcal{M}}\abs{\mathbb{P}F-\mathbb{P}_{n}F}
	\leq 
	C_3 \|f\|_\infty \frac{\sqrt{-\log \epsilon}+\sqrt{\log n}}{\sqrt{n}}%\,,
	\end{align}
	with probability $1-\mathcal{O}(n^{-2})$, where the first inequality comes from the fact that $\frac{1}{n}\sum_{i=1}^{n}\abs{f(x_{i})}\leq 1$. 
	
	The final statement is by a direct bound: 
	\begin{align}
	\nonumber\norm{(T_{\textup{ref},\epsilon,n}-T_{\textup{ref},\epsilon})g}_{\infty} &= \norm{\mathbb{P}_{n}M_{\textup{ref},\epsilon}(x,\cdot)f(\cdot) - \mathbb{P}M_{\textup{ref},\epsilon}(x,\cdot)f(\cdot)}_{\infty}
	\leq
	\sup_{F\in g\cdot\mathcal{M}}\abs{\mathbb{P}_{n}F - \mathbb{P}F}\,,
	\end{align}
	which leads to the conclusion by Lemma \ref{cant_class_rate}. Note that due to the finite sampling, we cannot control the error simply by the $\|g\|_\infty$.
\end{proof}

Based on the above preparation, the following proposition describes the spectral convergence of the operator {$\widehat{T}_{\textup{ref},\epsilon,n}$ } to the operator $T_{\textup{ref},\epsilon}$.  Denote $\lambda_{i,\epsilon}$ is the $i$-th smallest eigenvalue of $\frac{I-T_{\textup{ref},\epsilon}}{\epsilon}$,
and denote $u_{\epsilon,i}$ to be the associated eigenfunction.
Clearly, $\frac{I-T_{\textup{ref},\epsilon}}{\epsilon}$ and $T_{\textup{ref},\epsilon}$ share the same eigenfunctions.
Similarly, denote $\lambda_{\epsilon,n,i}$ to be the $i$-th smallest eigenvalue of {  $\frac{I-\widehat{T}_{\textup{ref},\epsilon,n}}{\epsilon}$}, and denote $u_{\epsilon, n,i}$ to be the associated eigenfunction. 
We assume that both $u_{\epsilon,i}$ and $u_{\epsilon, n,i}$ are normalized in the $L^2$ norm.

\begin{proposition}\label{Proposition: Tepsn conv to Teps with rate}
	Fix $K\in \mathbb{N}$. Assume that the eigenvalues of $\Delta$ are simple. Take $m=n^\beta$, where $\beta\in (0,1)$. Suppose $\epsilon=\epsilon(n)$ so that $\epsilon \to 0$ and $\frac{\sqrt{-\log\epsilon}+\sqrt{\log m}}{\sqrt{m}\epsilon^{d}} \rightarrow 0$, as $n \rightarrow \infty$, and $\epsilon \leq \mathcal{K}_1 \min \Bigg(\bigg(\frac{\min(\Gamma_K,1)}{\mathcal{K}_2+\lambda_K^{d/2+5}}\bigg)^2, \frac{1}{(2+\lambda_K^{d+1})^2}\Bigg)$, where $\Gamma_K$, $\mathcal{K}_1$ and $\mathcal{K}_2>1$  are introduced in Proposition \ref{T epsilon and Delta}, then there is a sequence $a_n \in \{1,-1\}$ such that with probability $1-\mathcal{O}(n^{-2})$,  for all $i < K$, 
	we have
	\begin{align}
	& |\lambda_{\epsilon,n,i}-\lambda_{\epsilon,i}|\leq  3\mathcal{K}_3\frac{\sqrt{-\log \epsilon}+\sqrt{\log m}}{\sqrt{m}\epsilon^{2d+2}}\,, \nonumber  \\
	& \|a_nu_{\epsilon,n,i}-u_{\epsilon,i}\|_{\infty} \leq \mathcal{K}_3\frac{\sqrt{-\log \epsilon}+\sqrt{\log m}}{\sqrt{m}\epsilon^{2d+3/2}}\,. \nonumber 
	\end{align}
	where $\mathcal{K}_3$ is a constant depending on the kernel, the curvature of $M$, $p_X$ and $p_Y$. 
\end{proposition}

Note that the imposed conditions, like $\frac{\sqrt{-\log\epsilon}+\sqrt{\log m}}{\sqrt{m}\epsilon^{d}}\to 0$, does not imply that $|\lambda_{\epsilon,n,i}-\lambda_{\epsilon,i}|\to 0$ or $\|a_nu_{\epsilon,n,i}-u_{\epsilon,i}\|_{\infty}\to 0$.

\begin{proof}
	Note that due to Proposition \ref{T epsilon and Delta} and the assumption that the eigenvalues of $\Delta$ are simple, we have that for any $K\in \mathbb{N}$, when $\epsilon>0$ is sufficiently small, the first smallest $K$ eigenvalues of $\frac{I-T_{\textup{ref},\epsilon}}{\epsilon}$ are simple. Specifically, it is shown in the proof of \cite[Proposition 1 (SI.20)]{DunsonWuWu2019} that when $\epsilon$ satisfies the assumption, for each $i<K$, we have
	\begin{equation}\label{Spectral gap of I-Teps/eps}
	\gamma_i\Big(\frac{I-T_{\textup{ref},\epsilon}}{\epsilon}\Big)\geq \frac{1}{12}\Gamma_K\,,
	\end{equation}
	where $\gamma_i$ is defined in \eqref{Proof proposition 1 Definition gamma i}. 
	
	Fix $i<K$. Take 
	\begin{equation}
	r=\frac{\Gamma_K}{24}\epsilon\,.
	\end{equation}
	We now quantify how the sequence $\{u_{\epsilon,n,i}\}_{n=1}^\infty$ converges to $u_{\epsilon,i}$. It is clear that $\frac{I-T_{\textup{ref},\epsilon}}{\epsilon}$ and $T_{\textup{ref},\epsilon}$ share the same eigenfunctions, with the eigenvalues directly related. Denote $\bar{\lambda}_{\epsilon,i}=1-\lambda_{\epsilon,i}\epsilon$ to be the $i$-th largest eigenvalue of $T_{\textup{ref},\epsilon}$. Similarly, this relationship holds for { $\frac{I-\widehat{T}_{\textup{ref},\epsilon,n}}{\epsilon}$ and $\widehat{T}_{\textup{ref},\epsilon,n}$}, and we denote $\bar{\lambda}_{\epsilon,n,i}=1-\lambda_{\epsilon,n,i}\epsilon$ to be the $i$-th largest eigenvalue of { $\widehat{T}_{\textup{ref},\epsilon,n}$}. Therefore, we can directly compare $T_{\textup{ref},\epsilon}$ and { $\widehat{T}_{\textup{ref},\epsilon,n}$}.

	By Proposition \ref{cong_1_dim_proj}, to control $\norm{a_nu_{\epsilon,n,i}-u_{\epsilon,i}}$, we need to bound $\norm{u_{\epsilon,i}-\textup{Pr}_{u_{\epsilon,n,i}}(u_{\epsilon,i})}$. 
	Since $T_{\textup{ref},\epsilon,n}$ converges to $T_{\textup{ref},\epsilon}$ collectively compactly a.e. by Lemma \ref{Proposition: Tepsn conv to Teps compactly}, we apply Theorem \ref{atkinson1967bound} to control 
	$\norm{u_{\epsilon,i}-\textup{Pr}_{u_{\epsilon,n,i}}(u_{\epsilon,i})}$. To apply Theorem \ref{atkinson1967bound}, we need to control  
	$\max_{z\in \Gamma_r(\bar\lambda_{\epsilon,i})} \|R_{z}(T_{\textup{ref},\epsilon})\|$, $\min_{z\in \Gamma_r(\bar\lambda_{\epsilon,i})}|z|$,
	$\norm{(\widehat{T}_{\textup{ref},\epsilon,n}-T_{\textup{ref},\epsilon})u_{\epsilon,i}}$ and $\norm{(T_{\textup{ref},\epsilon}-\widehat{T}_{\textup{ref},\epsilon,n})\widehat{T}_{\textup{ref},\epsilon,n}}$. 
	First, according to \eqref{Spectral gap of I-Teps/eps}, the spectral gap of the $i$-th largest eigenvalue of $T_{\textup{ref},\epsilon}$ is bounded from below by $\frac{\Gamma_K}{12}\epsilon$. Therefore, by the basic bound of the resolvent (see, for example, \cite[Lemma SI.16]{DunsonWuWu2019}) and the chosen $r$, we have
	\begin{equation}\label{Proof spectral rate bound1}
	\max_{z\in \Gamma_r(\bar\lambda_{\epsilon,i})} \|R_{z}(T_{\textup{ref},\epsilon})\|\leq \frac{1}{r}=\frac{24}{\Gamma_K \epsilon}\,.
	\end{equation}
	{
	By Proposition \ref{T epsilon and Delta}, under the assumption, we have $\lambda_{\epsilon,i}\leq \lambda_i+\epsilon^{3/4}$, and hence if $\epsilon$ is sufficiently small, we have}
	
	\begin{equation}\label{Proof spectral rate bound2}
	\min_{z\in \Gamma_r(\bar\lambda_{\epsilon,i})}|z|\geq 1/2\,.
	\end{equation}
	For the remaining terms, by a direct triangular inequality, we have
	\begin{align}
	\nonumber\norm{(\widehat{T}_{\textup{ref},\epsilon,n}-T_{\textup{ref},\epsilon})u_{\epsilon,i}}_\infty&\,\leq
	\norm{(\widehat{T}_{\textup{ref},\epsilon,n}-T_{\textup{ref},\epsilon,n})u_{\epsilon,i}}_\infty+\norm{(T_{\textup{ref},\epsilon,n}-T_{\textup{ref},\epsilon})u_{\epsilon,i}}_\infty\\
	\nonumber&\,\leq \norm{\widehat{T}_{\textup{ref},\epsilon,n}-T_{\textup{ref},\epsilon,n}}\norm{u_{\epsilon,i}}_\infty+\norm{(T_{\textup{ref},\epsilon,n}-T_{\textup{ref},\epsilon})u_{\epsilon,i}}_\infty\,.
	\end{align}
	Moreover, with probability $1-\mathcal{O}(m^{-2})$,
	\begin{align}
	\nonumber&\norm{(T_{\textup{ref},\epsilon}-\widehat{T}_{\textup{ref},\epsilon,n})\widehat{T}_{\textup{ref},\epsilon,n}}
	\\\nonumber\leq&\,
	\norm{T_{\textup{ref},\epsilon}\widehat{T}_{\textup{ref},\epsilon,n}-T_{\textup{ref},\epsilon}T_{\textup{ref},\epsilon,n}} + \norm{T_{\textup{ref},\epsilon}T_{\textup{ref},\epsilon,n}-T_{\textup{ref},\epsilon,n}T_{\textup{ref},\epsilon,n}}
	\\\nonumber&\qquad+\,
	\norm{T_{\textup{ref},\epsilon,n}T_{\textup{ref},\epsilon,n}-T_{\textup{ref},\epsilon,n}\widehat{T}_{\textup{ref},\epsilon,n}} + \norm{T_{\textup{ref},\epsilon,n}\widehat{T}_{\textup{ref},\epsilon,n}-\widehat{T}_{\textup{ref},\epsilon,n}\widehat{T}_{\textup{ref},\epsilon,n}}
	\\\nonumber\leq&\,
	\left(\norm{T_{\textup{ref},\epsilon}}+\norm{T_{\textup{ref},\epsilon,n}}+\norm{\widehat{T}_{\textup{ref},\epsilon,n}}\right)\norm{T_{\textup{ref},\epsilon,n}-\widehat{T}_{\textup{ref},\epsilon,n}}+\norm{(T_{\textup{ref},\epsilon}-T_{\textup{ref},\epsilon,n})T_{\textup{ref},\epsilon,n}}
	\\\nonumber\leq&\,
	C_{7}\norm{T_{\textup{ref},\epsilon,n}-\widehat{T}_{\textup{ref},\epsilon,n}}+\norm{(T_{\textup{ref},\epsilon}-T_{\textup{ref},\epsilon,n})T_{\textup{ref},\epsilon,n}}\,,
	\end{align}
	for $C_7>0$, where the last bound comes from Lemma \ref{Proposition: bound of several operators}.
	By Theorem \ref{atkinson1967bound}, we have:
	\begin{align}
	&\nonumber\quad\norm{u_{\epsilon,i}-\textup{Pr}_{u_{\epsilon,n,i}}u_{\epsilon,i}}
	\leq \max_{z\in \Gamma_r(\lambda_{\epsilon,i})} \frac{2r\|R_{z}(T_{\textup{ref},\epsilon})\|}{\min_{z\in \Gamma_r(\lambda_{\epsilon,i})}|z|}\Big(\norm{(\widehat{T}_{\textup{ref},\epsilon,n}-T_{\textup{ref},\epsilon})u_{\epsilon,i}} 
	\\\nonumber&\qquad\qquad\qquad\qquad+\|R_{z}(T_{\textup{ref},\epsilon})u_{\epsilon,i}\|_\infty\norm{(T_{\textup{ref},\epsilon}-\widehat{T}_{\textup{ref},\epsilon,n})\widehat{T}_{\textup{ref},\epsilon,n}}\Big)
	\\&\leq 4\Big(\norm{(\widehat{T}_{\textup{ref},\epsilon,n}-T_{\textup{ref},\epsilon})\frac{u_{\epsilon,i}}{\norm{u_{\epsilon,i}}_\infty}} +\frac{24}{\Gamma_K \epsilon}\norm{(T_{\textup{ref},\epsilon}-\widehat{T}_{\textup{ref},\epsilon,n})\widehat{T}_{\textup{ref},\epsilon,n}}\Big)\norm{u_{\epsilon,i}}_\infty  \nonumber \,, %
	\end{align}
	where we plug in \eqref{Proof spectral rate bound1} and \eqref{Proof spectral rate bound2}.
	Therefore, by plugging the bounds of $\norm{(\widehat{T}_{\textup{ref},\epsilon,n}-T_{\textup{ref},\epsilon})u_{\epsilon,i}}$ and $\norm{(T_{\textup{ref},\epsilon}-\widehat{T}_{\textup{ref},\epsilon,n})\widehat{T}_{\textup{ref},\epsilon,n}}$, with probability $1-\mathcal{O}(m^{-2})$, we have:
	\begin{align}
	\nonumber&\norm{u_{\epsilon,i}-\textup{Pr}_{u_{\epsilon,n,i}}u_{\epsilon,i}}\\
	\nonumber\leq&\, 4\bigg[\Big(\frac{2}{C^2_{1}}+\frac{24C_7}{\Gamma_K\epsilon}\Big)\frac{\sqrt{-\log \epsilon}+\sqrt{\log m}}{\sqrt{m}\epsilon^{2d}} +C_3\Big(1+\frac{24}{\Gamma_K }\Big)\frac{\sqrt{-\log \epsilon}+\sqrt{\log n}}{\sqrt{n}\epsilon}\bigg]\norm{u_{\epsilon,i}}_\infty\,.
	\end{align}
	Note that as discussed after Lemma \ref{cant_class_rate}, when we apply Lemma \ref{cant_class_rate}, the bound depends on the eigenfunction.
	To control $\norm{u_{\epsilon,i}}_\infty$, note that by Proposition \ref{T epsilon and Delta} and Lemma \ref{lemma hormander}, we have
		\begin{equation}
		\|u_{\epsilon,i}\|_\infty\leq \|u_i\|_\infty+\epsilon^{1/2} \leq C_1\lambda_K^{(d-1)/4}+\epsilon\leq 2C_1\lambda_K^{(d-1)/4}\,,
		\end{equation}
	where the last inequality comes from the assumption of $\epsilon$. Moreover, by the assumption of $\epsilon$, we have 
		\begin{equation}\label{controls of u infty u/Gamma and 1/Gamma}
		\max\Big\{\|u_{\epsilon,i}\|_\infty,\frac{\|u_{\epsilon,i}\|_\infty}{\Gamma_K},\frac{1}{\Gamma_K}\Big\}\leq \epsilon^{-1/2}\,.
		\end{equation}
	As a result, with probability $1-\mathcal{O}(m^{-2})$, we can find $a_n\in \{1,-1\}$ so that
		\begin{align}
		&\norm{a_nu_{\epsilon,n,i}-u_{\epsilon,i}}\leq 2\norm{u_{\epsilon,i}-\textup{Pr}_{u_{\epsilon,n,i}}u_{\epsilon,i}}_\infty
		\leq 192C_7\frac{\sqrt{-\log \epsilon}+\sqrt{\log m}}{\sqrt{m}\epsilon^{2d+3/2}}\,.\label{uepsni and uepsi difference L infty}
		\end{align}
	By setting $\mathcal{K}_3:=192C_7$, we get the claim for the eigenvectors.
	For eigenvalues, we have
	\begin{align}
	\nonumber&\abs{\bar\lambda_{\epsilon,i}-\bar\lambda_{\epsilon,n,i}}\norm{u_{\epsilon,i}}_\infty 
	= \nonumber\norm{\bar\lambda_{\epsilon,i}u_{\epsilon,i}-\bar\lambda_{\epsilon,n,i}u_{\epsilon,i}}_\infty
	\\\nonumber\leq&\,
	\norm{\bar\lambda_{\epsilon,i}u_{\epsilon,i}-\bar\lambda_{\epsilon,n,i}a_{n}u_{\epsilon,n,i}}_\infty + \abs{\bar\lambda_{\epsilon,n,i}}\norm{a_{n}u_{\epsilon,n,i}-u_{\epsilon,i}}_\infty
	\\\nonumber=&\,
	\norm{T_{\textup{ref},\epsilon}u_{\epsilon,i}-a_{n}\widehat{T}_{\textup{ref},\epsilon,n}u_{\epsilon,n,i}}_\infty + \abs{\bar\lambda_{\epsilon,n,i}}\norm{a_{n}u_{\epsilon,n,i}-u_{\epsilon,i}}_\infty\,,
	\end{align}
	which we further bound by
	\begin{align}
	\nonumber\quad&\,
	\norm{T_{\textup{ref},\epsilon}u_{\epsilon,i}-\widehat{T}_{\textup{ref},\epsilon,n}u_{\epsilon,i}}+\norm{\widehat{T}_{\textup{ref},\epsilon,n}u_{\epsilon,i}-a_{n}\widehat{T}_{\textup{ref},\epsilon,n}u_{\epsilon,n,i}}_\infty + \abs{\bar\lambda_{\epsilon,n,i}}\norm{a_{n}u_{\epsilon,n,i}-u_{\epsilon,i}}_\infty
	\\\nonumber\leq &\,
	\norm{(T_{\textup{ref},\epsilon,n}-\widehat{T}_{\textup{ref},\epsilon,n})u_{\epsilon,i}}_\infty +\norm{(T_{\textup{ref},\epsilon}-{T}_{\textup{ref},\epsilon,n})u_{\epsilon,i}}_\infty +\left(\norm{\widehat{T}_{\textup{ref},\epsilon,n}}+\abs{\bar\lambda_{\epsilon,n,i}}\right)\norm{a_{n}u_{\epsilon,n,i}-u_{\epsilon,i}}_\infty
	\\\nonumber\leq &\,
	\Big(\norm{T_{\textup{ref},\epsilon,n}-\widehat{T}_{\textup{ref},\epsilon,n}}+\norm{(T_{\textup{ref},\epsilon}-{T}_{\textup{ref},\epsilon,n})\frac{u_{\epsilon,i}}{\norm{u_{\epsilon,i}}_\infty}}_\infty\Big) \norm{u_{\epsilon,i}}_\infty+2\norm{\widehat{T}_{\textup{ref},\epsilon,n}}\norm{a_{n}u_{\epsilon,n,i}-u_{\epsilon,i}}_\infty\,,
	\end{align}
	where we use the fact that $\frac{u_{\epsilon,i}}{\norm{u_{\epsilon,i}}_\infty}$ has the $L^\infty$ norm $1$, and $\abs{\bar\lambda_{\epsilon,n,i}}\leq \norm{\widehat{T}_{\textup{ref},\epsilon,n}}$. 
	
	{
		As a result, by Lemmas \ref{Proposition: bound of several operators} and \ref{tech_bd}, \eqref{controls of u infty u/Gamma and 1/Gamma} and \eqref{uepsni and uepsi difference L infty}, with probability $1-\mathcal{O}(m^{-2})$,
		$\abs{\bar\lambda_{\epsilon,i}-\bar\lambda_{\epsilon,n,i}}$ is bounded by $4\frac{\sqrt{-\log \epsilon}+\sqrt{\log m}}{C_1^2\sqrt{m}\epsilon^{2d}}+ 2\mathcal{K}_3\frac{\sqrt{-\log \epsilon}+\sqrt{\log m}}{\sqrt{m}\epsilon^{2d+1}}$. Hence $|\lambda_{\epsilon,n,i}-\lambda_{\epsilon,i}|=\abs{\bar\lambda_{\epsilon,i}-\bar\lambda_{\epsilon,n,i}}/\epsilon\leq 3\mathcal{K}_3\frac{\sqrt{-\log \epsilon}+\sqrt{\log m}}{\sqrt{m}\epsilon^{2d+2}}$. We thus finish the proof.
		
	}
	
\end{proof}

\subsection{Finish the proof of Theorem \ref{spec_cong}}

With the above preparation, we are ready to prove the main theorem.
\begin{proof}[Proof of Theorem \ref{spec_cong}]
	
	By Lemma \ref{Proposition: Tepsn conv to Teps compactly}, $\widehat{T}_{\textup{ref},\epsilon,n}$
	converges to $T_{\textup{ref},\epsilon}$ compactly $\textup{a.s.}$ as $n\rightarrow\infty$. Therefore, Proposition \ref{Proposition: Tepsn conv to Teps with rate} leads to the spectral convergence of $\widehat{T}_{\textup{ref},\epsilon,n}$ to $T_{\textup{ref},\epsilon}$ with the rate.
	Next, we link $T_{\textup{ref},\epsilon}$ to $-\Delta$. By Theorem \ref{Proof: Theorem1: bias analysis}, when $p_Y$ is properly chosen so that $\frac{2\nabla p_{X}(x)}{p_{X}(x)}+\frac{\nabla p_{Y}(x)}{p_{Y}(x)}=0$, we have the convergence of the eigenvalue/eigenfunction of $\frac{1-T_{\textup{ref},\epsilon}}{\epsilon}$ to those of $-\Delta$. 
	Thus, by Proposition \ref{T epsilon and Delta}, we have the spectral convergence of $\frac{1-T_{\textup{ref},\epsilon}}{\epsilon}$ to $-\Delta$ with the rate.
	Finally, we put all the above together and finish the spectral convergence proof. 
	
\end{proof}

\section{Proof of Theorems \ref{robust thm roseland} -- robustness} \label{robust proof}

We start by preparing some generic lemmas.

\begin{lemma}\label{D-1/2W 2 norm general}
	Let $W$ and $\widetilde{W}$ be $n\times m$ matrices, whose entries are $W_{ik}$ and $\widetilde{W}_{ik}$ respectively. Let $D$ and $\widetilde{D}$ be two $n\times n$ diagonal matrices with entries
	$D_{ii} = \sum_{j=1}^{n}\sum_{k=1}^{m}W_{ik}W_{jk}$ and $\widetilde{D}_{ii}=\sum_{j=1}^{n}\sum_{k=1}^{m}\widetilde{W}_{ik}\widetilde{W}_{jk}$.
	Assume $D^{-1/2}$ and $\widetilde{D}^{-1/2}$ both exist. 
	Suppose $\sup_{i,k}\abs{W_{ik}-\widetilde{W}_{ik}}\leq\delta$, $0\leq W_{ik}\leq C$ for some constant $C>0$, and $\inf_{i}D_{ii}/mn>\gamma$, such that $\gamma> 2C\delta+\delta^2$. We have 
	\begin{align}
	\norm{D^{-1/2}W-\widetilde{D}^{-1/2}\widetilde{W}}_{2}\leq \frac{\delta}{\sqrt{\gamma}}+\frac{(2C\delta+\delta^2)(C+\delta)}{\gamma\sqrt{\gamma-2C\delta-\delta^2}+\sqrt{\gamma}(\gamma-2C\delta-\delta^2)}\,\nonumber
	\end{align}  
	and
	\begin{align}
	\norm{D^{-1/2}W-\widetilde{D}^{-1/2}\widetilde{W}}_{F}\leq \frac{\delta}{\sqrt{\gamma}}+\frac{(2C\delta+\delta^2)(C+\delta)}{\gamma\sqrt{\gamma-2C\delta-\delta^2}+\sqrt{\gamma}(\gamma-2C\delta-\delta^2)}\,.\label{D-1/2W F norm general}
	\end{align}  
\end{lemma}	

\begin{proof}
	Firstly we have
	\begin{align}
	\abs{\frac{D_{ii}}{mn}-\frac{\widetilde{D}_{ii}}{mn}}=&\frac{1}{mn}\abs{\sum_{j,k}W_{ik}W_{jk}-\sum_{j,k}\widetilde{W}_{ik}\widetilde{W}_{jk}}
	\leq\sup_{j,k}\abs{W_{ik}W_{jk}-\widetilde{W}_{ik}\widetilde{W}_{jk}}\nonumber
	\\\leq&\sup_{j,k}|W_{ik}||W_{jk}-\widetilde{W}_{jk}|+\sup_{j,k}|\widetilde{W}_{jk}||W_{ik}-\widetilde{W}_{ik}|
	\leq2C\delta+\delta^2\,,\nonumber
	\end{align}
	where we use the fact that $\sup_{i,k}\abs{W_{ik}-\widetilde{W}_{ik}}\leq\delta$ implies $\sup_{ik}\widetilde W_{ik}\leq C+\delta$.
	In particular, $\abs{D_{ii}-\widetilde{D}_{ii}}\leq mn(2C\delta+\delta^2)$ and $D_{ii}-mn(2C\delta+\delta^2)\leq\widetilde{D}_{ii}\leq D_{ii}+mn(2C\delta+\delta^2).$ Next, by the assumption we have $D_{ii}/mn>\gamma$, so $D_{ii}> mn\gamma$, $\widetilde{D}_{ii}> mn(\gamma-2C\delta-\delta^2),$ and $\norm{(D/mn)^{-1/2}}_{2}< 1/\sqrt{\gamma}$. Thus, we know 
	\begin{align}
	&\abs{\left(\frac{D_{ii}}{mn}\right)^{-1/2}-\left(\frac{\widetilde{D}_{ii}}{mn}\right)^{-1/2}} =\sqrt{mn}\abs{\frac{D_{ii}-\widetilde{D}_{ii}}{\sqrt{D_{ii}\widetilde{D}_{ii}}(\sqrt{D}_{ii}+\sqrt{\widetilde{D}}_{ii})}}\nonumber
	\\\leq&\,\sqrt{mn}\frac{mn(2C\delta+\delta^2)}{(mn)^{3/2}(\gamma\sqrt{\gamma-2C\delta-\delta^2}+\sqrt{\gamma}(\gamma-2C\delta-\delta^2))}
	=\frac{2C\delta+\delta^2}{\gamma\sqrt{\gamma-2C\delta-\delta^2}+\sqrt{\gamma}(\gamma-2C\delta-\delta^2)}\,.\nonumber
	\end{align}
	Also,
	\begin{align}
	\norm{\frac{W}{\sqrt{mn}}-\frac{\widetilde{W}}{\sqrt{mn}}}_{2}\leq&\,\norm{\frac{W}{\sqrt{mn}}-\frac{\widetilde{W}}{\sqrt{mn}}}_{F}=\sqrt{\frac{1}{mn}\sum_{i,k}\abs{W_{ik}-\widetilde{W}_{ik}}^{2}}\nonumber
	\leq\sqrt{\sup_{i,k}\abs{W_{ik}-\widetilde{W}_{ik}}^{2}}
	\leq \delta\,.\nonumber
	\end{align}
	Hence, $\norm{\widetilde{W}/\sqrt{mn}}_{2}\leq\norm{\widetilde{W}/\sqrt{mn}}_{F}\leq\norm{W/\sqrt{mn}}_{F}+\delta\leq C+\delta$.
	Finally, we conclude the operator norm bound by putting everything together:
	\begin{align}
	&\norm{D^{-1/2}W-\widetilde{D}^{-1/2}\widetilde{W}}_{2}
	\leq \norm{D^{-1/2}(W-\widetilde{W})}_{2}+\norm{(D^{-1/2}-\widetilde{D}^{-1/2})\widetilde{W}}_{2}\nonumber
	\\=&\,\norm{\left(\frac{D}{mn}\right)^{-\frac{1}{2}}\left(\frac{W}{\sqrt{mn}}-\frac{\widetilde{W}}{\sqrt{mn}}\right)}_{2} + \norm{\left(\left(\frac{D}{mn}\right)^{-\frac{1}{2}}-\left(\frac{\widetilde{D}}{mn}\right)^{-\frac{1}{2}}\right)\frac{\widetilde{W}}{\sqrt{mn}}}_{2}\nonumber
	\\\leq&\,\frac{\delta}{\sqrt{\gamma}}+\frac{(2C\delta+\delta^2)(C+\delta)}{\gamma\sqrt{\gamma-2C\delta-\delta^2}+\sqrt{\gamma}(\gamma-2C\delta-\delta^2)}\,.\nonumber
	\end{align}
	
	For the Frobenius norm, since $-C-\delta\leq\widetilde{W}_{ik}\leq C+\delta$ by assumption, we have
	\begin{align}
	\sup_{i,k}\left(\left(\left(\frac{D}{mn}\right)^{-\frac{1}{2}}-\left(\frac{\widetilde{D}}{mn}\right)^{-\frac{1}{2}}\right)\frac{\widetilde{W}}{\sqrt{mn}}\right)_{ik}\leq&\sup_{ii}\abs{\left(\frac{D_{ii}}{mn}\right)^{-1/2}-\left(\frac{\widetilde{D}_{ii}}{mn}\right)^{-1/2}}\sup_{i,k}\abs{\frac{\widetilde{W}_{ik}}{\sqrt{mn}}}\nonumber
	\\\leq& \frac{2C\delta+\delta^2}{\gamma\sqrt{\gamma-2C\delta-\delta^2}+\sqrt{\gamma}(\gamma-2C\delta-\delta^2)}\left(\frac{C+\delta}{\sqrt{mn}}\right)\,.\nonumber
	\end{align} 
	Next, note that
	\begin{align}
	\sup_{i,k}\abs{(D^{-1/2}(W-\widetilde{W}))_{ik}}\leq \sup_{i}(D^{-1/2})_{ii}\sup_{i,k}\abs{W_{ik}-\widetilde{W}_{ik}}
	\leq\frac{\delta}{\sqrt{mn\gamma}}\,.\nonumber
	\end{align}
	Finally, we conclude the Frobenius norm by putting everything together:
	\begin{align}
	%\norm{L(W)-L(\widetilde{W})}_{F}=
	&\norm{D^{-1/2}W-\widetilde{D}^{-1/2}\widetilde{W}}_{F}
	\leq \norm{D^{-1/2}(W-\widetilde{W})}_{F}+\norm{(D^{-1/2}-\widetilde{D}^{-1/2})\widetilde{W}}_{F}\nonumber
	\\=&\,\norm{D^{-1/2}(W-\widetilde{W})}_{F} + \norm{\left(\left(\frac{D}{mn}\right)^{-\frac{1}{2}}-\left(\frac{\widetilde{D}}{mn}\right)^{-\frac{1}{2}}\right)\frac{\widetilde{W}}{\sqrt{mn}}}_{F}\nonumber
	\leq\frac{\delta}{\sqrt{\gamma}}+\frac{(2C\delta+\delta^2)(C+\delta)}{\gamma\sqrt{\gamma-2C\delta-\delta^2}+\sqrt{\gamma}(\gamma-2C\delta-\delta^2)}\,.\nonumber
	\end{align}
\end{proof}

\begin{corollary}\label{D-1/2W 2 norm general corollary}
	Let $W$ and $\widetilde{W}$ be $n\times m$ matrices, whose entries are $W_{ik}$ and $\widetilde{W}_{ik}$ respectively. Let $D$ and $\widetilde{D}$ be two $n\times n$ diagonal matrices with entries $D_{ii} = \sum_{j=1}^{n}\sum_{k=1}^{m}W_{ik}W_{jk}$ and $\widetilde{D}_{ii}=\sum_{j=1}^{n}\sum_{k=1}^{m}\widetilde{W}_{ik}\widetilde{W}_{jk}.$
	Assuming $D^{-1/2}$ and $\widetilde{D}^{-1/2}$ both exist. 
	Suppose there exists $f>0$ so that
	$\sup_{i,k}\abs{W_{ik}-\frac{\widetilde{W}_{ik}}{f}}\leq\delta$,
	$0\leq W_{ik}\leq C$ for some constant $C>0$, and $\inf_{i}D_{ii}/mn>\gamma$, where $\gamma> 2C\delta+\delta^2$. We have 
	\begin{align}
	\norm{D^{-1/2}W-\widetilde{D}^{-1/2}\widetilde{W}}_{2}\leq\frac{\delta}{\sqrt{\gamma}}+\frac{(2C\delta+\delta^2)(C+\delta)}{\gamma\sqrt{\gamma-2C\delta-\delta^2}+\sqrt{\gamma}(\gamma-2C\delta-\delta^2)}\,\nonumber
	\end{align}  
	and
	\begin{align}
	\norm{D^{-1/2}W-\widetilde{D}^{-1/2}\widetilde{W}}_{F}\leq\frac{\delta}{\sqrt{\gamma}}+\frac{(2C\delta+\delta^2)(C+\delta)}{\gamma\sqrt{\gamma-2C\delta-\delta^2}+\sqrt{\gamma}(\gamma-2C\delta-\delta^2)}\,.\nonumber
	\end{align} 
\end{corollary}

\begin{proof}
	Let $\widetilde{W}_{f}$ be the matrix with entries $\widetilde{W}_{ij}/f$. By Lemma \ref{D-1/2W 2 norm general} we have
	\begin{align}
	\norm{D^{-1/2}W-\widetilde{D}_f^{-1/2}\widetilde{W}_{f}}_{2}\leq\frac{\delta}{\sqrt{\gamma}}+\frac{(2C\delta+\delta^2)(C+\delta)}{\gamma\sqrt{\gamma-2C\delta-\delta^2}+\sqrt{\gamma}(\gamma-2C\delta-\delta^2)}\,,
	\end{align}
	where $(\widetilde{D}_f)_{ii}=\sum_{j=1}^{n}\sum_{k=1}^{m}(\widetilde{W}_f)_{ik}(\widetilde{W}_f)_{jk}.$
	We then conclude by noticing $\widetilde{D}_f^{-1/2}\widetilde{W}_{f} = \widetilde{D}^{-1/2}\widetilde{W}$. Similar arguments apply to the Frobenius norm case.
\end{proof}

We need the following technical Lemma.

\begin{lemma}[\cite{el2016graph}]\label{general rv quad form}
	Suppose $Z_{1},\cdots,Z_{n}$ are random vectors in $\mathbb{R}^{q}$, with $Z_{i}=\Sigma_{i}X_{i}$, where $X_{i}$ has $0$ mean and covariance matrix $I_{q\times q}$. $Z_{i}$'s are possibly dependent. 
	Further assume for all convex $1$-Lipschitz function $f$, $\mathbb{P}(|f(X_{i})-m_{f(X_{i})}|>t)\leq 2\exp(-c_{i}t^{2})$, where $m_{f(X_{i})}$ is the median of $f(X_{i})$ and $c_{i}>0$. Let $\{Q_{i}\}$ be $q\times q$ positive definite matrices. Then we have
	\begin{align}
	\sup_{1\leq i\leq n}\abs{\sqrt{Z_{i}'Q_{i}Z_{i}}-\mathbb{E}\left(\sqrt{Z_{i}'Q_{i}Z_{i}}\right)}=\mathcal{O}_{P}\left(\sup_{i}\sqrt{\norm{Q_{i}\Sigma_{i}/c_{i}}}\sqrt{\log n}\right)\,.
	\end{align}
	This implies that if $\sup_{i}\sqrt{\norm{Q_{i}\Sigma_{i}/c_{i}}}\sqrt{\log n}\to 0$, we have
	\begin{align}
	\sup_{1\leq i\leq n}\abs{Z_{i}'Q_{i}Z_{i}-\textup{trace}(\Sigma_{i}Q_i)}=\mathcal{O}_{P}\left(\sup_{i}\sqrt{\norm{Q_{i}\Sigma_{i}/c_{i}}}\sqrt{\log n}\left(\sup_{i}\sqrt{\textup{trace}(\Sigma_{i}Q_{i})}\vee 1\right)\right)\,.
	\end{align}
	
\end{lemma}

The following theorem is the main theorem toward the robustness property of Roseland. It essentially says that the distance between two noisy points is a biased estimate of 
the associated clean distance,
and the difference is well controlled. Lemma \ref{general rv quad form} is essential for this theorem. 

\begin{theorem}\label{noisy d - clean d prob bound general}
	
	Assume the point clouds $\{x_{i}\}_{i=1}^{n}$ and $\{y_{j}\}_{j=1}^{m}\subseteq\mathbb{R}^{q}$ are i.i.d. sampled from the high dimensional model satisfying Assumption \ref{Assumption manifoldH}. 
	Let $\tilde{x_{i}} = x_{i} + \xi_{i}$ and $\tilde{y}_{j} = y_{j} + \eta_{j},$ where the noises $\xi_i$ and $\eta_j$ satisfy Assumption \ref{Assumption Noise} and are independent of $x_i$ and $y_j$. Assume $\sup_{i,j}\sqrt{\|(\Sigma_{i}+\overline{\Sigma}_{j})/c_{ij}\|_2}\sqrt{\log nm}\to 0$.
	As a result, we have
	\begin{align}
	\sup_{i,j}\abs{\tilde{d}_{ij}^{2}-d_{ij}^{2}-\textup{trace}(\Sigma_{i}+\overline{\Sigma}_{j})}=&\,\mathcal{O}_{P}\left( \delta_q
	\right)\,,
	\end{align}
	where $\delta_q$ is defined in \eqref{Assumption:CorollarySI34}
\end{theorem}

\begin{proof}
	By a direct expansion, we have
	\begin{align}
	\nonumber \norm{\tilde{x}_{i}-\tilde{y}_{j}}^{2}-\norm{x_{i}-y_{j}}^{2} = \norm{\xi_{i}-\eta_{j}}^{2}+2\langle x_{i}-y_{j}, \xi_{i}-\eta_{j}\rangle\,.
	\end{align}
	Clearly, since $\xi_i$ and $\eta_j$ are independent, $\xi_{i}-\eta_{j}$ has mean $0$ and covariance $\Sigma_i+\overline{\Sigma}_j$. Below we control $\norm{\xi_{i}-\eta_{j}}^{2}$ and $\langle x_{i}-y_{j}, \xi_{i}-\eta_{j}\rangle$ separately.
	
	We apply Lemma \ref{general rv quad form} with $Q_{i}=I_{q\times q}$ and get
	\begin{align}
	\max_{i,j}\abs{\norm{\xi_{i}-\eta_{j}}^{2}-\textup{trace}(\Sigma_{i}+\overline{\Sigma}_{j})}=\mathcal{O}_{P}\left(\sup_{i,j}\sqrt{\|(\Sigma_{i}+\overline{\Sigma}_{j})/c_{ij}\|_{2}}\sqrt{\log nm}\left(\sup_{i,j}\sqrt{\textup{trace}(\Sigma_{i}+\overline{\Sigma}_{j})}\vee 1\right)\right)\,.\nonumber
	\end{align}
	
	Next, we control $\langle \iota_q(x_{i})-\iota_q(y_{j}), \xi_{i}-\eta_{j}\rangle$ when $n\to \infty$. 
	It is trivial to know that $\langle \iota_q(x_{i})-\iota_q(y_{j}), \xi_{i}-\eta_{j}\rangle$ is sub-Gaussian
	with the variance $\gamma^{2}_{ij}:=(\iota_q(x_{i})-\iota_q(y_{j}))^\top(\Sigma_{i}+\overline{\Sigma}_{j})(\iota_q(x_{i})-\iota_q(y_{j}))$, where 
	$\gamma_{ij}^{2}\leq K^2\|\Sigma_{i}+\overline{\Sigma}_{j}\|_{2}$.
	By the maximal inequality, we have
	\begin{align}
	\sup_{i,j}|\langle x_{i}-y_{j}, \xi_{i}-\eta_{j}\rangle|=\mathcal O_{P}\left(\sqrt{\log nm}K\sup_{i,j}\sqrt{\|\Sigma_{i}+\overline{\Sigma}_{j}\|_{2}}\right)\,.
	\end{align}
	By putting two terms together, we conclude the bound of
	$\sup_{i,j}\abs{\tilde{d}_{ij}^{2}-d_{ij}^{2}-\textup{trace}(\Sigma_{i}+\overline{\Sigma}_{j})}$ by
	\begin{align}
	\nonumber &\sup_{i,j}\abs{\tilde{d}_{ij}^{2}-d_{ij}^{2}-\textup{trace}(\Sigma_{i}+\overline{\Sigma}_{j})}
	\\\nonumber\leq\,&\max_{i,j}\abs{\norm{\xi_{i}-\eta_{j}}^{2}-\textup{trace}(\Sigma_{i}+\overline{\Sigma}_{j})} + 2\sup_{i,j}|\langle x_{i}-y_{j}, \xi_{i}-\eta_{j}\rangle|
	\\\nonumber=\,&\mathcal{O}_{P}\left(\sup_{i,j}\sqrt{\|(\Sigma_{i}+\overline{\Sigma}_{j})/c_{ij}\|_{2}}\sqrt{\log nm}\left(\sup_{i,j}\sqrt{\textup{trace}(\Sigma_{i}+\overline{\Sigma}_{j})}\vee 1\right)\right)+\mathcal O_{P}\left(\sqrt{\log nm} K\sup_{i,j}\sqrt{\|\Sigma_{i}+\overline{\Sigma}_{j}\|_{2}}\right)
	\\\nonumber=\,&\mathcal{O}_{P}\left(\sqrt{\log nm}\sqrt{\sigma_{q}^{2}+\overline{\sigma}_{q}^{2}}\left[   \sup_{i,j}  \sqrt{c^{-1}_{ij}}\left(\sqrt{q(\sigma_{q}^{2}+\overline{\sigma}_{q}^{2})}\vee 1\right)+
	 K\right]\right)\,.
	\end{align}
	
\end{proof}

With the above theorem and Corollary \ref{D-1/2W 2 norm general corollary}, we immediately have the following corollary controlling the operator and Frobenius norms of ${D}^{-1/2}{W}-\widetilde{D}^{-1/2}\widetilde{W}$ coming from the Roseland algorithm.

\begin{corollary}\label{robust of operator norm general} 
	We follow the notation used in Theorem \ref{noisy d - clean d prob bound general}, and assume that $\{x_i\}_{i=1}^n$ and $\{y_j\}_{j=1}^m$  are i.i.d. sampled from the high dimensional model satisfying Assumption \ref{Assumption manifoldH}.	Let $\tilde{x_{i}} = x_{i} + \xi_{i}$ and $\tilde{y}_{j} = y_{j} + \eta_{j},$ where the noises $\xi_i$ and $\eta_j$ satisfy Assumption \ref{Assumption Noise} and are independent of $x_i$ and $y_j$.
	Let $W_{ij}=\exp(-d_{ij}^{2}/\epsilon)$ and $\widetilde{W}_{ij}=\exp(-\tilde{d}_{ij}^{2}/\epsilon)$, where $\epsilon>0$. Also, let $D$ and $\widetilde{D}$ be two $n\times n$ diagonal matrices with entries $D_{ii} = \sum_{j=1}^{n}\sum_{k=1}^{m}W_{ik}W_{jk}$ and $\widetilde{D}_{ii}=\sum_{j=1}^{n}\sum_{k=1}^{m}\widetilde{W}_{ik}\widetilde{W}_{jk}$. 
	Assume $\sup_{i,j}\sqrt{(\sigma^2_{q}+\overline{\sigma}^2_{q})/c_{ij}}\sqrt{\log nm}\to 0$.
	As a result, we have
	\begin{align}
	\norm{D^{-1/2}W-\widetilde{D}^{-1/2}\widetilde{W}}_{2}=\mathcal{O}_{P}\left(
	\frac{\delta_q}{\epsilon^{3d/2+1}}
	\right)
	\quad\mbox{and}\quad
	\norm{D^{-1/2}W-\widetilde{D}^{-1/2}\widetilde{W}}_{F}=\mathcal{O}_{P}\left(\frac{\delta_q}{\epsilon^{3d/2+1}}\right)\,.\nonumber
	\end{align}
\end{corollary}

\begin{proof}
	It is clearly that $D^{-1/2}$ and $\widetilde{D}^{-1/2}$ exist. 
	Let $f=\exp(-\textup{trace}(\Sigma+\overline{\Sigma})/\epsilon)$. Thus, we have
	\begin{align}
	W_{ij}-\frac{\widetilde{W}_{ij}}{f}=&\,\exp(-d_{ij}^{2}/\epsilon) - \frac{\exp(-\tilde{d}_{ij}^{2}/\epsilon)}{f}\nonumber
	\\=&\,\exp(-d_{ij}^{2}/\epsilon)\left(1-\exp\left(-(\tilde{d}_{ij}^{2}-d_{ij}^{2}-\textup{trace}(\Sigma+\overline{\Sigma}))/\epsilon\right)\right)\,.\nonumber
	\end{align}
	By the fact that $\lim_{x\to 0}(1-\exp(-x))/x\to 1$ and the assumptions about $\Sigma$, $\overline{\Sigma}$ and $\sup_{i,j}\norm{x_{i}-y_{j}}$, we know
	\begin{align}
	\sup_{i,j}\abs{W_{ij}-\frac{\widetilde{W}_{ij}}{f}}=\mathcal{O}_{P}\left(\frac{\delta_q}{\epsilon}\right)\,,\label{Proof:Wij tildeWij/f difference}
	\end{align}
	where we use the fact that $\textup{trace}(\Sigma_{i}+\overline{\Sigma}_{j})\leq q(\sigma^2_q+\overline{\sigma}^2_q)$ and $W_{ij}=\exp(-d_{ij}^{2}/\epsilon)\leq 1=:C$.
	Denote 
	{ 
	\begin{eqnarray}\label{z0 def}
	z_0:=\frac{1}{2}\min_{x\in M}p_X(x)p_Y(x)
	\end{eqnarray}
	}
	By the same argument as that of \eqref{denom_LF}, we have with probability higher than $1-n^{-2}$ that 
	{
	\begin{align}\label{control of lower bound of D}
	\inf_{i}D_{ii}/mn\geq \epsilon^d\min_{x\in M}p_X(x)p_Y(x)/2=\epsilon^d z_0
	\end{align}
	}
	when $n$ is sufficiently large.  
	By the assumption that $\delta_q=o_P(1)$ as $q\to \infty$, we know 
	{ $\epsilon^dz_0>2\delta_q/\epsilon+\delta_q^2/\epsilon^2$ } holds in probability when $n\to \infty$. We can thus conclude the result by applying Corollary \ref{D-1/2W 2 norm general corollary}, where we use the fact that $\delta_q\to 0$ when $n\to \infty$.
\end{proof}

Finally, we are ready to prove the main theorem. We need the following result to control the Roseland embedding.

\begin{theorem}(\cite{yu2015useful})\label{singular vec discrepancy bd} Let $A,\hat{A}\in\mathbb{R}^{p\times q}$ have singular values $s_{1}\geq\dots\geq s_{\min(p,q)}$ and $\hat{s}_{1}\geq\dots\geq\hat{s}_{\min(p,q)}$ respectively. Fix $1\leq r\leq l\leq\textup{rank}(A)$ and assume that $\min(s^{2}_{r-1}-s^{2}_{r},s^{2}_{l}-s^{2}_{l+1})>0$, where $s^{2}_{0}\coloneqq\infty$ and $s^{2}_{\textup{rank}(A)+1}\coloneqq -\infty.$ Let $d\coloneqq l-r+1,$ and let $V=(v_{r},\dots,v_{l})\in\mathbb{R}^{q\times d}$
	and $\hat{V}=(\hat{v}_{r},\dots,\hat{v}_{l})\in\mathbb{R}^{q\times d}$ have orthonormal columns satisfying $Av_{j}=s_{j}u_{j}$ and $\hat{A}\hat{v}_{j}=\hat{s}_{j}\hat{u}_{j}$ for $j=r,\dots,l$. Then
	\begin{align}
	\nonumber\norm{\sin\Theta(\hat{V},V)}_{F}\leq \frac{2(2s_{1}+\|\hat{A}-A\|_{\textup{op}})\min(d^{1/2}\|\hat{A}-A\|_{\textup{op}},\|\hat{A}-A\|_{F})}{\min(s^{2}_{r-1}-s^{2}_{r},s^{2}_{l}-s^{2}_{l+1})}\,.
	\end{align}
	Moreover, there exists an orthogonal matrix $\hat{O}\in\mathbb{R}^{d\times d}$ such that
	\begin{align}
	\nonumber\norm{\hat{V}\hat{O}-V}_{F}\leq \frac{2^{3/2}(2s_{1}+\|\hat{A}-A\|_{\textup{op}})\min(d^{1/2}\|\hat{A}-A\|_{\textup{op}},\|\hat{A}-A\|_{F})}{\min(s^{2}_{r-1}-s^{2}_{r},s^{2}_{l}-s^{2}_{l+1})}\,.
	\end{align}
	Identical bounds also hold if $\hat{V}$ and $V$ are replaced with the left singular vectors $\hat{U},U\in\mathbb{R}^{p\times d}$ accordingly.
\end{theorem}

Now, we prove Theorem \ref{robust thm roseland}.

\begin{proof}[Proof of Theorem \ref{robust thm roseland}]
	Let $f=\exp(-\textup{trace}(\Sigma+\overline{\Sigma})/\epsilon)$ and $\widetilde{W}_{f}$ be the $n\times n$ matrix with the $(i,j)$-th entry $\widetilde{W}_{ij}/f$. Note that $\widetilde{D}^{-1/2}\widetilde{W}=\widetilde{D}_{f}^{-1/2}\widetilde{W}_{f}$.
	Recall the Roseland algorithm in Section \ref{roseland alg} and consider $\Phi=D^{-1/2}U\in \mathbb{R}^{n\times q'}$ and  $\widetilde{\Phi}=\widetilde{D}_{f}^{-1/2}\widetilde{U}\in \mathbb{R}^{n\times q'}$, where $U\in \mathbb{R}^{n\times q'}$ and $\widetilde{U}\in \mathbb{R}^{n\times q'}$ are the top $q'$ non-trivial left singular vectors corresponding to the top $q'$ non-trivial  singular values of $D^{-1/2}W$ and $\widetilde{D}^{-1/2}\widetilde{W}$ respectively. We square the top $q'$ non-trivial singular values and put them in diagonal $q'\times q'$ matrices $L$ and $\widetilde{L}$ respectively. Then, for $t>0$ and an orthogonal matrix $O$ that we will choose later, we have
	\begin{align}
	\nonumber\norm{\Phi OL^{t} - \widetilde{\Phi}\widetilde{L}^{t}}_{F}=&\, \norm{D^{-1/2}UOL^{t} -\widetilde{D}_{f}^{-1/2}\widetilde{U}\widetilde{L}^{t}}_{F}
	\\\label{embedding frob norm}\leq&\norm{D^{-1/2}-\widetilde{D}_{f}^{-1/2}}_{F}\norm{UOL^{t}}_{F}+\norm{\widetilde{D}_{f}^{-1/2}}_{F}\norm{UOL^{t}-\widetilde{U}\widetilde{L}^{t}}_{F}\,.
	\end{align}
	Note $\|UOL^{t}\|_{F}=\sqrt{\textup{Tr}((UOL^{t})^{\top}UOL^{t})}=\sqrt{\textup{Tr}(L^{t}O^{\top}U^{\top}UOL^{t})}=\mathcal{O}\left(\sqrt{q's_{2}^{4t}}\right)$, where $0<s_{2}<1$. 
	By \eqref{Proof:Wij tildeWij/f difference}, we know
	\begin{align}
	\sup_{i,j}\abs{W_{ij}-(\widetilde{W}_{f})_{ij}}=\mathcal{O}_{P}\left(\frac{\delta_q}{\epsilon}\right)\,.\nonumber%\mathcal{O}_{P}\left(\frac{1}{\epsilon}\left(\sup_{i,j}\sqrt{\log nm}\sqrt{\frac{\sigma^{2}_{q}+\overline{\sigma}^{2}_{q}}{c_{ij}}}\left(\sqrt{q}\sqrt{\sigma_{q}^{2}+\overline{\sigma}_{q}^{2}}\vee 1\right)+K\sqrt{\log nm}\sqrt{\sigma_{q}^{2}+\overline{\sigma}_{q}^{2}}\right)\right)\,
	\end{align}
	Follow the same lines in the proof of Corollary \ref{robust of operator norm general}, 
	we know with probability greater than $1-n^{-2}$, {$\inf_iD_{ii}\geq mn\epsilon^dz_0$, where $z_0$ is defined in \eqref{z0 def},}
	and hence  {$\inf_i(\widetilde{D}_f)_{ii}\geq mn(\epsilon^dz_0-2\delta_q/\epsilon)$} in probability when $n\to \infty$. As a result, we have
	\begin{align} \label{D_ii tilde D_ii diff}
	\abs{D_{ii}^{-1/2}-(\widetilde{D}_{f})_{ii}^{-1/2}}=\mathcal{O}_{P}\left(\frac{\delta_q}{\epsilon^{3d/2+1}  \sqrt{nm}}\right)\,. 	\nonumber
	\end{align}
	Hence, for the first term in \eqref{embedding frob norm}, we have
	\begin{align}
	\norm{D^{-1/2}-\widetilde{D}_{f}^{-1/2}}_{F}\norm{UOL^{t}}_{F}
	=\mathcal{O}_{P}\left(\frac{\delta_q\sqrt{q's_{2}^{4t}}}{\epsilon^{3d/2+1}  \sqrt{m}}\right)\,.
	\end{align}
	Next, we control the second term in \eqref{embedding frob norm}. By \eqref{D_ii tilde D_ii diff}, we have
	\begin{align}
	(\widetilde{D}_{f})_{ii}^{-1/2}=&\,\mathcal{O}_{P}\left( D_{ii}^{-1/2}+\frac{\delta_q}{\epsilon^{3d/2+1}  \sqrt{nm}}\right)
	=\mathcal{O}_{P}\left( \frac{1}{\epsilon^{d/2} \sqrt{nm}}\right) \nonumber \,.
	\end{align}
	Hence, we have
	\begin{align}
	\norm{\widetilde{D}_{f}^{-1/2}}_{F}=\mathcal{O}_{P}\left( \frac{1}{\epsilon^{d/2} \sqrt{m}}\right).
	\end{align}
	For the last piece,
	\begin{align}
	\nonumber\norm{UOL^{t}-\widetilde{U}\widetilde{L}^{t}}_{F}\leq\norm{UO-\widetilde{U}}_{F}\norm{L^{t}}_{F}+\norm{\widetilde{U}}_{F}\norm{L^{t}-\widetilde{L}^{t}}_{F}\,,
	\end{align}
	where $\norm{L^{t}}_{F}=\mathcal{O}\left(\sqrt{q's_{2}^{4t}}\right)$ and $\|\widetilde{U}\|_{F}=\mathcal{O}(\sqrt{q'})$. 
	Next, we bound $\norm{L^{t}-\widetilde{L}^{t}}_{F}$. Recall that $L$ (resp. $\widetilde{L}$) is a diagonal matrice with positive eigenvalues $1>s^{2}_{2}\geq\cdots\geq s^{2}_{q'+1}$ (resp. $1>\tilde{s}^{2}_{2}\geq\cdots\geq \tilde{s}^{2}_{q'+1}$) of the positive definite matrice $D^{-1/2}WW^{\top}D^{-1/2}$ (resp. $\widetilde{D}_{f}^{-1/2}\widetilde{W}_{f}\widetilde{W}^{\top}_{f}\widetilde{D}_{f}^{-1/2}$). By a direct expansion, we have 
	\begin{align*}
	&\norm{D^{-1/2}WW^{\top}D^{-1/2} - \widetilde{D}_{f}^{-1/2}\widetilde{W}_{f}\widetilde{W}^{\top}_{f}\widetilde{D}_{f}^{-1/2}}_{2}\\\leq&\,\norm{D^{-1/2}W}_{2}\norm{W^{\top}D^{-1/2}-\widetilde{W}^{\top}_{f}\widetilde{D}_{f}^{-1/2}}_{2}+\norm{D^{-1/2}W-\widetilde{D}_{f}^{-1/2}\widetilde{W}_{f}}_{2}\norm{\widetilde{W}^{\top}_{f}\widetilde{D}_{f}^{-1/2}}_{2}\,.
	\end{align*}
	Since $\norm{D^{-1/2}W}_{2}\leq 1$ and $\norm{\widetilde{W}^{\top}_{f}\widetilde{D}_{f}^{-1/2}}_{2}\leq 1$, we have
	\begin{align*}
	&\norm{D^{-1/2}WW^{\top}D^{-1/2} - \widetilde{D}_{f}^{-1/2}\widetilde{W}_{f}\widetilde{W}^{\top}_{f}\widetilde{D}_{f}^{-1/2}}_{2}\leq 2\|D^{-1/2}W-\widetilde{D}^{-1/2}_{f}\widetilde{W}_{f}\|_{2} 
	=\mathcal{O}_{P}\left(\frac{\delta_q}{\epsilon^{d/2+1}}\right)
	\end{align*}
	by Corollary \ref{D-1/2W 2 norm general corollary}. Hence, the Weyl's inequality \cite{weyl1912asymptotische} tells us
	\begin{align}
	\abs{s_{i}^{2}-\tilde{s}_{i}^{2}} = \mathcal{O}_{P}\left(\frac{\delta_q}{\epsilon^{3d/2+1}}\right)
	\end{align}
	for all $i=2,\cdots,q'+1$. 
	Now, by a slight modification of the proof of Theorem \ref{spec_cong}, when $\epsilon$ is sufficiently small and $n$ is sufficiently large, the first $q'$ singular values are away from zero with probability higher than $1-n^{-2}$. 
	So, for a fixed $t>0$, by the binomial approximation, when $n\to \infty$, we have
	\begin{align}
	\abs{s_{i}^{2t}-\tilde{s}_{i}^{2t}} = \mathcal{O}_{P}\left(ts_i^{2t-2}\frac{\delta_q}{\epsilon^{3d/2+1}}\right) 
	\end{align}
	for all $i=2,\cdots,q'+1$. Thus, 
	\begin{align}
	\norm{L^{t}-\widetilde{L}^{t}}_{F}=\mathcal{O}_{P}\left(\sqrt{q'}ts_2^{2t-2}\frac{\delta_q}{\epsilon^{3d/2+1}}\right)\,.
	\end{align}
	Finally, for $\|UO-\widetilde{U}\|_{F}$, we apply Theorem \ref{singular vec discrepancy bd} and choose the orthogonal matrix $O$ so that
	\begin{align}
	\nonumber&\norm{UO-\widetilde{U}}_{F}\\
	\leq&\,\frac{2^{3/2}(2s_{1}+\|D^{-1/2}W-\widetilde{D}^{-1/2}_{f}\widetilde{W}_{f}\|_{2})\min(\sqrt{q'}\|D^{-1/2}W-\widetilde{D}^{-1/2}_{f}\widetilde{W}_{f}\|_{2},\,\|D^{-1/2}W-\widetilde{D}^{-1/2}_{f}\widetilde{W}_{f}\|_{F})}{\min(s^{2}_{1}-s^{2}_{2},s^{2}_{q'+1}-s^{2}_{q'+2})}\,,\nonumber
	\end{align}
	where $s_{1}=1$ and by assumption and Proposition \ref{Proposition: Tepsn conv to Teps with rate},  $\min(s^{2}_{1}-s^{2}_{2},s^{2}_{q'+1}-s^{2}_{q'+2})\asymp 1$ when $n$ is sufficiently large.  By Proposition \ref{robust of operator norm general},
	we have
	\begin{align}
	\norm{UO-\widetilde{U}}_{F}
	=\mathcal{O}_{P}\left(\frac{\delta_q}{\epsilon^{3d/2+1}}\right)\,.\label{Bound of UO-tildeU}
	\end{align}
	Thus,
	\begin{align}
	\nonumber&\norm{UOL^{t}-\widetilde{U}\widetilde{L}^{t}}_{F} = \mathcal{O}_{P}\left(\frac{\delta_q}{\epsilon^{3d/2+1}}\left(q'ts_2^{2t-2}+\sqrt{q'}s_2^{2t}\right) \right).
	\end{align}
	By putting all together and a simplification, we conclude that
	\begin{align}
	\nonumber&\norm{\Phi OL^{t} - \widetilde{\Phi}\widetilde{L}^{t}}_{F} =\mathcal{O}_{P}\left(\frac{\delta_q}{\sqrt{m}} \frac{q'ts_2^{2t-2}+\sqrt{q'}s_2^{2t}}{\epsilon^{2d+1} }\right)\,.\nonumber
	\end{align}
	{
	Finally, to control $\norm{\Phi L^{t}O - \widetilde{\Phi}\widetilde{L}^{t}}_{F} $, note that we have $\norm{\Phi L^{t}O - \widetilde{\Phi}\widetilde{L}^{t}}_{F}\leq  \norm{\Phi L^{t}O - \Phi OL^{t}}_{F} +
\norm{\Phi OL^{t} - \widetilde{\Phi}\widetilde{L}^{t}}_{F}
$. Since $\norm{\Phi OL^{t} - \widetilde{\Phi}\widetilde{L}^{t}}_{F}$ has been controlled above, it is sufficient to control $\norm{\Phi L^{t}O - \Phi OL^{t}}_{F}$. By definition of $\Phi$ and \eqref{control of lower bound of D}, we have
	\[
	\norm{\Phi L^{t}O - \Phi OL^{t}}_{F}\leq \frac{1}{\sqrt{m}\epsilon^{d/2}}\|U(OL^t-L^tO)\|_F\leq \frac{\sqrt{q'}}{\sqrt{m}\epsilon^{d/2}}\|OL^t-L^tO\|_F
	\]
	To control $\|OL^t-L^tO\|_F$, note that %$O=\Phi\Psi^\top$, where $\Phi$ and $\Psi$ come from the SVD of ${U}^\top \widetilde U=\Phi L\Psi^\top$. 
	by Theorem \ref{singular vec discrepancy bd} and the assumption that $\lambda_i$ is simple,  with one eigenvector at one time and the bound of the error shown in \eqref{Bound of UO-tildeU}, we know that $|e_i^\top{U}^\top \widetilde Ue_j-\delta_{ij}|=\mathcal{O}_{P}\left(\frac{\delta_q}{\epsilon^{3d/2+1}}\right)$, where $e_j$ is the unit vector with the $j$-th entry $1$ and $\delta_{ij}$ is the Kronecker delta. Thus, we have $U^\top \widetilde U=I+E\in \mathbb{R}^{q'\times q'}$, where each entry of $E$ is controlled by $\mathcal{O}_{P}\left(\frac{\delta_q}{\epsilon^{3d/2+1}}\right)$. On the other hand, since $\norm{{U}{O}- \widetilde U}_{F}=\norm{{O}- U^\top\widetilde U}_{F}=\mathcal{O}_{P}\left(\frac{\delta_q}{\epsilon^{3d/2+1}}\right) $ by \eqref{Bound of UO-tildeU}, we have $O=I+E'$, where each entry of $E'$ is controlled by $\mathcal{O}_{P}\left(\frac{\delta_q}{\epsilon^{3d/2+1}}\right)$. Since $\|OL^t-L^tO\|_F=\|OL^tO^\top-L^t\|_F=\|{E'} L^t+L^t{E'}^\top+E' L^t{E'}^\top\|_F=\mathcal{O}_{P}\left(\frac{{q'}\delta_q}{\epsilon^{3d/2+1}}\right)$, we conclude that
	\[
	\norm{\Phi L^{t}O - \Phi OL^{t}}_{F}\leq \frac{\sqrt{q'}}{\sqrt{m}\epsilon^{d/2}}\|OL^t-L^tO\|_F=\mathcal{O}\left(\frac{{q'}^{3/2}\delta_q}{\sqrt{m}\epsilon^{2d+1}} \right)\,,
	\]
	and hence the proof.
	}
	
\end{proof}

{
\section{Proof of \eqref{kernel ref convergence finite variance}} \label{Roseland kernel discrepancy proof}
Fix some $x_i,x_j\in M^d$, and define a random variable
	\begin{eqnarray}\label{roseland kernel rv}
	F\coloneqq\epsilon^{-d/2}K_{\epsilon}(x_{i},Y)K_{\epsilon}(Y,x_j)\,.
	\end{eqnarray}
Denote by $F_{k}$ one realization of $F$ when the realization of the random variable $Y$ is   $y_{k}$; in other words, $F_{k}=\epsilon^{-d/2}K_{\epsilon}(x_{i},y_{k})K_{\epsilon}(y_{k},x_{j})$. To get the convergence rate of 
\begin{eqnarray}\label{top_rate}
	\frac{1}{m}\sum_{k=1}^{m}F_{k}\longrightarrow\mathbb{E}(F)\,,
\end{eqnarray}
we follow the same lines in the variance theorem proof in Appendix \ref{biasproof}. We have 
\begin{align}
\nonumber\mathbb{E}(F)\, &= \epsilon^{-d/2}\int_{M}K_{\epsilon}(x_{i},y)K_{\epsilon}(y,x_j)p_{Y}(y)\,dV(y)= e^{-\|x_i-x_j\|^2/2\epsilon}c(x_i,x_j)%K_{\epsilon}(x_i,x_j)p_{Y}(x_{i}) + \mathcal{O}(\epsilon)
\end{align}
by \eqref{expansion landmark kernel2} and similarly
\begin{align}
\nonumber\mathbb{E}(F^{2}) \,&= \epsilon^{-d}\int_{M}K_{\epsilon}^{2}(x_{i},y)K_{\epsilon}^{2}(y,x_j)p_{Y}(y)\,dV(y)= \epsilon^{-d/2}e^{-\|x_i-x_j\|^2/\epsilon}\tilde c(x_i,x_j) \,,%\mu_{2,0}^{(0)}K^2_{\epsilon}(x_i,x_j)p_{Y}(x_{i}) + \mathcal{O}(\epsilon^{1-d/2})\,.
\end{align}
where $\tilde{c}$ is smooth, depends on $x_i$ and $x_j$ and is of order $1$.
Therefore, when $\epsilon>0$ is sufficiently small, $\mathbb{E}(F^{2})\asymp \epsilon^{-d/2}$.  
Since $[\mathbb{E}(F)]^2 =\mathcal{O}(1)$, we have $\textup{Var}(F)\asymp \epsilon^{-d/2}$.
By the Bernstein inequality, for all $t>0$, we have
\begin{eqnarray}
\nonumber\mathbb{P}\Big(\frac{1}{m}\sum_k F_k-\mathbb{E}(F)> t\Big)\leq\textup{exp}\left(-\frac{mt^2}{2\epsilon^{-d/2}+\frac{2}{3}\epsilon^{-d/2}t}\right)\,.
\end{eqnarray}
Recall we ask $\frac{t}{\mathbb{E}(F)}\to 0$ when $m\to \infty$. Since $\mathbb{E}(F)$ is of order $1$, with such $t$, the exponent becomes
\begin{eqnarray}
	\nonumber\frac{mt^2}{2\epsilon^{-d/2}+\frac{2}{3}\epsilon^{-d/2}t}\geq \frac{mt^2}{3\epsilon^{-d/2}}\,.
\end{eqnarray}
Then if we choose $m$ such that $\frac{mt^2}{3\epsilon^{-d/2}} = 2\log (nm)$, we have
\begin{eqnarray}
\nonumber t=\sqrt{6(1+\beta)} \frac{\sqrt{\log n}}{n^{\beta/2}\epsilon^{d/4}}\,,
\end{eqnarray}
which goes to $0$ by Assumption \eqref{variancethm}.
As a result
\begin{eqnarray}
\nonumber\mathbb{P}\Big(\frac{1}{m}\sum_k F_k-\mathbb{E}(F)> t\Big)\leq\textup{exp}(-2\log (nm)) = \frac{1}{n^2m^2}\,.
\end{eqnarray}
To have the above bound holds for all pairs of $i,j=1,2,\cdots, n$, we take the simple union bound and get
\begin{eqnarray}
\nonumber\mathbb{P}\Big(\frac{1}{m}\sum_k F_k-\mathbb{E}(F)> t; \forall i, j \Big)\leq  \frac{n^2}{n^2m^2} = \frac{1}{m^2}\,.
\end{eqnarray}
}

}

\end{document}